\documentclass[11pt,a4paper,notitlepage,twoside,appendixprefix=true,openright]{scrreprt}
\addtokomafont{disposition}{\rmfamily}

\usepackage[left=3cm,right=2cm,top=3cm,bottom=3cm]{geometry}
\usepackage[outline]{contour}

\usepackage{xspace}
\usepackage[ngerman,english]{babel}

\usepackage[headsepline=true]{scrlayer-scrpage} 
\clearpairofpagestyles
\lehead{\leftmark}
\rohead{\rightmark}
\automark[chapter]{chapter} 
\automark*[section]{} 
\usepackage{natbib}
\setcitestyle{authoryear,open={(},close={)}}

\usepackage{mathtools}
\usepackage{verbatim}   
\usepackage{ifthen}
\usepackage{bm}
\usepackage{lipsum}
\usepackage{amsmath, amssymb, enumerate, dsfont, comment, amsthm, thmtools}
\usepackage[utf8]{inputenc}
\usepackage{csquotes}
\usepackage{tikz}
\usepackage{float}
\usepackage{upgreek}

\usetikzlibrary{babel}
\usetikzlibrary{cd}

\newtheorem{lem}{Lemma}[section]
\newtheorem{lemma}[lem]{Lemma}
\newtheorem*{lem*}{lem}
\newtheorem{satz}[lem]{Theorem}
\newtheorem{theorem}[lem]{Theorem}
\newtheorem{prop}[lem]{Proposition}
\newtheorem{proposition}[lem]{Proposition}
\newtheorem*{prop*}{Proposition}

\newtheorem*{thm*}{Theorem}
\newtheorem{corollary}[lem]{Corollary}
\newtheorem{cor}[lem]{Corollary}
\newtheorem*{cor*}{Corollary}
\theoremstyle{definition}
\newtheorem{defi}[lem]{Definition}
\newtheorem*{defi*}{Definition}
\theoremstyle{remark}
\newtheorem*{example*}{Example}
\newtheorem{example}[lem]{Example}

\newtheorem{rem}[lem]{Remark}
\newtheorem{remark}[lem]{Remark}
\newtheorem*{rem*}{Remark}

\usepackage[hypertexnames=false]{hyperref}
\usepackage[nameinlink]{cleveref}
\usepackage{bookmark}

\newcounter{assum}
\newenvironment{assum}{\par\noindent\refstepcounter{assum}\textbf{Assumption (A\arabic{assum})} \itshape}{\par}

\newcommand{\assref}[1]{(A\ref{#1})}

\newcommand{\blue}{\textcolor{blue}}

\newcommand{\todoq}[1]{\blue{#1}}

\makeatletter
\@ifundefined{showdtodos}{
	\excludecomment{todo}
}{
	\newenvironment{todo}{\par\textcolor{blue}}{\par}
}

\definecolor{darkgreen}{rgb}{0.0196, 0.3412, 0.1647}

\makeatletter
\@ifundefined{showdocu}{
	\excludecomment{docu}
}{
	\newenvironment{docu}{\par\color{darkgreen}}{\par}
}
\makeatother

\newcommand*{\Cov}{\operatorname{Cov}}

\newcommand{\Ioif}{\Leftrightarrow}

\newcommand{\tiid}{i.i.d.}

\newcommand{\outSq}[1]{#1^{\otimes 2}}

\newcommand{\nrmps}[2]{{\nrm{#1^*}{#2}}}

\newcommand{\innp}[2]{\langle #1, #2\rangle}
\DeclareMathOperator{\Kurt}{Kurt}
\DeclareMathOperator*{\esssup}{esssup}

\newcommand{\set}[1]{\{#1\}}

\newcommand{\E}{\mathbb E}

\newcommand{\floor}[1]{\left\lfloor #1 \right\rfloor}
\newcommand{\ceil}[1]{\left\lceil #1 \right\rceil}

\newcommand{\der}{\partial}

\newcommand*{\cA}{\mathcal{A}}
\newcommand*{\cB}{\mathcal{B}}
\newcommand*{\cC}{\mathcal{C}}
\newcommand*{\cD}{\mathcal{D}}
\newcommand*{\cE}{\mathcal{E}}
\newcommand*{\cF}{\mathcal{F}}
\newcommand*{\cG}{\mathcal{G}}
\newcommand*{\cH}{\mathcal{H}}
\newcommand*{\cI}{\mathcal{I}}

\newcommand*{\cL}{\mathcal{L}}

\newcommand*{\cN}{\mathcal{N}}
\newcommand*{\cO}{\mathcal{O}}
\newcommand*{\cP}{\mathcal{P}}

\newcommand*{\cR}{\mathcal{R}}
\newcommand*{\cS}{\mathcal{S}}

\newcommand*{\cU}{\mathcal{U}}
\newcommand*{\cV}{\mathcal{V}}
\newcommand*{\cW}{\mathcal{W}}
\newcommand*{\cX}{\mathcal{X}}

\newcommand*{\cZ}{\mathcal{Z}}

\newcommand*{\N}{\mathbb{N}}
\newcommand*{\R}{\mathbb{R}}

\renewcommand*{\P}{\mathbb{P}}

\newcommand*{\IZ}{\mathbb{Z}}

\newcommand*{\om}{\omega}
\newcommand*{\Om}{\Omega}
\newcommand*{\si}{\sigma}
\newcommand*{\Si}{\Sigma}
\newcommand*{\al}{\alpha}
\newcommand*{\be}{\beta}
\newcommand*{\thet}{\theta}
\newcommand*{\Thet}{\Theta}
\newcommand*{\delt}{\delta}
\newcommand*{\Delt}{\Delta}
\newcommand*{\ph}{\varphi}
\newcommand*{\Ph}{\Phi}
\newcommand*{\la}{\lambda}
\newcommand*{\ka}{\kappa}
\newcommand*{\ga}{\gamma}
\newcommand*{\Ga}{\Gamma}

\newcommand*{\ep}{\varepsilon}

\newcommand{\Lip}{\operatorname{Lip}}

\newcommand*{\blnk}{\cdot}

\newcommand{\tIto}{Itô}

\newcommand{\tFre}{Fréchet}
\newcommand{\tCadl}{càdlàg}

\newcommand{\tIe}{that is}

\newcommand{\tCf}{see}

\newcommand*{\fa}{\forall}

\newcommand*{\intr}{\operatorname{int}}
\newcommand{\supp}{\operatorname{supp}}
\newcommand*{\nrm}[2]{\|#1\|_{#2}}

\newcommand*{\Var}{\operatorname{Var}}
\newcommand*{\prttn}[2]{\cS^{#1}_{#2}}
\newcommand*{\amin}{\operatorname{argmin}}
\newcommand*{\tr}{\operatorname{tr}}

\newcommand*{\id}[1]{\operatorname{id}_{#1}}

\newcommand*{\tSDE}{stochastic differential equation}
\newcommand*{\tYDE}{Young differential equation}
\newcommand*{\tODE}{ordinary differential equation}
\newcommand*{\tSME}{stochastic modified equation}
\newcommand*{\tSGD}{stochastic gradient descent}

\newcommand*{\mat}[1]{\begin{pmatrix}#1\end{pmatrix}}
\newcommand*{\idK}{\mathds 1}
\newcommand*{\pr}{\operatorname{pr}}
\newcommand*{\Bin}{\operatorname{Bin}}

\newcommand*{\unfovr}{\text{ uniformly over }}
\newcommand*{\unfin}{\text{ uniformly in }}

\newcommand{\Lfin}{L^{\infty-}}

\newcommand*{\LE}{\operatorname{LE}}
\newcommand{\sgn}{\operatorname{sgn}}
\newcommand{\XCC}{X^{\operatorname{CC}}}
\newcommand{\XCCh}[1]{X^{\operatorname{CC},#1}}
\newcommand{\XNCC}{X^{\operatorname{NCC}}}
\newcommand{\XNCCh}[1]{X^{\operatorname{NCC},#1}}
\newcommand{\BEq}{B^{\operatorname{Eq}}}
\newcommand{\BGF}{B^{\operatorname{GF}}}

\newcommand{\bx}{\bm{x}}
\newcommand{\by}{\bm{y}}
\newcommand{\bz}{\bm{z}}
\newcommand{\bep}{\bm{\varepsilon}}

\newcommand{\lrates}{T/\N \cap (0,1)}

\newcommand{\discX}{\tilde X}

\newcommand{\Leb}{\operatorname{Leb}}
\newcommand{\Perm}{\operatorname{Perm}}
\newcommand{\Sym}[1]{\cS_{#1}}
\newcommand{\dSym}[2]{\cS^{#1}_{#2}}
\newcommand{\Unif}[1]{\operatorname{Unif}(#1)}
\newcommand{\pvar}[2]{\nrm{#1}{#2\operatorname{-var}}}

\newcommand{\specnrm}[1]{\nrm{#1}{\operatorname{op}}}

\newcommand{\hoelH}[2]{\dot \cC^{#1,#2}}
\newcommand{\hoelHZ}[1]{\dot \cC^{#1}}
\newcommand{\wHoelH}[3]{\dot\cC^{#1,#2}_{#3}}
\newcommand{\wHoel}[3]{\cC^{#1,#2}_{#3}}
\newcommand{\wHoelNrm}[4]{\nrm{#1}{\wHoel#2#3#4}}
\newcommand{\wHoelHZ}[2]{\dot \cC^{#1}_{#2}}
\newcommand{\wHoelZ}[2]{\cC^{0,#1}_{#2}}

\newcommand{\wC}[2]{\cC^{#1}_{#2}}

\newcommand{\wCZ}[1]{\cC_{#1}}

\newcommand{\adptd}{\operatorname{ad}}

\newcommand{\Lpinfadt}[2]{L^{#1,\adptd}_{*{#2}}}

\newcommand{\Lpinft}[2]{L^{#1}_{*{#2}}}
\newcommand{\Lpinf}[1]{L^{#1}_*}
\newcommand{\Lpinfad}[1]{L^{#1,\adptd}_{*}}
\newcommand{\Lfininf}{L^{\infty-}_*}
\newcommand{\Lfininfad}{L^{\infty-,\adptd}_*}

\newcommand{\dC}[1]{\cC^{#1}}

\newcommand{\CHoelLoc}[1]{\cC^{0,#1}_{\operatorname{loc}}}

\newcommand{\LpLoc}[1]{L^{#1}_{\operatorname{loc}}}

\newcommand{\modu}{\operatorname{mod}}
\newcommand{\frk}[1]{\{#1\}}

\newcommand{\transp}{\intercal}

\newcommand{\condNr}[1]{\frac{\la_{\max}(#1)}{\la_{\min}(#1)}}

\newcommand*\pFqskip{8mu} 
\catcode`,\active
\newcommand*\pFq{\begingroup
	\catcode`\,\active
	\def ,{\mskip\pFqskip\relax}%
	\dopFq
}
\catcode`\,12
\def\dopFq#1#2#3#4#5{%
	{}_{#1}F_{#2}\biggl(\genfrac..{0pt}{}{#3}{#4};#5\biggr)%
	\endgroup
}

\allowdisplaybreaks

\title{Dissertation}

\author{Stefan Perko}

\begin{document}
\pagenumbering{roman}
\pagestyle{empty}

\vspace*{0.3cm}
\begin{figure}[h]
	\begin{center}
		\includegraphics[scale=0.4]{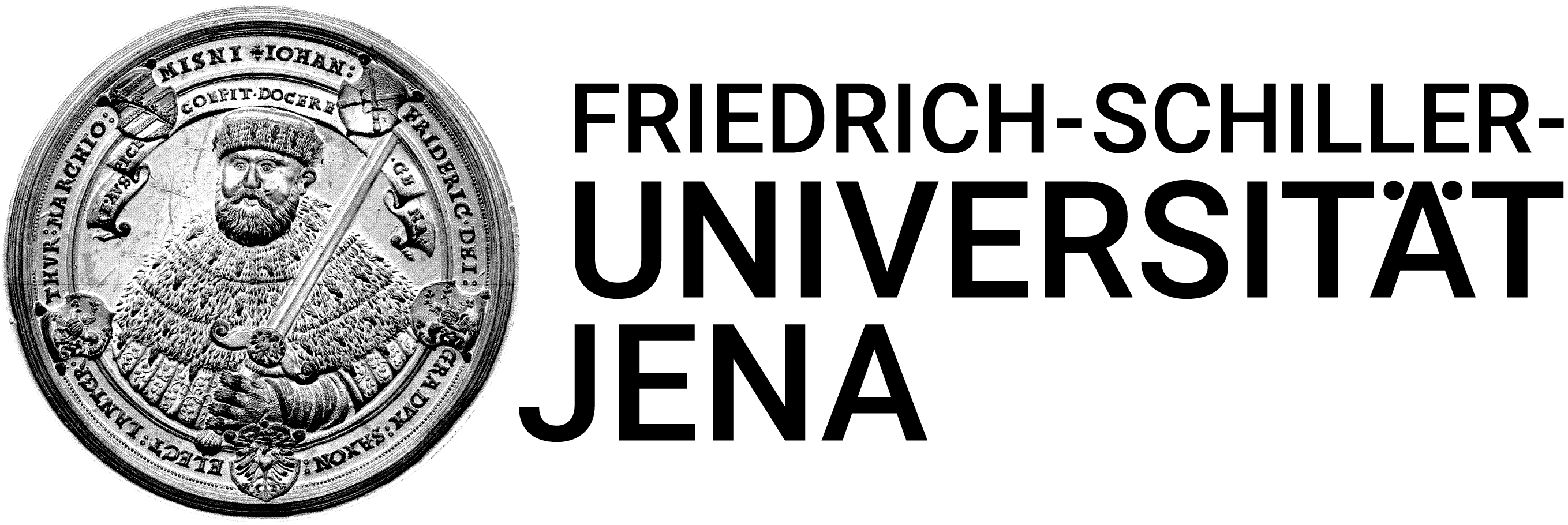}
	\end{center}
\end{figure}

\vspace{2cm}
\centerline{ {\Huge \textbf{Modified Equations}} }
\centerline{ { \LARGE \textbf{for Stochastic Optimization}} }



\vspace{2.8cm}

\centerline{ {\huge \textsc{ Dissertation}}}
\vspace{0.8cm}

\centerline{zur Erlangung des akademischen Grades}
\vspace{0.2cm}

\centerline{doctor rerum naturalium (Dr.\ rer.\ nat.)}
\vspace{2.8cm}

\centerline{vorgelegt dem Rat der}
\vspace{0.2cm}

\centerline{Fakult\"at f\"ur Mathematik und Informatik}
\vspace{0.2cm}

\centerline{der Friedrich-Schiller-Universit\"at Jena}
\vspace{2.8cm}

\centerline{von M.Sc.~Stefan Perko}
\vspace{0.2cm}

\centerline{geboren am 08.06.1996 in Erfurt}
\newpage

\text{}
\vfill
\begin{otherlanguage}{ngerman}
	
	{\noindent {\textbf{ Gutachter}}
		\vspace{0.4cm}
		
		\noindent
		1.  \mbox{  }
		Prof.~Dr.~Stefan Ankirchner, Friedrich-Schiller-Universität Jena, Germany
		\vspace{0.4cm}
		
		\noindent 2. \mbox{  }
		Prof.~Dr.~Steffen Dereich, Universität Münster, Germany
		\vspace{0.4cm}
		
		\noindent 3. \mbox{  } 
		Prof.~Dr.~Stefan Geiss, University of Jyväskylä, Finland
		\vspace{0.4 cm}
		
		\noindent {\textbf{ Tag der öffentlichen Verteidigung:}} \mbox{ TBA }
	}
\end{otherlanguage}

\chapter*{\centering Acknowledgments}

I extend my sincere gratitude to my supervisor Stefan Ankirchner for his guidance, fruitful discussions and support throughout the years working on this thesis. His willingness to explore new topics with me and pragmatic advice have been of substantial help.

\vspace*{.1cm}
I am grateful to Hannah and Stefan Geiss for their hospitality during my research stay in Jyväskylä.
Our stimulating discussions and hikes through snowy mountains and swamps has helped  open many doors.

\vspace*{.1cm}
I want to thank all the referees for taking the time to review my thesis.

\vspace*{.1cm}
Many thanks to my colleagues Dennis, Ilya, Jens, Julian, Hugalf, Marian, Max, Nicole, Patricia, Robert, Simon, Stefan, Stefan, Sooppawat, and Verena for the relaxed and motivating atmosphere, the joint lunches, and discussions over coffee.

\vspace*{.1cm}
The completion of my dissertation would have certainly been possible without my dearest friends Alexia, Denys, Maxi and Sam - but it would have been unimaginably poorer. Your unwavering support has carried me through even the hardest times and our continued shared interest in philosophy, science, technology and puzzling has inspired me time and time again.


\begin{otherlanguage}{ngerman}
	\chapter*{\centering Zusammenfassung}
In dieser Dissertation erweitern wir die neuartige Theorie der stochastischen modifizierten Gleichungen für stochastische Gradientenverfahren. Diese Theorie verbindet Ideen aus der Numerik von Differentialgleichungen mit Methoden der stochastischen Analysis und der Optimierung, um Fragestellungen aus dem maschinellen Lernen zu untersuchen.

Die Arbeit besteht aus zwei Hauptteilen. Im ersten Teil (Kapitel \ref{chap:mdfdeq} - \ref{chap:compare}) untersuchen wir modifizierte Gleichungen für stochastische Einschrittverfahren, darunter stochastischer Gradientenabstieg (SGD) ohne Zurücklegen bei unendlich vielen Daten, und SGD mit Zurücklegen. In Kapitel \ref{chap:mdfdeq} untersuchen wir zeitinhomogene stochastische Differentialgleichungen (SDEs), die von einer Brownschen Bewegung getrieben werden und deren Drift- und Diffusionskoeffizienten eine Entwicklung in der Schrittweite besitzen, wobei der führende Term im Drift durch das Einschrittverfahren bestimmt ist (und höhere Terme frei wählbar sind).
Unter bestimmten Regularitätseigenschaften beweisen wir für diese SDEs eine schwache Approximationseigenschaft erster Ordnung und wir bestimmen ihre linearen Fehlerterme explizit. Darüber hinaus beweisen wir eine schwache Approximationseigenschaft zweiter Ordnung für eine spezifische Familie von SDEs. In Kapitel \ref{chap:smeopt} instanziieren wir unsere Resultate für SGD. Wir arbeiten das Beispiel der linearen Regression vollständig aus und nutzen es in Kapitel \ref{chap:compare}, um die linearen Fehlerterme des Gradientenflusses mit zwei häufig verwendeten stochastischen modifizierten Gleichungen erster Ordnung für SGD zu vergleichen.

Im zweiten Teil (Kapitel \ref{chap:towardsSGDoSME} und \ref{chap:weakshuffle}) führen wir eine neuartige Diffusionsapproximation für SGD ohne Zurücklegen (SGDo) bei endlich vielen Daten ein und untersuchen diese. In Kapitel \ref{chap:towardsSGDoSME} motivieren und definieren wir den Begriff einer epochenweise zusammengesetzten Brownschen Bewegung. Wir argumentieren, dass Young-Differentialgleichungen (YDEs), die von solchen Prozessen getrieben werden, als kontinuierliche Modelle für SGDo dienen - und zwar für jede Mischungsstrategie, deren induzierte Permutationen im Grenzwert großer Stichproben gegen ein deterministisches Permuton konvergieren. Ferner beweisen wir fast sichere Konvergenz dieser YDEs im streng konvexen Fall. Zudem bestimmen wir eine asymptotische obere Schranke für die Konvergenzrate, die mindestens so scharf ist wie bisher bekannte Resultate für SGDo.

In Kapitel \ref{chap:weakshuffle} untersuchen wir Grenzwerte von Familien skalierter zufälliger Irrfahrten, die bis auf eine (möglicherweise zufällige) Permutation dieselben Inkremente haben. Wir zeigen schwache Konvergenz unter der Annahme, dass die Folge dieser Permutationen gegen ein deterministisches (höherdimensionales) Permuton konvergiert. Diese Permuton bestimmt die Kovarianzfunktion des Gaussprozesses im Limes.
Umgekehrt zeigen wir, dass jeder Gaussprozess mit einer Kovarianzfunktion, die durch ein solches Permuton bestimmt wird, als schwacher Grenzwert skalierter zufälliger Irrfahrten mit gemeinsamen Inkrementen auftritt. Schließlich wenden wir unsere Konvergenztheorie an, um zu zeigen, dass epochenweise zusammengesetzte Brownsche Bewegungen als Grenzwerte skalierter zufälliger Irrfahrten mit endlich vielen verschiedenen Inkrementen entstehen.
\end{otherlanguage}


\chapter*{\centering Abstract}	
In this thesis, we extend the recently introduced theory of stochastic modified equations for stochastic gradient optimization algorithms. This theory combines ideas first developed in the field of numerics of differential equations with stochastic calculus and optimization in order to study problems from machine learning.

This thesis consists of two main parts. In the first part (Chapters \ref{chap:mdfdeq} - \ref{chap:compare}) we study modified equations for stochastic one-step methods, including one-pass SGD and SGD with replacement. In Chapter \ref{chap:mdfdeq} we study time-inhomogeneous SDEs driven by Brownian motion whose drift and diffusion coefficients admit an expansion in the step size, with the leading-order term for the drift determined by the one-step method (and higher-order terms free).
For these SDEs we prove a first-order weak approximation property and we compute their linear error terms explicitly, under certain regularity conditions. Further, we prove a second-order weak approximation property for a specific family of SDEs. In Chapter \ref{chap:smeopt} we instantiate our results for SGD. We work out the example of linear regression explicitly. We use this example to compare the linear error terms of gradient flow and two commonly used first-order stochastic modified equations for SGD.

In the second part (Chapters \ref{chap:towardsSGDoSME} and \ref{chap:weakshuffle}) we introduce and study a novel diffusion approximation for SGD without replacement (SGDo) in the finite-data setting. In Chapter \ref{chap:towardsSGDoSME} we motivate and define the notion of an epoched Brownian motion. We argue that Young differential equations (YDEs) driven by such processes serve as continuous-time models for SGDo for any shuffling scheme whose induced permutations have a well-defined large-sample limit (i.e.\ they converge to a deterministic permuton). Further, we prove almost sure convergence for these YDEs in the strongly convex setting. Moreover, we compute an upper asymptotic bound on the convergence rate which is as sharp as, or better than previous results for SGDo. 
In Chapter \ref{chap:weakshuffle} we study scaling limits of families of random walks that share the same increments up to a (possibly random) permutation. We show weak convergence under the assumption that the sequence of permutations converges to a deterministic (higher-dimensional) permuton. This permuton determines the covariance function of the limiting Gaussian process.
Conversely, we show that every Gaussian process with a covariance function determined by a permuton in this way arises as a weak scaling limit of families of random walks with shared increments. Finally, we apply our weak convergence theory to show that epoched Brownian motions arise as scaling limits of random walks with finitely many distinct increments.

\cleardoublepage

%
%

%
%
%
%
%
%


\setcounter{tocdepth}{1}
\tableofcontents

\cleardoublepage

\pagestyle{headings}
\pagenumbering{arabic}

\chapter{Introduction}
In the last decade or so, the importance of \emph{machine learning} has increased dramatically.
Since the release of AlexNet in 2012, training \emph{neural networks} in particular has enjoyed rapidly growing attention.
Most models nowadays referred to as AI are in fact variants of neural networks trained with some \emph{stochastic (gradient) optimization} algorithm. While the field of AI seems to find success after success in applications, many aspects of the theoretical foundations of AI are poorly understood. Implementing a stochastic optimization algorithm is reasonably straightforward. Modern machine learning frameworks use automatic differentiation in backward mode (a.k.a. backpropagation) to compute gradients. Aside from computing gradients, the algorithms are usually rather simple.

However, it is not simple to see why they should compute anything useful on all but the simplest toy problems. At least for \emph{overparameterized} models, that is models where the number of parameters greatly exceeds the number of data points available. In practice we can effectively learn overparameterized models, although the reason why remains mostly elusive.

Consider, for concreteness, a $d$-dimensional stochastic optimization algorithm $\chi$ with dynamics given by 
\begin{align}\label{eq:SGDintro}
	\chi_{n+1}^h = \chi_n^h - h \nabla R_{\bz(n)}(\chi_n^h), \quad n \in \N_0, h \in (0,1), 
\end{align}
where $(R_z)_z$ is a family of differentiable functions from $\R^d$ to $\R$ and $(\bz(n))_{n \in \N_0}$ is an \tiid\ sequence of random variables in some measurable space $\cZ$. We interpret $(\chi_n^h)_{n \in \N_0}$ as the sequence of estimated parameters when applying a \emph{stochastic gradient descent} (SGD) method for minimizing the function $\cR(x) = \E[R_{\bz(0)}(x)]$ with constant \emph{step size} or \emph{learning rate} $h$. The function $\cR$ itself can be interpreted as \emph{empirical risk} (\tIe{} training error) or \emph{population risk}.
We refer to $h$ as the learning rate and $R_{\bz(n)}$ as the risk due to the $n$-th data point (or mini-batch: a small set of data points).
We also denote by $\Si(x) = \Cov[\nabla R_{\bz(0)}(x)]$ the covariance matrix of $\nabla R_{\bz(0)}(x)$.

Investigating the behavior of the discrete SGD dynamics directly is generally very difficult. 
To make the SGD process tractable with methods from mathematical analysis one frequently approximates the SGD dynamics with an ODE, usually referred to as gradient flow (GF), given by
\begin{equation}
	\label{eq:genericGF}
	\dot X_t^0 = -\nabla \cR(X_t^0), \quad X_0^0 = \chi_0,
\end{equation}
One can show, under certain regularity conditions on $\cR$, that Equation \eqref{eq:genericGF} is then a first-order approximation of SGD in the learning rate, \tIe{} for all $T > 0$ and sufficiently regular test functions $g$ we have
\[|\E g(\chi_{\floor{T/h}}^h) - g(X_T^0)| = \cO(h),\quad h\downarrow 0.\]
Here, first-order refers to first power $h^1$ of $h$ on the right-hand side. In other words, the global truncation error of $\chi$ converges to $0$ \emph{linearly} as we let the step size $h$ go to $0$.

Gradient flow dynamics are deterministic and hence ignore the randomness in SGD. Therefore, in recent years analytic approximations in terms of \tSDE s (SDEs) have become popular. They generally take the form
\begin{equation}
\label{eq:genericSGF}
dX_t^h = - \nabla \cR(X_t^h)\,dt + h b(X_t^h)\,dt + \sqrt{h D(X_t^h)}\,dW_t.
\end{equation}
Here, $W$ is a $d$-dimensional Brownian motion and $D(x)$ is a symmetric and positive semi-definite $d\times d$ matrix.
In particular, \citet{mandt2015continuous} introduced an approximation of the form $\eqref{eq:genericSGF}$ with $b = 0$ and $D$ constant in space such that $D \approx \Si(x)$ in a \enquote{region of interest} (e.g. around a stationary point of $\cR$). \citet{Li15} studied approximations with a state-dependent diffusion coefficient, in particular $b = 0$ and $D = \Si$. 
All approximations of SGD with $b = 0$ and sufficiently regular $D$ are in fact (weak) first-order approximations, just like gradient flow, and generally not approximations of higher order \citep[see][Appendix C, Remark 4]{Li15}.

We can gain further insight by exploring higher-order approximations for \emph{deterministic} $\chi$. In this case $\chi$ is simply the (forward) \emph{Euler method} applied to the gradient flow ODE. Indeed, the Euler method for given ODE is generally a first-order approximation of that ODE. In the literature on numerics of differential equations it is well-known that modifying the (drift) coefficient by setting $b = - \frac12\nabla^2 \cR \nabla \cR$ (and $D = 0$) in Equation \eqref{eq:genericSGF} yields a \emph{second-order approximation} of $\chi$ \citep[see][Section IX.7.2]{hairer2010geometric}. Here, $\nabla^2 \cR$ denotes the Hessian matrix of $\cR$. The resulting family of ODEs is called a \emph{second-order modified equation} of gradient flow. We also call it a second-order modified equation of the discrete process $\chi$.

\citet{Li15, Li18} showed that if we combine the drift modification $b = - \frac12\nabla^2 \cR \nabla \cR$ with the state-dependent diffusion coefficient $D = \Si$, then Equation \eqref{eq:genericSGF} becomes a genuine second-order approximation of $\chi$ in the stochastic case for sufficiently regular $\Si$.

Accordingly, we call Equation \eqref{eq:genericSGF} a first-order stochastic modified equation (SME) of $\chi$ if $b = 0$, and a second-order SME of $\chi$ if $b =- \frac12\nabla^2 \cR \nabla \cR$ and $D= \Si$.

Stochastic modified equations have been used as simplified models to study the dynamics of (variants of) SGD. In particular, SDE approximations have been used to optimize hyperparameters \citep[see][]{mandt2015continuous, mandt_stochastic_2017, Li15, malladi_SDEs_2022, zhao2022batch, perko2023unlocking}, to analyze the long-term behavior of SGD processes \citep[see][]{cao2020approximation, kunin_rethinking_2022, wojtowytsch2024stochastic}, to study the impact of normalization schemes  \citep[see][]{li_reconciling_2020}, to analyze the runtime until convergence \citep[see][]{hu2020runtime}, to study the transition between stationary points \citep[see][]{yang2021fast, zhou2020towards, xie2020diffusion, hu2017diffusion}, to study the implicit bias and regularization properties of SGD \citep[see][]{ali2020implicit, pesme_implicit_2021, li_what_2022}, and to study the effect of running SGD in parallel \citep[see][]{An2019, boffi2020continuous}.

In most of these articles, the authors prefer to work with first-order approximations, in particular $(b,D) = (0, \Si)$. However, upon closer inspection in some applications a constant diffusion coefficient is also frequently used, for example for computing optimal hyperparameters.
Given that order of approximation is the same as that of gradient flow, it is unclear whether first-order \tSME{} are more accurate. To compare first-order approximations including GF, we can aim to compute the linear error term, that is, the constant $C$ for which
\[\E g(\chi_{\floor{T/h}}^h) - \E g(X_T^0) = Ch + \cO(h^2), \quad h \downarrow 0.\]
\citet[Appendix C, Remark 5]{Li15} note that, in principle, the constant $C$ can be computed using the method by \citet{TalayTubaro} for the expansion of the global error of numerical schemes for \tSDE s. The underlying hope is to show that SMEs have a smaller approximation error compared to GF.

The first part of this thesis focuses on precisely this question. In Chapter \ref{chap:mdfdeq} we prove a general approximation result for \tSME s driven by Brownian motion and compute their linear error terms using a method inspired by \citet{TalayTubaro}. In contrast to previous works, we show regularity of the global truncation error in the initial condition. Additionally, we allow for time-inhomogeneous dynamics $\chi$ and do not assume a gradient field structure. Instead we work with a random, time-dependent increment function $f$. The time-inhomogeneous dynamics allows us, for example, to incorporate step size controls $u : [0,\infty) \to [0,1]$, that is step sizes that change over time. Another option we can consider is volatility controls, which leave $\bar f = \E f$ unchanged but change the covariance matrix $\Si = \Cov f$ over time.
We show that by choosing $(b,D) = (- \frac12\nabla^\transp \bar f \bar f + \der_t \bar f,\Si)$, the equation
\begin{equation}
\label{eq:genericSME}
dX_t^h = \bar f_t(X_t^h)\,dt + h b_t(X_t^h)\,dt + \sqrt{hD_t(X_t^h)}\,dW_t
\end{equation}
becomes a second-order SME of SGD. Here, $\nabla^\transp \bar f$ denotes the Jacobian matrix of $\bar f$. Moreover, we show that the linear error term for the general Equation \eqref{eq:genericSME} quantifies how much our chosen coefficients $(b,D)$ deviate from the coefficients $(- \frac12\nabla^\transp \bar f \bar f + \der_t \bar f,\Si)$ of the second-order SME.

Chapter \ref{chap:smeopt} is a transition away from the numerics of differential equations towards stochastic gradient optimization. In particular, we consider mini-batch SGD and its approximation by \tSME s. We discuss two main examples, namely SGD with replacement and one-pass SGD.
In SGD with replacement we draw our sequence of samples $(z(n))_{n\in \N}$ \tiid\ from a finite data set $\set{Z_1,\dots, Z_N}$ of size $N$. In contrast, in one-pass SGD we draw an \tiid\ sequence directly from the population, i.e. the true real-world distribution of our data. Focusing on the one-pass case, we work out the example of linear regression for observational data explicitly.

Using the linear regression example, in Chapter \ref{chap:compare} we perform a detailed comparison of first-order (stochastic) modified equations for SGD. We demonstrate that in this case, stochastic approximations are indeed usually better than gradient flow due to presence of residual noise in the data. Moreover, we show that the ranking of the different approximations is tightly linked to quantities like the \emph{batch size} and the \emph{kurtosis of the data features}. 

So far, we have always assumed that the SGD iterations \eqref{eq:SGDintro} use an \tiid\ sequence of data points, which corresponds either to SGD with replacement or one-pass SGD. Neither of these algorithms is actually used in practice. The first one is considered inefficient, since it can take a long time until all data points are sampled. The second one is impractical. After all, the abundance of data is limited, and if that is the case, then it seems wasteful to train on every data point only once. Instead, in practice one always uses SGD \emph{without replacement} (SGDo). That is, we use one-pass SGD until our finite data set is exhausted and then reuse the same data in subsequent training periods called \emph{epochs}. It is then up to us to define a new ordering for the data points in later epochs. Typical choices include using the same order as in the first epoch (\emph{single-shuffle}) or shuffling the data uniformly and independently across epochs (\emph{random reshuffling}). But in principle the options are endless.

Because we are reusing data points, the Markov property is longer satisfied for SGDo.
Thus, no version of SGDo is covered by the current theory of \tSME s. In the second part of this thesis, we take a step towards a theory of SMEs for SGDo. In Chapter \ref{chap:towardsSGDoSME} we introduce a novel approximation of SGDo using a family of Young (or rough) differential equations which are driven by a process we call an \emph{epoched Brownian motion} (EBM). An EBM is a Brownian motion up to some time point $T > 0$, and is then given by repeating the same Brownian path on the intervals $[jT,(j+1)T], j\in \N$, perhaps up to an \enquote{infinitesimal permutation} of the increments. We focus on approximations with state-independent diffusion coefficient, similar to \citet{mandt2015continuous} in the \tiid\ setting.
While establishing a rigorous approximation theory is beyond the scope of this thesis, we demonstrate the usefulness of EBM-driven approximations via an application.
Specifically, we show that given a step size control of the form $u_t = \frac{1}{(1+t)^\be}, \be \in (0,1)$ for the SGDo iterations and a strongly convex objective function $\cR$, the solution to the approximating EBM-driven equation converges almost surely. Further, we compute an asymptotic upper bound on the convergence speed which at least is as sharp as previous results by \citet{gurbuzbalaban_why_2021} on the convergence of single shuffle SGDo. Moreover, in the case of general random permutations, our results suggest markedly better upper bounds than the best results known for random reshuffling.

Finally, in Chapter \ref{chap:weakshuffle} we clarify our heuristic idea of considering the same Brownian path up to an \enquote{infinitesimal permutation} when approximating SGDo by an EBM-driven differential equation.
We establish weak approximation results for epoched Brownian motion by random walks. In particular, we consider families of random walks that share the same increments, up to a (random) permutation. We show the existence of Gaussian scaling limits of these random walks under natural assumptions on the sequence of permutations. The central assumption is the convergence of the sequence of random permutations to a deterministic \emph{permuton}, i.e. a probability measure on the unit square with uniform marginals. Further, we show that the covariance function of the limiting process is given exactly by the distribution function of the limiting permuton (a so-called \emph{copula}).
This covers the case of EBM with two epochs. More generally, we consider higher and infinite-dimensional permutons and copulas. This allows us to realize EBMs with arbitrarily many epochs as scaling limits.

\chapter{Preliminaries}
Before starting, let us briefly introduce notation and some basic properties that will be used repeatedly throughout this thesis.
Additional notation is introduced along the way as needed.

\section*{Words and multi-indices}
We write $\N = \set{1,2,\dots}$ and $\N_0 = \N \cup \set{0}$.
Given a set $A$ denote by $A^*$ the set of words over $A$, i.e.
\[A^* = \bigcup_{n\in \N_0} A^n, \quad A^0 = \set{()}.\]
We define the concatenation of two words $a \in A^m$ and $b\in A^n$ by
\[(a_1,\dots, a_m)(b_1,\dots, b_n) = (a_1,\dots, a_m, b_1, \dots, b_n)\in A^{m+n}.\]
The empty word $()$ is an identity for the concatenation operation. 
We also define the length of $a\in A^m$ by $|a| = m$.

Given $d,e,l\in \N$ write 
\[d\times e = \set{1,\dots, d}\times \set{1,\dots, e},\]
and
\[d^{\times l} := \underbrace{d\times \dots \times d}_{l \text{ times}}.\]

The words over the natural numbers may also be viewed as \emph{ordered multi-indices}. 
Given $d, e\in \N^*$ we write $d\leq e$ if $|d| = |e|$ and $d_i \leq e_i$ for all $i\leq |d|$.
We set
\[\Pi d := d_1 \times \dots \times d_{|d|}.\]

In particular, for a natural number $d\in \N$ we have
\[\Pi (\underbrace{d, \dots, d}_{l \text{ times}}) = d^{\times l}.\]

We also consider \emph{unordered multi-indices} which we typically simply call \emph{multi-indices}. An unordered multi-index is a function $\al : \set{1,\dots, d}\to \N_0$. Its size is defined by 
\[|\al| = \sum_{k=1}^d \al(k).\]
We also define
\[\al! = \prod_{k=1}^d \al(k)!, \quad x^{\al} = \prod_{k=1}^d x^{\al(k)}_k,\quad x\in \R^d.\]

\section*{Placeholders}
We write $\blnk$ to also denote a placeholder (if it cannot be confused with multiplication). For example, a norm $B\to [0,\infty), x\mapsto \nrm{x}{}$ on a normed space $B$ may be denoted by $\nrm{\blnk}{}$. For another example, consider a function $f : A\times B\to C$. Then $f(a,\blnk)$ denotes the function 
\[f(a,\blnk) : B\to C, b\mapsto f(a,b),\]
for all $a\in A$.

\section*{Uniform bounds}
Suppose we are given functions $f_1,\dots, f_n : A \to \R$. Then we write
\[f_1 \lesssim \dots \lesssim f_n,\]
or
\[f_1(x) \lesssim \dots \lesssim f_n(x), \text{ uniformly over } x\in A,\]
if there exist constants $C_1,\dots, C_n > 0$ such that
\[C_1 f_1(x) \leq \dots \leq C_n f_n(x), \quad x\in A.\]
We also write $f_1 \asymp f_2$ if $f_1 \lesssim f_2 \lesssim f_1$.
We can also mix $\lesssim$ and $\asymp$ with $\leq$ and $=$ as needed.

If $(E, (\nrm{\blnk}{n})_{n\in \N})$ is a vector space equipped with a family of seminorms (or just one norm), and $(x_j)_{j\in J}$ is a family (of elements in some space $F$ with $E\subseteq F$), then we write 
\[x_j \in E, \text{ uniformly in } j\in J,\]
if $x_j \in E, j \in J$, and $\sup_{j\in J} \nrm{x_j}{n} < \infty$ for all $n\in \N$.

\section*{Graded \tFre{} algebras}
A \tFre{} space is a complete metrizable Hausdorff topological vector space $F$ such that its topology is induced by a \emph{grading}. We call a countable family of seminorms $(\nrm{\blnk}{p})_{p\in \N}$ a grading if
\begin{itemize}
\item $\nrm{\blnk}{p} \leq \nrm{\blnk}{q}$ for $p\leq q$, and
\item $\bigcap_{p\in \N} \set{x\in F : \nrm{x}{p} = 0} = \set{0}$.
\end{itemize}
Thus, $x_n \to x$ in $F$ if and only if $\nrm{x_n-x}{p}\to 0, p\in \N$.
Given $x\in F$ we define
\[\nrm{x}{\infty} := \sup_{p\in \N} \nrm{x}{p} \in [0,\infty].\]

A \tFre{} algebra is a \tFre{} space $F$ together with a continuous operation $\cdot : F\times F\to F$ making it also an $\R$-algebra.
A graded \tFre{} space is a \tFre{} space equipped with a grading. 
A graded \tFre{} algebra is \tFre{} algebra with a grading such that for all $p\in \N$ there exists a $q\in \N$ with
\[\nrm{xy}{p} \lesssim \nrm{x}{q}\nrm{y}{q}, \unfovr x,y\in F.\]
We call the grading, and the graded \tFre{} algebra, \emph{Hölder-type} if for all $m\in \N, p, q_1,\dots, q_n\in \N \cup \set{\infty}$ with
\[\sum_{i=1}^n \frac1 {q_i} = 1,\]
we have
\[\nrm{\prod_{i=1}^n x_i}{p} \lesssim \prod_{i=1}^n \nrm{x_i}{pq_i}, \unfovr x_1,\dots, x_n\in F.\]
Here, we define $p\cdot \infty := \infty$ for $p\in \N$, and $\frac 1 {\infty} := 0$.
One simple example for a Hölder-type graded \tFre{} algebra is $\R$ with the grading $\nrm{\blnk}{p} := |\blnk|, p\in\N$.
Another less trivial example is $\Lfin(\Om, V) = \bigcap_{p\in \N} L^p(\Om, V)$, the space of $V$-valued random variables with finite moments. Here, $(\Om, \cF_\Om, \P)$ is a probability space and $V$ a closed linear subspace of some Banach space. It is a \tFre{} algebra under the standard operations, and the grading is given by the family of $p$-norms.

\section*{Arrays}
Let $F$ be a Hölder-type graded \tFre{} algebra.

Let $d, e\in \N^*, k = |d|, l = |e|$. 
An element of $F^{\Pi d}$ is a function $\Pi d \to F$, that is a \emph{$k$-array} (or \enquote{tensor}) with dimensions $d_1, d_2, \dots, d_k$ and values in $F$.
Given $A\in F^{\Pi d}$ we define
\[\nrm{A}{p} = \max_{i\leq \Pi d}\nrm{A_i}{p}, \quad p \in \N.\]
Here, the maximum runs over all ordered multi-indices $i$ with $i \leq \Pi d$.
This a seminorm and $(F^{\Pi d}, (\nrm{\blnk}{p})_{p\in \N})$ is a graded \tFre{} space with addition and scalar multiplication defined component-wise.
Given $e\in \N^*$, $B\in F^{\Pi d}$ and $C\in F^{\Pi e}$ we define the \emph{outer product} $B\otimes C\in F^{\Pi(de)}$ by
\[(B\otimes C)_{ij} = B_iC_j = B_{i_1\dots i_k}C_{j_1\dots j_l}\quad i \leq d, j \leq e\]
and we set 
\[B^{\otimes j} := \underbrace{B\otimes \dots \otimes B}_{j \text{ times}}, \quad j \in \N.\]
For example if $\Pi d = d$ and $\Pi e = e$, then $B,C$ are vectors and $B\otimes C = BC^\transp \in F^{d\times e}$ is a matrix. 
The identity array $1_{\Pi d} \in F^{\Pi d}$ is defined by
\[(1_{\Pi d})_i = \begin{cases}
1, & i_1 = \dots = i_k, \\
0, & \text{else}.
\end{cases}\]
Here, $1$ is the multiplicative identity in the $\R$-algebra $F$.
For $k = 1$ we have $1_d = (1,\dots, 1)$. For $k = 2$ we get the $d\times d$-identity matrix.

Given $d, e, f\in \N^*$ with $l = |d|$, $A \in F^{\Pi(de)}$ and $B\in F^{\Pi(df)}$, we define $\innp{A}{B}_l \in F^{\Pi(ef)}$, the \emph{$l$-th contraction} of $A$ and $B$, by 
\begin{equation}
\label{eq:arrayContr}
(\innp{A}{B}_l)_{jj'} = \sum_{i\leq d} A_{ij} B_{ij'}, \quad j\leq e, j'\leq f.
\end{equation}
If $e = ()$ or $f = ()$ then we simply write $\innp{A}{B} = (\innp{A}{B})_l$. In particular, if $e = f = ()$, then $\innp AB \in F$. In the case $\Pi d = d$ we recover the dot product of vectors and in the case $\Pi d = m \times n$ we get the \emph{Frobenius inner product} of matrices, that is
\[\innp A B = \tr(A^\transp B) = \sum_{k=1}^m \sum_{l=1}^n A_{kl} B_{kl}, \quad A,B\in F^{m\times n}.\]
If we are given $A\in \R^{\Pi(de)}$ instead (or alternatively a real-valued array $B$), Equation \eqref{eq:arrayContr} still makes sense: We can either view the product $A_{ij} B_{ij'}$ as a scalar-vector product, or we can view $\R^{\Pi (de)}$ as a subset of $F^{\Pi (de)}$ by identifying $A$ with $A'\in F^{\Pi(de)}$ given by
\[A_i' = A_i 1, \quad i \leq \Pi(de).\]
Both choices yield the same result.
For $A\in \R^{\Pi d}$ we also define \emph{Frobenius norm}
\[|A| = \sqrt{\innp A A}.\]
Of course, it is equivalent to the norm $\max_{i\leq \Pi d} |A_i|$ considered above, that is
\[|A| \asymp \max_{i\leq \Pi d} |A_i|, \unfovr A\in \R^{\Pi d}\]
However, the Frobenius norm is often nicer to work with when it is available. For matrices $A\in \R^{m\times n}$ we further define the \emph{spectral norm}
\[\specnrm{A} := \max_{x\neq 0} \frac{|Ax|}{|x|} = \max_{|x| = 1} |Ax|.\]
We have
\[\specnrm{A} = \sqrt{\la_{\max}(A^\transp A)},\]
where $\la_{\max}(B)$ is the largest eigenvalue of a square matrix $B$.

Note that the following properties hold true:
\begin{itemize}
	\item $\innp{A}{B\otimes C} = \innp{\innp{A}{B}}{C}\in F, \quad A \in F^{\Pi(de)}, B\in F^{\Pi d}, C\in F^{\Pi e}$.
	\item $\innp{A}{B}\innp{C}{D} = \innp{A\otimes C}{B\otimes D}, \quad A,B,C,D\in F^{\Pi d}$.
	\item $\innp{A}{u} = (u^\transp A)^\transp \in \R^n, \quad A\in \R^{m\times n}, u \in \R^m$.
	\item $\innp{A}{u\otimes v} = u^\transp A v, \quad A\in \R^{m\times n}, u \in \R^m, v\in \R^n$.
	\item $|\innp A B| \leq |A||B|,\quad  A,B\in \R^{\Pi d}$.
	\item $|A\otimes B| \leq |A||B|, \quad A,B\in \R^{\Pi d}$.
\end{itemize}
Further, for all $p\in \N$,
\begin{itemize}
\item $\nrm{\innp A B}{p} \lesssim |A| \nrm{B}{p}, \unfovr A\in \R^{\Pi d}, B\in F^{\Pi d}$.
\item $\nrm{\innp A B}{p} \lesssim \nrm{A}{pq} \nrm{B}{pr}, \unfovr A, B\in F^{\Pi d}$, for all $q,r\in \N \cup \set{\infty}$ with $\frac1q + \frac1r = 1$.
\item $\nrm{\bigotimes_{i=1}^n A_i}{p}\lesssim \prod_{i=1}^n\nrm{A_i}{pq_i}, \unfovr A_i\in F^{\Pi d_i}, i = 1,\dots, n$, for all $n\in \N$ and $q_1,\dots, q_n \in \N \cup \set{\infty}$ with $\sum_{i=1}^n q_i^{-1} = 1$.
\end{itemize}

\section*{Smooth functions}
Given $d\in \N, e \in \N^*$, an open set $U\subseteq \R^d$, and a function $f : U\to \R^{\Pi e}$, we write $f \in \dC l(U,\R^{\Pi e})$ if $f$ is $l$-times continuously differentiable (component-wise). We also allow $U$ to be a more general set. Given $U\subseteq \R^d$ we write $f \in \dC l(U,\R^{\Pi e})$ if $f$ is in $\dC l$ on the interior of $U$, and $f$ as well as its derivatives can be uniquely and continuously extended to the boundary.

An arbitrary subset of $U$ can be very degenerate. The subsets of interest to us are finite Cartesian products of the form $U = I_1\times \dots \times I_m \times \R^{d-m}$, where $I_1,\dots, I_m$ are bounded intervals.
If a function $f : U \to \R^{\Pi e}$ is in $\dC l$, then it can be extended to a function $\tilde f\in \dC l(\R^d, \R^{\Pi e})$ (see also Section \ref{sec:clext} in the Appendix). At the boundary of $U$, the derivatives of $\tilde f$ must then coincide with the continuous extensions of the derivative of $f$. In particular, at the boundary of $U$, the derivatives of $\tilde f$ are uniquely determined. Thus, for example, for the purpose of applying Taylor's theorem to $f$, boundary points of $U$ can be treated as interior points.

\section*{Function spaces}
For (families of) function spaces such as $\dC l$ we write $f\in \dC l, \dC l(U), \dC l(U, V)$ depending on the level of detail required. If $f$ is introduced as a function $U\to V$, then $f\in \dC l$ and $f\in \dC l(U)$ means $f\in \dC l(U,V)$. If the codomain of $f$ is not directly specified, then $f\in \dC l(U)$ means $f\in \dC l(U, \R)$. If neither the domain nor the codomain of $f$ is specified and we write $f\in \dC l$, then this means that there exist sets $U,V$ with $f : U\to V$ and $f\in \dC l(U,V)$.
Similarly, we treat other families of function spaces.

\chapter{Modified equations}
\label{chap:mdfdeq}
In this chapter, we consider a stochastic one-step method $\chi$ and study continuous-time approximations of $\chi$ called (stochastic) modified equations. We show first- and second-order weak approximation results, including regularity in the initial condition, and compute the linear error term for first-order approximations explicitly.
This chapter is inspired by and expands on \citet[Section 4]{perko2024compare}.

\section{Introduction}
In this chapter, our goal is to study the convergence and global error of a stochastic \emph{one-step method} given by
\begin{equation}
\label{eq:GSGD}
\chi_{n+1}^{h} = \chi_n^h + h f_{nh}^{h}(\chi_n^{h}), \quad \chi_0 = x \in \R^d,
\end{equation}
when approximating a differential equation.
Here, we consider a complete probability space $(\Om, \cF_\Om,\P)$ and a random \emph{increment function} \[f : \Om \times (0,1) \times [0,T] \times \R^d\to \R^d, (\om, h, t, x) \mapsto f_t^{h}(\om)(x).\]
The value $h\in (0,1)$ can be interpreted as \emph{discretization parameter} or \emph{step size}.
We assume that the finite family $(f_{nh}^h)_{n \leq \floor{T/h}}$ is independent for all $h\in (0,1)$. 

Typical choices for increment functions satisfy
\[f_t^h = u_t H_{\bz(\floor{t/h})}, \quad h\in (0,1), t\in [0,T].\]
Here, $u : [0,T]\to \R$ is sufficiently regular function, $(\bz(n))_{n\in \N}$ an \tiid\ sequence of random variables in some measurable space $\cZ$, and $H : \cZ \times \R^d \to \R^d$ is such that the random function $H_{\bz(0)}$ is sufficiently regular. In particular, we assume $H_{\bz(0)}(x)$ has finite moments for all $x\in \R^d$.

Assume that $\E f_t^h(x)$ exists and that it does not depend on $h\in (0,1)$.
Define
\[\bar f :[0,T]\times \R^d \to \R^d, (t,x) \mapsto \E f_t^{h}(x), \quad h\in (0,1),\]
and consider the \tODE{}
\begin{equation}
\label{eq:1stOrderOME}
\der_t X^0_t = \bar f_t(X^0_t), t\in [0,T], \quad X_0^0 = x.
\end{equation}
If $f$ is non-random, we have $f = \bar f$ and \eqref{eq:GSGD} is simply the \emph{Euler method} with step size $h$ applied to \eqref{eq:1stOrderOME}. Then, under certain conditions, the Euler method is known to be a \emph{first-order approximation} of \eqref{eq:1stOrderOME}. That is, for all $T > 0$ there exists a constant $C > 0$ such that
\[|\chi_{T/h}^h - X_T^0| \leq Ch, \quad h \in (0,1), T/h \in \N.\]

If $f$ is random, one may think of \eqref{eq:GSGD} as a noisy version of the Euler method that does not require us to calculate the full expectation $\bar f$ at every iteration. Several reasons come to mind for considering \eqref{eq:GSGD}:
\begin{itemize}
\item The expectation is given as an average over a very large number of realizations and is therefore expensive to compute.
\item The expectation is taken with respect to some continuous distribution and no explicit formula is known or useful.
\item The underlying distribution is not known at all, so no expectation can be computed.
\end{itemize}

In Section \ref{sec:ODE} we will show that the stochastic one-step method or \emph{noisy Euler method} \eqref{eq:GSGD} is still a \emph{first-order weak approximation} of the ODE \eqref{eq:1stOrderOME}, in the sense that for all $T > 0$ and sufficiently regular functions $g$ there exists a constant $C > 0$ such that
\[|\E g(\chi_{T/h}^h) - g(X_T^0)| \leq C h, \quad h \in (0,1), T/h \in \N.\]
It is somewhat remarkable that this is true despite the fact that $f$ may be a very crude estimator of $\bar f$.
Thus, the quality of the approximation is really captured in the constant $C$.
Consequently, we determine the linear error term
\[\frac1h(\E g(\chi_{T/h}^h) - g(X_T^0))\]
precisely as well.
Finally, we also investigate the regularity of $C$ as a function of the initial condition $\chi_0 = x$.

The one-step method \eqref{eq:GSGD} exhibits random effects that the ODE \eqref{eq:1stOrderOME} does not have. To better explain these effects it can be useful to approximate \eqref{eq:GSGD} using a family of \tSDE s, driven by a $d$-dimensional Brownian motion $W$, of the form
\[dX_t^h = \bar f_t(X_t^h)\,dt + b_t(X_t^h)\,dt + \sqrt{h D(X_t^h)}\,dW_t.\]
Here, $b$ and $D$ are sufficiently regular functions, and $D$ takes values in the space of positive semi-definite symmetric $d\times d$-matrices.
We study these so called \emph{stochastic modified equations} (SMEs) in Section \ref{sec:fstOrderSMEs}. We show that they are (at least) weak first-order approximations of \eqref{eq:GSGD}, in the sense that for all $T > 0$ and sufficiently regular functions $g$ there exists a constant $C > 0$ such that
\begin{equation}
\label{eq:1stOrderSME}
|\E g(\chi_{T/h}^h) - \E g(X_T^h)| \leq Ch, \quad h \in (0,1), T/h \in \N.
\end{equation}
and we compute their linear error terms as well. We also derive a second-order SME, meaning a particular family of SDEs where we can replace the $Ch$ on the right-hand side of \eqref{eq:1stOrderSME} with $\tilde Ch^2$ for some constant $\tilde C > 0$.

Even though we manage to compute the linear error terms for both stochastic and deterministic differential equations, the calculation does not immediately make clear whether first-order stochastic approximations are indeed better (according to their linear error term). We perform a thorough comparison in a special setting in Chapter \ref{chap:compare}.

\begin{docu}
I think, one cannot generalize Lemma \ref{lem:LGequiv} without assuming $L_p$ differentiability upfront. I mean even more basically the symbol $\der^\al X$ has no meaning given just $X\in LG^l$ and it wont be jointly measurable without assuming that.
For example $\nrm{X(x)}{p} \in \dC l$ is weaker than claiming $X : U \to L_p(\Om, V)\in \dC l$. Further existence of $DX$ along deterministic directions $v$ is weaker than Gateaux differentiability into $L_p$, which also considers random directions. Also note that there is a problem if $p$ is not an even integer, since $|x|^p$ is not smooth for those.
\end{docu}

\section{Preliminaries}
\subsection{Growth conditions}
In this section we discuss various regularity conditions that help us streamline the arguments in Section \ref{sec:ODE}.
We state some properties here without proof. We go in much further detail in Section \ref{sec:prelimSMEs} as preparation for discussing the more general theory of stochastic modified equations in Section \ref{sec:fstOrderSMEs}.

Let $m\in \N, e\in \N^*, V\subseteq \R^m, W\subseteq \R^{\Pi e}$, $g : V \to W$ be a continuous function, and $\ka \in \N_0$. We say $g$ \emph{has (at most) polynomial growth of order $\ka$} if
\[|g(x)| \lesssim 1 + |x|^\ka, \unfovr x \in V.\]
In this case we write $g\in \wCZ \ka$.
More generally, let $l\in \N_0$ and $g \in \dC l$.
Then we define
\[\nrm{g}{\wC l \ka} = \max_{|\al|\leq l}\sup_{x\in V} \frac{|\der^\al g(x)|}{1+|x|^\ka}.\]
Here, the maximum ranges over all (unordered) multi-indices $\al$ up to size $l$. We write $g\in \wC l \ka$ if $g\in \dC l$ with $\nrm{g}{\wC l \ka} < \infty$. In particular, $\wCZ \ka = \wC 0 \ka$.

Polynomial growth conditions are stable under various elementary operations, as the next lemma shows.
\begin{lem}
\label{lem:plyGrwthStab}
	Let $l, \ka, \la\in \N_0$. Then the following properties hold true:
	\begin{enumerate}[(i)]
		\item If $\la \leq \ka$, then $\wC l \la \subseteq \wC l \ka$ with $\nrm{\blnk}{\wC l \ka} \lesssim \nrm{\blnk}{\wC l \la}$.
		\item $\nrm{c f + g}{\wC l \ka} \leq |c| \nrm{f}{\wC l \ka} + \nrm{g}{\wC l \ka}, f,g\in \wC l \ka$.
		\item $\nrm{\innp{f}{g}}{\wC l {\ka+\la}}\lesssim \nrm f {\wC l \ka} \nrm g {\wC l \la}$, uniformly over $f\in \wC l \ka$ and $g\in \wC l \la$.
		\item $\nrm{f\otimes g}{\wC l {\ka+\la}}\lesssim \nrm f {\wC l \ka} \nrm g {\wC l \la}$, uniformly over $f\in \wC l \ka$ and $g\in \wC l \la$.
		\item $\nrm{g\circ f}{\wC l {\ka(\la+l)}} \lesssim \nrm{g}{\wC l {\la}}(1 \vee \nrm{f}{\wC l {\ka}}^{\la+l})$, uniformly over $f\in \wC l \ka$ and $g\in \wC l \la$.
	\end{enumerate}
\end{lem}
\begin{proof}
	This is a special case of Lemma \ref{lem:wCstab} below.
\end{proof}
We assume in the above statements that the expression are all well-defined. In particular, for (i) we fix the same domain and codomain for both function spaces. For the remaining properties, we assume that the domains and codomains are fixed such that $+,\innp\blnk\blnk, \otimes $ and $\circ$ are well-defined.

Next, we discuss continuity and polynomial growth conditions for functions which take random variables as values. Let $d\in \N, U\subseteq \R^d$, $V\subseteq \R^m$ be a linear subspace and $X : U \to \Lfin(\Om, V)$ be a function taking values in the space of $V$-valued random variables with finite moments.
Then we write $X\in \cC(U, \Lfin(\Om ,V))$ if given $x_n \to x$ in $U$ we have $X(x_n)\to X(x)$ in $L^p$, for all $p\geq 1$.
Further, we write $X\in \wCZ \ka(U, \Lfin(\Om,V))$ if:
\begin{itemize}
\item $X\in \cC(U, \Lfin(\Om ,V))$,
\item $\nrm{X(x)}{p} \lesssim 1 + |x|^\ka$, uniformly over $x\in U$, for all $p\geq 1$.
\end{itemize}
Note that since the sequence $\nrm{\blnk}{p}$ is non-decreasing it suffices to consider the condition $\nrm{X(x)}{p} \lesssim 1 + |x|^\ka$ for $p\in \N$.
We equip $\wCZ \ka(U, \Lfin(\Om,V))$ (which is indeed a vector space) with the family of norms given by
\[\nrm{X}{\wCZ \ka,p} := \nrm{x \mapsto \nrm{X(x)}{p}}{\wCZ \ka}, \quad p\geq 1.\]
Thus, given a family $(X^i)_{i\in I}$, writing 
\[X^i \in \wCZ \ka(U, \Lfin(\Om,V)), \unfin i\in I\]
means $X^i \in \wCZ \ka(U, \Lfin(\Om,V)), i\in I$ and $\sup_{i\in I} \nrm{X^i}{\wCZ \ka,p} < \infty$, for all $p\geq 1$.

In the following we call $X : \Om \times U \to V$ a \emph{random field} if $X$ is measurable; i.e.\ measurable with respect to the product $\sigma$-algebra $\cF_\Om \otimes \cB(U)$ and the Borel $\si$-algebra $\cB(V)$.
\begin{lem}
\label{lem:strngLGimpliesLG}
Let $X : \Om \times U\to V$ be a random field such that $X : U \to \Lfin(\Om, V)\in \cC$, and such that there exists a random variable $Z \in \Lfin(\Om, \R)$ with
\[|X(x)| \leq Z(1 + |x|^\ka),\text{ a.s.}, \quad x\in U.\]
Then $X\in \wCZ \ka(U, \Lfin(\Om, V))$ and
\[\nrm{X}{\wCZ \ka,p} \leq \nrm{Z}{p}, \quad p\geq 1.\]
\end{lem}

\begin{proof}
From the assumption we conclude
\[\nrm{X(x)}{p} \leq \nrm{Z}{p}(1 + |x|^\ka), \quad x\in U,\]
and continuity is already satisfied.
\end{proof}

\begin{lem}
	\label{lem:LGequiv} Consider a function $X : U \to \Lfin(\Om,V)$. The following are equivalent:
	\begin{enumerate}[(a)]
		\item $X \in \wCZ \ka(U, \Lfin(\Om, V))$.
		\item $g(X) \in \wCZ{\ka\la}(U, \Lfin(\Om, \R^{\Pi e}))$ for all $\la \in \N, e \in \N^*$ and $g\in \wCZ \la(V,\R^{\Pi e})$.
		\item $\E g(X) \in \wCZ{\ka \la}(U, \R^{\Pi e})$ for all $\la \in \N, e \in \N^*$ and $g\in \wCZ\la(V, \R^{\Pi e})$.
	\end{enumerate}
In this case,
\[\nrm{\E g(X)}{\wCZ{\ka\la}} \lesssim \nrm{g}{\wCZ \la}, \unfovr g\in \wCZ \la\]
\end{lem}
\begin{proof}
Assuming (a) and $g\in \wCZ \la$ we have 
	\[|g(X(x))| \leq \nrm{g}{\wCZ \la} (1 + |X(x)|^\la), \unfovr x \in U,\]
and so
		\begin{equation}
		\label{eq:gXLG}
	\nrm{g(X(x))}{p}\leq \nrm{g}{\wCZ \la} (1 + \nrm{X(x)}{p\la}^\la)\lesssim \nrm{g}{\wCZ \la}(1 + |x|^{\ka \la}), \unfovr x\in U.
		\end{equation}
	Further, if $x_n \to x$ in $U$, then $X(x_n) \to X(x)$ in $L^p$ (and in particular in probability), and so $g(X(x_n)) \to g(X(x))$ in probability. Using Inequality \eqref{eq:gXLG} and Vitali's convergence theorem, we conclude $g(X(x_n)) \to g(X(x))$ in $L^p$, for all $p\geq 1$. This proves (b).
	
  Assuming (b) and $g\in \wCZ \la(V)$, we have
  \[|\E g(X)| \leq \nrm{g(X)}{1} \lesssim \nrm{g}{\wCZ \la}(1 + |x|^{\ka\la}), \unfovr x\in U.\]
   Further, $x_n  \to x$ in $U$ implies $g(X(x_n)) \to g(X(x))$ in $L^1$, and so 
  \[\E g(X(x_n)) \to \E g(X(x)).\]
  Thus, (c) follows, as well as the estimate
	\[\nrm{\E g(X)}{\wCZ{\ka\la}} \lesssim \nrm{g}{\wCZ \la}, \unfovr g\in \wCZ \la.\]
	Assume (c). For $p\in \N$, define $g : V\to \R, x\mapsto |x|^p \in \wCZ p$. Then
	\[\E[|X(x)|^{p}] =  \E g(X)  \lesssim 1 + |x|^{p\ka}, \unfovr x\in U,\]
	Now, suppose $x_n \to x$ in $U$. We have 
	\[\sup_{n\in \N} \E[|X(x_n)|^p] = \sup_{n\in \N} \E g(X(x_n)) \lesssim 1 + \sup_{n\in \N} |x_n|^{p\ka} < \infty,\] for all $p\in \N$. In particular, $(\P_{X(x_n)})_{n\in \N}$ is a tight family of measures. Since $\E g(X(x_n)) \to \E g(X(x))$ for all $g\in \wCZ 0 \subseteq \wCZ \ka$, we have $X(x_n)\to X(x)$ in distribution by the Portmanteau theorem, and, by Vitali's convergence theorem, even in $L^p$. Thus, (c) implies (a).
\end{proof}
Properties similar to the ones in Lemma \ref{lem:plyGrwthStab} also hold for functions in $U \to \Lfin(\Om, V)$. For now, we just mention the following: 
\begin{lem}
\label{lem:plyGrwthStabRndSclrProd}
Let $\ka, \la \in \N_0, X\in \wCZ \ka(U, \Lfin(\Om, \R))$ and $Y\in \wCZ \la(U, \Lfin(\Om, \R^{\Pi d}))$.
Then $XY \in \wCZ{\ka+\la}(U, \Lfin(\Om, \R^{\Pi d}))$, and
\[\nrm{XY}{\wCZ {\ka+\la}, p} \lesssim \nrm{X}{\wCZ {\ka}, 2p}\nrm{Y}{\wCZ {\la}, 2p},\]
uniformly over $X\in \wCZ \ka(U, \Lfin(\Om, \R))$ and $Y\in \wCZ \la(U, \Lfin(\Om, \R^{\Pi d}))$, for all $p\in \N$.
\end{lem}
\begin{proof}
This is a special case of Lemma \ref{lem:wCstab} (c) below.
\end{proof}
Given $g : V \to W$ we say $g$ is \emph{Lipschitz (continuous)} if
\[|g(x) - g(y)| \lesssim |x-y|, \unfovr x,y\in V.\]
In this case we write $g\in \Lip$.
More generally, let $l\in \N_0$ and $g\in \dC l$. We define
\[\nrm{g}{\Lip^{l+1}} = \max_{|\al|\leq l}\sup_{x\neq y\in V} \frac{|\der^\al g(x) - \der^\al g(y)|}{|x-y|}\]
Here, the maximum is taken over all multi-indices $\al$ with $|\al| \in \set{0,\dots, l}$, and $\der^\al g = g$ for $|\al| = 0$.
Then we write $g\in \Lip^{l+1}$ if $\der^\al g \in \Lip$ for all multi-indices $\al$ with $|\al|\leq l$. We also write $\nrm{g}{\Lip} := \nrm{g}{\Lip^{1}}$. 
\subsection{General assumptions on the one-step method}
\label{sec:onstepassum}
Consider once more a random \emph{increment function} \[f : \Om \times (0,1) \times [0,T] \times \R^d\to \R^d, (\om, h, t, x) \mapsto f_t^{h}(\om)(x),\]
such that the finite family $(f_{nh}^h)_{n \leq \floor{T/h}}$ is independent for all $h\in (0,1)$. Define the family of random fields $\chi$ by
\[\chi_{n+1}^{h}(x) = \chi_n^h(x) + h f_{nh}^{h}(\chi_n^{h}(x)), \quad \chi_0^h(x) = \chi_0(x) = x,\]
for $x\in \R^d, h\in (0,1)$ and $n\in \set{0,\dots, \floor{T/h}}$. We write $\Delt \chi_n^h := \chi_{n+1}^h - \chi_n^h$.

Throughout this chapter we assume the following regularity conditions on $f$.
\begin{assum}
\label{assum:H}
There exists a measurable space $\cZ$, a measurable map $F : (0,1)\times [0,T]\times \R^d \times \cZ \to \R^d$, a measurable map $L : \cZ \to [0,\infty)$, and an \tiid\ sequence $(\bz(n))_{n\in \N_0}$ of $\cZ$-valued random variables  with $L(\bz(0)) \in \Lfin(\Om, \cZ)$ such that
\[f_t^h(x) = F(h,t,x,\bz(\floor{t/h})),\]
and
\[|F(h,t,x,z)| \leq L(z)(1 + |x|),\]
for all $ h\in (0,1), t\in [0,T], x\in \R^d$ and $z\in \cZ$.
Further, we have $f_t^h \in \cC(\R^d, \Lfin(\Om, \R^d))$ uniformly in $h\in (0,1)$ and $t\in [0,T]$.
Moreover, there exists a function $\bar f \in \Lip^{5+1}([0,T]\times \R^d, \R^d)$ with $\bar f = \E f^h$ for all $h \in (0,1)$,
and a function $\Si \in \wC 5 2([0,T]\times \R^d,\R^{d\times d})$ with $\Si =  \E[(f^h - \bar f)^{\otimes 2}]$ for all $h \in (0,1)$.
\end{assum}
Note that, using Assumption \assref{assum:H}, for $L_n := L(\bz(n))$ we have
\[|f_{nh}^h(x)| \leq |F(h,t,x,\bz_{n})| \leq L_n(1+ |x|),\]
and $L_n$ is independent of $\chi_n^h$ for all $h\in (0,1)$.
Thus, Lemma \ref{lem:strngLGimpliesLG} implies $f_t^h\in \wCZ 1(\R^d, \Lfin(\Om, \R^d))$, uniformly in $h\in (0,1)$ and $t\in [0,T]$. Also both the expectation and covariance matrix of $f$ do not depend on $h\in (0,1)$.

The following lemma and its proof are inspired by \citet[Lemma 29]{Li18}.
\begin{lemma}
	\label{lem:linGrowthEstSGD}
Assuming \assref{assum:H},	the following hold true: 
	\begin{enumerate}[(i)]
		\item We have 
		\[\max_{n\leq \floor{T/h}}|\chi^h_n| \in \wCZ 1 (\R^d, \Lfin(\Om, \R)), \unfin h\in (0,1).\]
		\item  We have
		\[h^{-1}\max_{n \leq \floor{T/h}}|\Delt\chi_n^h|\in \wCZ 1 (\R^d, \Lfin(\Om, \R)), \unfin h\in (0,1).\]
\end{enumerate}
\end{lemma}
\begin{proof}
The continuity conditions are straightforward to prove. We focus on the linear growth conditions.
	\begin{enumerate}[(i)]
		\item
	Define $M_n = \max_{m\leq n} |\chi_m^h|$. Then
	\[|\chi_{n+1}^h| \leq |\chi_n^h| + h |f_{nh}^h(\chi_n^h)| \leq M_n + h L_n(1 + |\chi_n^h|) \leq M_n + h L_n (1 + M_n),\]
	and so
	\[M_{n+1} = M_n \vee |\chi_{n+1}^h| \leq M_n \vee (M_n + h L_n (1 + M_n)) \leq M_n + h L_n (1 + M_n),\]
	for all $n\in \set{0,\dots, T/h}$.
 Let $p\in \N$.	Then we have
		\begin{align*}
			M_{n+1}^p \leq & (M_n +  h L_n (1 + M_n))^p\\
			\leq &M_n^p + \sum_{k=1}^p \binom{p}{k} M_n^{p-k} h^k L_n^k (1 + M_n)^k.
		\end{align*}
		for all $n\in \set{0,\dots, \floor{T/h}}$.
	Then, for $k\in \set{1,\dots p}$, $h\in (0,1)$ and $n\in \set{0,\dots, \floor{T/h}}$,
		\begin{align*}
\E[M_n^{p-k}  L_n^k (1 + M_n)^k] = & \E[M_n^{p-k}(1 + M_n)^k] \E[L_0^k]\\
\leq &2^k\E[M_n^{p-k}(1 + M_n^k)] \E[L_n^k]\\
\leq & 2^k\E[L_0^k]\E[M_n^{p-k} + M_n^p] \\
\leq & c_k (1 + \E[M_n^p]),
\end{align*}
where $c_k = 2^{k+1} \E[L_0^k]$. Here, we used that $L_n$ and $\chi_n^h$ are independent, and the inequality $y^q + y^p \leq 2(1 + y^p)$ for $0<q\leq p$ and $y \geq 0$.
Hence,
\begin{align*}
\E[M_{n+1}^p]\leq & \E[M_n^p]  + \sum_{k=1}^p \binom{p}{k}h^k \E[M_n^{p-k} L_n^k (1 + M_n)^k] \\
\leq & \E[M_n^p] + C h(1 + \E[M_n^p])\\
= & (1+Ch)\E[M_n^p] + Ch,
\end{align*}
where $C = \sum_{k=1}^p \binom p k c_k$.
By induction over $n$, we get
		\[\E[M_{n}^p] \leq (1 + Ch)^n |x|^p + Ch\left(\sum_{k=0}^{n-1}(1 + Ch)^k\right),\]
		for all $h\in (0,1)$ and $n\in \set{0,\dots, \floor{T/h}}$. Consequently,
		\begin{align*}
\E[M_{\floor{T/h}}^p] \leq & (1 + Ch)^{\frac Th} |x|^p + Ch\left(\sum_{k=0}^{\floor{\frac Th}-1}(1 + Ch)^k\right)\\
\leq & (1 + Ch)^{\frac Th} |x|^p + Ch\frac Th(1 + Ch)^{\frac Th} \\
=& (CT + |x|^p) (1 + Ch)^{\frac Th} \\
\leq & (CT + |x|^p) e^{\log(1 + Ch)\frac Th}\\
\leq & (CT + |x|^p) e^{CT}
		\end{align*}
		for all $h\in (0,1)$ and $x\in \R^d$, since $\log(1 + y) \leq y$ for all $y > -1$. Taking the $p$-th root, we get
		\[\nrm{\max_{n\leq \floor{T/h}}|\chi^h_n|}{p} = \nrm{M_{\floor{T/h}}}{p} \lesssim \left(1 + |x|^p\right)^{1/p} \lesssim 1 + |x|,\unfovr x\in \R^d, h\in (0,1).\]
		For arbitrary $p\geq 1$ we have $\nrmps{Y}{p} \leq \nrmps{Y}{\ceil{p}}$ and thus the result is proven.
		\item We have
		\[h^{-1}|\Delt \chi_n^{h}| = |f_{nh}^{h}(\chi_n^h)| \leq |L_n|(1 + |\chi_n^h|),\]
		for all $h\in (0,1)$ and $n \in \set{0,\dots, T/h}$.
		Thus, 
		\[\nrm{ h^{-1}\max_{n\leq T/h}|\Delt \chi_n^{h}|}{p} \leq \nrm{L_0}{2p}\left(1 + \nrm{\max_{n\leq T/h} |\chi_n^h|}{2p}\right),\]
		for all $h\in (0,1),x\in \R^d$ and $p\geq 1$.
		Hence, the result follows from (i).
	\end{enumerate}
\end{proof}

\section{Convergence of the noisy Euler method}
\label{sec:ODE}
In this section we show that the stochastic one-step method \eqref{eq:GSGD} is a first-order weak approximation of the ODE \eqref{eq:1stOrderOME}. Further, we calculate the linear error term explicitly (see Theorem \ref{thm:1stOrderOME} below), and we show regularity of the error term in the initial condition. In some sense, this is a warm-up for the more general theory of stochastic modified equations in Section \ref{sec:fstOrderSMEs}.
Given a time horizon $T > 0$ write
\begin{equation}
	\label{eq:acceptLR}
	\lrates := \set{h\in (0,1) : T/h \in \N}
\end{equation}
for the set of acceptable step sizes.
Let 
\[X^0 : [0,T]\times [0,T] \times \R^d \to \R^d, (t,s,x)\mapsto X_t^{0,s}\] 
be such that $X^{0,s}(x)$ is the unique solution to the ODE
\[\der_t X_t^{0,s}(x) = \bar f_t(X_t^{0,s}(x)), \quad t\in [s,T], X^0_s(x) = x,\]
for all $(s,x)\in [0,T]\times \R^d$. In the case $s>t$ we can simply set $X^{0,s}_t(x) = x$.
Assumption \assref{assum:H} implies that 
\[X^{0,\blnk}_t : [0,T]\times \R^d \to \R^d, (s,x)\mapsto X^{0,s}_t(x) \in \Lip^{5+1} \subseteq \wC 5 1,\]
uniformly in $t\in [0,T]$ (for example by Corollary \ref{cor:SDEDerInitPert} below). In detail, this means that $X^{0,s}_t(x)$ is five-times continuously differentiable in $(s,x)$ (with continuous extension to the boundary), and 
\[\sup_{t\in [0,T]} \nrm{X^{0,\blnk}_t}{\wC 5 1} = \sup_{t\in [0,T]}\max_{k\leq 5} \max_{|\al|\leq 5-k} \sup_{s\in [0,T]} \sup_{x\in \R^d} \frac{|\der_s^k \der^\al X^{0,s}_t(x)|}{1 + |(s,x)|} < \infty.\]
At the expense of re-scaling the constant, we can write $|x|$ instead of $|(s,x)|$, since $s\leq T$. We write $X^0_t := X^{0,0}_t, t\in [0,T]$ in the following. 
Given $g\in \dC 2(\R^d)$ we define 
\[v^g : [0,T]\times [0,T]\times \R^d \to \R, (r,t,x) \mapsto v^{g,r}_t(x) := g(X_r^{0,t}(x)).\]
We also write $v^{g} = v^{g,T}$ and $v = v^{g}$ if the choice of $g$ is clear from the context.
We denote by $\nabla^\transp w$ and $\nabla^2 w$ the Jacobi and the Hessian matrix of a function $w$ (with respect to the spatial variable), respectively.
\begin{theorem}
\label{thm:1stOrderOME}
Assume \assref{assum:H}. Then for all $\ka \in \N_0$ and $g\in \wC 5 \ka(\R^d)$ there exists a function $\rho^{g} : \lrates \to \wCZ{\ka + 13}(\R^d), h \mapsto \rho^{g,h}$ such that
\begin{align}\label{1st order expansion}
\E g(\chi_{T/h}^h) -  g(X_T^0)= \frac12 h \int_0^T (\innp{\nabla^2 v^{g}}{\Si} - \innp{\nabla v^{g}}{\nabla^\transp \bar f \bar f + \der_t \bar f})_t(X_t^0)\,dt + h^2 \rho^{g,h},
\end{align}
and
\[\nrm{\rho^{g,h}}{\wCZ {\ka + 13}} \lesssim \nrm{g}{\wC 5 \ka},\]
uniformly over $g\in \wC 5 \ka$ and $h\in \lrates$.
\end{theorem}
The degree $\ka+13$ is a bit surprising. Note that if $g\in \wC 5 \ka$ and $X\in \wC 5 1$, then we already have $\ph := g(X)\in \wC 5 {\ka + 5}$ (see Lemma \ref{lem:plyGrwthStab} (v)). In the proof we iterate compositions and pair this with estimates for Taylor remainders. A careful bookkeeping of these steps yields the value $\ka + 13$. We do not claim that this number is optimal.

For deterministic $f$ we have $\Si = 0$. Thus, the second summand on the right-hand side of Equation \eqref{1st order expansion} is the linear error term when approximating an ODE using the Euler method. The first summand accounts for the stochasticity of the noisy Euler method \eqref{eq:GSGD}.

Consider the linear operator
\[\cF : \dC 2([0,T]\times \R^d) \to \cC([0,T]\times \R^d)\]
given by
\begin{align}
	\label{eq:phi}
	\cF w := & \frac12 \innp{\nabla^2 w}{\outSq{\bar f} + \Si} + \innp{\partial_t \nabla w}{\bar{f}} +\frac12 \partial_t^2 w,\quad w\in \dC 2([0,T]\times \R^d),
\end{align} 
and we write $\cF_t w(x) := (\cF w)(t,x)$ for all $(t,x) \in [0,T] \times \R^d$.
We show that the linear error term is determined by $\cF v(X^0)$. We can rewrite $\cF v$ using the following lemma to get the desired expression in Equation \ref{1st order expansion}.
\begin{lem}
	\label{lem:FopSimpl}
	Let $g\in \dC 2 (\R^d)$. Then
	\[\cF v = \frac12 \innp{\nabla^2 v}{\Si} - \frac12\innp{\nabla v}{\nabla^\transp \bar f \bar f + \der_t \bar f}.\]
\end{lem}
\begin{proof}
Recall the transport equation \eqref{eq:KBEOME}
	\[\der_t v + \innp{\nabla v}{\bar f} = 0.\]
Using the formula
\[\nabla \innp{f}{g} = \innp{\nabla^\transp f}{g} + \innp{f}{\nabla^\transp g}, \quad f,g\in \dC 1(\R^d, \R^d)\]
we calculate
	\[\der_t\nabla v = - \nabla \innp{\nabla v}{\bar f} = - \innp{\nabla^2 v}{\bar f} - \innp{\nabla v}{\nabla^\transp \bar f}.\]
	Hence,
	\[-\innp{\der_t\nabla v}{\bar f} = \innp{\innp{\nabla^2 v}{\bar f}}{\bar f} +\innp{\innp{\nabla v}{\nabla^\transp \bar f}}{\bar f} = \innp{\nabla^2 v}{\bar f^{\otimes 2}} + \innp{\nabla v}{\nabla^\transp \bar f \bar f}.\]
	Therefore, again using Equation \eqref{eq:KBEOME},
	\begin{align*}
		\der_t^2 v = - \der_t \innp{\nabla v}{\bar f} = &- \innp{\nabla v}{\der_t \bar f} - \innp{\der_t\nabla v}{\bar f} \\
		= & - \innp{\nabla v}{\der_t \bar f} + \innp{\nabla^2 v}{\bar f^{\otimes 2}} + \innp{\nabla v}{\nabla^\transp \bar f \bar f},
	\end{align*}
and so
	\begin{align*}
		\cF v - \frac12 \innp{\nabla^2 v}{\Si} = & \frac12 \innp{\nabla^2 v}{\bar f^{\otimes 2}} + \innp{\partial_t \nabla v}{\bar{f}} +\frac12 \partial_t^2 v \\
		= & \frac12 \innp{\nabla^2 v}{\bar f^{\otimes 2}} - (\innp{\nabla^2 v}{\bar f^{\otimes 2}} + \innp{\nabla v}{\nabla^\transp \bar f \bar f}) \\
		& + \frac12 (- \innp{\nabla v}{\der_t \bar f} + \innp{\nabla^2 v}{\bar f^{\otimes 2}} + \innp{\nabla v}{\nabla^\transp \bar f \bar f})\\
		= & - \frac12 \innp{\nabla v}{\nabla^\transp \bar f \bar f + \der_t \bar f}. 
	\end{align*}
\end{proof}

\begin{lemma}
\label{lem:TTOperator}
Let $\ka,l \in \N_0$ with $l\leq 5$. Then
\[\nrm{\cF w}{\wC l {\ka + 2}} \lesssim \nrm{w}{\wC{l+2}{\ka}},\]
uniformly over $w\in \wC{l+2}\ka([0,T]\times \R^d)$.
\end{lemma}

\begin{proof}
The linearity of $\cF$ is straightforward to show. By Lemma \ref{lem:plyGrwthStab} we have $\innp{\nabla^2 w}{{\bar f}^{\otimes 2} + \Si}\in \wC l {\ka + 2}$ with
\[\nrm{\innp{\nabla^2 w}{{\bar f}^{\otimes 2} + \Si}}{\wC l {\ka + 2}} \lesssim \nrm{\nabla^2 w}{\wC{l} \ka}(\nrm{\bar f^{\otimes 2}}{\wC 5 2} + \nrm{\Si}{\wC 5 2})\lesssim \nrm{w}{\wC{l+2} \ka}(\nrm{\bar f}{\wC 5 1}^2 + \nrm{\Si}{\wC 5 2}),\]
uniformly in $w\in \wC{l+2}\ka([0,T]\times \R^d)$.
The other summands are treated similarly.
\end{proof}
To prove Theorem \ref{thm:1stOrderOME} we use the fact that, given $g\in \dC 2$, the function $v = v^{g,r} = g(X_r^{\blnk})$ satisfies following PDE, called \emph{transport equation} or \emph{Kolmogorov backward equation}:
\begin{equation}
	\label{eq:KBEOME}
	\begin{cases}
	\der_t v_t + \innp{\bar f}{\nabla v_t} = 0, & t\in [0,r],\\
	v_r = g,
	\end{cases}
\end{equation}
for all $r\in [0,T]$.

\begin{lem}
	\label{lem:talayTubaroSumOME}
	Given $g\in \wC 3 \ka(\R^d)$ there exists a function $\xi^{g} : (\lrates) \times \N \to \wCZ{\ka + 6}(\R^d), (h,n) \mapsto \xi^{g,h}_n$ such that
	\[\E g(\chi_n^{h}) - g(X_{nh}^{0}) = h^2  \sum_{k=0}^{n-1} \E[\cF_{kh}[g(X^{0,\blnk}_{nh})](\chi_k^h)] + h^2 \xi^{g,h}_n, \quad h \in \lrates,\]
and
\[\nrm{\xi^{g,h}_n}{\wCZ {\ka + 6}} \lesssim \nrm{g}{\wC 3 \ka},\]
	uniformly over $g\in \wC 3 \ka$, $h\in \lrates$, and $n\in \set{0,\dots, T/h}$.
\end{lem}
\begin{proof}
Note that $v^{g,r} \in \wC 3 {\ka + 3}([0,T]\times \R^d)$ uniformly in $r\in [0,T]$, with 
\[\sup_{r\in [0,T]}\nrm{v^{g,r}}{\wC 3 {\ka+3}([0,T]\times \R^d)} \lesssim \nrm{g}{\wC 3 \ka(\R^d)}, \unfovr g\in \wC 3 \ka,\] 
by Lemma \ref{lem:plyGrwthStab} (v) and since $X^0_r \in \wC 5 1([0,T]\times \R^d)$, uniformly in $r\in [0,T]$.
Writing $v := v^{g,r}$, Taylor's theorem implies
	\begin{align*}
		v_{t+h}(x+\delt) - v_t(x) = & h\der_t v_t(x) + \innp{\delt}{\nabla v_t(x)} + \frac{h^2} 2 \der_t^2 v_t(x) \\
		&+h \innp{\delt}{\der_t \nabla v_t(x)} + \frac 1 2 \innp{\delt^{\otimes 2}}{\nabla^2 v_t(x)}  \\
		&+ \zeta,
	\end{align*}
	where
	\begin{equation}
\label{eq:talayTubaroSumOMERem}
\zeta = \sum_{l=0}^3 \sum_{|\be| = 3-l} \frac{1}{\be!l!} \der_t^l \der_\be v_{t+\thet h}(x + \thet \delt) h^l \delt^\be,
	\end{equation}
	for some $\thet \in (0,1)$ depending on $t\in [0,T],h\in (0,1), x\in \R^d$ and $\delt \in \R^d$. By choosing $r = nh$, $t = kh$, $x = \chi_k^h$ and $\delt = \Delt \chi_k^h = h f_{kh}^h(\chi_k^h)$ we get
	\begin{align*}
		v_{(k+1)h}(\chi_{k+1}^h) - v_{kh}(\chi_k^h) = & h(\der_t v_{kh} + \innp{f_{kh}^h}{\nabla v_{kh}})(\chi_k^h) \\
		& + \frac12 h^2(\der_t^2 v_{kh} + 2\innp{f_{kh}^h}{\der_t \nabla v_{kh}} + \innp{(f^h_{kh})^{\otimes 2}}{\nabla^2 v_{kh}})(\chi_k^h) \\
		&+ \zeta_{k,n}^h,
	\end{align*}
with remainder term now denoted by $\zeta_{k,n}^h$.
Note that for a random field $Y : \Om \times \R^d\to \R$ and a random variable $Z : \Om \to \R^d$, such that $Y(x)$ is independent of $Z$ for all $x\in \R^d$, we have
\[\E[Y(Z)] = \E[\E[Y(Z)|Z]] = \E[(\E Y)(Z)].\]
Since $f_{kh}^h(x)$ is independent of $\chi_k^h$ and $\E[(f^h_{kh})^{\otimes 2}] = \Si_{kh} + \bar f_{kh}^{\otimes 2}$,
\begin{align*}
\E[v_{(k+1)h}(\chi_{k+1}^h) - v_{kh}(\chi_k^h)] = &  h\E[(\der_t v_{kh} + \innp{\bar f_{kh}}{\nabla v_{kh}})(\chi_k^h)] \\
& + \frac12 h^2\E[(\der_t^2 v_{kh} + 2\innp{\bar f_{kh}}{\der_t \nabla v_{kh}} + \innp{\Si_{kh} + \bar f_{kh}^{\otimes 2}}{\nabla^2 v_{kh}})(\chi_k^h)] \\
&+ \E \zeta_{k,n}^h.
\end{align*}
We have $\der_t v_{kh} + \innp{\bar f_{kh}}{\nabla v_{kh}} = 0$ by Equation \eqref{eq:KBEOME}, and so
\[\E[v_{(k+1)h}(\chi_{k+1}^h) - v_{kh}(\chi_k^h)] = h^2\E[\cF_{kh}[g(X^{0,\blnk}_{nh})](\chi_k^h)] + \E \zeta_{k,n}^h.\]
Therefore,
\begin{align*}
	\E g(\chi_n^h) - g(X_{nh}^0) =& \E v_{nh}(\chi_n^h) - \E v_{0}(\chi_0) \\
	=&\sum_{k=0}^{n-1} \E v_{(k+1)h}(\chi_{k+1}^h) - \E v_{kh}(\chi_k^h) \\
	=&h^2\sum_{k=0}^{n-1} \E[\cF_{kh}[g(X^{0,\blnk}_{nh})](\chi_k^h)] + \sum_{k=0}^{n-1} \E \zeta_{k,n}^h,
\end{align*}
for all $h\in \lrates$.
Recall Equation \eqref{eq:talayTubaroSumOMERem}. To estimate the sum of the remainder terms, note that by Lemma \ref{lem:linGrowthEstSGD} (ii) and Lemma \ref{lem:LGequiv} (b),
\[h^{-3}|h^l (\Delt \chi_k^h)^\be| \leq |h^{-(3-l)}(\Delt \chi_k^h)^{\otimes (3-l)}|\in \wCZ {3-l}(\R^d, \Lfin(\Om, \R)) \subseteq \wCZ 3(\R^d, \Lfin(\Om, \R)),\]
uniformly in $h\in [0,1]$ and $k\in \set{0,\dots, n}$, for $l\in \set{0,1,2,3}$ and $|\be| = 3-l$. Further,
\begin{align*}
	|\der_t^l \der_\be v_{kh+\thet h}^{nh}(\chi_k^h + \thet \Delt \chi_k^h)|\leq & \nrm{v^{nh}}{\wC 3 {\ka+3}}(1 + |(k+1)h|^{\ka} + |\chi_k^h + \thet \Delt \chi_k^h|^{\ka + 3}) \\
\lesssim &\nrm{g}{\wC 3 \ka}(1 + T^{\ka} + |\chi_k^h + \thet \Delt \chi_k^h|^{\ka + 3}) \in \wCZ{\ka+3}(\R^d, \Lfin(\Om, \R)),
\end{align*}
uniformly in $h \in (0,1), n \in \set{0,\dots, T/h}$ and $k\in \set{0,\dots, n}$.
Thus, $\E \zeta^h_{k,n} \in \wCZ{\ka + 6}$ by Lemma \ref{lem:plyGrwthStabRndSclrProd}, with 
\begin{align*}
\nrm{\E \zeta^h_{k,n}}{\wCZ{\ka + 6}(\R^d)} \lesssim & h^3\nrm{g}{\wC 3 \ka(\R^d)}\max_{l\in \set{0,1,2,3}}\nrm{h^{-(3-l)}(\Delt \chi_k^h)^{\otimes {3-l}}}{\wCZ 3, 2}\nrm{|\Delt \chi_k^h +\thet \Delt \chi_k^h|^{\ka + 3}}{\wCZ {\ka+3},2} \\
\lesssim & h^3 \nrm{g}{\wC 3 \ka(\R^d)},
\end{align*}
uniformly over $g\in \wC 3 \ka, h\in \lrates, n \in \set{0,\dots, T/h}$ and $k\in \set{0,\dots, n}$. Finally, for $\xi^{g,h}_n := h^{-2}\sum_{k=0}^{n-1} \E \zeta^h_{k,n}$, we have
\[\nrm{\xi^{g,h}_n}{\wCZ{\ka +6}}\leq h^{-2}\sum_{k=0}^{\frac Th-1} \max_{n\in \set{0,\dots, T/h}} \nrm{\E \zeta^h_{k,n}}{\wCZ{\ka + 6}} \lesssim \nrm{g}{\wC 3 \ka},\]
uniformly over $g\in \wC 3 \ka, h\in \lrates$ and $n \in \set{0,\dots T/h}$.
\end{proof}

\begin{docu}
It may be possible to work with $g\in \dC 2$ by using a Taylor formula with integral remainder of order $2$. However, in this case I think we actually need $g\in \Lip^3$.
\end{docu}

\begin{proof}[Proof of Theorem \ref{thm:1stOrderOME}]
Let $g\in \wC 5 \ka(\R^d)$ and $h\in \lrates$. Writing $\ph := \cF v^g$, Lemma \ref{lem:talayTubaroSumOME} implies
\[\E g(\chi_{T/h}^h) - g(X_T^0) = h \sum_{n=0}^{\frac Th-1} h\E\ph_{nh}(\chi_k^h) + h^2 \xi^{g,h}_{T/h},\]
with
\[\nrm{\xi^{g,h}_{T/h}}{\wCZ{\ka+6}} \lesssim \nrm{g}{\wC 3 \ka}, \unfovr g\in \wC 3 \ka, h \in \lrates.\]
We can then express the sum as follows:
	\begin{align*}
		\sum_{n=0}^{\frac T h -1} h\E [\ph_{nh}(\chi_n^h)] = &\int_0^T \ph_t(X_t^0)\,dt +  h\sum_{n=0}^{\frac T h -1} \E \ph_{nh}(\chi_n^h) - \ph_{nh}(X_{nh}^0) \\
		&+ \sum_{n=0}^{\frac T h-1} h\ph_{nh}(X_{nh}^0) - \int_0^T \ph_t(X_t^0)\,dt.
	\end{align*}
Note that $\ph \in \wC l {\ka + l+4}([0,T]\times \R^d)$, with
\begin{equation}
\label{eq:phEstODE}
\nrm{\ph}{\wC l {\ka +l+4}([0,T]\times \R^d)} \lesssim \nrm{g(X^{0,\blnk}_T)}{\wC {l+2} {\ka+l+2}([0,T]\times \R^d)} \lesssim \nrm{g}{\wC {l+2} \ka(\R^d)}
\end{equation}
uniformly over $g\in \wC 5 \ka$, by Lemma \ref{lem:TTOperator} and Lemma \ref{lem:plyGrwthStab} (v), for $l \leq 3$. Thus, we may use Lemma \ref{lem:talayTubaroSumOME} again to estimate
\begin{align*}
\sum_{n=0}^{\frac T h -1} |\E [\ph_{nh}(\chi_n^h)] - \ph_{nh}(X_{nh}^0)|\leq& h^2\sum_{n=0}^{\frac T h - 1} \sum_{k=0}^{n-1} |\E[\cF_{kh}[\ph_{nh}(X^{0,\blnk}_{nh})](\chi_k^h)] + \hat\xi^{g,h}_n|, \\
	\end{align*}
where by Inequality \eqref{eq:phEstODE}
\[\nrm{\hat \xi^{g,h}_n}{\wCZ{\ka + 13}} \lesssim \nrm{\ph_{nh}}{\wC 3 {\ka + 7}(\R^d)} \lesssim \nrm{\ph}{\wC 3 {\ka + 7}} \lesssim \nrm{g}{\wC 5 \ka},\]
uniformly over $g\in \wC 5 \ka, h\in \lrates$ and $n\in \set{0,\dots, T/h}$. Here (and implicitly also in other places) we use the fact
\[\sup_{t\in [0,T]} \nrm{w_t}{\wC l \ka(\R^d)} \lesssim \nrm{w}{\wC l \ka([0,T]\times \R^d)}, \unfovr w\in \wC l \ka([0,T]\times \R^d),\]
for all $\ka, l\in \N_0$.
Using Lemma \ref{lem:LGequiv}, we also have
\[\nrm{\E[\cF_{kh}[\ph_{nh}(X^{0,\blnk}_{nh})](\chi_k^h)]}{\wCZ {\ka + 10}} \lesssim \nrm{\cF_{kh}[\ph_{nh}(X^{0,\blnk}_{nh})]}{\wCZ {\ka + 10}} \lesssim \nrm{\ph_{nh}(X_{nh}^{0,\blnk})}{\wC 2 {\ka + 8}(\R^d)} \lesssim \nrm{\ph}{\wC 2 {\ka + 6}} \lesssim \nrm{g}{\wC 4 \ka},\]
uniformly over $h\in (0,1), k \leq n \in \set{0,\dots, T/h}$, and $g\in \wC 4 \ka$. We conclude
\[h \nrm{\sum_{n=0}^{\frac T h-1} \E\ph_{nh}(\chi_n^h) - \ph_{nh}(X_{nh}^0)}{\wCZ {\ka + 13}} \lesssim h \nrm{g}{\wC 4 {\ka}},\]
uniformly over $g\in \wC 4 \ka$ and $h\in (0,1)$.
Further, approximating the integral $\int \ph(X^0)\,dt$ by a left Riemann sum yields
	\begin{align*}
\left|\sum_{n=0}^{\frac T h-1} h\ph_{nh}(X_{nh}^0) -\int_0^T \ph_t(X_t^0)\,dt\right|\leq& \frac12 h T \sup_{t\in [0,T]} |\der_t(\ph_t(X_t^0))|,
	\end{align*}
Thus, using Lemma \ref{lem:plyGrwthStab} (v) and Inequality \eqref{eq:phEstODE} once more
\begin{gather*}
\nrm{\sum_{n=0}^{\frac T h-1} h\ph_{nh}(X_{nh}^0) -\int_0^T \ph_t(X_t^0)\,dt}{\wCZ {\ka + 6}} \lesssim h \sup_{t\in [0,T]}\nrm{\der_t (\ph_t(X^0_t))}{\wCZ {\ka +6}}\lesssim h \sup_{t\in [0,T]}\nrm{\ph_t(X^0_t)}{\wC 1 {\ka +6}}\\
\lesssim h \nrm{\ph}{\wC 1 {\ka+5}} \lesssim h \nrm{g}{\wC 3 \ka},
\end{gather*}
uniformly over $g\in \wC 3 \ka$ and $h\in \lrates$.
Putting all estimates together yields
\[\E g(\chi_{T/h}^h) - g(X_T^0) = h \int_0^T \ph_t(X_t^0)\,dt+ h^2\rho^{g,h},\]
with 
\[\nrm{\rho^{g,h}}{\wCZ{\ka + 13}}\lesssim \nrm{g}{\wC 5 \ka}\]
uniformly over $g\in \wC 5 \ka$ and $h\in \lrates$.
\end{proof}

\section{Differentiation and regularity}
\label{sec:prelimSMEs}
Let $d\in \N$, $F$ be a \tFre{} space, $U\subseteq \R^d$ be open and $f : U\to F$ be a continuous function.
The \emph{derivative} of $f$ at $x\in U$ in the direction $v\in \R^d$ is defined by 
\[Df(x)v = \lim_{h\to 0} \frac1h (f(x+hv) - f(x)).\]
If the limit exists, then we say $f$ is differentiable at $x$ in the direction $v$. We say $f$ is continuously differentiable on $U$ if the limit exists for all $x\in U$ and $v\in \R^d$, and if $Df : U\times \R^d \to F$ is continuous.

We further define higher derivatives recursively
\[D^{l+1} f(x)(v_1,\dots, v_l, v_{l+1}) = \lim_{h\to 0} \frac1h (D^lf(x+hv_{l+1})(v_1,\dots, v_l) - D^lf(x)(v_1,\dots v_l)), \quad l \in \N,\]
with $D^0 f = f$. We say $f$ is $l$-times continuously differentiable (on $U$) if $D^1 f,\dots, D^l f$ exist everywhere and $D^k f : U \times (\R^d)^k \to F$ is continuous for all $k\leq l$.
We write $\dC l(U,F)$ for the set of $l$-times continuously differentiable functions $U\to F$.

If $f\in \dC l(U,F)$, then $D^l f(x)(v_1,\dots, v_l)$ is symmetric in $v_1,\dots, v_l$. That is, its value is unchanged when permuting $v_1,\dots, v_l$.

If $U \subseteq \R^d$ is not necessarily open with interior $\intr U$, and $V\subseteq F$, then we write $\dC l(U,V)$ for the set of continuous functions $f : U\to F$ such that 
\begin{itemize}
	\item $f(U) \subseteq V$,
	\item $f : \intr U \to F \in \dC l$,
	\item $D^k f: \intr U \times (\R^d)^k \to F$ can be extended to a unique continuous function $D^kf : U\times (\R^d)^k \to F$, for all $k\leq l$.
\end{itemize}

For $U = I_1\times\cdots\times I_m\times \R^{d-m}$ where $I_1,\dots, I_m$ are bounded intervals, one can show that $f\in \dC l (U,V)$ implies that $f$ has a $\dC l$ extension to an open neighborhood of $U$ (see Lemma \ref{lem:FrechetExtension} in the Appendix). Moreover, the derivatives of extensions of $f$ are uniquely determined by $f$ on the boundary of $U$.

We can express (higher) derivatives using the standard basis $\set{e_1,\dots, e_d}$ in $\R^d$. Given a multi-index $\al$ with $|\al| \leq l$ we define the \emph{partial derivative} of $f$ at $x$ with respect to $\al$ by
\[\der^\al f(x) = D^{|\al|} f(x)(\overbrace{e_1,\dots, e_1}^{\al(1)},\dots, \overbrace{e_d,\dots, e_d}^{\al(d)}), \quad x\in U.\]
This defines a function $\der^\al f : U \to F$ if all the relevant limits exist. We define $\der^\al f = f$ for $|\al| = 0$.
The symmetry of $D^lf (x)$ implies that the order in which the vectors $e_1,\dots, e_d$ are listed does not matter, only how often each vector appears. It also clear that the definition of $\der^\al f$ makes sense as soon as the lower order derivatives $\der^\be f$ with $\be\leq \al$ are defined.
For $l\in \N$ and $x\in U$ we define the $l$-th order gradient $\nabla^l f(x) \in F^{d^{\times l}}$ of $f$ at $x$ by
\[(\nabla^l f(x))_i = \der_{i_l} \dots \der_{i_1} f(x),\quad i\in d^{\times l},\]
yielding a function $\nabla^l f : U \to F^{d^{\times l}}$ if all relevant partial derivatives exists. Also we set $\nabla^0 f = f$.

We can extend the domain of $D^l f(x)$ to the space of arrays $\R^{d^{\times l}}$ as follows:
\begin{align*}
D^l f(x)A = & D^l f(x)\left(\sum_{i\in d^{\times l}} A_i (e_{i_1}\otimes \dots \otimes e_{i_l})\right)\\
 := & \sum_{i\in d^{\times l}} A_i D^l f(x)(e_{i_1},\dots, e_{i_l}),
\end{align*}
for $x\in U$ and $A\in \R^{d^{\times l}}$.

For most purposes, we consider the spaces $(F^{d_2})^{d_1}$ and $F^{d_1\times d_2}$ the same, via the canonical bijection
\[(F^{d_2})^{d_1} \to F^{d_1\times d_2}, ((x_{1,1},\dots, x_{1,d_2}), \dots, (x_{d_1,1},\dots, x_{d_1,d_2})) \mapsto \mat{x_{1,1} & \dots&  x_{1,d_2} \\ \vdots & \ddots & \vdots \\ x_{d_1,1} & \dots & x_{d_1,d_2}}.\]
Similarly, we  identify $((F^{d_l})^{\dots})^{d_1} \cong F^{d_1\times \dots \times d_l}$.
For example, given $f : \R^d\to \R^m$ we have $\nabla f : \R^d \to (\R^m)^d$ by definition. However, we prefer using the equivalent version $\nabla f : \R^d \to \R^{d\times m}$, so that $\nabla f(x)$ denotes the \emph{transpose} of the Jacobi matrix\footnote{Working with the Jacobi matrix instead of its transpose is awkward in our framework. We would have to define $\innp{\blnk}{\blnk}$ differently.} of $f$ at $x\in \R^d$. We denote the Jacobi matrix of $f$ by $\nabla^\transp f$ instead.
Moreover, given a \tFre{} space $F, U \subseteq \R^d$, and $f : U \to F$ we consider $\nabla^2 f : U\to F^{d\times d}$ and $\nabla \nabla f : U\to (F^d)^d$ the same (and the same as $\nabla^2 f : U \to \R^{d\times d\times m}$ if $F = \R^m$). Similar identifications are applied to higher order gradients.

Now with that mind, if $F$ is a \tFre{} algebra, then
\[D^l f(x)A = \innp{A}{\nabla^l f(x)}, \quad x\in U, A\in \R^{d^{\times l}}.\]
We can also characterize smoothness through partial derivatives or higher order gradients.
\begin{lem}
	Let $U\subseteq \R^d, V\subseteq F$ and $f : U \to V$. Then $f\in \dC l(U,V)$ if and only if $\der^\al f : U \to F$ exists for all $\al$ with $|\al|\leq l$ and $\der^\al f \in \cC(U,F)$.
\end{lem}
Thus, the continuity of $D^l f$ is also equivalent to the continuity of $\nabla^l f : U\to F^{d^{\times l}}$ if $D^l f$ exists. Also if $F$ is a Banach space, then $f\in \dC l(U,V)$ is equivalent to $f$ being $l$-times continuous differentiable in the sense of  \tFre{} derivatives.

Familiar rules for differentiating linear combinations, products and compositions apply, see Proposition \ref{prop:diffRules} below. For the Faa di bruno formula we need some additional notation.
Given a multi-index $\al$, a \emph{partition} of $\al$ is a set $\cB := \set{\be_1,\dots, \be_k}$ of multi-indices such that $\sum_{i=1}^k \be_i = \al$. We denote by $\prttn \al k$ the set of partitions of $\al$ with $k$ elements. Further, we define $\cB! = \prod_{i=1}^k \be_i!$.
Denote the symmetric group on $\set{1,\dots, k}$ by $\Sym k$. Given arrays $(A_1,\dots, A_k)$ we define
\[\bigodot_{i=1}^k A_i = \frac1{k!}\sum_{\si \in \Sym k} \bigotimes_{i=1}^k A_{\si(i)}.\]
Note that the order of the $A_i$ on the left-hand side does not matter.
Given $\cB = \set{\be_1,\dots, \be_k} \in \prttn \al k$ and $f\in \dC {|\al|}$ we define
\[\der^\cB f = \bigodot_{i=1}^k \der^{\be_i} f.\]
\begin{prop}
\label{prop:diffRules}
Let $l\in \N_0$.
\begin{enumerate}[(a)]
\item If $c\in \R, f,g\in \dC l$, then $cf+g\in \dC l$ and $\der^\al(cf+g) = c\der^\al f + \der^\al g$ for $|\al|\leq l$.
\item If $n\in \N$ and $f_1,\dots, f_n\in \dC l$, then $\bigotimes_{i=1}^n f_i\in \dC l$ and 
\[\der^\al \left(\bigotimes_{i=1}^n f_i\right) = \sum_{\be_1 + \dots + \be_n = \al} \frac{\al!}{\be_1! \cdots \be_n!} \bigotimes_{i=1}^n \der^{\be_i} f_i, \quad |\al|\leq l \text{ (Leibniz)}.\]
\item $f,g\in \dC l$, then $g\circ f\in \dC l$ and 
\[\der^\al(g\circ f) = \sum_{k=1}^{|\al|} \sum_{\cB\in \prttn \al k} \frac{\al!}{\cB!} \innp{\nabla^k g\circ f}{\der^\cB f}, \quad |\al|\leq l \ \text{ (Faa di Bruno)}.\]
\end{enumerate}
\end{prop}
In Proposition \ref{prop:diffRules} we implicitly assume that the sums, products and compositions are well-defined.


\subsection{Weighted Hölder spaces}
We introduce weighted Hölder spaces, generalizing spaces of functions of polynomial growth $\wC l \ka$ and Lipschitz functions $\Lip^l$. The main goal here is just to prove various stability properties for different special cases all at once, including smooth functions with Lipschitz derivatives or derivatives of polynomial growth. We do not directly use general weighted Hölder spaces in any of our applications.

We call a function $\cV : [0,\infty) \to [1,\infty)$ a \emph{weight function} if
\begin{itemize}
\item $\cV(x+y)\lesssim \cV(x) + \cV(y)$,
\item $\cV(xy) \lesssim \cV(x) \cV(y)$, 
\end{itemize}
uniformly over $x,y \geq 0$.
Given $x,y\in \R^d$ we write $\cV(x) = \cV(|x|)$ and $\cV(x,y) = \cV(x) \vee \cV(y)$.
The only weight functions of interest to us are polynomial weights of the form $\cV(x) = C(1 + x^\ka), x \geq 0$ for some $\ka \in \N_0$. Nevertheless, for elegance and brevity reasons we work with abstract weights $\cV$ in the following.

Note that set of weight functions is stable under products and compositions. This is proven in the polynomial case in the next lemma.
We also show that polynomial weights are indeed weight functions.
\begin{lem}
\label{lem:polynWeights}
Write $\cV_\ka(x) := 1 + x^\ka, x\geq 0, \ka \in \N_0$.
	Let  $\ka, \la \in \N_0$. Then:
	\begin{enumerate}[(a)]
		\item $\cV_\ka(x+y) \lesssim \cV_\ka(x) + \cV_\ka(y), \unfovr x,y \geq 0$.
		\item $\cV_\ka(xy) \leq \cV_\ka(x)\cV_\ka(y), \quad x,y \geq 0$.
		\item If $\ka \leq \la$, then $\cV_\ka\lesssim \cV_\la$.
		\item $\cV_\ka \cV_\la \lesssim \cV_{\ka + \la}$.
		\item $\cV_\la \circ \cV_\ka \lesssim \cV_{\ka \la}$.
	\end{enumerate}
\end{lem}
\begin{proof}
Let $x,y \geq 0$. Then the estimates follow from simple calculations:
	\begin{enumerate}[(a)]
		\item $1+(x+y)^\ka \leq 1 + 2^{\ka-1} x^\ka + 2^{\ka-1} y^\ka \leq 2^{\ka-1}(1+x^\ka + 1 + y^\ka)$.
		\item $(1+x^\ka)(1+y^\ka) = (1 + x^\ka + y^\ka + (xy)^\ka) \geq 1 + (xy)^\ka$.
		\item $1 + x^\ka \leq 1 + (1 \vee x^\la) \leq 2(1 + x ^\la)$.
		\item $(1 + x^\ka)(1 + x^\la) \leq 1 + x^\ka + x^\la + x^{\ka+\la} \leq (3 + x^{\ka+\la}) \vee (1 + 3 x^{\ka+\la}) \leq 3(1 + x^{\ka+\la})$.
		\item $1 + (1 + x^\ka)^\la \leq 1 + \sum_{k=0}^{\la} \binom {\la} k x^{\ka k} \leq 1 + 2^\la x^{\ka\la}$. \qedhere
	\end{enumerate}
\end{proof}
In the following, we consider sequences $\N \to [0,\infty]$. Given sequences $a,b$ we write $a\leq b$ if $a_p \leq b_p$ for all $p\in \N$, and similarly for $<$. Constants $c\in [0,\infty]$ are identified with the constant sequence $(c,c,\dots)$. In particular, $a < \infty$ means $a_p < \infty$ for all $p \in \N$. We extend all functions $\R\to \R$ and $\R^2 \to \R$ to sequences $a < \infty$ in the usual way (\enquote{component-wise}).

Let $F$ be a graded \tFre{} space. We denote the grading $\N \to [0,\infty), p \mapsto \nrm{\blnk}{p}$ also simply by $\nrm{\blnk}{}$.
Let $U\subseteq \R^d, V\subseteq F$, $\delt \in [0,1]$, $\cV$ be a weight function and $f : U\to V \in \cC$.
We define non-decreasing sequences $\nrm{f}{\wCZ \cV}, \nrm{f}{\wHoelHZ \delt \cV} : \N \to [0,\infty]$ by
\[\nrm{f}{\wCZ \cV,p} := \sup_{x\in U} \frac{\nrm{f(x)}{p}}{\cV(x)}, \quad \nrm{f}{\wHoelHZ \delt \cV,p} := \sup_{x\neq y \in U} \frac{\nrm{f(x)-f(y)}{p}}{\cV(x,y)|x-y|^\delt}, \quad p \in \N.\]
Recall that if $F$ is a graded \tFre{} algebra, then for all $k\in \N$ we define a grading $(\nrm{\blnk}{p})_{p\in \N}$ on $F^{d^{\times k}}$ so that we have
\[\nrm{\nabla^k f}{p} = \max_{i\leq d^{\times k}} \nrm{\nabla^k_i f}{p} = \max_{|\al|\leq k} \nrm{\der^\al f}{p}\]
for $f \in \dC k$. The latter maximum is taken over all multi-indices with sizes up $k$ (including size $0$).
Given $l\in \N_0$ we further define
\begin{align*}
\nrm{f}{\wC l \cV} := & \max_{k\leq l} \nrm{\nabla^k f}{\wCZ \cV} = \max_{|\al|\leq l} \nrm{\der^\al f}{\wCZ \cV},\\
\nrm{f}{\wHoelH l \delt \cV} := & \nrm{\nabla^l f}{\wHoelHZ \delt \cV} = \max_{|\al|=l} \nrm{\der^\al f}{\wHoelH l \delt \cV},\\
\wHoelNrm f l \delta \cV = & \nrm{f}{\wC l \cV} \vee \nrm{f}{\wHoelH l \delt \cV}.
\end{align*}
Next, for $\delt \in (0,1]$ we define
\[\nrm f {\Lip_\cV} := \nrm f {\wHoelHZ 1 \cV}, \quad \nrm f {\Lip^{l + \delt}_\cV} := \max_{k < l} \nrm{\nabla^k f}{\Lip_\cV} \vee \nrm{f}{\wHoelH l \delt \cV}.\]
We write $f\in E$ if $\nrm f E < \infty$ for $E\in \set{\wC l \cV, \wHoelH l \delt \cV, \wHoel l \delt \cV,\dots}$ (also implying all derivatives necessary for this statement to be meaningful exist). We also write $\wC l \cV = \wHoel l 0 \cV, \wCZ \cV = \wC 0 \cV$, as well as $\wHoel l \delt b := \wHoel l \delt {(x\mapsto 1)}, \wC l b := \wHoel l 0 b, \wCZ b := \wC 0 b,\hoelH l \delt = \wHoelH l \delt {x\mapsto 1}, \hoelHZ \delt = \hoelH 0 \delt$ , $\Lip = \Lip_{(x\mapsto 1)}$ and $\Lip^{l+\delt} = \Lip^{l+\delt}_{(x\mapsto 1)}$. Here, $x\mapsto 1$ is the constant weight, which is equal to $1$ everywhere. Note that $\Lip^l \subseteq \dC {l-1}$, but $\Lip^l \nsubseteq \dC l$.
Finally, we define $\Lip^{l+\delta}_b = \Lip^{l+\delt} \cap \wC l b$ with
\[\nrm{f}{\Lip^{l+\delta}_b} := \nrm{f}{\Lip^{l+\delta}} \vee \nrm{f}{\wC l b}.\] 
Here is a short glossary for the most important spaces in applications:
\begin{itemize}
\item $\hoelHZ \delt$ - $\delt$-Hölder continuous functions,
\item $\hoelH l \delt$ -  $l$-times continuously differentiable functions such that their $l$-th derivative is $\delt$-Hölder continuous,
\item $\wHoel l \delt b$ - bounded $l$-times continuously differentiable functions such that all their derivatives are bounded, and their $l$-th derivative is $\delt$-Hölder continuous
\item $\Lip$ - Lipschitz continuous functions,
\item $\Lip^l$ - $l-1$-times continuously differentiable functions such that they and all their derivatives are Lipschitz continuous,
\item $\Lip^{l+\delt}$ - $l$-times continuously differentiable functions such that they and their derivatives up to order $k-1$ are Lipschitz continuous, and such that their $l$-th derivative is $\delt$-Hölder continuous,
\item $\wCZ {\cV_\ka}$ ($=\wCZ \ka$) - continuous functions with at most polynomial growth of order $\ka$,
\item $\wC l {\cV_\ka}$ ($=\wC l \ka$) - $l$-times continuously differentiable functions such that they and all their derivatives have at most polynomial growth of order $\ka$.
\end{itemize}
A lowercase $b$ (as in $\wC l b$) means functions and all their derivatives (which are implied to exist) are bounded.

The following statement is a generalization of the fact that a function $f\in \dC 1$ on a convex domain is Lipschitz if and only if its derivative is bounded.
\begin{lem}
\label{lem:wghtdRdemchr}
	Let $f : U\to V \in \dC 1$ and $U$ convex.
	Then $f \in \Lip_\cV$ if and only if $\nabla f\in \wCZ \cV$.
	Further,
 \[\nrm{f}{\Lip_\cV} \asymp \nrm{\nabla f}{\wCZ \cV}, \unfovr f\in \dC 1 \cap \Lip_\cV.\]
\end{lem}
\begin{proof}
	Suppose $f\in \Lip_\cV$. Then 
	\[\nrm{f(x+hv)-f(x)}{} \leq \nrm f {\Lip_\cV} \cV(x, x+hv) |h| |v|,\]
	and
	\[\cV(x, x+hv) \lesssim \cV(x) \vee (\cV(x) + \cV(hv)) \lesssim \cV(x),\]
	uniformly over $x\in U$ and $|hv|\leq 1$.
	Hence,
	\[\nrm{D f(x)(v)}{} \lesssim \nrm f {\Lip_\cV} \cV(x), \unfovr x\in U, |v|\leq 1.\]
In particular,
	\[\nrm{\der_j f(x)}{} \lesssim \nrm f {\Lip_\cV} \cV(x) \unfovr x\in U.\]
	
	Conversely, suppose $\nabla f\in \wCZ \cV$. Then, by Hadarmard's lemma \citep[see][Theorem 3.2.2.]{hamilton1982inverse}
	\begin{align*}
		\nrm{f(x) - f(y)}{} 	\leq & \int_0^1 \nrm{Df(x+t(y-x))(y-x)}{} \,dt \\
		\leq & \int_0^1 \nrm{\innp{\nabla f(x+t(y-x))}{y-x}}{}\,dt\\
		\lesssim & \int_0^1 \nrm{\nabla f(x+t(y-x))}{}|y-x|\,dt\\
		\leq &  \int_0^1 \nrm{\nabla f}{\wCZ \cV} \cV(|x+t(y-x)|)|y-x|\,dt\\
		\leq & \nrm{\nabla f}{\wCZ \cV}\cV(x,y)|x-y|,
	\end{align*}
uniformly over $x,y\in U$.
\end{proof}
\begin{lem}
\label{lem:LiplAltRep}
	Let $l\in \N_0$ and $\delt \in (0,1]$. On a convex domain $U$ we have
\begin{enumerate}[(a)]
	\item $\Lip^{l+\delt}_\cV = \set{f\in \wHoelH l \delt \cV : \nabla f\in \wC {l-1} \cV}$
	with
\[\nrm{f}{\Lip^{l+\delt}_\cV} \asymp \nrm{\nabla f}{\wC {l-1} \cV}\vee \nrm{f}{\wHoelH l \delt \cV}, \unfovr f\in \Lip^{l+\delt}_\cV.\]
	\item $\wHoel l \delt \cV = \Lip^{l+\delt}_\cV \cap \wCZ \cV = \set{f \in \dC l : \nabla^k f \in \wHoelZ \delt \cV, k = 0,\dots, l}$, with
		\[\nrm{f}{\wHoel l \delt \cV} \asymp \nrm{f}{\Lip^{l+\delt}_\cV} \vee \nrm{f}{\wCZ \cV} \asymp \max_{k\leq l} \nrm{\nabla^k f}{\wHoelZ \delt \cV} = \max_{|\al|\leq l} \nrm{\der^\al f}{\wHoel 0 \delt \cV}, \unfovr f\in \Lip^{l+\delt}_\cV.\]
\item $\Lip^{l+\delt}_b = \wHoel l \delt b$ with
\[\nrm{f}{\Lip^{l+\delt}_b} \asymp \nrm{f}{\wHoel l \delt b}, \unfovr f\in \Lip^{l+\delt}_b.\]
\end{enumerate} 
\end{lem}
\begin{proof}
\begin{enumerate}[(a)]
\item By Lemma \ref{lem:wghtdRdemchr}
\[\nrm{f}{\Lip^{l+\delt}_\cV} = \max_{k<l}\nrm{\nabla^k f}{\Lip_\cV} \vee  \nrm{f}{\wHoelH l \delt \cV} \asymp  \max_{k<l}\nrm{\nabla^{k+1} f}{\wCZ \cV} \vee  \nrm{f}{\wHoelH l \delt \cV} = \nrm{\nabla f}{\wC {l-1} \cV}\vee \nrm{f}{\wHoelH l \delt \cV},\]
uniformly over $f\in \Lip^{l+\delt}_\cV$.
\item By Lemma  \ref{lem:wghtdRdemchr}
\begin{align*}
\nrm{f}{\wHoel l \delt \cV} = & \nrm{f}{\wC l \cV} \vee \nrm{f}{\wHoelH l \delt \cV} = \max_{k\leq l} \nrm{\nabla^k f}{\wCZ \cV}  \vee \nrm{f}{\wHoelH l \delt \cV} \asymp \nrm{f}{\wCZ \cV}\vee \max_{k < l} \nrm{\nabla^k f}{\Lip_\cV}\vee \nrm{f}{\wHoelH l \delt \cV} \\
= &\nrm{f}{\Lip^{l+\delt}_\cV} \vee \nrm{f}{\wCZ \cV}.
\end{align*}
For the second representation, note that 
\[\nrm{f}{\wHoel l \delt \cV} = \max_{k\leq l} \nrm{\nabla^k f}{\wCZ \cV} \vee \nrm{f}{\wHoelH l \delt \cV} \lesssim \max_{k\leq l} \nrm{\nabla^k f}{\wHoelZ \delt \cV},\]
and, using the fact that $\nrm{\blnk}{\wHoelHZ \delt \cV} \lesssim \nrm{\blnk}{\Lip_\cV}$,
\begin{align*}
\max_{k\leq l} \nrm{\nabla^k f}{\wHoelZ \delt \cV} = & \max_{k\leq l} (\nrm{\nabla^k f}{\wCZ \cV} \vee \nrm{\nabla^k f}{\wHoelHZ \delt \cV}) \\
\lesssim  &\max_{k\leq l} \nrm{\nabla^k f}{\wCZ \cV} \vee \max_{k<l} \nrm{\nabla^k f}{\Lip_\cV} \vee \nrm{\nabla^l f}{\wHoelHZ \delt \cV} \\
= & \nrm{f}{\wC l \cV} \vee \nrm{f}{\Lip^{l+\delt}_\cV} \\
\asymp & \nrm{f}{\wHoel l \delt \cV},
\end{align*}
uniformly over $f\in \wHoel l \delt \cV$. 
\item 
Using (b), we have
\[\nrm{f}{\Lip^{l+\delta}_b} = \nrm{f}{\Lip^{l+\delta}} \vee \nrm{f}{\wC l b} \asymp \nrm{f}{\wHoel l \delt b}, \unfovr f\in \Lip^{l+\delt}_b.\]
\qedhere
\end{enumerate}
\end{proof}

\begin{lem}
\label{lem:wCstab}
	Let $F$ be a Hölder-type graded \tFre{} algebra, $\cV, \cW$ be weight functions, $l\in \N_0$, $p,d\in \N$, $e\in \N^*$, $U\subseteq \R^d$ and $V\subseteq \R^m$ convex, and $\delt, \ga \in [0,1]$. Then:
	\begin{enumerate}[(a)]
		\item If $WV\leq \cW$ and $\delt \leq \ga$, then $\wHoel l \ga \cV \subseteq \wHoel l \delt \cW$ with $\wHoelNrm \blnk l \delt \cW \lesssim \wHoelNrm \blnk l \ga \cV$.
		\item $\wHoelNrm{cf + g} l \delt \cV \leq |c|\wHoelNrm{f} l \delt \cV +\wHoelNrm{g} l \delt \cV $, $f,g \in \wHoel l \delt \cV(U,F^{\Pi e}), c\in \R$.
		\item $\wHoelNrm{\bigotimes_{i=1}^n f_i} l \delt {{\Pi \cV}, p}\lesssim \prod_{i=1}^n \wHoelNrm{f_i} l \delt {{\cV^i},pq_i}$, uniformly over $f_i \in \wHoel l \delt {\cV^i} (U,F^{\Pi e^i}), i = 1,\dots, n$.\\ Here, $n\in \N, q_1,\dots, q_n \in \N \cup \set{\infty}$ with $\sum_{i=1}^n q_i^{-1} = 1, \Pi \cV := \prod_{i=1}^n \cV^i$ for weight functions $\cV^1,\dots, \cV^n$, and $e^1,\dots e^n \in \N^*$.
		\item $\wHoelNrm{\innp fg} l \delt {{\cV\cW}, p} \lesssim \wHoelNrm{f} l \delt {\cV,pq} \wHoelNrm{g} l \delt {\cW,pr}$, uniformly over $f\in \wHoel l \delt \cV(U,F^{\Pi e})$ and $g\in \wHoel l \delt \cW(U,F^{\Pi e})$. Here, $q,r\in \N \cup \set{\infty}$ with $\frac1q + \frac1r = 1$.
		\item  $\nrm{g\circ f}{\wHoel l {\delt\ga} {\cW \circ \cV \cdot \cV^l}} \lesssim \nrm{g}{\wHoel l {\ga} \cW}\cW(\nrm{f}{\wHoel l \delt \cV}) (1 \vee\nrm{f}{\wHoel l \delt \cV}^{\ga \vee l})$, uniformly over $f\in \wHoel l \delt \cV(U,V)$ and $g\in \wHoel l \ga \cW(V,F^{\Pi e})$.
	\end{enumerate}
\end{lem}

Note that by considering polynomial weights, $F = \R$ and $\delt = \ga = 0$, with the help of Lemma \ref{lem:polynWeights}, we can conclude Lemma \ref{lem:plyGrwthStab} (which we have already used before). 

\begin{proof}[Proof of Lemma \ref{lem:wCstab}]
In the following, we use Lemma \ref{lem:LiplAltRep} (b) whenever needed.
	\begin{enumerate}[(a)]
		\item 
		Given $f\in \wHoel 0 \ga \cV$ we have
		\[\nrm{f(x)}{} \leq \nrm{f}{\wCZ \cV}\cV(x) \leq \nrm{f}{\wCZ \cV} \cW(x), \quad x\in U.\]
		Further, for $x,y\in U$ with $|x-y|\leq 1$ we have
		\[\nrm{f(x) - f(y)}{} \leq \wHoelNrm f 0 \ga \cV \cV(x, y) |x-y|^\ga \leq \wHoelNrm f 0 \ga \cV \cW(x, y) |x-y|^\delt.\]
		In the case $|x-y|>1$ we instead estimate
		\[\nrm{f(x) - f(y)}{}\leq \nrm{f(x)}{} + \nrm{f(y)}{}\leq 2 \nrm{f}{\wCZ \cV}\cW(x,y)|x-y|^\delt.\]
		Hence, $\wHoelNrm f 0 \delt \cW \leq 2 \wHoelNrm f 0 \ga \cV$.
		For general $f\in \wHoel l \ga \cV$ we have $\nrm{\nabla^k f}{\wCZ \cW} \leq 2\nrm{\nabla^k f}{\wCZ \cV}$ for all $k < l$ and $\nrm{\nabla^l f}{\wHoel 0 \delt \cW} \leq 2\nrm{\nabla^l f}{\wHoel 0 \ga \cV}$.
		\item
		Given $f, g \in \wHoel 0 \delt \cV, c\in \R$ and $x,y \in U$ we have
		\[\nrm{cf(x) + g(x)}{} \leq |c|\nrm{f(x)}{} + \nrm{g(x)}{} \leq (|c|\nrm f {\wCZ \cV} + \nrm g {\wCZ \cV})\cV(x),\]
		and
		\begin{align*}
			\nrm{cf(x) + g(x) -  cf(y) - g(y)}{} \leq & |c|\nrm{f(x) - f(y)}{} + \nrm{g(x) - g(y)}{}\\
			\leq & (|c|\nrm f {\wHoel 0 \delt \cV}+\nrm g {\wHoel 0 \delt \cV})\cV(x,y)|x-y|^\delt
		\end{align*}
		Thus, for $f,g\in \wHoel l \delt \cV$ we have $\nrm{\nabla^k (cf+g)}{\wHoel 0 \delt \cV} \leq |c|\nrm{\nabla^k f}{\wHoel 0 \delt \cV} + \nrm{\nabla^k g}{\wHoel 0 \delt \cV}, k\leq l$.
		\item In the case $l = 0$, we have
		\[\nrm{\bigotimes_{i=1}^n f_i(x)}{p} \lesssim \prod_{i=1}^n \nrm{f_i(x)}{pq_i}\leq \prod_{i=1}^n \cV^i(x) \nrm{f_i}{\wCZ{\cV^i}, pq_i} \leq \prod_{i=1}^n \nrm{f_i}{\wCZ{\cV^i}, pq_i}\Pi \cV(x),\]
		uniformly over $f_i \in \wHoel l \delt {\cV^i} (U,F^{\Pi e^i}), i = 1,\dots, n$ and $x\in U$.
		Further, we have
		\begin{align*}
			\bigotimes_{i=1}^n f_i(x) - \bigotimes_{i=1}^n f_i(y) = \sum_{i=1}^n \left(\bigotimes_{j=1}^{i-1} f_j(x)\right) \otimes (f_i(x) - f_i(y)) \otimes \left(\bigotimes_{j=i+1}^n f_j(y)\right).
		\end{align*}
	Therefore, 
		\begin{align*}
			& \nrm{\bigotimes_{i=1}^n f_i(x) - \bigotimes_{i=1}^n f_i(y)}{p}\\
		\lesssim	& \sum_{i=1}^n\left( \prod_{j=1}^{i-1} (\cV^j(x) \nrm{f_j}{\wCZ{\cV^j}, pq_j})(\cV^i(x,y)\nrm{f_i}{\wHoel 0 \delt {\cV^i},pq_i} |x-y|^\delt)\prod_{j={i+1}}^{n} (\cV^j(y) \nrm{f_j}{\wCZ{\cV^j}, pq_j})\right) \\
			\lesssim &  \Pi \cV(x,y)\prod_{i=1}^{n} \nrm{f_i}{\wHoel 0 \delt {\cV^i},pq_i} |x-y|^\delt,
		\end{align*}
	uniformly over $f_i \in \wHoel l \delt {\cV^i} (U,F^{\Pi e^i}), i = 1,\dots, n$ and $x,y\in U$.
		For general $l\in \N_0$ recall the Leibniz formula (Proposition \ref{prop:diffRules} (b)) 
		\[\der^\al \left(\bigotimes_{i=1}^n f_i\right) = \sum_{\be_1 + \dots + \be_m = \al} \frac{\al!}{\be_1!\cdots\be_n!}\bigotimes_{i=1}^n \der_{\be_i} f_i.\]
		Hence,
		\begin{gather*}
			\nrm{\nabla^k (\bigotimes_{i=1}^n f_i)}{\wHoel 0 \delt {\Pi \cV},p} \leq \max_{|\al|=k} \nrm{\der^\al (\bigotimes_{i=1}^n f_i)}{\wHoel 0 \delt {\Pi \cV},p} \lesssim \max_{|\al|=k} \sum_{\Si \be = \al} \nrm{\bigotimes_{i=1}^n \der_{\be_i} f_i}{\wHoel 0 \delt {\Pi \cV},p} \\
			\lesssim \max_{|\al|=k} \sum_{\Si \be = \al} \prod_{i=1}^n \nrm{\der_{\be_i} f_i}{\wHoel 0 \delt {\cV^i},pq_i} \lesssim \prod_{i=1}^n \nrm{f_i}{\wHoel l \delt {\cV^i},pq_i},
		\end{gather*}
		uniformly over $k\leq l$ and $f_i \in \wHoel l \delt {\cV^i} (U,F^{\Pi e^i}), i = 1,\dots, n$.
		\item Analogous to (c) for $n = 2$.
		\item Note that $\nrm{f(x)}{p} = |f(x)|, x\in U, p \geq 1$. Given $f\in \wHoelZ \delt \cV$ and $g\in \wHoelZ \ga \cW$, we have
		\begin{align*}
		\nrm{g(f(x))}{} \leq &  \nrm{g}{\wCZ \cW} \cW(f(x)) \\
		\leq & \nrm{g}{\wCZ \cW} \cW(\nrm{f}{\wCZ \cV}\cV(x))\\
		\lesssim & \nrm{g}{\wCZ \cW} \cW(\nrm{f}{\wCZ \cV}) (\cW \circ \cV)(x),
		\end{align*}
	uniformly over $x\in U$, and
	\begin{align*}
	\nrm{g(f(x)) - g(f(y))}{} \leq & \nrm{g}{\wHoelHZ \ga \cW}\nrm{f(x) - f(y)}{}^\ga \cW(f(x),f(y)) \\
							\leq & \nrm{g}{\wHoelHZ \ga \cW}\nrm{f}{\wHoelHZ \delt \cV}^\ga|x-y|^{\delt\ga} \cW(\nrm{f}{\wCZ \cV}\cV(x),\nrm{f}{\wCZ \cV}\cV(y)) \\
							\lesssim & \nrm{g}{\wHoelHZ \ga \cW}\nrm{f}{\wHoelHZ \delt \cV}^\ga|x-y|^{\delt\ga} \cW(\nrm{f}{\wCZ \cV}) (\cW \circ\cV)(x,y),
	\end{align*}
uniformly over $x,y\in U$. Thus,
\begin{equation}
\label{eq:wC0StabComp}
\nrm{g\circ f}{\wHoelZ {\delt\ga}{\cW\circ \cV}} \lesssim \nrm{g}{\wHoelZ \ga \cW} \cW(\nrm{f}{\wCZ \cV})(1 \vee \nrm{f}{\wHoelZ \delt \cV}^\ga),
\end{equation}
uniformly over $f\in \wHoelZ \delt \cV$ and $g\in \wHoelZ \ga \cW$. Now, consider $f\in \wHoel l \delt \cV$ and $g\in \wHoel l \ga \cW$. Let $\al$ be a multi-index with $|\al|\leq l$. Recall Faa di Bruno's formula (Proposition \ref{prop:diffRules} (c)):
\[\der^\al(g\circ f) = \sum_{k=1}^{|\al|} \sum_{\cB\in \prttn{\al}{k}} \frac{\al!}{\cB!} \innp{\nabla^k g\circ f}{\der^\cB f}.\]
Given a partition $\cB$ of $\al$ into $k$ multi-indices, we have using (a), (b) and (c)
\[\nrm{\der^\cB f}{\wHoelZ \delt {\cV^l}} \lesssim \nrm{\der^\cB f}{\wHoelZ \delt {\cV^k}} \lesssim \prod_{\be \in \cB} \nrm{\der^\be f}{\wHoelZ \delt \cV} \leq \nrm{f}{\wHoel l \delt \cV}^k \leq1 \vee \nrm{f}{\wHoel l \delt \cV}^l,\]
uniformly over $f\in \wHoel l \delt \cV$. Further, by using Inequality \eqref{eq:wC0StabComp}
\[\nrm{\nabla^k g \circ f}{\wHoelZ {\delt\ga}{\cW\circ \cV}} \lesssim \nrm{\nabla^k g}{\wHoelZ \ga \cW} \cW(\nrm f {\wHoelZ \delt \cV})(1\vee \nrm f {\wHoelZ \delt \cV}^\ga),\]
uniformly over $f\in \wHoel l \delt \cV$ and $g\in \wHoel l \ga \cW$.
Thus, using (d) with $q = \infty, r = 1$,
\begin{align*}
\nrm{\innp{\nabla^k g\circ f}{\der^\cB f}}{\wHoelZ {\delt\ga}{\cW\circ\cV\cdot \cV^l}} & \lesssim \nrm{\nabla^k g}{\wHoelZ \ga \cW} \cW(\nrm f {\wHoelZ \delt \cV})(1\vee \nrm f {\wHoelZ \delt \cV}^\ga) (1\vee \nrm{f}{\wHoel l \delt \cV}^l)\\
& \lesssim \nrm{g}{\wHoel l \ga \cW} \cW(\nrm f {\wHoel l \delt \cV})(1\vee \nrm f {\wHoel l \delt \cV}^{\ga \vee l}),
\end{align*}
uniformly over $f\in \wHoel l \delt \cV$ and $g\in \wHoel l \ga \cW$.	
Using (b), the desired result follows.
\end{enumerate}
\end{proof}

\subsection{Regularity of random fields}
In the previous subsection, we already encountered various regularity properties that also apply to random fields. Here, we provide additional properties specific to random fields.

Fix a complete probability space $(\Om,\cF_\Om, \P)$. We consider functions $X : U \to \Lfin(\Om, V)$, where $U\subseteq \R^m, V\subseteq B$ and $B$ is a Banach space.
\begin{lem}
\label{lem:expDerCmmte}
Let $l\in \N_0$ and $X: U\to \Lfin(\Om, V) \in \dC l$. Then $\E X \in \dC l(U, V)$ with
\[\der^\al \E X = \E[\der^\al X], \quad |\al|\leq l.\]
\end{lem}
\begin{proof}
Let $j \in \set{1,\dots, d}$. Then, for all $x\in U$,
\[\frac1h (X(x + e_j h) - (X(x))) \to \der_j X(x)\text{ in } L^1.\]
In particular,
\[\frac1h \E[X(x + e_j h) - (X(x))] \to \E[\der_j X(x)].\]
Thus, $\E X \in \dC 1(U, V)$ and $\der_j \E X = \E[\der_j X]$. By induction we get the desired property for all $|\al|\leq l$.
\end{proof}

\begin{lem}
\label{lem:expWHoelFld}
Let $l\in \N_0, \delt\in [0,1]$, $\cV$ be a weight function and $X: U \to \Lfin(\Om, V) \in \wHoel \delt l \cV$.
Then $\E X \in \wHoel \delt l \cV(U,V)$ with
\[\nrm{\E X}{\wHoel \delt l \cV} \leq \nrm{X}{\wHoel \delt l \cV,1}.\]
\end{lem}
\begin{proof}
We have
\[\nrm{\E X(x)}{B} \leq \E[\nrm{X(x)}{B}] \leq \nrm{\nrm{X(x)}{B}}{1} \leq \nrm{X}{\wHoel \delt l \cV,1}\cV(x), \quad x \in U,\]
and similarly
\[\nrm{\E X(x) - \E X(y)}{B} \leq \nrm{X}{\wHoelHZ \delt \cV,1}\cV(x,y)|x-y|^\delt,\quad  x,y\in U.\]
By replacing $X$ with $\der^\al X$ for all $|\al|\leq l$ and using Lemma \ref{lem:expDerCmmte}, the desired result follows.
\end{proof}

Let $e\in \N^*$ and $T >0$.
Next, we consider random variables valued in the space of a.s.\ bounded paths, that is the space $B = (L^\infty([0,t], \R^{\Pi e}), \nrm{\blnk}{\infty})$ for $t\in [0,T]$, and a filtration $(\cF_t)_{t\in [0,T]}$ on $(\Om, \cF_\Om, \P)$ satisfying the usual condition.

The norm on the $L^p$-space 
\[\Lpinft p t:= \Lpinft p t(\Om) := \Lpinft p t(\Om, \R^{\Pi e}) := \cL_p(\Om, L^\infty([0,t], \R^{\Pi e}))/\sim,\] where $\sim$ is equality almost surely, is given by
\[\nrm X {\Lpinft pt} := \nrm{X : \Om \to L^\infty([0,t],\R^{\Pi e})}{p} = \E\left[\esssup_{s\in [0,t]} |X_s|^p\right]^{1/p},\]
for all $p\geq 1$ and $t\in [0,T]$.
In this case, almost surely equality is equivalent to indistinguishability of stochastic processes. We identify each element of $L^\infty([0,t], \R^{\Pi e})$ with a measurable representative $f$ such that $\sup_{s\in [0,t]} |f_s| = \esssup_{s\in [0,t]} |f_s|$. Hence, we can write $\sup$ instead of $\esssup$ in the following. Further, we write $\Lpinf p := \Lpinft pT$,
 As usual, we consider the space of random variables with finite moments 
 \[\Lfininf := \Lfininf(\Om) := \Lfininf(\Om, \R^{\Pi e}) = \bigcap_{p\in \N} \Lpinf p(\Om, \R^{\Pi e})\] as a \tFre{} algebra with the grading $(\nrm{\blnk}{\Lpinf p})_{p\in \N}$.
To make some estimates work for all parameter values we define $\nrm{X}{0} = 0$ and $\nrm{X}{\Lpinft 0 t} = 0$, etc.

\begin{docu}
	\begin{lem}
		Let $X,Y : \Om \times [0,T]\times \R^d\to \R^d$ be random fields with $X\in \wHoel 0 \delt \cV(\R^d, \Lfininf(\Om)), Y\in \wHoel 0 \ga \cW(\R^d, \Lfininf(\Om))$. Define
		\[(Y\circ X)(\om, t, x) = Y(\om,t(X(\om,t,x))), \quad \om \in \Om, t\in [0,T], x\in \R^d.\]
		Then $Y\circ X$ is a random field with $Y\circ X\in \wHoel 0 {\delt\ga}{\cW\circ \cV}(\R^d, \Lfininf(\Om))$ and
		\[\nrm{Y\circ X}{\wHoel 0 {\delt\ga}{\cW\circ \cV}(\R^d, \Lpinf{p}(\Om))} \leq \nrm{Y}{\wHoel 0 \ga \cW(\R^d, \Lpinf{2p}(\Om))}(1 + \nrm{X}{\wHoel 0 \delt \cV(\R^d, \Lpinf{2p}(\Om))}), \quad p\geq 1.\]
	\end{lem}
\end{docu}

Let $X : \Om \times [0,T]\times U\to \R^d$ be a random field and $g : [0,T]\times \R^d \to \R^{\Pi e}$ be jointly measurable. Define
\[g(X)(\om,t,x) = g_t(X_t(\om, x)), \quad \om \in \Om, t\in [0,T], x\in \R^d.\]
Then $g(X) : \Om \times [0,T]\times U \to \R^{\Pi e}$ is a random field as well.

\begin{lem}
\label{lem:gXwghtHoelIneq}
Let $\delt\in (0,1], \la\in \N_0,\cV_\la = 1 + |\blnk|^\la, X,Y : \Om \times [0,T]\times U\to \R^d$ be random fields and $g : [0,T]\times \R^d \to \R^{\Pi e}$ be jointly measurable.
Then
\begin{enumerate}[(a)]
\item If $g\in \wCZ {\cV_\la}(\R^d, L^\infty([0,T], \R^{\Pi e}))$, then 
\[\nrm{g(X)}{\Lpinft p t} \leq \nrm{g}{\wCZ {\cV_\la}}\cV_\la(\nrm{X}{\Lpinft {p\la} t}), \quad p \geq 1, t\in [0,T].\]
\item If $g\in \wHoelHZ \delt {\cV_\la}(\R^d, L^\infty([0,T], \R^{\Pi e}))$ and $q,r\in [0,\infty]$ with $1/q + 1/q = 1$, then
\[\nrm{g(X)-g(Y)}{\Lpinft pt} \leq \nrm g {\wHoelHZ \delt {\cV_\la}} (\cV_\la(\nrm{X}{\Lpinft {pq\la} t}) + \cV_\la(\nrm{Y}{\Lpinft {pq\la} t}))  \nrm{X-Y}{\Lpinft {pr} t}^\delt, \quad p \geq 1, t\in [0,T].\]
\end{enumerate}
\end{lem}
\begin{proof}
Firstly, note that 
\[\cV_\la(\sup_{s\in [0,t]} |x_s|) = \sup_{s\in [0,t]} \cV_\la(|x_s|), \quad t\in [0,T], x \in L^\infty([0,t], \R^d).\]
Moreover, given $p\geq 1$ and $Z\in L^{p\la}$
\[\nrm{\cV_\la(Z)}{p} \leq \cV_\la(\nrm{Z}{p\la}),\]
and so
\[\nrm{\cV_\la(X)}{\Lpinft p t} \leq V_\la(\nrm{X}{\Lpinft{p\la} t}), \quad t\in [0,T].\]
Further, given $Z,Z'\in L^{p\la}$ we have
\[\nrm{\cV_\la(Z,Z')}{p} \leq \nrm{\cV_\la(Z)+\cV_\la(Z')}{p} \leq \cV_\la(\nrm{Z}{p\la}) + \cV_\la(\nrm{Z'}{p\la}),\]
and so 
\[\nrm{\cV_\la(X,Y)}{\Lpinft p t} \leq \cV_\la(\nrm{X}{\Lpinft{p\la} t}) + \cV_\la(\nrm{Y}{\Lpinft{p\la} t}), \quad t\in [0,T].\]
\begin{enumerate}[(a)]
\item In this case, we have
	\begin{align*}
	\sup_{s\in [0,t]} |g_s(X_s)| \leq & \sup_{s\in [0,t]} \sup_{r\in [0,T]} |g_r(X_s)| \leq \nrm{g}{\wCZ {\cV_\la}} \sup_{s\in [0,t]}  \cV_\la(X_s).
\end{align*}
Hence, by applying $\nrm{\blnk}{p}$,
\[\nrm{g(X)}{\Lpinft p t} \leq \nrm{g}{\wCZ {\cV_\la}} \nrm{\cV_\la(X)}{\Lpinft p t} \leq \nrm{g}{\wCZ {\cV_\la}}\cV_\la(\nrm{X}{\Lpinft{p\la} t}), \quad p\geq 1.\]
\item In this case, we have
	\begin{align*}
	\sup_{s\in [0,t]} |g_s(X_s) - g_s(Y_s)| \leq & \sup_{s\in [0,t]} \sup_{r\in [0,T]} |g_r(X_s) - g_r(Y_s)| \\
	\leq & \nrm{g}{\wHoelHZ \delt {\cV_\la}} \sup_{s\in [0,t]} (\cV_\la(X_s,Y_s) |X_s - Y_s|^\delt).
\end{align*}
Further, by Hölder's inequality
\begin{align*}
\nrm{\cV_\la(X,Y) |X - Y|^\delt}{\Lpinft pt} \leq & \nrm{\cV_\la(X,Y)}{\Lpinft {pq}t} \nrm{|X - Y|^\delt}{\Lpinft {pr}t}\\
 \leq & (\cV_\la(\nrm{X}{\Lpinft{pq\la} t}) + \cV_\la(\nrm{Y}{\Lpinft{pq\la} t}))\nrm{X - Y}{\Lpinft {pr}t}^\delt,
\end{align*}
for all $p\geq 1$ and $t\in [0,T]$.
\end{enumerate}
\end{proof}

\begin{lem}
	\label{lem:wHoelHgRndFld}
	Let $l,\ka,\la \in \N_0$, $X : \Om \times [0,T]\times U\to \R^d$ be a random field with $X\in \wHoelHZ \delt {\cV_\ka}(U, \Lfininf)$ and $g : [0,T]\times \R^d \to \R^{\Pi e} \in \wHoelHZ \ga {\cV_\la}(\R^d, L^\infty([0,T], \R^{\Pi e}))$ jointly measurable.
	Then $g(X)\in \wHoelHZ {\delt\ga}{\cV_{\ka\la}}(U, \Lfininf)$ with
	\[\nrm{g(X)}{\wHoelHZ {\delt\ga}{\cV_{\ka\la}},p} \lesssim \nrm{g}{\wHoelHZ \ga {\cV_\la}}\cV_\la(\nrm X {\wHoelHZ \delt {\cV_\ka},2p\la})(1\vee \nrm X {\wHoelHZ \delt {\cV_\ka},2p}^{\ga}),\]
	uniformly over random fields $X : \Om \times [0,T]\times U\to \R^d \in \wHoelHZ \delt {\cV_\ka}(U, \Lfininf)$, and jointly measurable functions $g : [0,T]\times \R^d \to \R^{\Pi e} \in \wHoelHZ \ga {\cV_\la}(\R^d, L^\infty([0,T], \R^{\Pi e}))$.
\end{lem}
\begin{proof}
By Lemma \ref{lem:gXwghtHoelIneq} (b)
	\begin{align*}\nrm{g(X(x)) - g(X(y))}{\Lpinf p} \leq & \nrm{g}{\wHoelHZ \ga {\cV_\la}} ({\cV_\la}(\nrm{X(x)}{\Lpinf {2p\la}}) + \cV_\la(\nrm{X(y)}{\Lpinf {2p\la}})) \nrm{X(x) - X(y)}{\Lpinf {2p}}^\ga \\
	\leq & \nrm{g}{\wHoelHZ \ga {\cV_\la}} ({\cV_\la}(\nrm{X}{\wCZ {\cV_\ka},{2p\la}}\cV_\ka(x)) + \cV_\la(\nrm{X}{\wCZ {\cV_\ka} ,{2p\la}}\cV_\ka(y))) \\
	& \cdot \nrm{X}{\wHoelHZ \delt {\cV_\ka},{2p\ga}}^\ga |x-y|^{\delt \ga} \\
	\lesssim & \nrm{g}{\wHoelHZ \ga {\cV_\la}} \cV_\la(\nrm{X}{\wCZ {\cV_\ka},{2p\la}})\cV_{\ka\la}(x,y)\nrm{X}{\wHoelHZ \delt {\cV_\ka},{2p}}^\ga |x-y|^{\delt \ga},
\end{align*}
uniformly over $x,y,X$ and $g$.
\end{proof}
The following lemma contains a stability property similar to Lemma \ref{lem:wCstab} (e), except for $g$ deterministic, $f = X$ random, and with index set $[0,T]\times \R^d$.
\begin{lem}
\label{lem:wHoelgRndFld}
Let $l,\ka,\la \in \N_0$, $X : \Om \times [0,T]\times U\to \R^d$ be a random field with $X\in \wHoel l \delt {\cV_\ka}(U, \Lfininf)$ and $g : [0,T]\times \R^d \to \R^{\Pi e} \in \wHoel l \ga {\cV_\la}(\R^d, L^\infty([0,T], \R^{\Pi e}))$ jointly measurable.
Then $g(X)\in \wHoel l {\delt\ga}{\cV_{\ka(\la+l)}}(U, \Lfininf)$ with
\[\nrm{g(X)}{\wHoel l {\delt\ga}{\cV_{\ka(\la+l)}},p} \lesssim \nrm{g}{\wHoel l \ga {\cV_\la}}\cV_\la(\nrm X {\wHoel l \delt {\cV_\ka},4p\la})(1\vee \nrm X {\wHoel l \delt {\cV_\ka},4p(1\vee l)}^{\ga \vee l}),\]
uniformly over random fields $X : \Om \times [0,T]\times U\to \R^d \in \wHoel l \delt {\cV_\ka}(U, \Lfininf)$, and jointly measurable functions $g : [0,T]\times \R^d \to \R^{\Pi e} \in \wHoel l \ga {\cV_\la}(\R^d, L^\infty([0,T], \R^{\Pi e}))$.
Further, we have $\E g(X) \in \wHoel l {\delt \ga}{\cV_{\ka(\la + l)}}(U, L^\infty([0,T],\R^{\Pi e}))$ with
\[\nrm{\E g(X)}{\wHoel l {\delt \ga}{\cV_{\ka(\la + l)}}}\lesssim \nrm{g}{\wHoel l \ga {\cV_\la}},\]
uniformly over jointly measurable functions $g : [0,T]\times \R^d \to \R^{\Pi e} \in \wHoel l \ga {\cV_\la}(\R^d, L^\infty([0,T], \R^{\Pi e}))$.
\end{lem}
\begin{proof}
Suppose $l=0$.
	By Lemma \ref{lem:gXwghtHoelIneq} (a)
	\begin{align*}
		\nrm{g(X(x))}{\Lpinf p} \leq & \nrm{g}{\wCZ {\cV_\la}} \cV_\la(\nrm{X(x)}{\Lpinf {p\la}})\\
		\leq & \nrm{g}{\wCZ {\cV_\la}} {\cV_\la}(\nrm{X}{\wCZ {\cV_\ka},p\la} {\cV_\ka}(x)) \\
		\lesssim &\nrm{g}{\wCZ {\cV_\la}} {\cV_\la}(\nrm{X}{\wCZ {\cV_\ka},p\la}) \cV_{\ka\la}(x),
	\end{align*}
	uniformly over $x, X$ and $g$. Together with Lemma \ref{lem:wHoelHgRndFld}, the statement for $l = 0$ follows.

For $l > 0$ we proceed similarly to the proof of Lemma \ref{lem:wCstab} (e).
Let $\al$ be a multi-index with $|\al|\leq l$. By Faa di Bruno's formula (see Proposition \ref{prop:diffRules} (c))
\[\der^\al(g(X)) = \sum_{k=1}^{|\al|} \sum_{\cB} \frac{\al!}{\cB!} \innp{\nabla^k g(X)}{\der^\cB X}.\]
Given a partition $\cB$ of $\al$ into $n$ multi-indices, we have, using Lemma \ref{lem:wCstab} (a) and (c)
\[\nrm{\der^\cB X}{\wHoelZ \delt {\cV_{\ka l}},p} \lesssim \nrm{\der^\cB X}{\wHoelZ \delt {\cV_{\ka n}},p} \lesssim \prod_{\be \in \cB} \nrm{\der^\be X}{\wHoelZ \delt {\cV_\ka},np} \leq \nrm{X}{\wHoel l \delt {\cV_\ka},np}^n \leq1 \vee \nrm{X}{\wHoel l \delt {\cV_\ka},lp}^l,\]
uniformly over $X$, for all $p\in \N$. Further, using the previous calculation
\[\nrm{\nabla^k g(X)}{\wHoelZ {\delt\ga}{\cV_{\ka\la}},p} \lesssim \nrm{\nabla^k g}{\wHoelZ \ga {\cV_\la}} \cV_\la(\nrm X {\wHoelZ \delt {\cV_\ka},2p\la})(1\vee \nrm X {\wHoelZ \delt {\cV_\ka},2p}^\ga),\]
uniformly over $X$ and $g$, for all $p\geq 1$.
Thus, using Lemma \ref{lem:wCstab} (d) with $q = r = 2$
\begin{align*}
	\nrm{\innp{\nabla^k g(X)}{\der^\cB X}}{\wHoelZ {\delt\ga}{\cV_{\ka \la}\cdot \cV_{\ka l}},p} & \lesssim \nrm{\nabla^k g}{\wHoelZ \ga {\cV_\la}} \cV_\la(\nrm X {\wHoelZ \delt \cV,4p\la})(1\vee \nrm X {\wHoelZ \delt {\cV_\ka},4p}^\ga) (1\vee \nrm{X}{\wHoel l \delt {\cV_\ka},2lp}^l)\\
	& \lesssim \nrm{g}{\wHoel l \ga {\cV_\la}} \cV_\la(\nrm X {\wHoel l \delt {\cV_\ka},4p\la})(1\vee \nrm X {\wHoel l \delt {\cV_\ka},4p(1\vee l)}^{\ga \vee l}),
\end{align*}
uniformly over $g$ and $X$.	
By Lemma \ref{lem:LiplAltRep} (b) and Lemma \ref{lem:wCstab} (a), (b), the desired result follows.
\end{proof}

Fix a filtration $\cF =(\cF_t)_{t\in [0,T]}$ on $(\Om, \cF_\Om, \P)$ satisfying the usual conditions.
We denote by $\Lfininfad$ the subset of $\Lfininf$ of $\cF$-adapted processes, and similarly we define $\Lpinfad p$ and $\Lpinfadt p t$.

\begin{lem}
	\label{lem:intgpInfnrmIneq}The following estimates hold true.
	\begin{enumerate}[(a)]
		\item $\nrm{\int_0^\blnk Y_s \,ds}{\Lpinft pt} \leq\int_0^t \nrm{Y_s}{p}\,ds\leq \int_0^t \nrm{Y}{\Lpinft p s}\,ds$, for all $p\geq 1, t\in [0,T]$ and $Y\in \Lpinft p t(\Om, \R^d)$.
		\item $\nrm{\int_0^\blnk Y_s \,ds}{\Lpinft pt} \lesssim \left(\int_0^t \nrm{Y}{\Lpinft p s}^2\,ds\right)^{1/2}$,
		uniformly over $t\in [0,T]$ and $Y\in \Lpinft p t(\Om, \R^d)$, for all $p\geq 1$.
		\item $\nrm{\int_0^\blnk Z_s \,dW_s}{\Lpinft pt} \lesssim \left(\int_0^t \nrm{Z_s}{p}^2\,ds\right)^{1/2}\leq \left(\int_0^t \nrm{Z}{\Lpinft p s}^2\,ds\right)^{1/2}$,
		uniformly over $t\in [0,T]$ and $Z\in \Lpinfadt p t(\Om, \R^{d\times d})$, for all $p\geq 2$.
		\item $\nrm{\int_0^\blnk Y_s\,ds+ \int_0^\blnk Z_s \,dW_s}{\Lpinft pt}^2\lesssim \int_0^t \nrm{Y}{\Lpinft p s}^2 + \nrm{Z}{\Lpinft p s}^2\,ds$,
		uniformly over $t\in [0,T], Y\in \Lpinfadt p t(\Om, \R^d)$ and $Z\in \Lpinfadt p t(\Om, \R^{d\times d})$, for all $p\geq 2$.
	\end{enumerate}
\end{lem}
\begin{proof}
Let $p\geq 1$ and $Y\in \Lpinft p t$. By Minkowski's integral inequality we have \[\nrm{\nrm{Y}{L^1([0,t])}}{L^p(\Om)} \leq \nrm{\nrm{Y}{L^p(\Om)}}{L^1([0,t])}.\] Hence,
	\[\nrm{\int_0^\blnk Y_s \,ds}{L_p^{*t}} \leq \nrm{\int_0^t |Y_s|\,ds}{p} \leq\int_0^t \nrm{Y_s}{p}\,ds\leq \int_0^t \nrm{Y}{L_p^{*s}}\,ds,\]
proving (a). Note that 
\[\left(\int_0^t u_s\,ds \right)^2 \leq t\int_0^t u_s^2\,ds \leq T\int_0^t u_s^2\,ds, \quad u \in L^2([0,t]), t\geq 0,\]
by Jensen's inequality.
Thus, (b) follows from (a).

For $p\geq 2$, we have
\[\nrm{\nrm{Z}{L^2([0,t])}}{L^p(\Om)} \leq \nrm{\nrm{Z}{L^p(\Om)}}{L^2([0,t])},\]
again by Minkowski's integral inequality.
Thus, using the Burkholder–Davis–Gundy inequality
	\[\nrm{\int_0^\blnk Z \,dW}{L_p^{*t}} \lesssim \E\left[\left(\int_0^t |Z_s|^2\,ds\right)^{p/2}\right]^{1/p} \leq \E\left[\left(\int_0^t |Z_s|^p\,ds\right)^{2/p}\right]^{1/2}  \leq \left(\int_0^t \nrm{Z}{L_p^{*s}}^2\,ds\right)^{1/2},\]
	uniformly over $t\in [0,T]$ and $Z\in \Lpinfadt p t$. This proves (c).
Using the inequality $(x+y)^2 \leq 3(x^2+y^2)$, we conclude (d) from (b) and (c).
\end{proof}

\subsection{Differentiating \tSDE s}
In the following, we consider family of \tSDE s parameterized by a set 
\[\Thet = I_1\times\cdots\times I_d\times \R^{m-d} \subseteq \R^m\] 
for some $m\in \N$ and bounded intervals $I_1,\dots, I_d$.
\begin{satz}
	\label{thm:SDESol}
	Suppose $b \in \Lip(\Thet\times \R^d, L^\infty([0,T], \R^d))$, $\si \in \Lip(\Thet\times \R^d, L^\infty([0,T], \R^{d\times d}))$, and $\ph \in \hoelHZ \delt(\Thet,\Lfininfad)$. Then, the family of \tSDE s
	\[X_t = \ph_t + \int_0^tb_s(X_s)\,ds + \int_0^t\si_s(X_s)\,dW_s,\quad t\in [0,T]\]
	admits a unique solution $X : \Om \times [0,T]\times \Thet \to \R^d$. That is,
	\[X_t(x) = \ph_t(x) + \int_0^tb_s(x,X_s(x))\,ds + \int_0^t\si_s(x,X_s(x))\,dW_s,\quad t\in [0,T]\]
	up to indistinguishability, for all $x\in \Thet$, and given another solution $Y :  \Om \times [0,T]\times \Thet \to \R^d$ we have 
	\[X_t(x) = Y_t(x),\quad t\in [0,T],\]
	up to indistinguishability, for all $x\in \Thet$.
  Further, $X\in \hoelHZ \delt(\Thet,\Lfininfad)$ and for all $p\geq 2$ there exists a constant $c > 0$ such that
	\[\nrm{X}{\hoelHZ \delt(\Thet, \Lpinft p t)} \lesssim \nrm{\ph}{\hoelHZ \delt(\Thet, \Lpinft p t)} e^{c (\nrm{b}{\Lip}^2 + \nrm{\si}{\Lip}^2)},\]
	uniformly over $b \in \Lip(\Thet\times \R^d, L^\infty([0,T], \R^d))$, $\si \in \Lip(\Thet\times \R^d, L^\infty([0,T], \R^{d\times d}))$, \\$\ph \in \hoelHZ \delt(\Thet,\Lfininfad)$ and $t\in [0,T]$.
\end{satz}

\begin{proof}
	Existence, uniqueness and adaptedness is essentially due to a standard result, cf.\ \citet[Theorem 3.1 and 3.2]{Kunita2004} for example. The extension from an initial value $x\in \R^d$ to a process $\ph$ is discussed in \citet[Theorem 18 and 19]{Li18}.
	
	We prove $X\in \hoelHZ \delt(\Thet,\Lfininf)$ and the inequality. 
	Let $p\geq 2$. By Lemma \ref{lem:intgpInfnrmIneq} (d) and Lemma \ref{lem:gXwghtHoelIneq} (b) with $q=\infty$ and $r = 1$
	\begin{align}
		\label{eq:SDESolproof}
		\nrm{X(x) - X(y)}{\Lpinft p t}^2 \lesssim & \nrm{\ph(x) - \ph(y)}{\Lpinft p t}^2 + \int_0^t\nrm{b(X(x)) - b(X(y))}{\Lpinft p s}^2\,ds \nonumber \\
		&+ \int_0^t\nrm{\si(X(x)) - \si(X(y))}{\Lpinft p s}^2\,ds\nonumber \\
		\leq & \nrm{\ph}{\hoelHZ \delt(\Thet, \Lpinft p t)}^2|x-y|^{2\delt}+ (\nrm{b}{\Lip}^2+\nrm{\si}{\Lip}^2) \int_0^t\nrm{X(x) - X(y)}{\Lpinft p s}^2\,ds,
	\end{align}
uniformly over $x,y,t,b,\si$ and $\ph$.	
Grönwall's inequality implies
	\[\nrm{X(x) - X(y)}{\Lpinft p t}^2 \lesssim \nrm{\ph}{\hoelHZ \delt(\Thet, \Lpinft p t)}^2 e^{c(\nrm{b}{\Lip}^2 + \nrm{\si}{\Lip}^2)}|x-y|^{2\ga},\]
uniformly over $x,y,t,b,\si$ and $\ph$, for some $c > 0$.
	Taking the square root gives the desired estimate. Note that the estimate for large $p$ (here $p\geq 2$) suffices to show $X\in \hoelHZ \delt(\Thet,\Lfininf)$.
\end{proof}

\begin{docu}
	At one point I did the estimate like this
	Thus, there exist a constant $c > 0$, such that for all $x,y\in D$ and $t\in [0,T]$
	\begin{align*}
		\nrm{X(x) - X(y)}{L_p^{*t}} \lesssim & \nrm{\ph(x) - \ph(y)}{L_p^{t*}} + T^p \nrm{b(X(x)) - b(X(y))}{L_p^{t*}} \\
		&+ T^{p/2} \nrm{\si(X(x)) - \si(X(y))}{L_p^{t*}} \\
		\lesssim & \nrm{\ph}{L_p^{*t}\Lip^\ga}|x-y|^\ga + (T^p \nrm{b}{\Lip} + T^{p/2} \nrm{\si}{\Lip})\nrm{X(x) - X(y)}{L_p^{t*}}.
	\end{align*}
	Seems strange, but it is correct. However, this only gives you a bound on $\nrm{X(x) - X(y)}{L_p^{*t}}$ if $(T^p \nrm{b}{\Lip} + T^{p/2} \nrm{\si}{\Lip}) < 1$.
\end{docu}

\begin{docu}
	We need the following
	\[\nrm{f}{\Lip^{l+\delt}} = \begin{cases}
		\nrm{\nabla f}{\wHoel{l-1} \delt b}, & l,\delt > 0,\\
		\nrm{\nabla f}{\wHoelZ{l-2} b}, & l\geq 2, \delt = 0.
	\end{cases}\]
\end{docu}
\begin{satz}
	\label{thm:SDEDer}
	Let $l\in \N_0, \delt \in (0,1]$ with $l+\delt \geq 1$, $b \in \Lip^{l+\delt}(\Thet\times \R^d, L^\infty([0,T], \R^d))$, $\si \in \Lip^{l+\delt}(\Thet\times \R^d, L^\infty([0,T], \R^{d\times d}))$, and $\ph \in \Lip^{l+\delt}(\Thet,\Lfininfad)$. Let $X : \Om \times [0,T]\times \Thet \to \R^d$ be the unique solution to the family of \tSDE s
	\[X_t = \ph_t + \int_0^t b_u(\blnk, X_u)\,du + \int_0^t\si_u(\blnk, X_u)\,dW_u, \quad t\in [0,T].\]
	Then,
	\[X : \Om \times [0,T]\times \Thet \to \R^d, (\om,t,x)\mapsto X_t(x) \in \Lip^{l+\delt}(\Thet,\Lfininfad).\]
	Further, for every multi-index $\al$ with $|\al| \leq l$, $\der^\al X$ satisfies the \tSDE{}
	\[\der^\al X_t = \der^\al \ph_t + \int_0^t \der^\al(b_u(\blnk, X_u))\,du + \int_0^t \der^\al(\si_u(\blnk, X_u))\,dW_u, \quad t\geq 0.\]
\end{satz}
\begin{proof}
By extending $b,\si$ and $\ph$ to an open neighborhood of $\Thet \times \R^d$, we can assume wlog that $\Thet$ is open.
	To show $X\in \dC l$ and SDE representation of $\der^\al X$ cf.\ \cite{Kunita2004} Theorem 3.4. They cover the case $l = 1$ and, where the dependence is only on the initial condition. The generalization to our case is straightforward (although cumbersome).
	
	We show $X\in \Lip^{l+\delt}$ by proving inductively: For all $k\in \N_0$ with $k\leq l$ we have $X\in \Lip^{k+\ga}$, where
	\[\ga = \begin{cases}
	1, & k < l,\\
	\delt, & k = l.
	\end{cases}\]
	By Theorem \ref{thm:SDESol}, $X\in \hoelHZ \ga = \Lip^\ga$ and so the statement is true for $k = 0$. Now, assume $k\in \N, k \leq l$ and that $X\in \Lip^{k}$.
Recall that
	\[\Lip^{k+\ga} = \set{f\in \wHoelH k \ga {} : \nabla^m f \in \Lip, m = 0,\dots, k-1} = \set{f\in \wHoelH k \ga {} : f\in \Lip^k}.\]
	Thus, it suffices to prove $X\in \wHoelH k \ga {}$ or equivalently $\der^\al X \in \wHoelHZ \ga {}$ for all multi-indices $|\al|$ of size $k$.
	
Let $\al$ be a multi-index of size $k$.
	We write $f^1 = b, f^2 = \si, W^1_t = t$ and $W^2 = W$ in the following.
	Firstly, note that Faa di Bruno's formula implies
	\[\der^\al f^i(\blnk, X) = \sum_{\be \leq \al} \sum_{j=0}^{|\al - \be|} \sum_{\cB\in \cS_j^{\al - \be}} \frac{(\al - \be)!}{\cB!} \innp{\der^\be_x \nabla^j_y f^i(\blnk, X)}{\der^\cB X}.\]
	Here, we denote derivatives with respect to the two coordinates of $f^i$ with indices $x$ and $y$ respectively.
	For most of the terms of the form $\der^\cB X$, we can use the inductive assumption to estimate the corresponding summands above. The only exception is for $|\be| = 0$ and $j = 1$. In this case $|\prttn{\al-\be} j| = 1$ and the inner sum is given by the sole summand
	\[\innp{\nabla_y f^i(\blnk, X)}{\der^\al X}.\]
	That is, only in this case does the inner sum depend on the left-hand side of the differential equation.
	Define
	\[\psi_t^i = \int_0^t (\der^\al(f^i_u(\blnk, X_u)) - \innp{\nabla_y f^i_u(\blnk, X_u)}{\der^\al X_u})\,dW^i_u, \quad i = 1,2,\]
	Then
	\begin{equation}
		\label{eq:deralXAlt}
		\der^\al X_t = \der^\al\ph_t + \sum_{i=1}^2 \psi_t^i + \int_0^t \innp{\nabla_y f_u^i(X_u)}{\der^\al X_u}\,dW^i_u, \quad t\in [0,T],
	\end{equation}
	and only the last summand depends on $\der^\al X$.
	Now, let $p\geq 1$.
	\paragraph{Estimating $\psi$:}
	By defining
	\[A(\al) = \set{(\be, j, \cB) : \be \leq \al, j \in \set{0,\dots, |\al-\be|}, \cB\in \cS_j^{\al - \be}, |\be| \geq 1 \text{ or } j \geq 2}\]
	we can write
	\[\psi^i_t = \int_0^t \sum_{(\be,j,\cB) \in A(\al)} \frac{(\al - \be)!}{\cB!} \innp{\der^\be_x \nabla^j_y f^i(\blnk, X)}{\der^\cB X} \,dW^i_u, \quad t\geq 0.\]
	Let $(\be, j, \cB) \in A(\al)$. Note that $\al \notin \cB$ and so $|\tilde \be|<k$ for all $\tilde \be \in \cB$.
	Note that by Lemma \ref{lem:wHoelHgRndFld}
\begin{align*}
\nrm{g(X)}{\wHoelZ \ga b,p} \leq & \nrm{g(X)}{\wHoelHZ \ga {},p} \vee \nrm{g(X)}{\wCZ b,p}\\
\leq & \nrm{g}{\wHoelHZ \ga {}}(1 \vee \nrm{X}{\Lip,4p})\vee \nrm{g}{\wCZ b}\\
\lesssim & \nrm{g}{\wHoelZ \ga b}(1 \vee \nrm{X}{\Lip,4p}),
\end{align*}
uniformly over $g\in \wHoelZ \ga b(\Thet \times \R^d, L^\infty([0,T], E))$ with $E\in \set{\R^d,\R^{d\times d}}$. Further, by Lemma \ref{lem:wghtdRdemchr}
\[\max_{1\leq m\leq k} \nrm{\nabla^m g}{\wHoelZ \ga b} \lesssim \max_{1\leq m<k}\nrm{\nabla^m g}{\Lip_b}  \vee \nrm{g}{\wHoel k \ga b}  \asymp \max_{m<k}\nrm{\nabla^m g}{\Lip} \vee \nrm{g}{\wHoel k \ga b} \asymp  \nrm{g}{\Lip^{k+\ga}},\]
uniformly over $g\in \wHoelZ \ga b(\Thet \times \R^d, L^\infty([0,T], E)), E\in \set{\R^d,\R^{d\times d}}$.
Similarly $\max_{1\leq m \leq k-1} \nrm{\nabla^m X}{\wHoelZ \ga b} \lesssim \nrm{X}{\Lip^{k}}$, uniformly over $X\in \Lip^k$.
	 Thus, by Lemma \ref{lem:wCstab} (c) and (d)
	\begin{align*}
		\nrm{\innp{\der^\be_x\nabla^j_y f^i(\blnk, X)}{\der^\cB X}}{\wHoelZ \ga b,p} \lesssim &\nrm{\der^\be_x \nabla^j_y f^i(\blnk, X)}{\wHoelZ \ga b,2p}\nrm{\der^\cB X}{\wHoelZ \ga b,2p} \\
		\lesssim & (\nrm{\der^\be_x \nabla^j_y f^i}{\wHoelZ \ga b} (1\vee\nrm{X}{\Lip,8p}) \prod_{\tilde \be \in \cB} \nrm{\der^{\tilde \be} X}{\wHoelZ \ga b,2pj}\\
		\leq & \nrm{f^i}{\Lip^{k+\ga}} (1\vee \nrm{X}{\Lip,8p})\nrm{X}{\Lip^{k},2pj}^j\\
		\lesssim & \nrm{f^i}{\Lip^{k+\ga}}(1 \vee \nrm{X}{\Lip^{k},8pk}^k),
	\end{align*}
uniformly over $f^i\in \Lip^{k+\ga}(\Thet \times \R^d, L^\infty([0,T],\R^{d^{\times i}}))$ and $X\in \Lip^k$, for $i=1,2$.
Hence,
	\begin{align*}
		\nrm{\psi^i}{\wHoelZ \ga b, p} \lesssim & \sum_{(\be,j,\cB) \in A(\al)} \nrm{\innp{\der^\be_x \nabla^j_y f^i(X)}{\der^{\cB} X}}{\wHoelZ \ga b, p}\\
		\lesssim & \nrm{f^i}{\Lip^{k+\ga}} (1\vee \nrm{X}{\Lip^{k},8pk}^k),
	\end{align*}
uniformly over $f^i\in \Lip^{k+\ga}$ and $X\in \Lip^{k}$, for $i = 1,2$.
	\paragraph{Estimating $\int\innp{\nabla_y f(X)}{\der^\al X}$:}
	For $g\in\wHoelZ \ga b(\Thet \times \R^d, L^\infty([0,T], E)), E\in \set{\R^d,\R^{d\times d}}$, we have
	\begin{align*}
		|\innp{g_t(x,y)}{z} - \innp{g_t(x',y')}{z'}| \leq & |\innp{g_t(x,y) - g_t(x',y')}{z} - \innp{g_t(x',y')}{z - z'}| \\
		\leq & \nrm g {\wHoelZ \ga b}|z|\left|\mat{x-x'\\y-y'}\right|^\ga + \nrm g {\wHoelZ \ga b}|z-z'|\\
		\leq & \nrm{g}{\wHoelZ \ga b}(|z||x-x'|^\ga + |z||y-y'|^\ga + |z-z'|),
	\end{align*}
	for $x,x'\in \Thet, y,y' \in \R^d$, and $z,z'\in E$.
	Hence, using the Hölder inequality
	\begin{align*}
		&\nrm{\innp{g(x,X(x))}{\der^\al X(x)} - \innp{g(x',X(x'))}{\der^\al X(x')}}{\Lpinft p t}\\
		\lesssim & \nrm{g}{\wHoelZ \ga b}\nrm{\der^\al X(x)}{\Lpinft {2p} t} (|x-x'|^\ga + \nrm{X(x)-X(x')}{\Lpinft {2p\ga} t}^\ga)+ \nrm{g}{\wHoelZ \ga b}\nrm{\der^\al X(x)-\der^\al X(x')}{\Lpinft p t}\\
		\lesssim &\nrm{g}{\wHoelZ \ga b} \nrm{\der^\al X}{\wCZ b,2p} (1 + \nrm{X}{\Lip,2p}^\ga) |x-x'|^\ga + \nrm{g}{\wHoelZ \ga b}\nrm{\der^\al X(x)-\der^\al X(x')}{\Lpinft p t},
	\end{align*}
uniformly over $t\in [0,T], x,x'\in \Thet, g\in \wHoelZ \ga b$ and $X\in \Lip^{k}\cap \dC k$.
Moreover, by Lemma \ref{lem:wghtdRdemchr}
\begin{align*}
\nrm{\der^\al X}{\wCZ b,2p} (1 + \nrm{X}{\Lip,2p}^\ga) \lesssim & \nrm{\nabla^{k-1}X}{\Lip,2p} (1 + \nrm{X}{\Lip,2p}^\ga) \lesssim 1 \vee \nrm{X}{\Lip^k,2p}^{1+\ga},
\end{align*}
uniformly over $X\in \Lip^k$.
	Therefore, by Lemma \ref{lem:intgpInfnrmIneq} (c)
	\begin{align*}
		&\nrm{\int_0^\blnk \innp{\nabla_y f_u^i(x,X_u(x))}{\der^\al X_u(x)} - \innp{\nabla_y f_u^i(x', X_u(x'))}{\der^\al X_u(x')}\,dW^i_u}{\Lpinft p t}^2 \\
		\lesssim & \int_0^t \nrm{\innp{\nabla_y f^i(x, X(x))}{\der^\al X(x)} - \innp{\nabla_y f^i(x', X(x'))}{\der^\al X(x')}}{\Lpinft p s}^2\,ds\\
		\lesssim & \nrm{\nabla_y f^i}{\wHoelZ \ga b}^2 (1 \vee \nrm{X}{\Lip^k,2p}^{1+\ga})^2 |x-x'|^{2\ga} + \nrm{\nabla_y f^i}{\wHoelZ \ga b}^2\int_0^t\nrm{\der^\al X(x)-\der^\al X(x')}{\Lpinft p s}^2 \,ds,
	\end{align*}
uniformly over $t\in [0,T], x,x'\in \Thet, f^i \in \Lip^{1+\delt}_b$ and $X\in \Lip^{k}\cap \dC k$, for $i =1,2$.
	\paragraph{Finishing the inductive step}
	By Equation \eqref{eq:deralXAlt} and Lemma \ref{lem:intgpInfnrmIneq} 
	\begin{align*}
\nrm{\der^\al X(x) - \der^\al X(x')}{\Lpinft p t}^2\lesssim & \nrm{\der^\al \ph}{\wHoelZ \ga b}^2 + \sum_{i=1}^2 (\nrm{\psi^i}{\wHoelZ \ga b}^2 +  \nrm{\nabla_y f^i}{\wHoelZ \ga b}^2 (1 \vee \nrm{X}{\Lip^k,2p}^{1+\ga})^2) |x-x'|^{2\ga} \\
&+ \sum_{i=1}^2\nrm{\nabla_y f^i}{\wHoelZ \ga b}^2\int_0^t\nrm{\der^\al X(x)-\der^\al X(x')}{\Lpinft p s}^2 \,ds\\
		\lesssim & |x-x'|^{2\ga}+ \int_0^t\nrm{\der^\al X(x)-\der^\al X(x')}{\Lpinft p s}^2\,ds,
	\end{align*}
uniformly over $t\in [0,T], x,x'\in \Thet$ and $X\in \Lip^{k}\cap \dC k$.
	Thus, by Grönwall's inequality
	\begin{align*}
		\nrm{\der^\al X(x) - \der^\al X(x')}{\Lpinft p t}^2\lesssim & |x-x|^{2\ga} \unfovr x,x'\in \Thet, X\in \Lip^{k}\cap \dC k.
	\end{align*}
Hence, $\der^\al X \in \wHoelHZ \ga {}$ for all multi-indices $\al$ of size $k$, as desired.
\end{proof}

\begin{docu}
	I tried briefly generalizing the proof to Lipschitz up to linear growth. The problem is that when we estimate $\der^\al X$ the left-hand side features an $L_p$ norm, while the RHS has a $L_{2p}$ or $L_{3p}$ norm after applying Hölders inequality, due to the extra $V(\nrm{X}{})$ terms, etc.. So we cannot apply Gröwnwall. 
\end{docu}

Using Theorem \ref{thm:SDEDer} we can deduce the smoothness of solutions of SDEs in the initial condition $x$, the initial time point $s$ and a small parameter $\ep$.
\begin{cor}
	\label{cor:SDEDerInitPert}
	Let $l\in \N_0, \delt \in (0,1], \Thet := [0,1]\times [0,T] \times \R^d$ and suppose we are given functions
	\begin{align*}
		b :&\, [0,T]\times [0,1]\times \R^d \to \R^d, (t,\ep, x)\mapsto b^\ep_t(x) \in \Lip^{l+\delt}(\Thet\times \R^d, L^\infty([0,T], \R^d)), \\
		\si :&\, [0,T]\times [0,1]\times \R^d \to \R^{d\times d}, (t,\ep, x)\mapsto \si^\ep_t(x) \in \Lip^{l+\delt}(\Thet\times \R^d, L^\infty([0,T], \R^{d\times d})).
	\end{align*}
	Let $X : \Om \times [0,T]\times \Thet \to \R^d$ be the unique solution to the family of \tSDE s
	\[X_t^{\ep,s}(x) = x + \int_s^t b_u^\ep(X_u^{\ep,s}(x))\,du + \int_s^t\si_u^\ep(X_u^{\ep,s}(x))\,dW_u, \quad t\in[0,T],\]
	with the convention that $\int_s^t\,du = \int_s^t\,dW_u = 0$ for $s > t$.
	Then,
	\[X : \Om \times [0,T]\times \Thet \to \R^d, (\om,t,(\ep,s,x))\mapsto X_t^{\ep,s}(x) \in \Lip^{l+\delt}(\Thet,\Lfininfad).\]
	Further, for all $k,m\in \N_0$ and every multi-index $\al$ with $|\al| \leq l - k - m $, $\der^\al_x \der_\ep^k \der_s^m X$ satisfies the \tSDE{}
	\[\der^\al_x \der_\ep^k \der_s^m X_t = \ph^{\al,k,m} + \int_s^t \der^\al_x \der_\ep^k \der_s^m (b_u^\ep(X_u^{\ep,s}))\,du + \int_s^t \der^\al_x \der_\ep^k \der_s^m (\si_u^\ep(X_u^{\ep,s}))\,dW_u, \quad t\in[0,T].\]
	Here, 
	\[\ph^{\al,k,m}(x) = \begin{cases}
	x, & |\al| = k = m = 0,\\
	e_j & \al = \set{j}, k = m = 0,\\
	0, & |\al|>1 \text{ or } k + m > 0,
	\end{cases}\quad x\in \Thet.\]
\end{cor}
\begin{proof}
	We augment the state space $\R^d$ by time. 
	Define
	\[\tilde b^\ep(t,x) = \mat{1\\ b_t^\ep(x)}, \tilde \si^\ep(t,x) = \mat{0 \\ \si_t^\ep(x)}, \quad t\in [0,T], x \in \R^d.\]
	Consider the family of \tSDE s
	\[\tilde X_t(\ep,s,x) = \mat{s\\x} + \int_0^t \tilde b^\ep(\tilde X_u(\ep,s,x))\,du + \int_0^t \tilde \si^\ep(\tilde X_u(\ep,s,x))\,dW_u.\]
	The assumptions of Theorem \ref{thm:SDEDer} are satisfied for this system.
	Write $\tilde X = (Y,Z)$ with $Y$ one-dimensional.
	By construction, $Y_t(\ep,s,x) = s + t$. Thus,
	\begin{align*}
		\tilde X_t(\ep,s,x) = \mat{s+t\\ Z_t(\ep,s,x)} = & \mat{s+t\\ x + \int_0^t b_{u+s}^\ep(Z_u(\ep,s,x))\,du + \int_0^t\si_{u+s}^\ep(Z_u(\ep,s,x))\,dW_u},
	\end{align*}
	and so $Z_{t-s}(\ep,s,x) = X_t^{\ep,s}(x)$, up to indistinguishability in $t\in [0,T]$, for all $(\ep,s,x)\in \Thet$. Thus, the desired properties for $X$ follow from the ones for $\tilde X$.
\end{proof}
\section{Stochastic modified equations}
\subsection{Introduction and main results}
\label{sec:fstOrderSMEs}
For the remainder of this chapter, we write $\wC l \ka := \wC l {\cV_\ka}$, where $\cV_\ka(x) = 1 + x^\ka$, and $\wCZ \ka := \wC 0 \ka$.
We work again in the setting of Subsection \ref{sec:onstepassum}. In addition to $f, \bar f$ and $\Si$, we consider functions $b : [0,T]\times \R^d\to \R^d$ and $D : [0,T]\times \R^d \to \R^{d\times d}$.\\
\begin{assum}
\label{assum:bD}
$D_t(x)$ is a symmetric positive semi-definite matrix for all $(t,x)\in [0,T]\times \R^d$.
We have $b \in \Lip^{5+1}([0,T]\times \R^d, \R^d)$ and $\sqrt D \in \Lip^{5+1}([0,T]\times \R^d, \R^{d\times d})$.\\
\end{assum}
By Assumption \assref{assum:bD}, we have $\sqrt D \in \wC 5 1$. Hence, $D\in \wC 5 2$ by Lemma \ref{lem:plyGrwthStab} (iii).
For all $h \in [0,1]$ we consider the following family of \tSDE s
\begin{align}\label{eq:NCC-SGF}
	dX_t^h = \bar f_t(X^h_t)+hb_t(X_t^h)\,dt + \sqrt{hD_t(X^h_t)}\,dW_t, \quad t\in [0,T].
\end{align}
Notice that, as $h \downarrow 0$, Equation \eqref{eq:NCC-SGF} becomes the ODE
\[\der_t X_t^0 = \bar f_t(X^0_t), \quad t\in [0,T],\]
which we investigated in Section \ref{sec:ODE}.
Denote by $X^{h,t}(x)$ the (essentially) unique solution of \eqref{eq:NCC-SGF} with $X_t^{h,t}(x) = x$ and set $X_r^{h,t}(x) = x$ for $r < t$.
Given $g\in \dC 2(\R^d)$ we define 
\[v^g : [0,1]\times [0,T]\times [0,T]\times \R^d \to \R, (h,r,t,x) \mapsto v^{g,h,r}_t(x) := \E g(X_r^{h,t}(x)).\]
We also write $v^{g,h} = v^{g,h,T}$ and $v^{h} = v^{g,h}$ if the choice of $g$ is clear from the context.
Notice that $v^{g,0}_t(x) = g(X^{0,t}_T(x))$.
\begin{theorem}
	\label{thm:firstOrderSME}
	Assume \assref{assum:H} and \assref{assum:bD}.
	Then for all $g\in \wC 5 \ka(\R^d)$ there exists a function $\rho^{g} : \lrates\to \wCZ{\ka +13}(\R^d), h \mapsto \rho^{g,h}$, such that
	\begin{align}\label{better 1st order expansion}
		\E g(\chi_{T/h}^h) -  \E g(X_T^{h})= \frac12 h \int_0^T (\innp{\nabla^2 v^{g,0}}{\Si-D} - \innp{\nabla v^{g,0}}{2 b + \nabla^\transp \bar f \bar f + \der_t \bar f})_t(X_t^0)\,dt + h^2 \rho^{g,h},
	\end{align}
and
	\[\nrm{\rho^{g,h}}{\wCZ{\ka+13}} \lesssim \nrm{g}{\wC 5 \ka}\]
	uniformly over $g\in \wC 5 \ka$ and $h\in \lrates$.
\end{theorem}

By choosing $b = 0$ and $D = \Si$ we recover the first-order \tSME{} introduced by \citet[][Theorem 1]{Li15}, with the difference that our $f$ can be time-dependent and does not need to be a gradient field.
Alternatively, we may choose a state-independent diffusion coefficient such as $\Si(x^*)$ for some $x^*\in \R^d$. The resulting equation is easier to work with, but may still provide important insight. Moreover, in some cases setting $(b,D) = (0, \Si(x^*))$ gives us a smaller absolute linear error term than setting $(b,D) = (0, \Si)$ (see Chapter \ref{chap:compare}).

By choosing $b = -\frac12 \nabla^\transp \bar f \bar f + \der_t \bar f$ and $D = \Si$ we recover the second-order \tSME{} introduced by \citet[][Theorem 1]{Li15}.
\begin{cor}
\label{cor:SME2}
Assume \assref{assum:H}, $\bar f \in \Lip^{6+1}([0,T]\times \R^d)$ and $\sqrt \Si\in \Lip^{6+1}([0,T]\times \R^{d\times d})$. Suppose $X$ is the solution to the family of \tSDE{}
\begin{equation}
\label{eq:SME2}
dX_t^h = \left(\bar f_t - \frac12 h (\nabla^\transp \bar f_t \bar f_t + \der_t \bar f_t)\right)(X_t^h)\,dt + \sqrt{h\Si(X_t^h)}\,dW_t, \quad t\in [0,T].
\end{equation}
Then for all $g\in \wC 5 \ka(\R^d)$ we have
\[\nrm{\E g(\chi_{T/h}^h) -  \E g(X_T^{h})}{\wCZ{\ka+13}} \lesssim h^2\nrm{g}{\wC 5 \ka},\]
uniformly over $g\in \wC 5 \ka$ and $h\in \lrates$.
\end{cor}
\begin{proof}
Apply Theorem \ref{thm:firstOrderSME} in the case $b = -\frac12 \nabla^\transp \bar f \bar f + \der_t \bar f$ and $D = \Si$.
\end{proof}

Corollary \ref{cor:SME2} gives a nice interpretation of the linear error term in Equation \eqref{eq:NCC-SGF}. For general (stochastic) modified equations of the form \eqref{eq:NCC-SGF} it measures how much our approximation differs from the second-order SME \eqref{eq:SME2}.

\subsection{Proof of Theorem \ref{thm:firstOrderSME}}
As in Subsection \ref{sec:ODE}, the linear error term in Equation \eqref{better 1st order expansion} is determined by a linear operator
\[\cF : \dC 2([0,T]\times \R^d) \to \cC([0,T]\times \R^d)\]
which is now given by
\begin{align}
	\label{eq:phiN}
	\cF w := & \frac12 \innp{\nabla^2 w}{\outSq{\bar f} + \Si - D} - \innp{\nabla w}{b} + \innp{\partial_t \nabla w}{\bar{f}} +\frac12 \partial_t^2 w,\quad w\in \dC 2([0,T]\times \R^d).
\end{align} 
We write $\cF_t w(x) := (\cF w)(t,x)$ for all $(t,x) \in [0,T] \times \R^d$. By Lemma \ref{lem:FopSimpl}, we have
\[\cF {v^{g,0}} = \frac12 \innp{\nabla^2 v^{g,0}}{\Si - D} - \frac12\innp{\nabla v^{g,0}}{2 b + \nabla^\transp \bar f \bar f + \der_t \bar f},\]
for all $g\in \dC 2$.
Similar to the proof of Lemma \ref{lem:TTOperator}, we can show
\[\nrm{\cF w}{\wC l {\ka + 2}} \lesssim \nrm{w}{\wC{l+2}{\ka}}, \unfovr w\in \wC{l+2}\ka([0,T]\times \R^d),\]
for all $\ka,l \in \N_0$ with $l\leq 5$.

Note that given $g\in \dC 2$ and $r\in [0,T]$, the function $v = v^{g,h,r}$ satisfies the Kolmogorov backward equation
\begin{equation}
	\label{eq:KBESME1}
	\begin{cases}
		\der_t v + \innp{\nabla v}{\bar f} + h \innp{\nabla v}{b} + \frac 1 2 h \innp{\nabla^2 v}{D} = 0, & t\in [0,r],\\
		v_r = g,
	\end{cases}
\end{equation}
for all $h\in [0,1]$.

\begin{lemma}
	\label{lem:OMEvsSME1}
	For all $l\leq 5$ and $g\in \wC l \ka(\R^d)$, we have $v^{g,h} \in \wC l {\ka + l}([0,T]\times \R^d, L^\infty([0,T], \R))$, uniformly in $h\in [0,1]$, with
	\[\nrm{v^{g,h}}{\wC l {\ka + l}([0,T]\times \R^d, L^\infty([0,T], \R))} \lesssim \nrm{g}{\wC l \ka},\]
	uniformly over $h\in [0,1]$ and $g\in \wC l \ka$.
	 Moreover,
	\[\nrm{v^{g,0} - v^{g,h}}{\wC l{\ka + 2}([0,T]\times \R^d,L^\infty([0,T], \R))} \lesssim h\nrm{g}{\wC {l+2} \ka},\]
	uniformly over $g\in \wC{l+2} \ka$ and $h\in [0,1]$, for $l = 0,1,2$.
\end{lemma}
\begin{proof}
	We apply Corollary \ref{cor:SDEDerInitPert} to the family of \tSDE s \eqref{eq:NCC-SGF} with $\ep = \sqrt h$.
	More precisely, let 
	\[\Thet = \set{(\ep, s, x) : \ep \in [0,1], s\in [0,T], x\in \R^d}.\]
	We consider the $\Thet \times \R^d$-indexed family of \tSDE s
	\[dY_t^{\ep,s}(x) = \bar f_t(Y_t^{\ep,s}(x)) + \ep^2 b_t(Y_t^{\ep,s}(x))\,dt + \ep\si_t(Y_t^{\ep,s}(x))\,dW_t, \quad t\in [s,T],\]
	with $Y_t^{\ep,s}(x) = x$ for $t\in [0,s]$. Here, $\si = \sqrt{D}$.
By Corollary \ref{cor:SDEDerInitPert}, we have $Y \in \Lip^{5+1}([0,1]\times [0,T]\times \R^d, \Lfininf)$.
In particular, $Y \in \wC 5 1([0,1]\times [0,T]\times \R^d, \Lfininf)$. Thus, by Lemma \ref{lem:wHoelgRndFld}, we have $\E g(Y) \in \wC l {\ka+l}([0,1]\times [0,T]\times \R^d, L^\infty([0,T], \R^d))$ for all $g\in \wC l \ka$, with
\[\nrm{\E g(Y)}{\wC l{\ka+l}}\lesssim \nrm{g}{\wC l \ka}, \unfovr g\in \wC l \ka(\R^d),\]
for every $l\leq 5$.
In particular,
\[v^{g,h} = \E g(X^{h,\blnk}) = \E g(Y^{\sqrt h,\blnk}) \in \wC l{\ka+l}([0,T]\times \R^d, \R),\]
uniformly in $h\in [0,1]$, with
\[\nrm{v^{g,h}}{\wC l {\ka + l}} \lesssim \nrm{g}{\wC l \ka},\]
uniformly over $h\in [0,1]$ and $g\in \wC l \ka$, for every $l\leq 5$. 

Next, we consider derivatives of $Y^\ep$ in $\ep\in (0,1)$. By Corollary \ref{cor:SDEDerInitPert},
	\begin{align*}
		\der_\ep Y^{\ep, s}_t = & \int_s^t \der_\ep(\ep^2 b_u(Y_u^\ep))\,du +  \int_s^t \der_\ep (\ep \si_u(Y_u^\ep))\,dW_u \\
		= & \int_s^t  2\ep b_u(Y_u^\ep) + \ep^2   \innp{\nabla b(Y_u^\ep)}{\der_\ep Y_u^\ep}\,dt +  \int_s^t \si_u(Y_u^\ep) + \ep \innp{\nabla \si_u(Y_u^\ep)}{\der_\ep Y_u^\ep}\,dW_u \\
		\der_\ep^2 Y^{\ep, s}_t = & \int_s^t 2 b_u(Y^\ep_u) + \cO(\ep)\,du + \int_s^t 2\innp{\nabla \si_u(Y_u^\ep)}{\der_\ep Y_u^\ep} + \cO(\ep)\,dW_u.
	\end{align*}
	Note that
\[(\der_\ep Y^\ep)|_{\ep = 0} = \int_s^\blnk \si_u(Y_u^0)\,dW_u\]
is a martingale, as $\si(Y^\ep) \in \Lpinf 2$ and so $\E\int_0^T |\si_u(Y_u^\ep)|^2\,du < \infty$. In particular, $\E[(\der_\ep Y^\ep)]|_{\ep = 0} = 0$ and further
	\[\E[\innp{\nabla g(Y^\ep)}{\der_\ep Y^\ep}]|_{\ep = 0} = \innp{\nabla g(Y^0)}{\E[\der_\ep Y^\ep]|_{\ep = 0}} = 0,\]
	for all $g\in \dC 1$.
	Let $l \in \set{0,1,2}$ and $g\in \wC {l+2} \ka$. By Taylor's theorem and Lemma \ref{lem:expDerCmmte}
	\begin{align}
		\label{eq:prtbExpan}
		\E g(Y^{\ep}) 	= & g(Y^{0}) + \ep \der_\ep(\E g(Y^{\ep}))|_{\ep = 0} + \frac12\ep^2 \der_\ep^2(\E g(Y^{\ep}))|_{\ep = \xi}\nonumber \\
		= & g(Y^0) + \ep \E[\innp{\nabla g(Y^\ep)}{\der_\ep Y^\ep}]|_{\ep = 0} + \frac12\ep^2 \E[\innp{\nabla^2 g(Y^\ep)}{(\der_\ep Y^\ep)^{\otimes 2}}]|_{\ep = \xi}\nonumber\\
		&+  \frac12\ep^2 \E[\innp{\nabla g(Y^\ep)}{\der_\ep^2 Y^\ep}]|_{\ep = \xi}\nonumber\\
		= & g(Y^0) + \frac12\ep^2 \E[\innp{\nabla^2 g(Y^\xi)}{(\der_\ep Y^\xi)^{\otimes 2}}+\innp{\nabla g(Y^\xi)}{\der_\ep^2 Y^\xi}],
	\end{align}
	for some $\xi \in (0,1)$.
Since $Y\in \wC 5 1$, we have $\innp{\nabla^2 g(Y^\xi)}{(\der_\ep Y^\xi)^{\otimes 2}},\innp{\nabla g(Y^\xi)}{\der_\ep^2 Y^\xi} \in \wC 2 {\ka + l + 2}([0,T]\times \R^d, \Lfininf)$, with
\begin{align*}
\nrm{\innp{\nabla g(Y^\xi)}{\der_\ep^2 Y^\xi}}{\wC l {\ka + l + 2},1} \lesssim \nrm{\innp{\nabla g(Y^\xi)}{\der_\ep^2 Y^\xi}}{\wC l {\ka + l + 1},1} \lesssim &\nrm{\nabla g(Y^\xi)}{\wC l {\ka + l},2}\nrm{\der_\ep^2 Y^\xi}{\wC l 1,2} \\
\lesssim & \nrm{\nabla g}{\wC l {\ka}},\\
\nrm{\innp{\nabla^2 g(Y^\xi)}{(\der_\ep Y^\xi)^{\otimes 2}}}{\wC l {\ka + l + 2},1} \lesssim &\nrm{\nabla^2 g(Y^\xi)}{\wC l {\ka + l},2}  \nrm{(\der_\ep Y^\xi)^{\otimes 2}}{\wC l 2,2} \\
\lesssim & \nrm{\nabla^2 g(Y^\xi)}{\wC l {\ka + l},2} \nrm{\der_\ep Y^\xi}{\wC l 1,2}^2\\
\lesssim & \nrm{\nabla^2 g}{\wC l{\ka}}, \unfovr g\in \wC{l+2} {\ka},
\end{align*}	
by Lemma \ref{lem:wCstab} and \ref{lem:wHoelgRndFld}. 	 
	 Hence, by Lemma \ref{lem:expWHoelFld} and Equation \eqref{eq:prtbExpan}
	\[\nrm{v^{g,0} - v^{g,h}}{\wC l{\ka + l + 2}} = \nrm{g(Y^{0}) - \E g(Y^{\sqrt h})}{\wC l {\ka + l + 2}} \lesssim h \nrm{g}{\wC {l+2} \ka},\]
	uniformly over $g\in \wC{l+2} \ka$ and $h\in [0,1]$
\end{proof}


\begin{lem}
\label{lem:talayTubaroSumSME1}
Given $g\in \wC 3 \ka$ there exists a function $\xi^g: (\lrates)\times \N \to \wCZ{\ka + 6}(\R^d)$ such that
\[\E g(\chi_n^h) - \E g(X_{nh}^h) = h^2 \sum_{k=0}^{n-1} \E[\cF_{kh}[v^{g,h}](\chi_k^h)] + h^2 \xi^{g,h}_n, \quad h\in \lrates,\]
and
\[\nrm{\xi^{g,h}_n}{\wCZ{\ka +6}} \lesssim \nrm{g}{\wC 3 \ka},\]
uniformly over $g\in \wC 3 \ka, h\in \lrates$ and $n\in \set{0,\dots, T/h}$.
\end{lem}
\begin{proof}
	We follow the proof of Lemma \ref{lem:talayTubaroSumOME}. Note that for $g\in \wC 3 \ka(\R^d)$ we have $v^{g,h,s}\in \wC 3 \ka([0,T]\times \R^d)$, uniformly in $h\in [0,1]$ and $s\in [0,T]$, with
	\begin{equation}
\label{eq:vghEst}
\sup_{h\in [0,1]} \sup_{s\in [0,T]} \nrm{v^{g,h,s}}{\wC 3 {\ka + 3}([0,T]\times \R^d)} \lesssim \nrm{g}{\wC 3 \ka(\R^d)}, \unfovr g\in \wC 3 \ka,
	\end{equation}
	by Lemma \ref{lem:OMEvsSME1}.
	 A Taylor expansion of $v^h := v^{g,h,T}$ yields
	\[\E v_{(k+1)h}^h(\chi_{k+1}^h) - \E v_{kh}^h(\chi_k^h) = hA_1^h + h^2 A_2^h + \E \zeta^h_{k,n},\]
	as before, except with
	\begin{align*}
		A_1^h = &\E[(\der_t v_{kh}^h + \innp{\bar f_{kh}}{\nabla v_{kh}^h } + h\innp{b_{kh}}{\nabla v_{kh}^h} +  \frac 1 2 h \innp{D_{kh}}{\nabla^2v_{kh}^h})(\chi_k^h)] = 0,
	\end{align*}
	by Equation \eqref{eq:KBESME1}. 
To compensate for the $h\innp{b_{kh}}{\nabla v_{kh}^h} + \frac 1 2 h \innp{D_t}{\nabla^2 v_{kh}^h}$ term we have
\begin{align*}
A_2^h = & \frac12\E[(\der_t^2 v_{kh}^h  + 2 \innp{\bar f_{kh}}{\der_t \nabla v_{kh}^h } - 2 \innp{b_{kh}}{\nabla v_{kh}^h} + \innp{\Si_{kh} - D_{kh} + \bar f_{kh}^{\otimes 2}}{\nabla^2 v_{kh}^h })(\chi_k^h)]\\
 = &\E[\cF_{kh}[v^h](\chi_k^h)].
\end{align*}	
Using Inequality \eqref{eq:vghEst} we deduce that $\xi^{g,h}_n := h^{-2}\sum_{k=0}^{n-1} \E \zeta_{k,n}^h$ satisfies 
\[\nrm{\xi_n^{g,h}}{\wCZ{\ka + 6}}\lesssim \nrm{g}{\wC 3 \ka},\]
uniformly over $g\in \wC 3 \ka, h \in \lrates$ and $n\in \set{0,\dots, T/h}$.
\end{proof}
\noindent 
We do an initial approximation just as in the ODE case.
\begin{lemma}
\label{lem:talayTubaroSME1withH}
For all $g\in \wC 5 \ka(\R^d)$ there exists a function $\rho^{g} : \lrates\to \wCZ{\ka +13}(\R^d), h \mapsto \rho^{g,h}$ such that
\begin{align}
	\E g(\chi_{T/h}^h) -  \E g(X_T^{h})= h \int_0^T \E[\cF_t[v^{g,h}](X_t^h)]\,dt + h^2 \rho^{g,h},
\end{align}
and
\[\nrm{\rho^{g,h}}{\wCZ{\ka+13}} \lesssim \nrm{g}{\wC 5 \ka}\]
uniformly over $g\in \wC 5 \ka$ and $h\in \lrates$.
\end{lemma}
\begin{proof}
The proof is analogous to the proof of Theorem \ref{thm:1stOrderOME}. Let $g\in \wC 5 \ka$ and define $\ph^h = \cF v^h, h \in [0,1]$. Lemma \ref{lem:talayTubaroSumSME1} implies 
\begin{align*}
\frac1h(\E g(\chi_{T/h}^h) - \E g(X_T^h)) =& h \sum_{n=0}^{\frac Th-1} \E\ph_{nh}^h(\chi_n^h) + h\xi_{T/h}^{g,h}\\ 
 = & \int_0^T \E \ph_t^h(X_t^h)\,dt +  h\sum_{n=0}^{\frac T h -1} \E \ph_{nh}^h(\chi_n^h) - \E \ph_{nh}^h(X_{nh}^h) \\
 &+ \sum_{n=0}^{\frac T h-1} h\E \ph_{nh}^h(X_{nh}^h) - \int_0^T \E \ph_t^h(X_t^h)\,dt + h\xi_{T/h}^{g,h},
\end{align*}
with 
\[\nrm{\xi_{T/h}^{g,h}}{\wCZ{\ka+6}} \lesssim \nrm{g}{\wC 3 \ka}, \unfovr g\in \wC 3 \ka, h \in \lrates.\]
Note that $\ph^h \in \wC l {\ka + l+4}([0,T]\times \R^d)$, uniformly in $h\in [0,1]$, with
\begin{equation}
\label{eq:phiBnd}
\nrm{\ph^h}{\wC l {\ka + l+4}([0,T]\times \R^d)} \lesssim \nrm{v^h}{\wC{l+2} {\ka+l+2}([0,T]\times \R^d)} \lesssim \nrm{g}{\wC {l+2} \ka(\R^d)},
\end{equation}
uniformly over $g\in \wC {l+2} \ka$ and $h\in [0,1]$, for all $l \leq 3$.
Thus, we may use Lemma \ref{lem:talayTubaroSumSME1} again to estimate
\begin{align*}
	\sum_{n=0}^{\frac T h -1} |\E \ph_{nh}^h(\chi_n^h) - \E\ph_{nh}^h(X_{nh}^h)|\leq& h^2\sum_{n=0}^{\frac T h - 1} \sum_{k=0}^{n-1} |\E[\cF_{kh}[v^{\ph_{nh}^h,h}](\chi_k^h)] + \hat\xi^{g,h}_n|, \\
\end{align*}
where by Inequality \eqref{eq:phiBnd}
\[\nrm{\hat \xi^{g,h}_n}{\wCZ{\ka + 13}} \lesssim \nrm{\ph^h_{nh}}{\wC 3 {\ka + 7}(\R^d)} \lesssim \nrm{g}{\wC 5 \ka},\]
uniformly over $g\in \wC 5 \ka, h\in \lrates$ and $n\in \set{0,\dots, T/h}$.
Since by Lemma \ref{lem:LGequiv}, Lemma \ref{lem:OMEvsSME1} and Inequality \eqref{eq:phiBnd},
\[\nrm{\E[\cF_{kh}[v^{\ph_{nh}^h,h}](\chi_k^h)]}{\wCZ {\ka + 10}} \lesssim \nrm{\cF_{kh}[v^{\ph_{nh}^h,h}]}{\wCZ {\ka + 10}} \lesssim \nrm{v^{\ph_{nh}^h,h}}{\wC 2 {\ka + 8}} \lesssim \nrm{\ph_{nh}^h}{\wC 2 {\ka + 6}(\R^d)} \lesssim \nrm{g}{\wC 4 \ka},\]
uniformly over $h\in \lrates, k\leq n \in \set{0,\dots, T/h}, g\in \wC 4 \ka$, we conclude
\[h \nrm{\sum_{n=0}^{\frac T h-1} \E\ph_{nh}^h(\chi_n^h) - \E \ph_{nh}^h(X_{nh}^h)}{\wCZ {\ka + 13}} \lesssim h \nrm{g}{\wC 4 {\ka}},\]
uniformly over $g\in \wC 4 \ka$ and $h\in \lrates$.
Further, approximating the integral $\int \ph\,dt$ by a left Riemann sum yields
\begin{align*}
	\left|\sum_{n=0}^{\frac T h-1} h\E \ph_{nh}^h(X_{nh}^h) -\int_0^T \E \ph_t^h(X_t^h)\,dt\right|\leq& \frac12 h T \sup_{t\in [0,T]} |\der_t \E \ph_t^h(X_t^h)|,
\end{align*}
Thus, by Lemma \ref{lem:wHoelgRndFld} and Inequality \eqref{eq:phiBnd}
\[\nrm{\sum_{n=0}^{\frac T h-1} h\E \ph_{nh}^h(X_{nh}^h) -\int_0^T \E \ph_t^h(X_t^h)\,dt}{\wCZ {\ka + 6}} \lesssim h \nrm{\E \ph^h(X^h)}{\wCZ {\ka+6}} \lesssim \nrm{\ph^h}{\wC 1 {\ka+5}} \lesssim h \nrm{g}{\wC 3 \ka},\]
uniformly over $g\in \wC 3 \ka$ and $h\in \lrates$.
Putting all estimates together yields
\[\E g(\chi_{T/h}^h) - \E g(X_T^h) = h \int_0^T \E \ph_t^h(X_t^h)\,dt+ h^2\rho^{g,h},\]
with 
\[\nrm{\rho^{g,h}}{\wCZ {\ka + 13}} \lesssim \nrm{g}{\wC 5 \ka},\]
uniformly over $g\in \wC 5 \ka$ and $h\in \lrates$.
\end{proof}

\begin{proof}[Proof of Theorem \ref{thm:firstOrderSME}]
Firstly, Lemma \ref{lem:OMEvsSME1} implies
\[\nrm{v_t^{h} - v_t^{0}}{\wC 2 {\ka + 2}(\R^d)} \lesssim h \nrm{g}{\wC 4 \ka},\]
and also
\begin{align*}
\nrm{\cF_t[v^h](X_t^0) - \E[\cF_t[v^h](X_t^h)]}{\wCZ{\ka + 8}(\R^d)} = & \nrm{v_t^{\cF_t[v^h],0} - v_t^{\cF_t[v^h],h}}{\wCZ {\ka + 8}} \\
\lesssim & h \nrm{ \cF_t[v^h]}{\wC 2 {\ka + 6}} \\
\lesssim & h \nrm{v^h}{\wC 4 {\ka+ 4}} \\
\lesssim & h\nrm{g}{\wC 4 \ka},
\end{align*}
uniformly over $g\in \wC 4 \ka, h\in \lrates$ and $t\in [0,T]$.
Thus, by Lemma \ref{lem:plyGrwthStab} (v)
\begin{align*}
\nrm{\E[\cF_t[v^h](X_t^{h})] - \cF_t[v^0](X_t^{0})}{\wCZ {\ka + 8}(\R^d)} \leq & \nrm{\E[\cF_t[v^h](X_t^{h})] - \cF_t[v^h](X_t^{0})}{\wCZ {\ka + 8}}\\
& + \nrm{\cF_t[v^h- v^0](X_t^{0})}{\wCZ {\ka + 8}} \\
\lesssim & h (\nrm{g}{\wC 4 \ka} + \nrm{\cF_t[v^h- v^0]}{\wCZ {\ka + 8}})\\
 \lesssim & h (\nrm{g}{\wC 4 \ka} + \nrm{v_t^h - v_t^0}{\wC 2 {\ka + 6}})\\
 \lesssim &h \nrm{g}{\wC 4 \ka},
\end{align*}
uniformly over $g\in \wC 4 \ka, h\in [0,1]$ and $t\in [0,T]$.
Hence,
\[\nrm{\int_0^T \cF_t[v^0](X_t^0)\,dt - \int_0^T \E[\cF_t[v^h](X_t^h)]\,dt}{\wCZ {\ka + 8}}\lesssim h \nrm{g}{\wC 4 \ka},\]
uniformly over $g\in \wC 4 \ka$ and $h\in [0,1]$.
Consequently, Theorem \ref{thm:firstOrderSME} follows from Lemma \ref{lem:talayTubaroSME1withH}.
\end{proof}

\chapter{SMEs for optimization and learning}
\label{chap:smeopt}
So far we have discussed (stochastic) modified equations as continuous-time models for numerical methods of differential equations. In this brief chapter, we shift our focus towards stochastic gradient optimization, statistical learning, and their relation to \tSME s.
This chapter is adapted from and expands on \citet[Section 3.1, 6.1 and 6.2]{perko2024compare}.

\section{Stochastic optimization and SGD}
Let $d \in \N$. Given a function $\cR : \R^d \to [0,\infty) \in \dC 1$, which we call the \emph{objective} or \emph{risk}, we consider the following optimization problem
\[\min_{\thet\in \R^d} \cR(\thet).\]
To solve this problem, consider the following ODE
\begin{equation}
\label{eq:GFsimple}
\dot X_t^0 = - \nabla \cR(X_t^0).
\end{equation}
Under certain assumptions $X^0$ exists and it converges to a critical point $\thet^*$ of $\cR$, that is $\nabla\cR(\thet^*) = 0$. We also call $X^0$ \emph{gradient flow}.
In some situations $\thet^*$ is indeed the global minimum of $\cR$, for example if $\cR$ is strongly convex with Lipschitz gradient.
In any case, if we wanted to find $\thet^*$, one possible avenue would be applying the Euler method with step size $h\in (0,1)$ to gradient flow:
\[\chi_{n+1} = \chi_n - h \nabla \cR(\chi_n).\]
This resulting algorithm is called \emph{gradient descent}. Note that our goal here has shifted compared to the numerics of differential equations.
It is not necessary for gradient descent to be very close to gradient flow. The only important thing for us is that it (approximately) minimizes $\cR$. Our main interest is now in gradient descent and we view gradient flow as continuous-time approximation to it (instead of thinking of approximating the other way around). 

Consider now a family of functions 
\[R : \R^d \times \cZ \to [0,\infty), (\thet, z) \mapsto R_z(\thet)\]
and a probability measure $\nu$ on a measurable space $\cZ$. We may call $R_z$ the risk at the point $z\in \cZ$. Suppose $\cR(\thet) = \E_{z\sim \nu}[R_z(\thet)]$. Under mild assumptions we have
\[\nabla \cR(\thet) = \E_{z\sim \nu}[\nabla R_z(\thet)], \quad \thet \in \R^d.\]

In practice, it may be very costly or even impossible to compute the full gradient $\nabla \cR$.
In this case, we can instead apply a noisy Euler method:
\begin{equation}
\label{eq:SGDsimple}
\chi_{n+1} = \chi_n - h \nabla R_{\bz(n)}(\chi_n)
\end{equation}
where $\bz(n) \in \cZ$ with $\bz(n) \sim \nu$ for all $n\in \N_0$.
This resulting algorithm is called \emph{stochastic gradient descent}.

We think of $(\bz(n))_{n\in \N}$ as a sequence of data points.
We highlight two conceptually different scenarios. 

\begin{enumerate}[(a)]
\item Suppose $\nu$ is an empirical measure, that is there exist $z_1',\dots, z_N'$ such that $\nu$ is given by the average of the Dirac measures
\[\nu = \frac 1N \sum_{n=1}^N \idK_{\blnk}(z_n').\]
Then we can think of $\nu$ as a finite data set, or \emph{sample} of size $N$, and the SGD method uses this sample to minimize the \emph{empirical risk} or \emph{training error}
\[\cR(\thet) = \frac1N \sum_{n=1}^N R_{z_n'}(\thet), \quad \thet \in \R^d.\]
If $(\bz(n))_{n\in \N_0}$ is \tiid\ with $\bz(0)\sim \nu$, then we call \eqref{eq:SGDsimple} SGD \emph{with replacement}. That is because we can think of $(\bz(n))_{n\in \N}$ as sampling uniformly from an urn containing $z_1',\dots, z_N'$ and after sampling the urn is refilled.
\item Suppose $\nu$ is a non-atomic measure. Then we can think of $\nu$ as a \emph{population}, that is the true real-world process by which data $z \sim \nu$ is generated. In this case the SGD method is used to minimize the \emph{population risk} $\cR$ (which is usually measured by the \emph{test error}).
If $(\bz(n))_{n\in \N_0}$ is \tiid\, then we call \eqref{eq:SGDsimple} \emph{one-pass} SGD. We sample from $\nu$ and every sampled data point is only passed over once. In this setting we essentially assume that we are given an infinite sequence of distinct data points (which is true almost surely since $\nu$ is non-atomic).
\end{enumerate}

The goal of learning is ultimately to (approximately) minimize the population risk. Empirical risk minimization (ERM) is a proxy task that we can perform when we only have a finite amount of data. However, the result can in some cases deviate greatly from population risk minimization (for example via \enquote{overfitting} in neural networks). 
Instead of focusing abstractly on ERM, we focus our analysis on the actual optimization algorithm SGD.

\section{Running examples}
\subsection{Mini-batch SGD with learning rate schedule}
Let $T > 0$.
Consider a risk minimization problem $(R : \R^d \times \cZ \to [0,\infty), \nu)$ on a measurable space $\cZ$. Fix a batch size $B \in \N$ and an \tiid\ sequence $(\bz(n))_{n\in \N}$ in $\cZ$ with $\bz(n) \sim \nu$. We consider \emph{mini-batch SGD} with \emph{batch size} $B$ and \emph{learning rate schedule} $u : [0,T]\to [0,1]$, given by
\begin{equation}
	\label{eq:SGDminibtchU}
	\chi_{n+1}^h = \chi_n^h - u_{nh}\frac{h}{B} \sum_{k=0}^{B-1} \nabla R_{\bz({k + Bn})} (\chi_n^h), \quad h \in(0,1), n \in \set{0,\dots, \floor{T/h}}
\end{equation}
The learning rate in the $n$-th step is $hu_{nh}$, and $h$ is interpreted as the \emph{maximal learning rate}.
We want to prove that under certain assumptions Equation \eqref{eq:SGDminibtchU} is a one-step method in the sense of Chapter \ref{chap:mdfdeq} satisfying Assumption \assref{assum:H}.

Let $\cZ \subseteq \R^m$ be bounded by a constant $C$ and $\bz$ be a $\cZ$-valued random variable.

\begin{lem}
\label{lem:randRtoR}
Let $l\in \N$ and $g \in \dC l(\R^d, L^\infty(\cZ,\R^k))$ with $\nabla g \in \Lip^l(\R^d, L^\infty(\cZ,\R^{d\times k}))$. Then $g_{\bz}\in \dC l(\R^d, L^\infty(\Om, \R^k))$ with $\nabla g_{\bz}\in \Lip^l(\R^d,L^\infty(\Om,\R^{d\times k}))$. In particular, $\E g_{\bz} \in \dC l(\R^d, \R^k)$ with $\nabla \E g_{\bz} \in \Lip^l(\R^d, \R^{d \times k})$ and
\[\der^\al \E g_{\bz} = \E \der^\al g_{\bz}, \quad |\al|\leq l.\]
\end{lem}
\begin{proof}
Consider the linear operator
\[A : L^\infty(\cZ, \R^k) \to L^\infty(\Om, \R^k)\]
given by
\[(A f)(\om) = f(\bz(\om)), \quad f\in L^\infty(\cZ, \R^k), \om \in \Om.\]
Indeed, $A$ is well-defined, since $f = \tilde f$, a.e.\ on $\cZ$ implies $f(\bz) = \tilde f(\bz)$ almost surely.
Linearity is obvious. Moreover,
\[\nrm{ A f}{\infty} \leq \nrm{f}{\infty},\quad f\in L^\infty(\cZ, \R^k),\]
so $A$ is bounded. In particular, $A \in \dC l( L^\infty(\cZ, \R^k),L^\infty(\Om, \R^k))$ in the sense of \tFre{} differentiability. Hence, $g_{\bz}= A\circ g \in \dC l(\R^d, L^\infty(\Om, \R^k)) \subseteq \dC l(\R^d, \Lfin(\Om, \R^k))$ with $\der^\al g_{\bz} = (\der^\al g)_{\bz}, |\al|\leq l$. Therefore,
\[\nrm{\der^\al g_{\bz}(x) - \der^\al g_{\bz}(y)}{L^\infty(\Om)} \leq \nrm{\der^\al g_{\blnk}(x) - \der^\al g_{\blnk}(y)}{L^{\infty}(\cZ)} \lesssim |x-y|,\]
uniformly over $x,y\in \R^d$ for all $|\al|\leq l$, i.e. $\nabla g_{\bz}\in \Lip^l(\R^d,L^\infty(\Om,\R^{d\times k}))$.
By Lemma \ref{lem:expDerCmmte} and Lemma \ref{lem:expWHoelFld} we conclude $\E g_{\bz} \in \dC l(\R^d, \R^k)$ with $\nabla \E g_{\bz}\in \wHoel {l-1} 1 {b} = \Lip^l(\R^d,\R^{d\times k})$, and 
\[\der^\al \E g_{\bz} = \E \der^\al g_{\bz}, \quad |\al|\leq l.\]
\end{proof}

\begin{docu}
\begin{lem}
\label{lem:stdLip}
Let $X\in \Lip(\R^d, L^2(\Om, \R^d))$. Then $\sqrt{\Cov X} \in \Lip(\R^d, \R^{d\times d})$.
\end{lem}
\begin{proof}
Define the map
\[\si : L^2(\Om, \R) \to [0,\infty), Y\mapsto \nrm{Y - \E Y}{2} = \sqrt{\Var Y}.\]
Then
\[|\si(Y) - \si(Z)|\leq \nrm{Y - Z}{2} + |\E X - \E Y| \leq 2\nrm{Y - Z}{2},\quad Y,Z\in L^2(\Om, \R).\]
Let $x,u\in \R^d$. Then for $\Si := \Cov X$ we have
\[|\sqrt{\Si(x)}u|^2 = \innp{\Si(x)}{u^{\otimes 2}} = \Var(\innp{X(x)}{u}),\]
and so
\[|\sqrt{\Si(x)} u| = \si(\innp{X(x)}{u}).\]
Thus, for $x,y, u \in \R^d$ with $|u| = 1$ we have
\begin{align*}
||\sqrt{\Si(x)} u| - |\sqrt{\Si(y)} u|| = & |\si(\innp{X(x)}{u})-  \si(\innp{X(y)}{u})|\\
 \leq & 2 \nrm{\innp{u}{X(x)-X(y)}}{2} \\
 \leq & 2 \nrm{X(x) - X(y)}{2} \\
 \leq & 2 \nrm{X}{\Lip,2}|x-y|.
\end{align*}
Thus,
\[\specnrm{\sqrt{\Si(x)} - \sqrt{\Si(y)}} \leq 2 \nrm{X}{\Lip,2}|x-y|,\quad x,y\in \R^d.\]
\end{proof}
\end{docu}

\begin{lem}
\label{lem:btchSGDtoSME}
Let $T > 0$ and $B, l\in \N$.
Suppose $u : [0,T]\to [0,1] \in \Lip^l$ and $R \in \dC l(\R^d, L^\infty(\cZ,\R))$ with $\nabla R \in \Lip^l(\R^d, L^\infty(\cZ,\R^d))$. Define
\begin{align*}
F(h,t,\thet,z) := &  -\frac{u_t} B\sum_{k=0}^{B-1}\nabla R_{z(k)}(\thet), \quad h\in (0,1),t\in [0,T], \thet \in \R^d, z = (z_0,\dots, z_{B-1})\in \cZ^B,\\
Z_n := & (\bz(0 + Bn),\dots, \bz((B-1) + Bn)), \quad n\in \N,\\
f_t^h(\thet) := & F(h,t,\thet, Z_{\floor{t/h}}), \quad h\in (0,1),t\in [0,T], \thet \in \R^d.
\end{align*}
Then 
\[|F(h,t,\thet,z)| \lesssim 1 + |\thet|,\unfovr h\in (0,1),t\in [0,T],\thet \in \R^d, z\in \cZ^B,\]
and $f_t^h \in \Lip^l(\R^d, L^\infty(\Om, \R^d))$ uniformly in $t\in [0,T]$ and $h\in (0,1)$.
Further, by defining
\[\bar f_t(\thet) := -u_t \nabla \cR(\thet), \quad \Si_t(\thet) := \frac{u_t^2}{B}\E[(\nabla R_{\bz(0)}(\thet) - \nabla \cR(\thet))^{\otimes 2}],\]
for all $\thet\in \R^d$ and $t\in [0,T]$, we have $\bar f\in \Lip^l([0,T]\times \R^d, \R^d), \Si \in \wC l 2([0,T]\times \R^d, \R^{d\times d})$, and
\[\bar f_t(\thet) = \E[f_t^h(\thet)], \quad  \Si_t(\thet) = \Cov f_t^h(\thet),\]
for all $\thet\in \R^d, h \in (0,1)$ and $t\in [0,T]$.
\end{lem}
\begin{proof}
From $\nabla R \in \Lip(\R^d, L^\infty(\cZ, \R^d))$ we deduce the inequality for $F$.
By Lemma \ref{lem:randRtoR} we have $\nabla R_{\bz(n)} \in \Lip^l(\R^d, L^\infty(\Om, \R^d))$. 
Thus, also $f_t^h \in \Lip^l(\R^d, L^\infty(\Om, \R^d))$ uniformly in $t\in [0,T], h\in (0,1)$. Further, we have $\E \nabla R_{\bz(n)} \in \dC l(\R^d, \R^d)$ with 
\[\E[\nabla R_{\bz(n)}] =  \nabla \E[R_{\bz(n)}] = \nabla \cR\in \Lip^{l}(\R^d, \R^d), \quad n \in \N_0.\]
Thus, $\bar f\in \Lip^l([0,T]\times \R^d, \R^d)$. Hence, also $\nabla R_{\bz(0)} - \nabla \cR \in \Lip^l \subseteq \wC{l-1} 1$ and so $\Si \in \wC l 2([0,T]\times \R^d, \R^{d\times d})$ by Lemma \ref{lem:plyGrwthStab} (iv) and Lemma \ref{lem:expWHoelFld}.
Finally, we compute $\E f_t^h(\thet) = -u_t \nabla \cR(\thet) = \bar f_t(\thet)$ and
\begin{align*}
\Cov f_t^h(\thet) = &\Cov\left[ -\frac{u_t} B\sum_{k=0}^{B-1}\nabla R_{\bz(k+B\floor{t/h})}(\thet)\right] \\
= & \frac{u_t^2}{B^2} \sum_{k=0}^{B-1}\Cov[\nabla R_{\bz(k+B\floor{t/h})}(\thet)] \\
= &  \frac{u_t^2}{B}\E[(\nabla R_{\bz(0)}(\thet) - \nabla \cR(\thet))^{\otimes 2}]\\
= & \Si_t(\thet),
\end{align*}
for all $\thet\in \R^d, h \in (0,1)$ and $t\in [0,T]$.
\end{proof}

\subsection{Linear regression}
\label{sec:linreg}
\label{sec:comparison_statl}
In this subsection we introduce our default example: population risk minimization for linear regression.

Suppose we are given an $\R^d$-valued random variable $\bx$ and an $\R$-valued random variable $\bep$ defined on a probability space $(\Om, \cF, \P)$ such that $\bx$ and $\bep$ are independent, $\E \bep = 0, \si_\ep^2 := \E \bep^2 < \infty$, the covariance matrix $\ka$ of $\bx$ is positive definite, and $\bx$ has finite joint fourth moments
\[\E|\bx_i\bx_j\bx_k\bx_l| < \infty, \quad i,j,k,l \in \set{1,\dots, d}.\]
We define $\mu_x^4 \in \R^{d^{\times 4}}$ by
\[(\mu_x^4)_{i,j,k,l} = \E[\bx_i\bx_j\bx_k\bx_l],\quad i,j,k,l \in \set{1,\dots, d}.\]
For simplicity we make the stronger assumption that $|\bx| + |\bep| < C$ for some constant $C > 0$.
Let $\thet^* \in \R^d$. We define the $\R$-valued random variable $\by$ by
\[\by = \innp{\thet^*}{\bx} + \bep.\]
Denote the distribution of $(\bx, \by)$ by $\nu$. By assumption $(\bx,\by)$ take values in a bounded set $\cZ \subseteq \R^d\times \R$.

Note in the literature on linear regression the \emph{features} $\bx$ are often assumed to be deterministic. This is an appropriate assumption if the data is generated by a \emph{randomized controlled trial}. In these experiments, experimental conditions are controlled by the researchers, so the features can be treated as deterministic. In contrast, in machine learning and especially anything using the label \enquote{big data}, the data is generated by a process outside of our control. Thus, we are working with \emph{observational data}. In this case, it is more appropriate to assume $\bx$ to be random, which is why we do so here.

Let $\ell$ be the \emph{square loss}, given by $\ell(y,y') = \frac12(y-y')^2$. The goal is to fit the data drawn from $\nu$ using a linear predictor $\thet \mapsto \innp{\thet}{x}$.
Thus, for any data point $(x,y) \in \R^d\times \R$ we consider the squared risk
\[R_{x,y}(\thet) = \ell(\innp{\thet}{x}, y) = \frac12 (\innp{\thet}{x} - y)^2.\]
We stress that the bold letters $\bx, \by$ denote random variables, while $x,y$ represent realizations.
Note that $R_{x,y}(\thet)$ is smooth in $\thet$, with derivatives uniformly in $(x,y)\in \cZ$ (since $\cZ$ is bounded). That is, $R\in \dC \infty(\R^d, L^\infty(\cZ, \R))$.
Further,
\[\nabla R_{x,y}(\thet) = (\innp{\thet}{x} - y) x, \nabla^2 R_{x,y}(\thet) = x^{\otimes 2}, \quad (x,y)\in \cZ, \thet \in \R^d.\]
Thus, $\nabla R \in \Lip^\infty(\R^d, L^\infty(\cZ, \R^d))$. Hence, $R$ satisfies the assumptions of Lemma \ref{lem:btchSGDtoSME}.
The population risk is given by
\begin{align*}
\cR(\thet) 	= &\E[R_{\bx,\by}(\thet)] = \frac12\E[(\innp{\thet - \thet^*}{\bx} - \bep)^2]	= \frac12 \innp{\ka}{(\thet-\thet^*)^{\otimes 2}}+ \frac{\si_\ep^2}{2}.
\end{align*}
The minimum of $\cR$, \tIe{} the best possible fit for the linear model, is given by the population parameter $\thet^*$.

\begin{todo}
	Hence, $\cR$ is a quadratic form with minimum at $\thet^*$. Now, consider $D\in \set{0, \Si(\thet^*), \Si}$ and
	\[dX_t^h = - \nabla \cR(X_t^h) + \sqrt{hD(X_t^h)}\,dW_t, \quad t\in [0,T], h \in (0,1).\]
	One can show (see Section \ref{sec:LEquadObj}) the linear error term is given by
	\[\LE(X) = - \frac12 T \innp{\ka^3 e^{-2T\ka}}{(\thet - \thet^*)^{\otimes 2}} + \frac12 \innp{\ka}{\al_T^D},\]
	where 
	\[\al_T^D = \int_0^T e^{-(T-t)\ka}(\Si(X^0_t) - D(X_t^0))e^{-(T-t)\ka}\,dt.\]
	To calculate $\al_T$ we will start with the covariance matrix of the gradient noise at batch size $1$, which is given by
	\begin{align*}
		S(\thet):= \Cov[\nabla_\thet R_{\bx,\by}(\thet)] = \innp{\mu_x^4 - \ka^{\otimes 2}}{(\thet - \thet^*)^{\otimes 2}}  + \si_\ep^2 \ka
	\end{align*}
	where $\mu_x^4 \in \R^{d\times d\times d\times d}$ with
	\[(\mu_x^4)_{i,j,k,l} = \E[\bx_i \bx_j \bx_k \bx_l], \quad i,j,k,l\in \set{1,\dots, d}.\]
	Note that the covariance matrix of the gradient noise is given by $\Si(\thet) = \frac1B S(\thet)$.
	
	In the next subsection we add a natural assumption on $S$ that allows for computation of explicit linear error terms not involving any integrals.
\end{todo}

For $S(\thet) := \Cov[\nabla R_{\bx,\by}(\thet)]$ we have
\begin{align*}
	S(\thet)= & \E[(\innp\thet\bx - \by)^2 \bx^{\otimes 2}] - (\ka(\thet - \thet^*))^{\otimes 2}\\
	= & \E[(\innp{\thet-\thet^*} \bx - \bep )^2 \bx^{\otimes 2}] - \ka (\thet - \thet^*)^{\otimes 2} \ka\\
	= & \E[\innp{\thet - \thet^*}{\bx}^2 \bx^{\otimes 2}] - 2\E[\ep \innp{\thet - \thet^*}{\bx}\bx^{\otimes 2}] \\
	&+ \E[\bep^2 \bx^{\otimes 2}] - \ka (\thet - \thet^*)^{\otimes 2} \ka \\
	= & \innp{\mu_x^4}{(\thet - \thet^*)^{\otimes 2}} - \ka (\thet - \thet^*)^{\otimes 2} \ka^\transp + \si_\ep^2 \ka\\
	= & \innp{\mu_x^4 - \ka^{\otimes 2}}{(\thet - \thet^*)^{\otimes 2}}  + \si_\ep^2 \ka
\end{align*}

We are mostly interested in the following two settings, where we can simplify $S$ further (ignoring the boundedness of $\cZ$ for a moment).
\begin{example}
	\label{ex:4thMomSetts}
	\begin{enumerate}[(a)]
	\item We assume that the features are centered Gaussian, \tIe{} $\bx\sim \cN(0,\ka)$.
	Then we can simplify the covariance matrix of the gradient noise to
	\[S(\thet) = 2\ka (\thet - \thet^*)^{\otimes 2} \ka + \si_\ep^2 \ka.\]
	To see this, let $\tau$ be a permutation of the set $\set{1,\dots, l}$ and $B\in \R^{d^{\times l}}$. Then we write $B_{\tau} \in \R^{d^{\times l}}$ for
	\[(B_{\tau})_{i_1,\dots, i_l} = B_{i_{\tau(1)}, \dots, i_{\tau(l)}}.\]
	For example if $B$ is matrix, then $B^\transp = B_{(12)}$. Here we use the cycle notation for permutations.
	By Isserli's theorem \citep[see][]{bose2021random}, the joint fourth moments of a centered Gaussian satisfy
	\[\mu_x^4 = \ka^{\otimes 2} + \ka^{\otimes 2}_{(23)} + \ka^{\otimes 2}_{(13)}.\]
	Given matrices $U,A\in \R^{d\times d}$ we have
	\begin{align*}
		\innp{U^{\otimes 2}_{(23)}}{A}_{i,j} = & \sum_{k,l} U_{i,k} U_{j,l} A_{k,l}\\
		= & U A U^\transp,\\
		\innp{U^{\otimes 2}_{(13)}}{A}_{i,j} = & \sum_{k,l} U_{k,j} U_{i,l} A_{k,l} \\
		= & U A^\transp U.
	\end{align*}
	Therefore, $S(\thet) = 2\ka (\thet - \thet^*)^{\otimes 2} \ka + \si_\ep^2 \ka$.
		\item We assume that $d = 1$, but not that $\bx$ is Gaussian. Then, we can write
		\begin{align*}
			S(\thet) = \ka^2(\Kurt \bx - 1)(\thet- \thet^*)^2 + \ka \si_\ep^2,
		\end{align*}
		where $\Kurt \bx := \E[\bx^4]/\ka^2$ is the \emph{kurtosis} of $\bx$ (see Section \ref{sec:kurtosis} in the appendix for more information about kurtosis).
	\end{enumerate}
\end{example}
We can join these two settings by assuming that there exists a constant $c > 0$, such that
\begin{equation}
\label{eq:niceSiLinReg}
S(\thet) = c\ka(\thet -\thet^*)^{\otimes 2}\ka + \si_\ep^2 \ka, \quad \thet \in \R^d.
\end{equation}
In particular, in Example \ref{ex:4thMomSetts} (a) we have $c = 2$ and for (b) we have $c = \Kurt \bx - 1$.

Under this condition, we $\sqrt{\Si} = u_t \sqrt{S}$ satisfies the assumptions of Corollary \ref{cor:SME2} as the next lemma implies.
\begin{lem}
Let $S$ be defined by \eqref{eq:niceSiLinReg}. Then $\sqrt S\in \Lip^\infty(\R^d, \R^{d\times d})$.
\end{lem}
\begin{proof}
Let $\thet \in \R^d$ and set $w := \ka^{1/2}(\thet - \thet^*)$ and assume $w\neq 0$. Then
\[S(\thet) = c (\ka^{1/2} w)(\ka^{1/2} w)^\transp + \si_\ep^2 \ka = \ka^{1/2}(c ww^\transp + \si_{\ep}^2 1_{d\times d})\ka^{1/2}.\]
Writing $M = c ww^\transp + \si_\ep^2 1_{d\times d}$ we get
\[S(\thet)^{1/2} = \ka^{1/2} M^{1/2} \ka^{1/2}.\]
Note that $M$ is a rank-one perturbation of $\si_\ep^2 1_{d\times d}$. We have
\[ww^\transp w = |w|^2w, \quad ww^\transp u = 0, \quad u \in w^\perp,\]
where $w^\perp = \set{u \in \R^d : \innp u v = 0}$.
Thus, $ww^\transp$ has eigenvalues $|w|^2$ and $0$, and by the rank-nullity theorem the following orthogonal eigenspace decomposition 
\[\R^d = \operatorname{Ker}(ww^\transp) \oplus \operatorname{Im}(ww^\transp) = w^\perp \oplus \operatorname{span} w.\]
Accordingly,
\[Mw = c|w|^2w + \si^2_\ep w = (c|w|^2 + \si^2_\ep)w, \quad Mu = \si_\ep^2 u, \quad u\in w^\perp,\]
and $M$ has the same eigenspace decomposition. Therefore, $M$ has eigenvalues $c|w|^2 + \si^2_\ep$ and $\si_\ep^2$ with (algebraic) multiplicities  $1$ and $d-1$ respectively. 
There is an orthonormal basis $\set{v_1,\dots, v_d}$ of $\R^d$ such that $v_1 = \frac{w}{|w|}$ and by defining a matrix $Q\in \R^{d\times d}$ with columns $v_1,\dots, v_d$ we can transform $M$ into the diagonal matrix
\[Q^\transp MQ= \operatorname{diag}(c|w|^2 + \si^2_\ep,\si^2_\ep,\dots, \si^2_\ep).\]
Then,
\[Q^\transp M^{1/2}Q = \operatorname{diag}(\la,\si_\ep,\dots, \si_\ep),\]
where $\la := \sqrt{c|w|^2 + \si^2_\ep}$.
Note that 
\[Q(e_1e_1^\transp)Q^\transp = (Qe_1)(Qe_1)^\transp = \frac{ww^\transp}{|w|^2} =: P_w.\]
Thus,
\[M^{1/2} = Q(\si_\ep 1_{d\times d} + (\la - \si_\ep)e_1e_1^\transp)Q^\transp = \si_\ep 1_{d\times d} + a(|w|^2)P_w,\]
where 
\[a_t = -\si_\ep + \sqrt{ct+\si_\ep^2} = \frac{ct}{-\si_\ep + \sqrt{ct+\si_\ep^2}}, \quad \geq 0.\]
Note that $a \in \dC \infty([0,\infty))$, and
\[|\der^{k}a_t| \asymp (ct+\si_\ep^2 )^{-\frac12(2k-1)}, \unfovr t\geq 0, k\in \N_0.\]
In particular,
\[|\der^{k} a_t| = O(t^{1/2-k}), \quad t\to \infty.\]
Define $b(t) = \frac{a(t)}{t}, t>0$. By Leibniz's formula we have  $b \in \dC \infty([0,\infty))$ with
\[|\der^k b(t)| \lesssim \sum_{j=0}^k  |\der^j a(t)| t^{-(k+1-j)},  \unfovr t\geq 0, k\in \N_0.\]
Hence,
\begin{equation}
\label{eq:beBddRght}
|\der^k b(t)| = O(t^{-1/2-k}), \quad t\to \infty,\quad k \in \N_0.
\end{equation}
Note that
\[a(|w|^2)P_w = b(|w|^2)ww^\transp\]
and by Equation \eqref{eq:beBddRght}
\[|b(|w|^2)ww^\transp| = O(|w|), \quad |w|\to \infty.\]
Further, by Faa di Bruno's formula
\begin{align*}
|\der^\be (b(|w|^2)) |\lesssim & \sum_{k=1}^{|\be|} \sum_{\cB \in \prttn{\be}{k}} \der_t^k b(|w|^2) \der^\cB(|w|^2).
\end{align*}
Since $|\der^\ga(|w|^2)| = O(|w|^{2-|\ga|})$ we have for $\cB = \set{\ga_1,\dots,\ga_k}\in \prttn \be k$ with $\sum_{i=1}^k \ga_i = \be$
\[|\der_t^k b(|w|^2) \der^\cB(|w|^2)| = O(|w|^{-1-2k} |w|^{\sum_{i=1}^k (2-|\ga_i|)}) = O(|w|^{-1-|\be|}), \quad |w|\to \infty,\]
for all $k\leq |\be|$.
Thus, using Leibniz's formula
\begin{align*}
|\der^\al(b(|w|^2)ww^\transp)| \lesssim & \sum_{\be \leq \al} |\der^\be(b(|w|^2)) \der^{\al-\be}(ww^\transp)|\\ = &O\left(\sum_{\be \leq \al} |w|^{-1-|\be|} |w|^{2 - |\al| + |\be|}\right)\\
= &O(|w|^{1-|\al|}), \quad |w|\to \infty.
\end{align*}
Thus, the derivatives of $w\mapsto b(|w|^2)ww^\transp)$ are bounded at $0$ and $\infty$, i.e. $w\mapsto b(|w|^2)ww^\transp) \in \Lip^\infty(\R^d, \R^{d\times d})$.
Since
\[\sqrt{S(\thet)} = \ka^{1/2}(\si_\ep 1_{d\times d} + a(|w|^2)P_w)\ka^{1/2} =  \ka^{1/2}(\si_\ep 1_{d\times d} + b(|w|^2)ww^\transp)\ka^{1/2},\]
even if $w = 0$, we conclude $\sqrt{S} \in \Lip^\infty(\R^d, \R^{d\times d})$.
\end{proof}

\begin{docu}
\begin{proof}
	Write $w(\theta)=\kappa^{1/2}(\theta-\theta^*)$.  
	Then
	\[
	S(\theta)
	= c \,\kappa^{1/2} w(\theta)w(\theta)^\top \kappa^{1/2}
	\,+\, \sigma_\epsilon^2\,\kappa.
	\]
	
	Since $\kappa$ is positive definite, the second term is uniformly
	elliptic. Therefore $S(\theta)$ takes values in the open cone
	$\mathcal S_{++}^d$ of positive definite matrices, uniformly bounded
	away from the boundary.
	
	The map
	\[
	\Phi : \mathcal S_{++}^d \to \mathcal S_{++}^d, \qquad \Phi(A)=A^{1/2},
	\]
	is $C^\infty$ on $\mathcal S_{++}^d$ (see, e.g.,
	Bhatia, \emph{Positive Definite Matrices}, §1.7).  
	Moreover,
	\[
	\theta \mapsto S(\theta)
	\]
	is a $C^\infty$ function, because it is polynomial in the components of
	$w(\theta)$.
	
	Since the composition of $C^\infty$ maps is $C^\infty$, we obtain
	\[
	\theta \mapsto S(\theta)^{1/2} = \Phi(S(\theta)) \in C^\infty(\R^d).
	\]
	
	Finally, every $C^\infty$ function on $\R^d$ with polynomial growth has
	bounded derivatives of all orders, hence belongs to
	$\Lip^\infty(\R^d)$.
	
	Thus $\sqrt{S} \in \Lip^\infty(\R^d,\R^{d\times d})$.
\end{proof}

\end{docu}

So far we described one-pass SGD for linear regression. Let us contrast this with SGD with replacement.
Consider a finite sample $(x_n,y_n)_{n=1}^N$ of size $N$.

We now consider the risk minimization problem $(R, \tilde \nu)$ with
\[R : \R^d \times \cZ \to [0,\infty), (\thet, (x,y)) \mapsto R_z(\thet) = \frac12 (\innp{\thet}{x} - y)^2\]
and where $\tilde \nu$ is the empirical measure
\[\tilde\nu = \frac 1N \sum_{n=1}^N \idK_{\blnk}(x_n,y_n).\]
We can set $\cZ = \R^d \times \R$ here since $\tilde \nu$ automatically has bounded support.
Let $\bz$ be a random variable $\bz \sim \tilde \nu$.
The expected risk for $(R,\tilde \nu)$ is the empirical risk for our finite sample:
\[\hat \cR(\thet) = \E[R_{\bz}(\thet)] = \frac1N\sum_{n=1}^N R_{(x_n,y_n)}(\thet) = \frac1{2N}\sum_{n=1}^N (\innp{\thet}{x_n} - y_n)^2, \quad \thet \in \R^d.\]
Its minimum given by the OLS estimator:
\[\hat \thet := \left(\sum_{n=1}^N x_n x_n^\transp\right)^{-1}\left(\sum_{n=1}^N x_n y_n\right).\]
If we want to compare one-pass SGD and SGD with replacement, then the right way to proceed is to replace $(x_n,y_n)_{n=1}^N$ with an \tiid\ sequence $(\bx_n,\by_n)_{n=1}^N$ drawn from the population $\nu$. Thus, in this situation $\tilde \nu, \hat \cR$ and the corresponding gradient covariance matrix $\Cov_{z\sim \tilde \nu}[\nabla R_z]$ are random quantities.

\chapter{A comparison of first-order SMEs}
\label{chap:compare}
In this chapter we investigate and compare first-order continuous-time approximations to stochastic gradient descent.
We focus on the case of population risk minimization and one-pass SGD for linear regression. 
This chapter is adapted from \citet[Section 3 and 6]{perko2024compare}.

\section{Introduction and problem statement}
Consider a risk minimization problem $(R : \R^d \times \cZ \to [0,\infty), \nu)$ on a measurable space $\cZ$. Fix a batch size $B \in \N$, an \tiid\ sequence $(\bz(n))_{n\in \N}$ in $\cZ$ with $\bz_n \sim \nu$ and consider the mini-batch SGD method
\begin{equation}
	\label{eq:SGDminibtch}
	\chi_{n+1}^h = \chi_n^h - \frac{h}{B} \sum_{k=0}^{B-1} \nabla R_{\bz{(k + Bn)}} (\chi_n^h).
\end{equation}
Note that for $S(\thet) := \Cov[\nabla \cR_{\bz(0)}(\thet)]$, we have
\[\Si(\thet) = \Cov\left(\frac{1}{B} \sum_{k=0}^{B-1} \nabla R_{\bz(k+Bn)} (\thet)\right) = \frac1B \Cov[\nabla \cR_{\bz(0)}(\thet)] = \frac{1}{B}S(\thet), \quad \thet \in \R^d.\]

We can approximate \eqref{eq:SGDminibtch} using gradient flow given by
\[\dot X_t^0 = - \nabla \cR(X_t^0),\]
where $\cR = \E[R_{\bz_1}]$. More generally, to model the noise inherent to SGD, we can approximate \eqref{eq:SGDminibtch} using a first-order \tSME{}
\begin{equation}
\label{eq:1stOrdrSGFs}
dX_t^h = - \nabla \cR(X_t^h)\,dt + \sqrt{h D(X_t^h)}\,dW_t.
\end{equation}
There are arguably two main choices considered (more or less explicitly) in the literature. Set $\Si(\thet) = \Cov[R_{\bz_1}]$. 
Then, for $D = \Si$, Equation \eqref{eq:1stOrdrSGFs} becomes a family of \tSDE s first introduced by \citet{Li15} to approximate SGD. We refer to a family of processes solving \eqref{eq:1stOrdrSGFs} with $D = \Si$ as (first-order) \emph{stochastic gradient flow with non-constant covariance} or NCC-SGF for short \citep[in accordance with the terminology used by][]{ali2020implicit}, and we denote it as $\XNCC$.

In order to simplify the analysis of Equation \ref{eq:1stOrdrSGFs}, in many cases the covariance matrix $\Si$ is assumed to be well approximated by a constant. For example, frequently one is interested in the behavior of SGD around a stationary point. In fact, suppose gradient flow $X^0$ converges to a, necessarily stationary, point $X^0_\infty \in \R^d$ and set $D = \Si(X^0_\infty)$.
Then we refer to a solution of Equation \ref{eq:1stOrdrSGFs} as (first-order) \emph{stochastic gradient flow with constant covariance} or CC-SGF for short \citep[again, in accordance with the terminology used by][]{ali2020implicit}, and we denote it as $\XCC$. This is essentially the continuous-time approximation introduced by \citet{mandt2015continuous}. Note that the diffusion coefficient may depend on the initial condition, since $X_\infty^0$ may already depend on it.

The introduction of a non-zero diffusion coefficient comes with the implicit promise of having a better approximation to SGD.
However, without an additional modification of the \emph{drift coefficient} $-\nabla \cR$ in Equation \eqref{eq:1stOrdrSGFs} the SGF dynamics are still merely a first-order approximation. Given that the order of approximation is not improved, does it make sense at all to add a diffusion term to the gradient flow dynamics? And if it does, how can one quantify the benefit? 

Theorem \ref{thm:firstOrderSME} indicates that the linear error term when approximating \eqref{eq:SGDminibtch} by \eqref{eq:1stOrdrSGFs} depends on $D$.
We find that the linear error terms for GF, CC-SGF and NCC-SGF are generally all different.

Further, we show that for linear regression models, the linear error terms for the objective function can be calculated in closed form. A comparison then reveals that any of three continuous-time approximations can be the best, depending on the batch size (see Theorem \ref{thm:linRegBS} below).
However, there is a notable caveat for the case of gradient flow being the best approximation:
Note that the dynamics of learning a linear model using SGD with constant learning rate can be roughly separated into the initial \emph{descent phase} and the final \emph{fluctuation phase}, where SGD, due to the variance of the stochastic gradients, is mostly fluctuating around the global minimum. The batch size at which gradient flow becomes the best approximation increases as the duration of the fluctuation phase increases, relative to the time horizon. On the other hand, the approximation quality of the stochastic approximations is unaffected by the relative duration of the fluctuation phase.
In fact, we show that there are two special batch sizes $\BEq$ and $\BGF$, such that for batch sizes $B < \BEq$ the NCC approximation is the best, followed by CC-SGF for $\BEq < B < \BGF$ and GF for $B > \BGF$. However, we also observe that $\BGF$ increases with the duration of the fluctuation phase of SGD. On the other hand, $\BEq$ only depends on the kurtosis of the features.

Fix $T > 0$.
Given a continuous-time (stochastic) approximation $Y = (Y^h_t)_{t\in [0,T], h \in T/\N}$ of SGD \eqref{eq:SGDminibtch} we define the linear error term (with respect to $\cR$) by
\[\LE(Y) := \lim_{h\downarrow 0} \frac{\E \cR(\chi_{T/h}^h) - \E \cR(Y_T^h)}{h},\]
where the limit is taken in $\lrates$.

\section{Comparison of the Linear Error Terms}
\label{sec:comparison}
In this section we compare gradient flow and the two stochastic gradient flow approximations (NCC and CC) in the setting of linear regression using mini-batch SGD.

Firstly, we provide a theoretical comparison using Theorem \ref{thm:firstOrderSME} (see Theorem \ref{thm:linRegBS}). We will see that the comparison highly depends on the batch size and on the kurtosis of the features (also called independent variables).
Secondly, we substantiate the theoretical findings using a numerical example.
Proofs are postponed until Section \ref{sec:comparisonproofs}.

We work in the setting of linear regression, as in Subsection \ref{sec:linreg}, and Example \ref{ex:4thMomSetts}. 
That is, there exist $\thet^*\in \R^d, \si_\ep > 0$ and a symmetric and positive definite matrix $\ka$ such that
\[\cR(\thet) = \frac12 \innp{\ka}{(\thet-\thet^*)^{\otimes 2}}+ \frac{\si_\ep^2}{2},\quad \thet\in \R^d\]
Further, we assume the existence of a constant $B^{\operatorname{Eq}} > 0$, such that
\[S(\thet) = 2B^{\operatorname{Eq}}\ka(\thet -\thet^*)^{\otimes 2}\ka + \si_\ep^2 \ka, \quad \thet \in \R^d.\]
In particular, in Example \ref{ex:4thMomSetts} (a) we have $B^{\operatorname{Eq}} = 1$ and for (b) we have $B^{\operatorname{Eq}} = \frac12 (\Kurt \bx - 1)$.
Proposition \ref{prop:TTlinReg} below implies that if $B^{\operatorname{Eq}}\in \N$, then it is the batch size $B$ where the NCC and CC approximation have the same error, up to flipping the sign.

For technical reasons we need to assume that the population is bounded. However, in our applications we sometimes ignore this extra assumption and also allow, say, for Gaussian features.

Now, the three continuous-time approximations \eqref{eq:1stOrdrSGFs} with $D = \set{0, \Si, \Si(X_\infty^0)}$ take the form
\begin{align}\label{ou dynamics}
	dX_t^0= & - \ka(X_t^0 - \thet^*)\,dt \nonumber \\
	d\XNCCh{h}_t = & - \ka(\XNCCh{h}_t - \thet^*)\,dt + \sqrt{\frac h B} \sqrt{2\BEq\ka (\XNCCh{h}_t - \thet^*)^{\otimes 2} \ka + \si_\ep^2 \ka} \,dW_t \nonumber \\
	dX_t^{\operatorname{CC},h}	= & - \ka(X_t^{\operatorname{CC}, h} - \thet^*)\,dt + \sqrt{\frac h B \si_\ep^2 \ka } \,dW_t.
\end{align}
Note that the process with constant covariance dynamics \eqref{ou dynamics} is an Ornstein-Uhlenbeck process. Using \eqref{eq:LEquad} we can derive the following expressions for the linear error terms of the three continuous-time approximations of SGD.
\begin{proposition}
	\label{prop:TTlinReg}
	Suppose $\chi_0^h = X_0^0 = \XNCCh{h}_0 = \XCCh{h}_0 = \thet \in \R^d$ for all $h\in \cH$. Then, we have
	\begin{align}
		\label{eq:TTlinReg}
		\E \cR(\chi_{T/h}^h) - \E \cR(X_T^h) = & -\frac h2 T \innp{\ka^3 e^{-2T\ka}}{(\thet - \thet^*)^{\otimes 2}} + \cO(h^2), \nonumber\\
		\E \cR(\chi_{T/h}^h) - \E \cR(X_T^{\operatorname{CC},h}) = &  h\left(\frac{B^{\operatorname{Eq}}}{B} - \frac12\right) T\innp{\ka^3 e^{-2T\ka}}{(\thet - \thet^*)^{\otimes 2}} + \cO(h^2), \nonumber\\
		\E \cR(\chi_{T/h}^h) - \E \cR(X_T^0) = & h\left(\frac{B^{\operatorname{Eq}}}{B} - \frac12\right) T\innp{\ka^3 e^{-2T\ka}}{(\thet - \thet^*)^{\otimes 2}} \nonumber\\
		&+ \frac h{4B} \si_\ep^2 \innp{\ka}{1_{d\times d} - e^{-2\ka T}} + \cO(h^2).
	\end{align}
	as $h \downarrow 0$, with $h\in \lrates$.
\end{proposition}

We introduce some additional notation to succinctly state the following theorem. Given two continuous-time approximations $Y,Z$ we write $Y \preceq Z$ if $|\LE(Y)|\geq |\LE(Z)|$,
that is if the approximation of SGD with $Y$ has (in absolute terms) a greater linear error term than the one using $Z$. More briefly it means that $Z$ is not worse than $Y$.
Evidently $\preceq$ is a reflexive and transitive relation. We write $Y \asymp Z$ if $Y \preceq Z$ and $Z \preceq Y$, that is if $Y$ and $Z$ are equally good approximations. Further, we write $Y\prec Z$ if $Y\preceq Z$ and $Z \npreceq Y$, that is if $Z$ is strictly a better approximation than $Y$.

\begin{theorem}
	\label{thm:linRegBS}
	Suppose $B^{\operatorname{Eq}} > 0$ and we are given an initial value $\thet\neq \thet^*$.
	Define
	\[B^{\operatorname{GF}} = 2 B^{\operatorname{Eq}} + \frac{\si_\ep^2 \innp{\ka}{1-e^{-2T\ka}}}{4T\innp{\ka^3e^{-2T\ka}}{(\thet-\thet^*)^{\otimes 2}}}.\]
	Then, we have the following
	\begin{enumerate}[(i)]
		\item $X^0 \prec X^{\operatorname{CC}} \prec X^{\operatorname{NCC}}$, if $B < B^{\operatorname{Eq}}$,
		\item $X^0 \prec \XCC \asymp \XNCC$, if $B = \BEq$,
		\item $X^0 \prec X^{\operatorname{NCC}} \prec X^{\operatorname{CC}}$, if $B^{\operatorname{Eq}} < B < B^{\operatorname{GF}} - B^{\operatorname{Eq}}$,
		\item $X^{\operatorname{NCC}} \prec X^0 \prec X^{\operatorname{CC}}$, if $B^{\operatorname{GF}} - B^{\operatorname{Eq}} < B < B^{\operatorname{GF}}$,
		\item $X^{\operatorname{NCC}} \prec X^{\operatorname{CC}} \prec X^0$, if $B>B^{\operatorname{GF}}$,
		\item $\LE(X^{\operatorname{CC}}) = 0$, if $B = 2 B^{\operatorname{Eq}}$.
	\end{enumerate}
\end{theorem}
In other words, for small batch sizes the best approximation is NCC-SGF, followed by CC-SGF and then gradient flow. If we increase the batch size, then NCC and CC switch places. After that NCC and GF switch places. Finally, for large batch sizes GF becomes the best approximation. Somewhere in between CC is not only the best approximation among the three, but also has a linear error of $0$.

Even though the gradient flow approximation can be the best approximation for large batch sizes, the lower bound $\BGF$ for this to occur diverges to $\infty$ as
\begin{equation}
	\label{eq:fluctuationDiv}
	T\to \infty\text{, or }\si_\ep \to \infty\text{, or } \ka \to \infty \text{ (for } d=1\text{)}  \text{, or } \thet - \thet^* \to 0  \text{ (for } d=1\text{)}.
\end{equation}
In fact, one can summarize \eqref{eq:fluctuationDiv} by saying $\tau\to \infty$, where $\tau$ is the time that SGD spends fluctuating around the global minimum $\thet^*$. Therefore, for large $\tau$ the SGF approximations are preferable to gradient flow, for all reasonably large batch sizes.

When it comes to deciding between NCC and CC-SGF, the important quantity is $\BEq$. This quantity only depends on the distribution of $\bx$ and \emph{not} on $T, \ka, \si_\ep$ or $\thet - \thet^*$.
For $\bx$ Gaussian we have $\BEq = 1$, so the CC-SGF approximation is, perhaps surprisingly, almost always preferred over the NCC approximation. We also consider the case where $d = 1$ and $\BEq = \frac12 (\Kurt x - 1)$. In this case we observe for batch sizes that are small, relative to the kurtosis of the features $\bx$, the NCC approximation can still be the best one (see also Section \ref{sec:kurtosis} in the appendix for more information on kurtosis).

Overall, one can also say that for highly leptokurtic features, the NCC approximation is the best across a large range of batch sizes. On the other hand, for lower kurtosis the CC approximation is best.

Figure \ref{fig:diffvskurt} below provides a visual comparison of the three approximations in terms of kurtosis and batch size in two specific examples.
\begin{figure}
	\centering
	\begin{minipage}[l]{0.485\textwidth}
		\includegraphics[width=\linewidth]{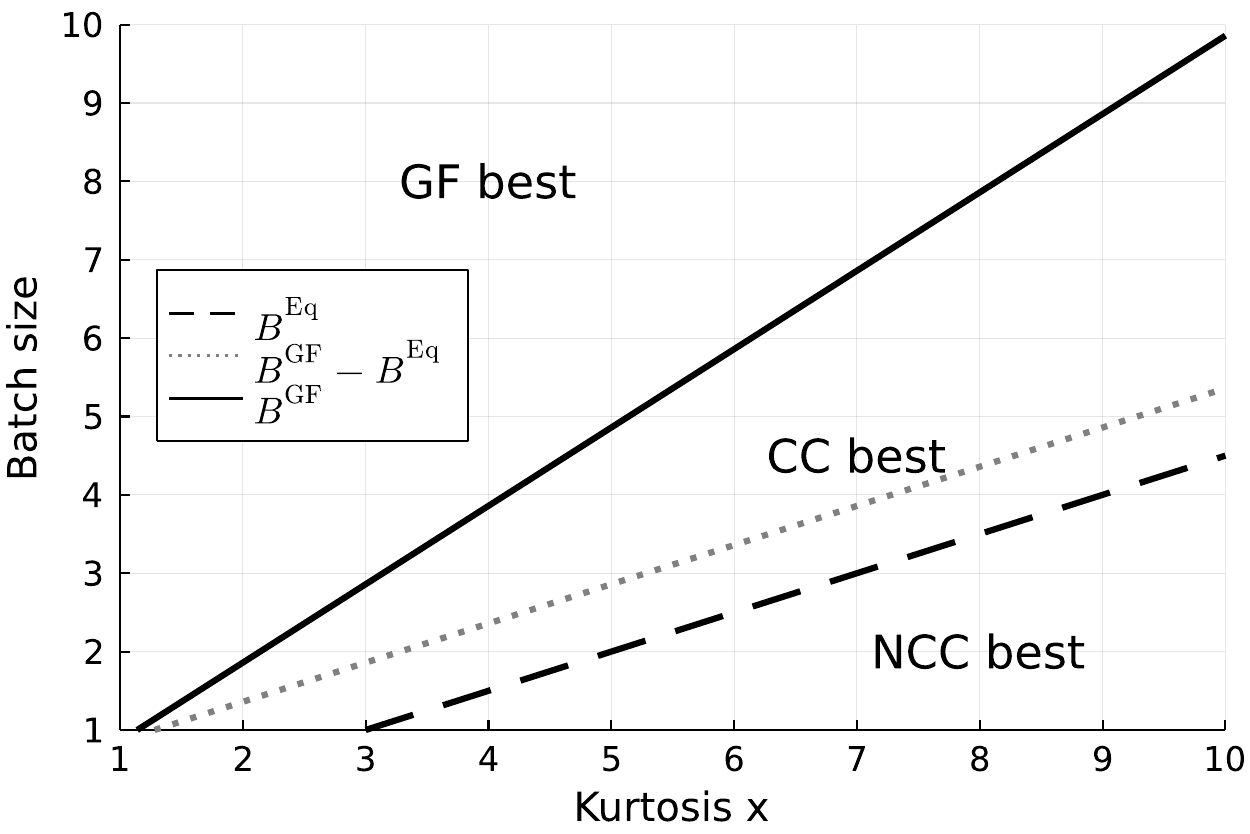}
	\end{minipage} 
	\hfill
	\begin{minipage}[r]{0.485\textwidth}
		\includegraphics[width=\linewidth]{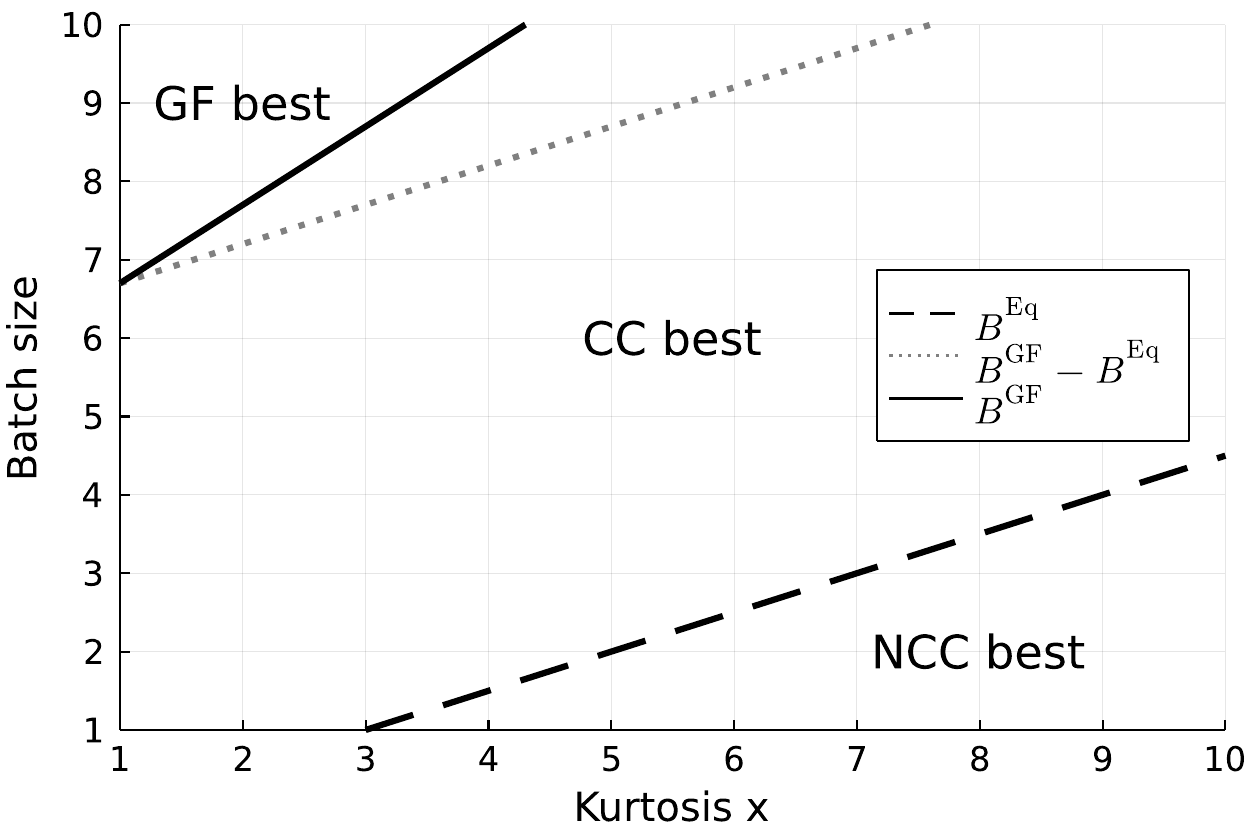}
	\end{minipage} 
	\caption{The best continuous-time approximation of SGD for linear regression in dimension $1$ in terms of the kurtosis of the features and the batch size. Here $\ka = 1$, $(\thet-\thet^*)^2 = 1$ and $T = 0.5$ ($T = 2.0$) in the left (right) plot. In the lower part of the middle region, where CC-SGF is the best approximation, gradient flow is worse than NCC-SGF. In the upper part of the middle region, gradient flow is better than NCC-SGF.}
	\label{fig:diffvskurt}
\end{figure}

\subsubsection{The Case of Batch Size 1}
\label{sec:batchSize1}
Here, we specifically study the case $B = 1$. Firstly, we have 
\[X^0 \prec \XCC \asymp \XNCC, \text{ if }\bx\text{ is Gaussian}.\]
Secondly, if $d = 1$ and so $\BEq = \frac12 (\Kurt \bx - 1)$, then
\begin{enumerate}[(i)]
	\item $X^0 \prec \XCC \prec \XNCC$, if $\Kurt \bx > 3$,
	\item $X^0 \prec \XCC \asymp \XNCC$, if $\Kurt \bx = 3$,
	\item $\XNCC \prec \XCC$, if $\Kurt \bx \in (1,3)$,
	\item $\XNCC = \XCC$, if $\Kurt \bx = 1$,
	\item $\LE(\XCC) = 0$, if $\Kurt\bx = 2$.
\end{enumerate}
Note that distributions with kurtosis $<3$ / $=3$ / $>3$ are also called platykurtic / mesokurtic / leptokurtic (see also Section \ref{sec:kurtosis})

Gradient flow is always the worst approximation for $\Kurt x \geq 3$. 
Assume we are in the platykurtic setting $\Kurt \bx \in (1,3)$. Then gradient flow is the worst / second-best / best approximation if 
\[1 < \BGF - \BEq \quad / \quad \BGF -\BEq < 1 < \BGF \quad / \quad \BGF < 1.\]

\subsection{A Numerical Example}
\label{sec:simulation}
In this subsection we present results from a numerical experiment confirming the theoretical results presented in Theorem \ref{thm:linRegBS}. We also compare the three approximations to the second-order SME (see Corollary \ref{cor:SME2}), which we here call \emph{second-order stochastic gradient flow}, or SGF2 for short. The corresponding family of \tSDE s is given by
\begin{align}
	\label{eq:SGF2}
	dX_t^{2,h} = &- \cR'(X_t^{2,h}) - \frac h 2 \cR''(X_t^{2,h}) \cR'(X_t^{2,h})\,dt + \sqrt{h \Si(X_t^{2,h})}\,dW_t \\ \nonumber
	= & -\ka\left(1_{d\times d} + \frac h2\ka\right)(X_t^h - \thet^*)\,dt + \sqrt{\frac h B} \sqrt{2\BEq\ka (X_t^{2,h} - \thet^*)^{\otimes 2} \ka + \si_\ep^2 \ka} \,dW_t
\end{align}
with $X_0^{2,h} = \chi_0$.

For the remainder of this section we exclusively work in setting (b) from Example \ref{ex:4thMomSetts}.

\subsubsection{Experimental Setup}
We consider using SGD for fitting the particular one-dimensional linear model
\begin{equation}
	\label{eq:exLinModl}
	\by = - \bx + \bep
\end{equation}
with $\bx, \bep$ independent, centered and of variance $1$, where $\bep$ is Gaussian. Note that in this case we have $\thet^* = -1$.
We compare the weak errors of the population risk $\cR$ for different continuous-time approximations of SGD. Here we use time horizons $T = 0.5$ and $T = 2.0$, varying distributions of $\bx$ and initial values $\thet$. We use a Monte Carlo approximation to estimate $\E \cR(\chi_{T/h}^h)$, \tIe
\[\E \cR(\chi_{T/h}^h) \approx \frac1 M \sum_{i=1}^M \E \cR(\hat\chi_{T/h}^{i,h})\]
where $\hat \chi^1, \dots, \hat \chi^M$ are independent copies of $\chi$. More precisely, to compute one copy $\hat \chi^i$ we draw $BT/h$ \tiid\ samples from the data-generating model \eqref{eq:exLinModl} and then perform SGD for $T/h$ steps using a batch of $B$ samples in each step, never using any sample twice. Thus, every copy of $\hat\chi$ uses a different (pseudo-) data set.
For the experiments we have chosen $M$ large enough (between $10^8$ and $2\cdot 10^9$) so that the variance of the Monte Carlo estimator is negligible compared to the weak error. 
Moreover, to reduce the computational burden significantly, we determine $\E \cR^e(Y_T^h)$ for $Y = X^0, \XNCC{h}, X^{\operatorname{CC}, h}, X^{2,h}$ using explicit formulas, which can be derived in this example (see Proposition \ref{prop:contTimeLinRegRisk} in Section \ref{sec:explFormLinReg}).
We consider the learning rates $h = 0.5, 0.1, 0.05, 0.01, 0.005, 0.001$. Notice that $T/h$ is an integer in each case, where $T \in \set{0.5, 2.0}$. Plotted is the dependence of the weak error 
\[\frac1 \ka |\E \cR(\chi_{T/h}^h) - \E \cR(Y_T^h)|,\]
\emph{divided by} $\ka$ (!), on the learning rate $h$. 
\subsubsection{Results}
In the following $\nu_x$ denotes any distribution with expectation $m$, such that $\bx + m \sim \nu_x$. That is $\bx$ has distribution $\nu_x$, but shifted to have expectation zero.
Figure \ref{fig:kurtcomparison} depicts the weak error's dependence on the learning rate in the following settings:
\begin{center}
	\begin{tabular}{c|c|c|c|c|c|c|c|c|c}
		Nr &	$T$ & $\thet$ & $\nu_x$  & $\kappa$ & $\operatorname{Kurt} \bx$ & $B$ & $B^{\operatorname{Eq}}$ & $B^{\operatorname{GF}} -B^{\operatorname{Eq}}$ & $B^{\operatorname{GF}}$ \\
		\hline
		(1) & $0.5$	& $0$ & $\operatorname{Exp}(0.1)$ & $10$  & $9$ & $1$  & $4$  & $114.127$ & $118.127$\\
		(2) & $0.5$	& $0$ & $\cN(0,1)$ & $1$  & $3$ & $1$  & $1$  & $1.85914$ & $2.85914$ \\
		(3) & $2.0$	& $0$ & $\cN(0,1)$ & $1$  & $3$ & $4$  & $1$  & $7.69977$ & $8.69977$ \\
		(4) & $0.5$	& $0$ & $\operatorname{Exp}(1)$ & $1$  & $9$ & $8$  & $4$  & $4.85914$ & $8.85914$\\
		(5) & $0.5$	& $0$ & $\cN(0,1)$ & $1$  & $3$ & $4$  & $1$  & $1.85914$ & $2.85914$ \\
		(6) & $0.5$	& $-0.9$ & $\cN(0,1)$ & $1$  & $3$ & $2$  & $1$  & $86.9141$ & $87.9141$ \\
	\end{tabular}
\end{center}

\begin{figure}
	\begin{minipage}[l]{0.485\textwidth}
		\includegraphics[width=\linewidth]{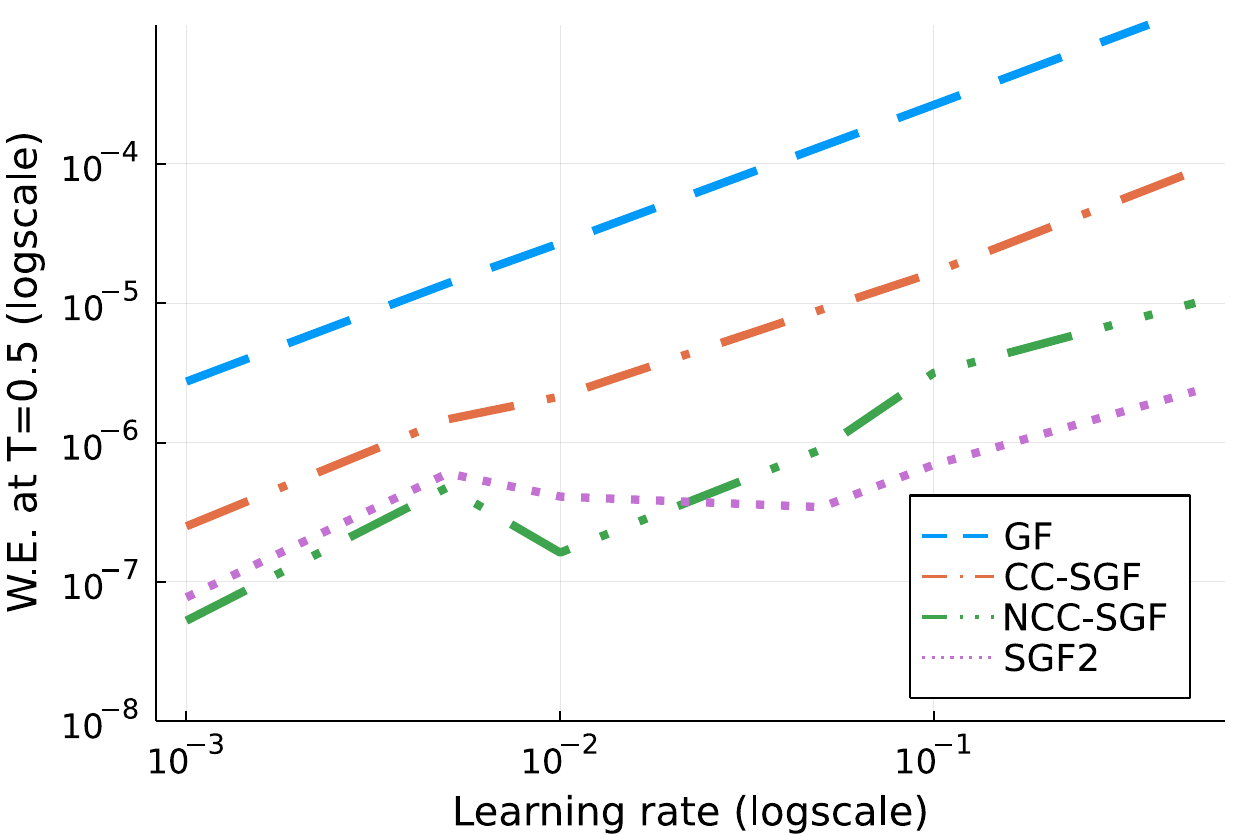} \centering{(1)}
	\end{minipage} 
	\hfill
	\begin{minipage}[r]{0.485\textwidth}
		\includegraphics[width=\linewidth]{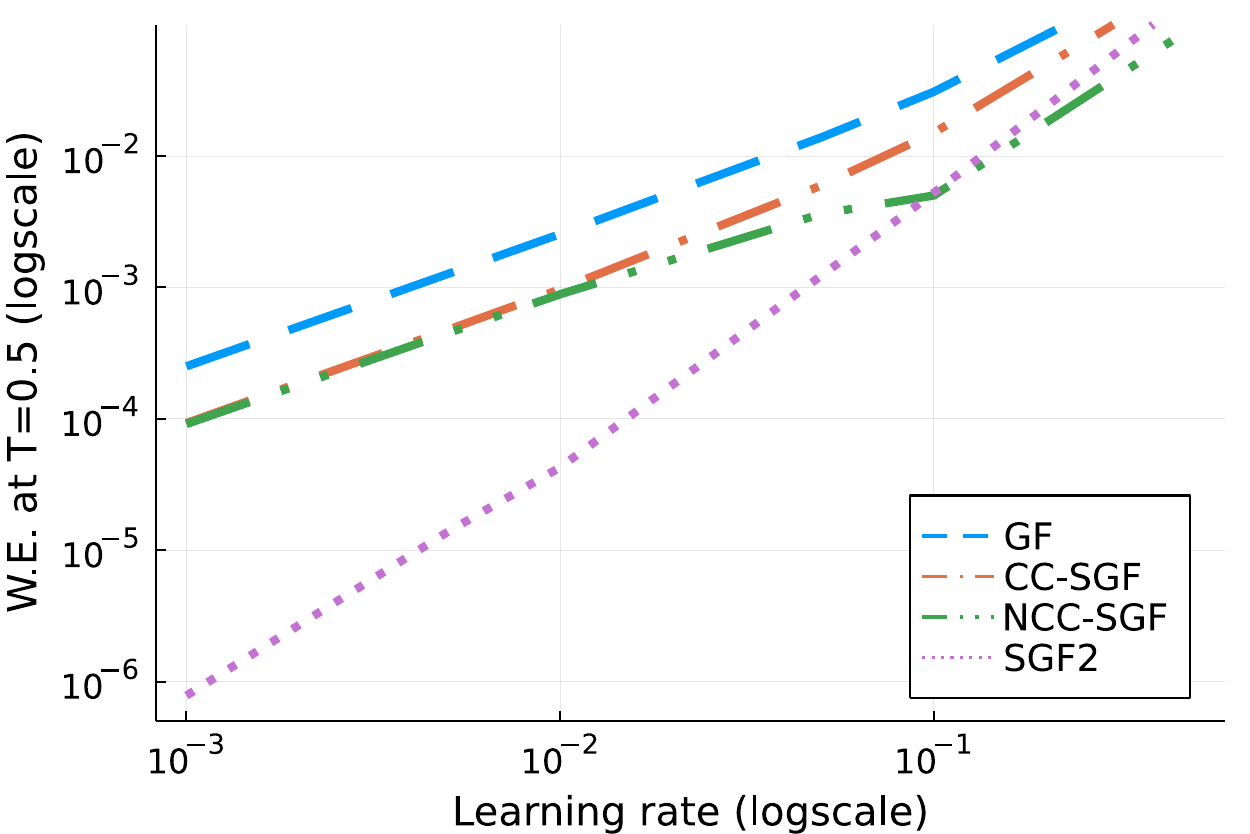}
		\centering{(2)}
	\end{minipage} 
	\begin{minipage}[l]{0.485\textwidth}
		\includegraphics[width=\linewidth]{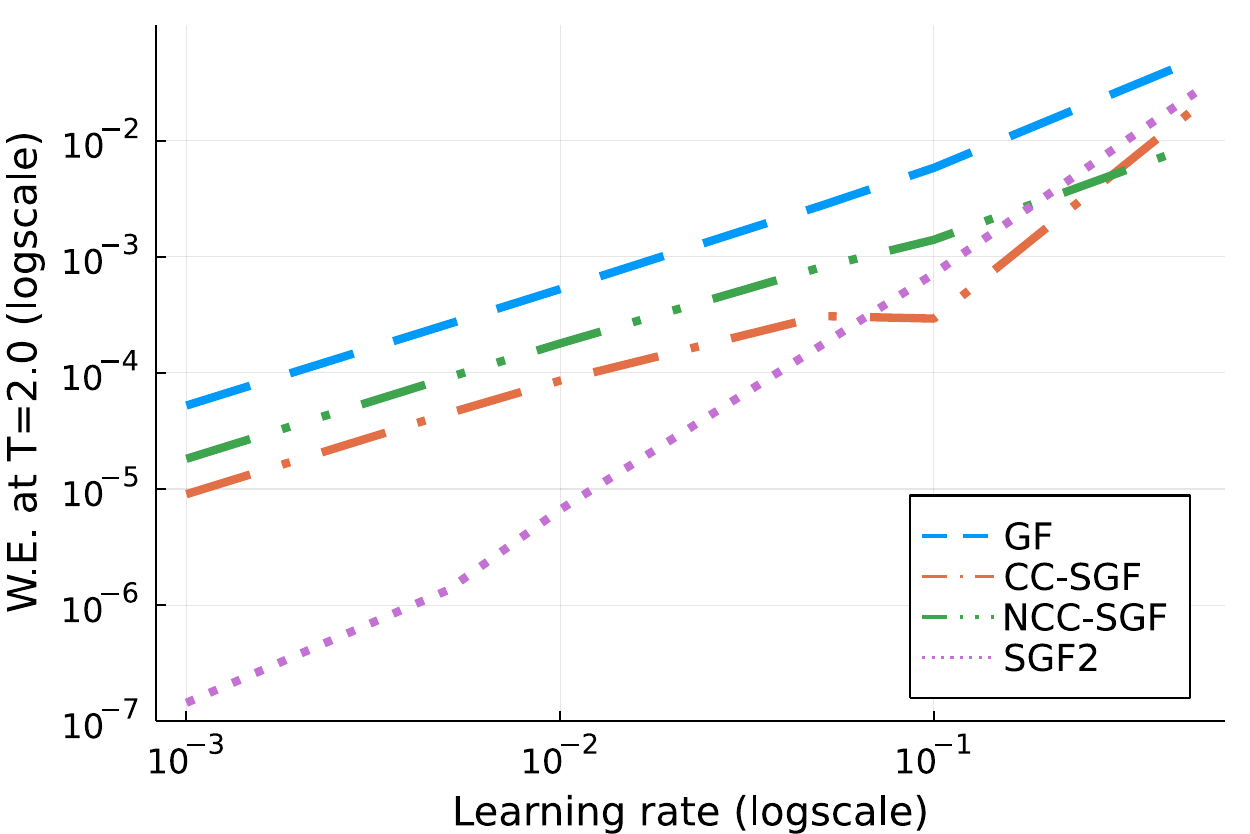}
		\centering{(3)}
	\end{minipage} 
	\hfill
	\begin{minipage}[r]{0.485\textwidth}
		\includegraphics[width=\linewidth]{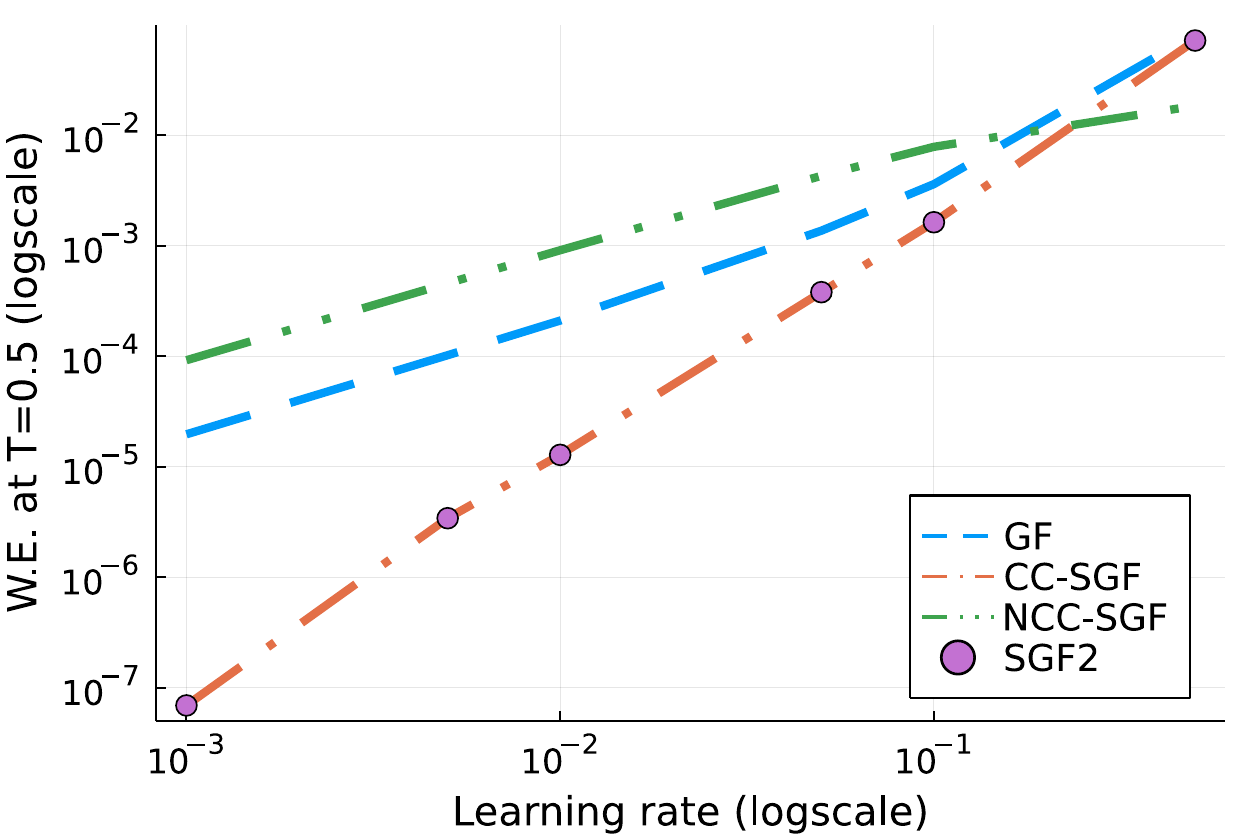}
		\centering{(4)}
	\end{minipage} 
	\begin{minipage}[l]{0.485\textwidth}
		\includegraphics[width=\linewidth]{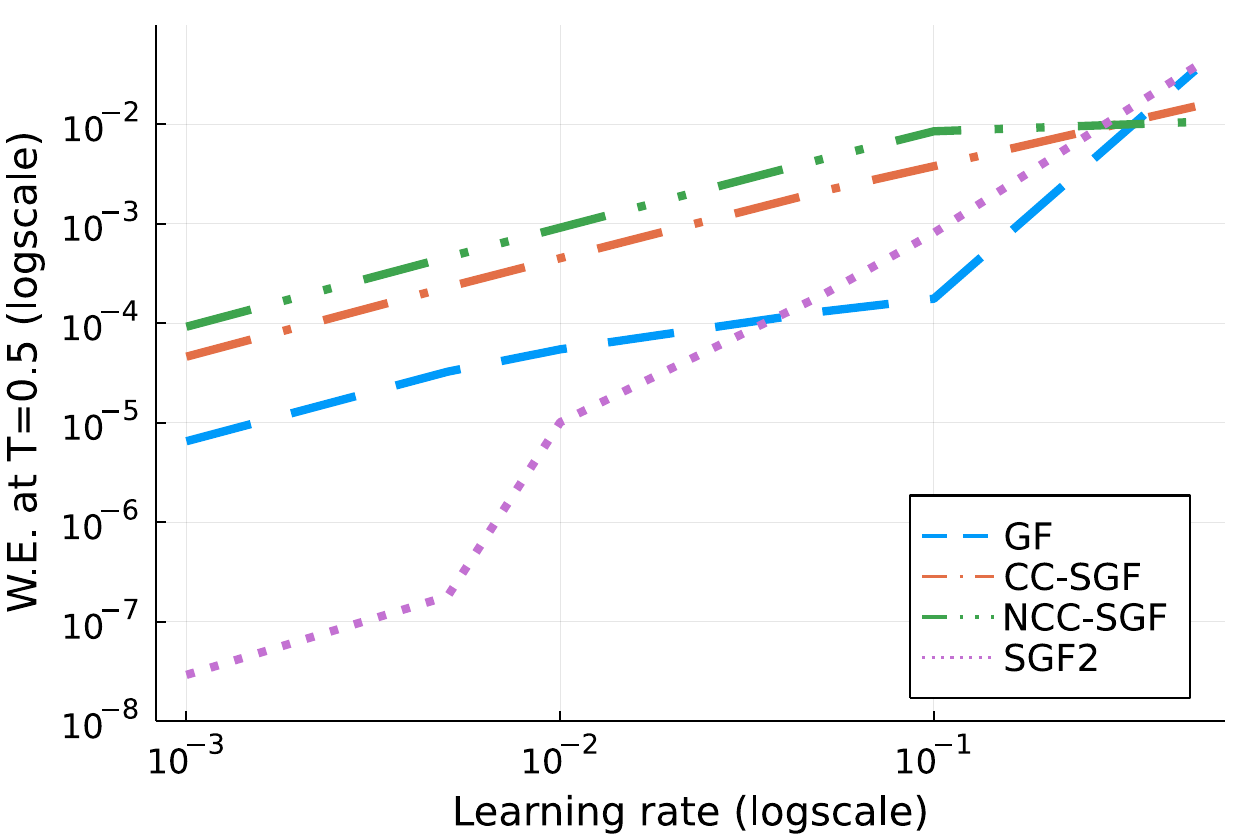}
		\centering{(5)}
	\end{minipage} 
	\hfill
	\begin{minipage}[r]{0.485\textwidth}
		\includegraphics[width=\linewidth]{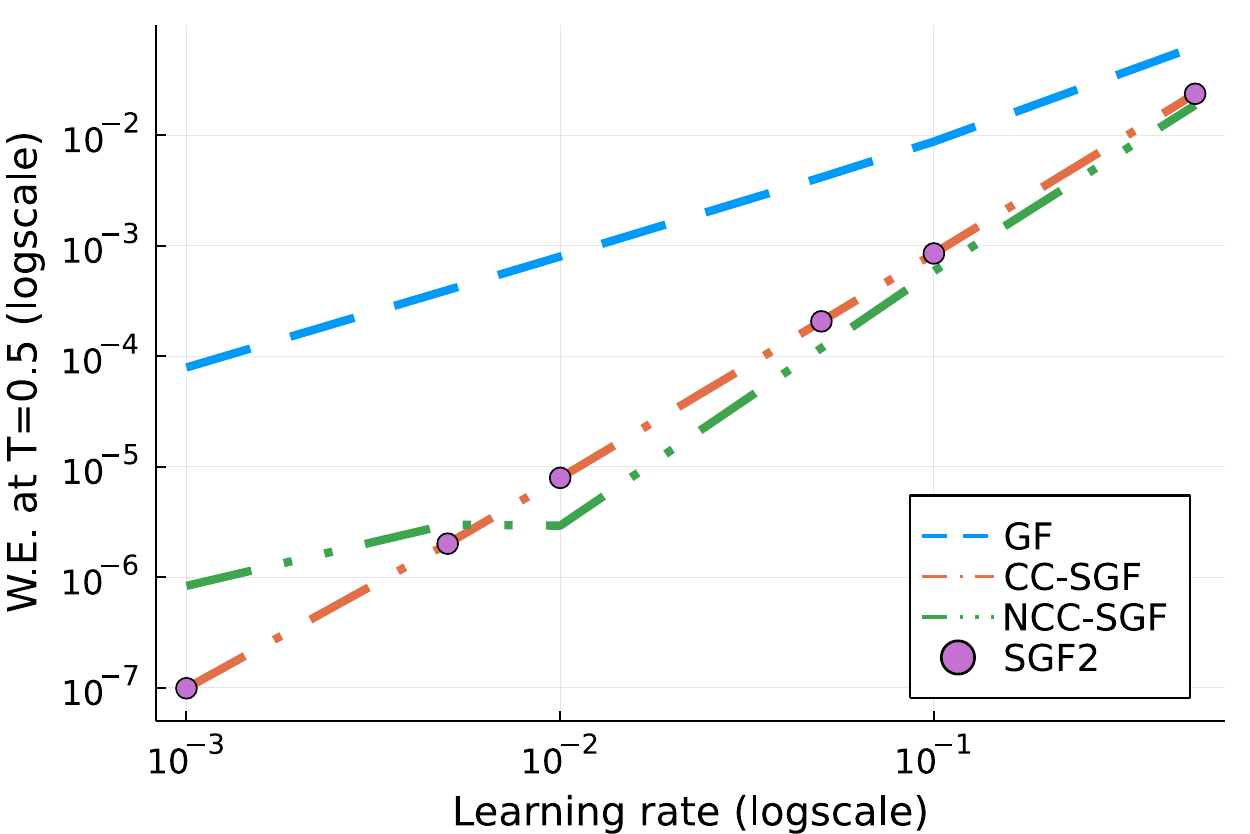}
		\centering{(6)}
	\end{minipage} 
	\caption{The weak error's dependence on the learning rate for several continuous-time approximations to SGD, in various settings. The plots (1)-(5) correspond to the settings (i)-(v) in Theorem \ref{thm:linRegBS}. Further, (4) and (6) also correspond to (vi). Finally, (6) depicts a situation where $X^{\operatorname{CC}} \not\asymp X^{\operatorname{NCC}}$, but the weak errors are close to each other since the common initial value is close to the minimum.}
	\label{fig:kurtcomparison}
\end{figure}

Aside from minor deviations stemming from the Monte Carlo estimation, the empirical results in Figure \ref{fig:kurtcomparison} confirm the theoretical results in the last subsection. In particular, we observe:
\begin{enumerate}[(i)]
	\item The experimental settings (1)---(5) correspond exactly to the settings (i)---(v) in Theorem \ref{thm:linRegBS}. Note that instead of merely varying the batch size $B$ we also varied $\BEq$ and $\BGF$ by choosing different $T$ and distributions of $\bx$.
	\item As indicated by Proposition \ref{prop:TTlinReg}, the experimental setting (6) shows that for $\thet \approx \thet^*$ and only moderately small learning rates there is little difference between the NCC- and the CC-SGF approximations, while gradient flow is lagging behind by neglecting to model the variance of the residuals $\si_\ep^2$.
	\item For $B = B^{\operatorname{Eq}}$, NCC- and CC-SGF are equally good (setting (2)).
	\item For $B = 2B^{\operatorname{Eq}}$ the CC-SGF approximation is of second order\footnote{More precisely, the approximation is of order $2$ for the chosen test function $\cR$. This is a weaker property than being a second-order SME.} (settings (4) and (6)).
	\item The SGF2 approximation is always best, irrespective of batch size.
\end{enumerate}
We remark that the theoretical rates of convergence are difficult to observe without using a high number of Monte Carlo samples.
Moreover, note that in the experiments we always plotted the weak error while Theorem \ref{thm:linRegBS} only applies to the linear error term. The results indicate that the higher order error terms have negligible impact on the total error.

\section{Derivations and Proofs for Section \ref{sec:comparison}}
\label{sec:comparisonproofs}
In this section we give proper justifications for the results of Section \ref{sec:comparison}.
\subsection{Quadratic Objectives}
\label{sec:LEquadObj}
Here, we derive the linear error terms for the three continuous-time approximations when the objective function is quadratic. This includes ordinary linear regression with SGD using the population risk, but the derivation applies more generally. 

Suppose we are given a symmetric and positive definite matrix $\ka \in \R^{d\times d}$ and a quadratic form
\[\cR(\thet) = \frac12\thet^\transp \ka \thet + \thet^\transp c' + d',\quad  \thet \in \R^d,\]
where $c' \in \R^d$ and $d' \in \R$.
Then $\cR$ has a global minimum $\thet^* \in \R^d$ and so we may rewrite it as 
\[\cR(\thet) = \frac12\innp{\ka}{(\thet-\thet^*)^{\otimes 2}} + d, \quad \thet \in \R^d\]
for some $d\in \R$.
Now, consider SGD with $\bar H(\thet) = - \nabla \cR(\thet)$. The gradient flow equation
\[dX_t^0 = - \nabla \cR(X_t^0) = - \ka (X_t^0 - \thet^*)\,dt,\]
has the unique solution
\[X_t^0(\thet) = e^{-t\ka}(\thet - \thet^*) + \thet^*, \quad t\in [0,T],\]
for every initial condition $\thet \in \R^d$.
Note that $X_t^0(\thet) \to X_\infty^0(\thet) = \thet^*$, as $t\to \infty$, for every $\thet\in \R^d$.
Set $v_t(\thet) = \cR(X_T^{0,t}(\thet)), t\in [0,T], \thet \in \R^d$. Then, given $\thet \in \R^d$ and $t\in [0,T]$,
\begin{align*}
v_t(\thet) 	= & \cR(e^{-(T-t)\ka}(\thet - \thet^*) + \thet^*) \\
					= & \frac12 \innp{\ka}{(e^{-(T-t)\ka}(\thet - \thet^*))^{\otimes 2}} + d\\
					= &  \frac12 \innp{\ka}{(e^{-(T-t)\ka}(\thet - \thet^*))(e^{-(T-t)\ka}(\thet - \thet^*))^\transp} + d\\
					= & \frac12 \innp{\ka}{e^{-(T-t)\ka}(\thet - \thet^*)^{\otimes 2}e^{-(T-t)\ka}} + d\\
					= & \frac12 \innp{\ka e^{-2(T-t)\ka}}{(\thet - \thet^*)^{\otimes 2}},
\end{align*}
Here, we used the property
\[\innp{A}{BCD} = \innp{B^\transp A D^\transp}{C}, \quad A,B,C,D\in \R^{d\times d},\]
and the fact that $\ka$ and $e^{-2(T-t)\ka}$ are symmetric and commute with each other.
Further,
\[\nabla \cR(\thet) = \ka(\thet-\thet^*), \nabla^2\cR(\thet) =  \ka.\]
Therefore,
\begin{align*}
\innp{\nabla v_t(\thet)}{\nabla^2 \cR \nabla \cR} = & \innp{\ka e^{-2(T-t)\ka}(\thet - \thet^*)}{\ka^2(\thet-\thet^*)} \\
					= & \innp{\ka^3 e^{-2(T-t)\ka}}{(\thet-\thet^*)^{\otimes 2}},\\
\innp{\nabla v_t}{\nabla^2 \cR \nabla \cR}(X_t^0(\thet)) = & \innp{\ka^3 e^{-2(T-t)\ka}}{(X_t^0-\thet^*)^{\otimes 2}}\\
	= & \innp{\ka^3 e^{-2(T-t)\ka}}{(e^{-t\ka}(\thet - \thet^*))^{\otimes 2}}\\
	= & \innp{\ka^3 e^{-2T\ka}}{(\thet - \thet^*)^{\otimes 2}},\\
\nabla^2 v_t(\thet) = & \ka e^{-2(T-t)\ka}.
\end{align*}
Now, consider $D\in \set{0, \Si(\thet^*), \Si}$ and
\[dX_t^h = - \nabla \cR(X_t^h) + \sqrt{hD(X_t^h)}\,dW_t, \quad t\in [0,T], h \in \cH.\]
By Theorem \ref{thm:firstOrderSME} we have
\begin{align}
	\label{eq:LEquad}
	\LE(X) = & \frac12\int_0^T \innp{\nabla^2 v_t}{(\Si-D)}(X_t^0)-\innp{\nabla v_t}{\nabla^2 \cR \nabla \cR}(X_t^0)\,dt \nonumber\\
	= &  \frac12\int_0^T\innp{\ka e^{-2(T-t)\ka}}{(\Si-D)(X_t^0)}\,dt-\frac12 T\innp{\ka^3 e^{-2T\ka}}{(\thet - \thet^*)^{\otimes 2}}.
\end{align}

\begin{proof}[Proof of Proposition \ref{prop:TTlinReg}]
	Recall Equation \eqref{eq:LEquad}. The first equation in Proposition \ref{prop:TTlinReg} follows by setting $\Si = D$. Moreover,
	\[\Si(X_t^0) - \Si(\thet^*) =2\frac{B^{\operatorname{Eq}}}{B}\ka e^{-t\ka}(\thet - \thet^*)^{\otimes 2}e^{-t\ka}\ka,\]
	and so
\begin{align*}
\innp{\ka e^{-2(T-t)\ka}}{\Si(X_t^0)-\Si(\thet^*)}\,dt = & 2\frac{B^{\operatorname{Eq}}}{B}\innp{\ka e^{-2(T-t)\ka}}{\ka e^{-t\ka}(\thet - \thet^*)^{\otimes 2}e^{-t\ka}\ka}\\
= &2\frac{B^{\operatorname{Eq}}}{B}\innp{\ka^3 e^{-2T\ka}}{(\thet - \thet^*)^{\otimes 2}}.
\end{align*}	
	Therefore, by Equation \eqref{eq:LEquad},
	\begin{align*}
		\LE(X^{\text{CC}}) = & T \left(\frac{B^{\operatorname{Eq}}}{B} - \frac12\right)\innp{\ka^3  e^{-2T\ka}}{(\thet - \thet^*)^{\otimes 2}}.
	\end{align*}
	Moreover,
	\begin{align*}
		\LE(X^0) = & \frac12\int_0^T\innp{\ka e^{-2(T-t)\ka}}{\Si(X_t^0)}\,dt-\frac12 T\innp{\ka^3 e^{-2T\ka}}{(\thet - \thet^*)^{\otimes 2}}\\
		= & \LE(X^{\text{CC}}) + \frac12\int_0^T\innp{\ka e^{-2(T-t)\ka}}{\Si(\thet^*)}\,dt\\
		= & \LE(X^{\text{CC}}) + \frac{1}{2B}\si_\ep^2 \innp{\ka^2}{\int_0^T e^{-2(T-t)\ka}\,dt}.
	\end{align*}
	Finally, since $\ka$ is positive definite, we may simplify
	\[\frac1{2B} \si_\ep^2 \innp{\ka^2}{\int_0^T e^{-2(T-t)\ka}\,dt} = \frac1{4B} \si_\ep^2 \innp{\ka^2}{(1_{d\times d} - e^{-2\ka T}) \ka^{-1}} = \frac1{4B} \si_\ep^2 \innp{\ka}{1_{d\times d} - e^{-2\ka T}}.\]
\end{proof}
The following lemma is used in the proof of Theorem \ref{thm:linRegBS}.
\begin{lemma}
	\label{lem:absComp}
	Let $a,b_1,b_2,B > 0$ with $b_1 < b_2$ and set $e_i = -a + \frac{b_i}B$
	Then,
	\[\sgn(|e_1|-|e_2|) = \sgn\left(B - \frac{b_1 + b_2}{2a}\right).\]
\end{lemma}
\begin{proof}
	Note that $B \leq \frac{b_i}{a}$ if and only if $e_i \geq 0$, and $B \geq \frac{b_i}a$ if and only if $e_i \leq 0$. Moreover,
	\[\frac{b_1}{a} < \frac{b_1 + b_2}{2a} < \frac{b_2}{a}.\]
	Thus, we have $|e_1| < |e_2|$ if and only if
	\begin{enumerate}[(a)]
		\item $B \leq \frac{b_1}{a}$ and $e_1 < e_2$, or
		\item $\frac{b_2}{a} < B \leq \frac{b_1}{a}$ and $e_1 < - e_2$, or
		\item $\frac{b_1}{a} < B \leq \frac{b_2}{a}$ and $-e_1 < e_2$, or
		\item $B > \frac{b_2}{a}$ and $-e_1 < -e_2$.
	\end{enumerate}
	Since $e_1 < e_2$, case (d) can never occur and (a) is equivalent to $B \leq \frac{b_1}{a}$. Further, since $b_1 < b_2$, (b) is also impossible.
	Moreover, (c) is equivalent to
	\[\frac{b_1}{a}\leq B < \frac{b_1 + b_2}{2a}.\]
	Putting (a) and (c) together yields
	\[|e_1| < |e_2| \Ioif B < \frac{b_1 + b_2}{2a}.\]
	Finally, since
	\[|e_1| = |e_2| \Ioif e_1 = - e_2 \Ioif B = \frac{b_1 + b_2}{2a},\]
	the result follows.
\end{proof}

\begin{proof}[Proof of Theorem \ref{thm:linRegBS}]
	Set 
	\[a := \frac 12 T \innp{\ka^3 e^{-2T\ka}}{(\thet - \thet^*)^{\otimes 2}}, b := B^{\operatorname{Eq}} T \innp{\ka^3 e^{-2T\ka}}{(\thet - \thet^*)^{\otimes 2}}, c := \frac14 \si_\ep^2 \innp{\ka}{1_{d\times d} - e^{-2\ka T}} > 0.\] By definition and Proposition \ref{prop:TTlinReg}
	\[\LE(X^{\operatorname{NCC}}) = -a,\quad \LE(X^{\operatorname{CC}}) = -a + \frac{b}{B}, \quad \LE(X^0) = -a + \frac{b}{B} + \frac{c}{B}.\]
	Lemma \ref{lem:absComp} implies
	\begin{align*}
		B < \frac b {2a} \Ioif |\LE(X^{\operatorname{NCC}}) | < |\LE(X^{\operatorname{CC}})|, & \quad B > \frac b {2a} \Ioif |\LE(X^{\operatorname{NCC}}) | > |\LE(X^{\operatorname{CC}})|, \\
		B < \frac{b+c}{2a} \Ioif |\LE(X^{\operatorname{NCC}}) | < |\LE(X^0)|, & \quad B > \frac{b+c}{2a} \Ioif  |\LE(X^{\operatorname{NCC}}) | > |\LE(X^0)|,\\
		B < \frac{2b + c}{2a} \Ioif |\LE(X^{\operatorname{CC}})| < |\LE(X^0)|, & \quad B > \frac{2b + c}{2a} \Ioif |\LE(X^0)| > |\LE(X^{\operatorname{CC}})|.
	\end{align*}
	Further,
	\[B^{\operatorname{Eq}} = \frac{b}{2a}, \quad \frac{c}{2a} = \frac{\si_\ep^2 \innp{\ka}{1-e^{-2T\ka}}}{4T\innp{\ka^3e^{-2T\ka}}{(\thet-\thet^*)^{\otimes 2}}}, \quad B^{\operatorname{GF}} = \frac{2b + c}{2a},\]
	and so the cases (i) - (iv) are proven. Finally,
	\begin{align*}
		\LE(X^{\operatorname{CC}}) = 0 &\Ioif B = \frac{b}{a} = 2 B^{\operatorname{Eq}}, \\
		\LE(\XCC) = 0 &\Ioif B = \frac{b+c}{a} = \frac{2b+c}{2a} + \frac{c}{2a} = B^{\operatorname{GF}} +  B^{\operatorname{GF}} - 2  B^{\operatorname{Eq}},
	\end{align*}
	proving (v) and (vi).
\end{proof}

\begin{remark}
	There are few additional statements one can make, adding to the list in Theorem \ref{thm:linRegBS}. Firstly,
	\begin{enumerate}[(i)]
		\item $X^0 \asymp \XNCC$, if $B = \BGF - \BEq$,
		\item $X^0 \asymp \XCC$, if $B = \BGF$,
		\item $\LE(X^0) = 0$, if $B = 2(B^{\operatorname{GF}} - B^{\operatorname{Eq}})$.
	\end{enumerate}
	Note however that these will almost never occur in practice because it is unlikely that $\BGF$ is an integer. That is, unless one specifically designs the problem in such a way. On the other hand, notice that $\BEq = 1$ if $\bx$ is Gaussian and $\BEq = 4$ if $d = 1$ and $\bx$ is exponentially distributed and so the case (ii) in Theorem \ref{thm:linRegBS} can realistically occur in applications.
	
	Moreover, note that for $\BEq = 0$ we have $\Si(\thet) = \Si(\thet^*)$ for all $\thet\in \R^d$ and so $\XCC = \XNCC$. In particular, this happens for $d = 1$ and if $\bx$ has a symmetric Rademacher distribution, since then $\Kurt x = 1$ (recall example \ref{ex:4thMomSetts}).
	Thus, we are left with the cases
	\begin{enumerate}[(i)]
		\item $X^0 \prec \XNCC$, if $B < \BGF$,
		\item $\XNCC \prec X^0$, if $B > \BGF$.
	\end{enumerate}
\end{remark}

\subsection{Explicit Formulas for the Expected Risk of the Continuous-Time Approximations of SGD for Linear Regression}
\label{sec:explFormLinReg}
Here, we derive explicit formulas for the expected (excess) population risk for four\\ continuous-time approximation of SGD for linear regression. These are used in the numerical experiments to compute the continuous-time half of the weak error.
Firstly, we consider the following families of modified equations
\begin{align*}
	dX_t^0= & - \ka(X_t^0 - \thet^*)\,dt, \nonumber \\
	d\XNCCh{h}_t = & - \ka(X_t^h - \thet^*)\,dt + \sqrt{\frac h B} \sqrt{2\BEq\ka (\XNCCh{h}_t - \thet^*)^{\otimes 2} \ka + \si_\ep^2 \ka} \,dW_t, \nonumber \\
	dX_t^{\operatorname{CC},h}	= & - \ka(X_t^{\operatorname{CC}, h} - \thet^*)\,dt + \sqrt{\frac h B \si_\ep^2 \ka } \,dW_t.\\
	dX_t^{2,h} = &-\ka\left(1_{d\times d} + \ka\frac h2\right) (X_t^h - \thet^*)\,dt + \sqrt{\frac h B} \sqrt{2\BEq\ka (\thet - \thet^*)^{\otimes 2} \ka + \si_\ep^2 \ka}\,dW_t.
\end{align*}
For simplicity we set $d = 1$ and $\BEq = \frac12(\Kurt \bx - 1)$.
The next Proposition gives explicit formulas for the expected \emph{excess population risk} $\E[\cR^e(Y_t)]$ for $Y \in \set{X^0, \XNCCh{h}, X^{\operatorname{CC},h}, X^{2,h}}$, where $\cR^e(\thet) = \frac12(\thet-\thet^*)^2$. The actual population risk is also given by $\cR = \ka \cR^e + \frac{\si_\ep^2}{2}$. Note that
\[\cR^e(\thet) - \cR^e(\tilde \thet) = \frac{1}{\ka}(\cR(\thet) - \cR(\tilde\thet)), \quad \thet,\tilde \thet \in \R.\]
\begin{proposition}
	\label{prop:contTimeLinRegRisk}
	Define
	\[\zeta^h = 1 - \frac h {2B} \ka(\Kurt \bx - 1), \quad \xi^h := \zeta^h + \frac h2 \ka = 1 + \frac{h}{2B}\ka (B + 1 - \Kurt \bx), \quad h\in [0,1).\]
	Then, we have
	\begin{align*}
		\cR^e(X_t^0) = & e^{-2\ka t}\cR^e(\thet), \\
		\E[\cR^e(X_t^{\operatorname{CC},h})] = & e^{-2\ka t}\cR^e(\thet) + \frac{h \si_\ep^2}{4B} (1 - e^{-2\ka t}),\\
		\E[\cR^e(\XNCCh{h}_t)] = &e^{-2\ka \zeta_h t}\cR^e(\thet) + \frac{h\si_\ep^2}{4B \zeta_h}(1 - e^{-2\ka \zeta_h t}),\\
		\E[\cR^e(X_t^{2,h})] = &e^{-2\ka \xi_h t}\cR^e(\thet) + \frac{h\si_\ep^2}{4B \xi_h}(1 - e^{-2\ka \xi_h t}),
	\end{align*}
	for all $h\in (0,1)$ and $t\geq 0$.
\end{proposition}
\begin{proof}
	Recall that
	\[X_t^0 = e^{-\ka t}(\thet - \thet^*) + \thet^*,\]
	and so
	\[\cR^e(X_t^0) = e^{-2\ka t}\cR^e(\thet).\]
	Further, $X^{\text{CC},h}$ is an Ornstein-Uhlenbeck process and so
	\[X_t^{\operatorname{CC},h} = X_t^0 + \sqrt{\frac{h \si_\ep^2}{2B}} W_{1 - e^{-2\ka t}}.\]
	Hence,
	\[\E[\cR^e(X_t^{\operatorname{CC},h})] = e^{-2\ka t}\cR^e(\thet) + \frac{h \si_\ep^2}{4B} (1 - e^{-2\ka t}).\]
	Now, by \tIto's formula
	\begin{align*}
		d\cR^e(\XNCCh{h}_t) = &-\ka(\XNCCh{h}_t - \thet^*)^2 + \frac h {2B} \ka^2(\Kurt \bx - 1)(\XNCCh{h}_t- \thet^*)^2 + \frac h {2B} \ka \si_\ep^2\,dt + M_t\\
		= &\left(\frac h B \ka^2(\Kurt \bx - 1) - 2\ka\right) \cR^e(\XNCCh{h}_t) + \frac h {2B} \ka \si_\ep^2\,dt + M_t
	\end{align*}
	where $M$ is a martingale starting in $0$, a.s. Hence, by optional stopping
	\[d\E[\cR^e(\XNCCh{h}_t)] = -2 \ka \zeta_h \E[\cR^e(\XNCCh{h}_t)] + \frac h {2B} \ka \si_\ep^2\,dt,\]
	and so
	\[\E[\cR^e(\XNCCh{h}_t)] = e^{-2\ka \zeta_h t}\cR^e(\thet) + \frac{h\si_\ep^2}{4B \zeta_h}(1 - e^{-2\ka \zeta_h t}).\]
	Similarly,
	\[\E[\cR^e(X_t^{2,h})] = e^{-2\ka \xi_h t}\cR^e(\thet) + \frac{h\si_\ep^2}{4B \xi_h}(1 - e^{-2\ka \xi_h t}).\]
\end{proof}


\chapter{SMEs for SGD without replacement}
\label{chap:towardsSGDoSME}
Training algorithms using epochs, that is those based on stochastic gradient descent without replacement (SGDo) are predominantly used to train machine learning models in practice. However, the mathematical theory of SGDo and related algorithms remain underexplored compared to their \enquote{with replacement} and \enquote{one-pass} counterparts. Further, there is no existing theory of modified equations for SGDo. 
In this chapter, we propose a stochastic, continuous-time approximation to SGDo with additive noise based on a \emph{Young differential equation} driven by a stochastic process we call \emph{epoched Brownian motion}. We show its usefulness by proving the convergence of the solution of this continuous-time approximation for strongly convex objectives and the learning rate schedule $u_t = \frac{1}{(1+t)^\be}, \be \in (0,1)$, and computing an upper bound on the asymptotic rate of convergence.

This chapter is inspired by and significantly expands on \citet{perko2022towards}.

\section{Introduction}
Consider a risk minimization problem $(R : \R^d \times \cZ \to [0,\infty), \nu)$ on a measurable space $\cZ$. Fix an \tiid\ sequence $(\bz(n))_{n\in \N_0}$ in $\cZ$ with $\bz(0) \sim \nu$. For now, consider one-pass SGD with a sequence of learning rates $(\eta_n)_{n\in \N}$, given by
\begin{equation}
	\label{eq:introSGDlr}
	\chi_{n+1} = \chi_n - \eta_n \nabla R_{\bz({n})} (\chi_n), \quad h \in(0,1), n \in \N_0.
\end{equation}
In order to better understand SGD several authors have proposed approximating their dynamics by the solution of an SDE. In particular, in the case of a constant learning rate ($\eta_n = h$), \citet{mandt2015continuous} propose the following family of \tSDE s as an approximation of \eqref{eq:introSGDlr}
\[dY_t^h = - \nabla \cR(Y_t^h)\,dt + \sqrt{h}\si\,dW_t.\]

Here, $\si$ is a symmetric and positive semi-definite matrix approximating the gradient covariance in a \enquote{region of interest}, $W$ is a $d$-dimensional Brownian motion, and $\cR = \E R_{z(0)}$. Time is scaled in such a way that heuristically we have $Y_{nh}^h \approx \chi_n$.
Consider now a learning rate schedule $u : [0,\infty) \to [0,1]$ such that $\eta_n = h u_{nh}$.
\citet{Li15} further investigated this case of a non-constant learning rate schedules, and they heuristically used the following non-homogeneous dynamics
\begin{equation}
\label{eq:SME1}
dY_t^h = - u_t\nabla \cR(Y_t^h)\,dt + u_t\sqrt{h\Si(Y_t^h)}\,dW_t.
\end{equation}
The presence of $u$ in both coefficients can be motivated as follows. By multiplying the stochastic gradients with $u$, the expected gradients are multiplied by $u$ and their covariance by $u^2$. Thus, the diffusion coefficient - being the square root of the covariance is multiplied by $u$ as well.
While high learning rates seem to promise fast convergence via the drift, they also increase the variance of the gradients. A well-chosen learning rate schedule should thus balance both effects to ensure convergence.

Theorem \ref{thm:firstOrderSME} implies that under certain regularity conditions \eqref{eq:SME1} is a first-order SME of SGD. However, we know from Chapter \ref{chap:compare} that, among first-order SMEs, choosing a state-dependent diffusion coefficient is not always better than a state-independent one (see in particular Theorem \ref{thm:linRegBS}).
Therefore, in the following we elect to work with the simpler additive noise approximation of the form 
\begin{equation}
\label{eq:introCCSGFu}
dY_t^h = - u_t\nabla \cR(Y_t^h)\,dt + \sqrt{h}u_t\si\,dW_t,
\end{equation}
in the spirit of \citet{mandt2015continuous}.

The Markov property of Brownian motion says that the future is independent of the past given the current state. In the approximation \eqref{eq:SME1} this reflects the idea that all future data points of SGD are new data points, independent of those we have seen so far.

Consider now a finite \tiid\ sequence $(\bz(n))_{n=0}^{N-1}$ with $\bz(0)\sim \nu$, and the following variant of SGD, called SGD \emph{without replacement (with finite data)} (SGDo)
\begin{equation}
	\label{eq:genericSGDo}
	\chi_{n+1} = \chi_n - hu_{nh} \nabla R_{\bz({\pi^{\floor{n/N}}(n \modu N)})} (\chi_n), \quad n \in \N_0.
\end{equation}
Here, $(\pi^j)_{j\in \N_0}$ is a sequence of permutations of the set $\set{0,\dots, N-1}$. Wlog we set $\pi^0 = \id{}$.
Then the dynamics \eqref{eq:genericSGDo} and \eqref{eq:introSGDlr} coincide for $n\in \set{0,\dots, N-1}$. In the following \emph{epoch}, i.e.\ for $n \in \set{N,\dots, 2N-1}$, we reuse the same finite sample $(\bz(k))_{k=0}^{N-1}$ in perhaps a different order $(\bz(\pi^1(k)))_{k=0}^{N-1}$. We continue on like this in subsequent epochs using the sequence of permutations $(\pi^j)_{j\in \N_0}$. In general, we allow $(\pi^j)_{j\in \N_0}$ to be random, but independent of $(\bz(n))_{n=0}^{N-1}$.

For $t\in [0,T]$ with $T = Nh$, Equation \eqref{eq:introCCSGFu} is a reasonable approximation of \eqref{eq:genericSGDo}. However, Equation \eqref{eq:genericSGDo} no longer defines a Markov process for $n\geq N$ on the state space $\R^d$, because it cannot be written in the form $\chi_{n+1} = g(\chi_n, Z_n)$ for some \tiid\ sequence $(Z_n)_{n\in \N_0}$.
Thus, the Markov property for the driver $W$ in Equation \eqref{eq:introCCSGFu} is no longer appropriate if we try to find a continuous-time model for SGDo (for finite data).

For now, let us consider \emph{single-shuffle} SGDo, that is we choose\footnote{Technically, in the literature on SGDo \enquote{single shuffle} means \enquote{shuffle once}. We assume no shuffling here because it makes no difference: the distribution of the sample is unaffected.} $\pi^j = \id{}, j \geq 1$.
Given $T > 0$ and a Brownian motion $W : \Om \times [0,T]\to \R^d$, define
\[\hat W_t := W_{\frk{t/T}T} + \floor{t/T}W_T, \quad t\geq 0.\]
Here, $\frk r = r - \floor r$ is the fractional part of $r\in \R$.
Note that $\hat W$ is a Brownian motion when restricted to the interval $[0,T)$, and $\hat W$ satisfies
\begin{equation*}
\label{eq:epochProperty}
\hat W_{t+jT} = \hat W_t + jW_T, \qquad t \geq 0, j\in \N_0.
\end{equation*}
Note that $\hat W$ is almost surely continuous and even locally Hölder continuous.
The increments of $\hat W$ on $[jT, (j+1)T]$ coincide with the increments of $W$ on $[0, T]$ (up to translating time).  
We call $\hat W$ a \emph{single shuffle Brownian motion} with period $T$. The fact that we reuse the same Brownian path $(W_t)_{t\in [0,T]}$ corresponds to using the same data points in the same order in later epochs (single-shuffle).

By replacing the driving path of the diffusion in \eqref{eq:introCCSGFu} by single shuffle Brownian motion, we arrive at the following differential equation with additive noise
\begin{equation}
	\label{eq:epochedOU}
	dY_t = - u_t \nabla \cR(Y_t) \,dt + u_t \sqrt h \si\,d\hat W_t.
\end{equation}
Since $\hat W$ is not a semimartingale we cannot interpret the term $ u_t\,d\hat W_t$ using \tIto{} integration. Instead, we interpret it pathwise as the Young integral 
\[\int_0^t u_s \,d\hat W_s = \lim_{|\cP| \to 0} \sum_{[r,s]\in \cP} u_r (\hat W_s - \hat W_r),\]
where the limit is taken with respect to all partitions of $[0,t]$ with mesh size $|\cP|$. The integral exists for example if $u$ is Lipschitz. Thus, we understand \eqref{eq:epochedOU} as Young differential equation.

More generally, we allow the driver $\hat W$ in Equation \eqref{eq:epochedOU} to be an \emph{epoched Brownian motion} (EBM). An EBM $\hat W$ is roughly speaking a single shuffle Brownian motion, except on $[jT, (j+1)T]$ the increments of $\hat W$ may be \enquote{infinitesimally shuffled} according to $\pi^j$ (see Section \ref{sec:ebms} for a proper explanation). We can thereby encode different shuffling schemes for SGDo in the approximating equation \eqref{eq:epochedOU}.

Previous works on SGDo have mainly focused on comparing the convergence rates of SGD with replacement and SGDo, where empirically the latter is known to converge faster. \citet{shamir2016} establishes lower bounds on \emph{in expectation} convergence rates for SGDo with constant learning rates.
\citet{nagaraj2019} use the method of exchangeable pairs to derive non-asymptotic in expectation convergence results for general smooth, strongly convex functions.

\citet{gurbuzbalaban_why_2021} focuses on the speed of in expectation and almost sure convergence for single-shuffle and random reshuffling SGDo. The later algorithm uses an \tiid\ sequence $(\pi^j)_{j\in \N_0}$ of permutations where $\pi^0$ uniformly distributed. Using martingale techniques, they prove convergence results for learning rates decaying like the schedule $u_t = \frac{1}{(1+t)^\be}, t\geq 0$ with $\be \in(1/2,1]$, and strongly convex $\cR$.

To demonstrate the usefulness of our heuristic SME approximation \eqref{eq:epochedOU}, we study the almost sure convergence of the solution of \eqref{eq:epochedOU} for Lipschitz and strongly convex $\cR$ with Hölder continuous Hessian matrix, and with $u_t = \frac{1}{(1+ct)^\be}, t\geq 0$ with $\be \in (0,1)$ and $c > 0$. Here, we leave out the case $\be = 1$ for brevity reasons. On the other hand, we cover the case $\be \in (0,1/2]$ as well since our main strategy uses the Young-Lóeve inequality instead of martingale techniques.
We show convergence to a random point depending on $\hat W_T$ and compute an asymptotic upper bound on the convergence speed. Our result for the single shuffle cases matches previous results by \citet{gurbuzbalaban_why_2021}. In the case of general random permutations, our results suggest markedly better upper bounds than the best results known for random reshuffling. Note that, heuristically speaking, $\hat W_T$ encodes information about the random sample $(\bz(n))_{n=1}^N$ including the sample size $N$, which is why the limit depends on it. In the setting of linear regression, we identify the random limit with the (random) OLS estimator which further substantiates the legitimacy of our approximation.

\section{SMEs driven by epoched Brownian motions}
\label{sec:ebms}
Let $(\Om, \cF_\Om, \P)$ be a complete probability space, $d\in \N$ and $T > 0$.
Recall that $\hat W$ is a single shuffle Brownian motion (of period $T$) if there exists a Brownian motion $W : \Om \times [0,T]\to \R^d$ with
\[\hat W_t := W_{\frk{t/T}T} + \floor{t/T}W_T, \quad t\geq 0.\]

Note that given a single shuffle Brownian motion $\hat W$ we can define a \emph{Brownian bridge} $B : \Om \times [0,1] \to \R^d$ from $0$ to $0$ by setting
\[B_t = \frac{1}{\sqrt T}(\hat W_{tT} - t\hat W_T),\quad t\in [0,1].\]
Then,
\[\hat W_t = \sqrt T B_{\frk{t/T}} + \frac t {\sqrt T} V, \quad t\geq 0.\]
with $V := \frac{1}{\sqrt T} \hat W_T$ a standard Gaussian.

More generally, we may replace the single Brownian bridge $B$ with a sequence of bridges $(B^j)_{j\in \N}$, one for each epoch. This motivates the following definition.
\begin{defi}
	A stochastic process $X : \Om \times [0,\infty) \to \R^d$ is called an \emph{epoched Brownian bridge} if there exists a jointly Gaussian\footnote{Jointly Gaussian family means $(B^{j_1}_{t_1}, \dots, B^{j_m}_{t_m})$ is Gaussian for all $j_1,\dots, j_m \in \N_0$ and $t_1,\dots, t_m \in [0,1]$.} family $(B^j : \Om \times [0,1] \to \R^d)_{j \in \N_0}$ of Brownian bridges from $0$ to $0$, such that
	\[X_t = B^{\floor t}_{\frk t}, \quad t\geq 0.\]
	A stochastic process $\hat W : \Om \times [0,\infty) \to \R^d$ is called an \emph{epoched Brownian motion} of period $T > 0$ if there exists an epoched Brownian bridge $X$ and a random variable $V\sim \cN(0, 1_{d\times d})$ independent of $X$, such that
	\[\hat W_t = \sqrt T X_{t/T} + \frac{t}{\sqrt T} V, \quad t\geq 0.\]
\end{defi}

We highlight the following examples:
\begin{enumerate}[(a)]
	\item Single shuffle (SS): $B^0 = B^1 = \dots$,
	\item Random reshuffling (RR): $(B^j)_{j \in \N_0}$ are independent,
	\item Flip-flop single shuffle: $B^0 = B^2 = \dots$, and $B^{j+1}_t = -B^j_{1-t}, t\in [0,1]$,
	\item Flip-flop random reshuffling: $(B^{2j})_{j\in \N_0}$ are independent, $B^{j+1}_t = -B^j_{1-t}, t\in [0,1]$.
\end{enumerate}

In our framework, the epoched Brownian motion $\hat W$ corresponds to the versions of SGDo with the same name. That is, they correspond to the following shuffling schemes for SGDo for large samples sizes $N$:
\begin{enumerate}[(a)]
	\item Single shuffle (SS): $\pi^j = \id{N}, j \in \N$,
	\item Random reshuffling (RR): $(\pi^j)_{j \in \N_0}$ are independent with $\pi^j$ uniformly distributed on the symmetric group of order $N$,
	\item Flip-flop single shuffle: $\pi^{2j} = \id{N}, \pi^{2j+1} = \tau, j \in \N_0$, where $\tau(n) = N - n + 1$ is the \emph{reversal} permutation\footnote{Not to be confused with the inverse of a permutation.},
	\item Flip-flop random reshuffling: $(\pi^{2j})_{j \in \N_0}$ are independent with $\pi^j$ uniformly distributed on the symmetric group of order $N$, and $\pi^{2j+1} = \tau \circ \pi^{2j}, j \in \N_0$.
\end{enumerate}

We do not claim that every epoched Brownian motion or bridge correspond to a shuffling scheme for SGDo. Instead, a \emph{one-dimensional} epoched Brownian motion (or bridge) given by a family of Brownian bridges $(B^n : \Om \times [0,1]\to \R)_{n\in \N_0}$ corresponds to a shuffling scheme for SGDo for large sample sizes $N$ if there exists a measure $\mu$ on $[0,1]^\N$ with uniform marginals, such that
\[\E[B_s^i B_t^j] = C^{ij}(s,t) - st, \quad i \neq j \in \N, s,t\in [0,1],\]
where
\[C^{ij}(s,t) = \mu([0,1]\times \dots \times [0,1] \times \overbrace{[0,s]}^i \times [0,1]\times \dots \times [0,1] \times \overbrace{[0,t]}^j \times [0,1] \times \dots), \quad i\neq j\]
and $C^{ii}(s,t) = s\wedge t$, $i\in \N$.
Note that the functions $C^{ij}$ are $2$-copulas. A $d$-dimensional epoched Brownian bridge corresponding to a shuffling scheme consists of $d$ independent copies of such a one-dimensional process (the same measure is used for all dimensions).

The reason we claim correspondence to shuffling schemes, provided such a measure $\mu$ exists, is that these processes arise as scaling limits of the joint distributions of random walks that have the same increments, up to a (random) permutation, see Chapter \ref{chap:weakshuffle}.

All our previous examples satisfy this condition, with
\begin{enumerate}[(a)]
	\item Single Shuffle (SS): $C^{ij}(s,t) = s\wedge t$,
	\item Random reshuffling (RR): $C^{ij}(s,t) = st$,
	\item Flip-flop single shuffle: 
	\[C^{ij}(s,t) = \begin{cases}
		s\wedge t, & i,j \text{ are both odd or even},\\
		(s+t-1)\vee 0, & \text{else}, 
	\end{cases}\]
	\item Flip-flop random reshuffling: 
	\[C^{ij}(s,t) = \begin{cases}
		(s+t-1)\vee 0, & i\text{ is even and } i+1 = j, \\
		st, & \text{else},
	\end{cases}\]
\end{enumerate}
for $i\neq j$.

The first formula is simply stating that the covariance of a single Brownian bridge is given by 
\[\Cov(B_s,B_t) = s\wedge t - st = s(1-t) \wedge t(1-s), \quad s,t\in [0,1].\]
The second formula just says that independent Brownian bridges have covariance $0$. To show (c) and (d) it remains the consider a Brownian bridge $B$ and calculate
\begin{align*}
	\Cov(B_s,-B_{1-t}) = & -(s\wedge(1-t)) + s(1-t) \\
	= & (-s)\vee (t-1) + s - st \\
	= &(s+t-1)\vee 0 - st, \quad s,t\in [0,1].
\end{align*}

Since most of our results do not depend on the existence of such a measure $\mu$ we will not assume such a covariance structure in general. 

\section{Main result}
Let $d\in \N$ and $\la > 0$. We say a function $\cR : \R^d \to \R \in \dC 2$ is \emph{$\la$-strongly convex} if it satisfies any of the following equivalent properties:
\begin{itemize}
\item $\innp{\nabla \cR(x) - \nabla \cR(y)}{x-y} \geq \la |x-y|^2, \quad x,y\in \R^d$,
\item $\cR(y) \geq \cR(x) + \innp{\nabla \cR(x)}{y-x} + \frac12\la |x-y|^2, \quad x,y\in \R^d$,
\item $\nabla^2\cR(x) - \la 1_{d\times d}$ is a positive semi-definite matrix, for all $x\in \R^d$.
\end{itemize}
Let $L > 0$. We say $\cR$ is \emph{$L$-smooth} if $\nabla \cR$ is Lipschitz, with $\nrm{\nabla \cR}{\Lip} \leq L$.
Our main (mathematical) result is the following.

\begin{satz}
\label{thm:main}
Let $\be \in (0,1)$, $c > 0$, $L, \la > 0$ and $\cR : \R^d\to \R \in \dC 2$ be $\la$-strongly convex and $L$-smooth such that $\nabla^2 \cR$ is Hölder continuous.
Let $Y$ be the solution to the Young differential equation
\begin{equation}
\label{eq:epochedCvx}
dY_t = - \frac{1}{(1+ct)^\be} \nabla \cR(Y_t)\,dt + \frac{1}{(1+ct)^\be} \si\,d\hat W_t,
\end{equation}
driven by an epoched Brownian motion $\hat W$ with period $T$. Then
\[\left|Y_t - (\nabla \cR)^{-1}(T^{-1}\si \hat W_T)\right| \leq T^{1/2-\be} |\si| \left(4.7\frac{L}{\la} + 1.2\right)c^{-\be}\frac{\sqrt{\log t}}{t^{\be}} + o\left(\sqrt{\log t}\cdot  t^{-\be}\right), t\to \infty, \quad a.s.\]
\end{satz}
Theorem \ref{thm:main} may give the impression that its optimal to let $\be\to 1$-. After all, that choice gives us the fastest asymptotic rate of convergence. However, in actuality the constant hidden in $o(\sqrt{\log t}\cdot  t^{-\be})$ diverges to $\infty$, as $\be \to 1$. Therefore, we cannot conclude that $\be \to 1$ is optimal. In fact, in practice setting $\be = 1$ makes the learning rates decay much too fast.

In certain situations we can get a better decay rate compared to Theorem \ref{thm:main}.
The following theorem applies to all epoched Brownian motions which have only finitely many different epochs over their entire time horizon. For example, this is the case for single shuffle Brownian motion, which only has a single repeated epoch.
\begin{satz}
\label{thm:mainAlt}
Let $\be \in (0,1)$, $c > 0$, $L, \la > 0$ and $\cR : \R^d\to \R \in \dC 2$ be $\la$-strongly convex and $L$-smooth, such that $\nabla^2 \cR$ is Hölder continuous.
Let $Y$ be the solution to the Young differential equation
\begin{equation}
dY_t = - \frac{1}{(1+ct)^\be} \nabla \cR(Y_t)\,dt + \frac{1}{(1+ct)^\be} \si\,d\hat W_t,
\end{equation}
driven by an epoched Brownian motion $\hat W$ with period $T$. Suppose further there exists a number $J\in \N$, such that $\cI := \set{(\hat W_{(j+t)T}-\hat W_{jT})_{t\in [0,1]} : j\in \N}|$ satisfies $|\cI| = J$, almost surely. Then, for all $\al \in(0,1/2)$,
\[\left|Y_t - (\nabla \cR)^{-1}(T^{-1}\si \hat W_T)\right| \leq C_\al T^{1/2-\be} |\si| \left(\frac{1}{1-2^{-\al}} \frac{L}{\la} + 1\right)\frac{1}{t^{\be}} + o\left(C_\al t^{-\be}\right),  t\to \infty, \quad a.s.\]
where $C_\al = \max_{w\in \cI} \nrm{w}{\al}$.
\end{satz}
Note that the only random factor in $o(C_\al t^{-\be})$ is $C_\al$.

As an example, consider SGDo applied to linear regression, which corresponds to the \tYDE{}
\[dY_t = - \frac{1}{(1+t)^\be} \ka (Y_t - \thet^*)\,dt + \frac{1}{(1+t)^\be} \sqrt{h \si_\ep^2 \ka}\,d\hat W_t.\]
Here, $\hat W$ has period $T = Nh$ where $N$ is the sample size and $h$ the maximal learning rate. We implicitly assume we are in the underparameterized regime $N\gg d$.

Then
\begin{align*}
(\nabla \cR)^{-1}(T^{-1}\si \hat W_T) = & \thet^* + \ka^{-1}((Nh)^{-1/2}\sqrt{h \si_\ep^2 \ka} T^{-1/2}\hat W_{T}) \\
= & \thet^* + \frac{\si_\ep}{\sqrt N} \ka^{-1/2} (T^{-1/2}\hat W_T)\\
\sim &\cN\left(\thet^*, \frac{\si_\ep^2}{N} \ka^{-1}\right),
\end{align*}
and Theorem \ref{thm:main} implies
\begin{align*}
\left|Y_t - \left(\thet^* + \frac{\si_\ep}{\sqrt N} \ka^{-1/2} (T^{-1/2}\hat W_T)\right)\right| \leq & (Nh)^{1/2-\be} \sqrt h \si_\ep|\sqrt \ka| \left(4.7\condNr{\ka}+1.2\right)c^{-\be}\frac{\sqrt{\log t}}{t^{\be}}\\
 & + o\left(\sqrt{\log t}\cdot  t^{-\be}\right)\\
\leq & N^{1/2-\be}d h^{1-\be} \si_\ep  \sqrt{\la_{\max}(\ka)} \left(4.7\condNr{\ka}+1.2\right)\frac{\sqrt{\log t}}{t^{\be}} \\
&+ o\left(\sqrt{\log t}\cdot  t^{-\be}\right),
\end{align*}
as $t\to \infty$, almost surely. 
The limit $Y_\infty := \thet^* + \frac{\si_\ep}{\sqrt N} \ka^{-1/2} T^{-1/2}\hat W_T$ of $Y$ has the same mean and covariance matrix as the OLS estimator
\[\hat\thet = \left(\sum_{n=1}^N \bx_n \bx_n^\transp\right)^{-1}\left(\sum_{n=1}^N \bx_n \by_n\right),\]
if $(\bx_n, \by_n)_{n=1}^N$ is a finite \tiid\ sample with $(\bx_0,\by_0) \sim \nu$, and $\nu$ is the corresponding population.
Since $\hat W$ is independent of $(\bx_n,\by_n)_{n\in \N}$ we \emph{do not} have $\hat \thet = Y_\infty$, even if $\hat \thet$ was Gaussian. Nevertheless, this result suggests that spiritually $Y_\infty$ represents the OLS estimator in our model in the case of linear regression.

The factor $T^{1/2-\be}$ (or $N^{1/2-\be}$ after setting $T = Nh$) in the convergence speed may be surprising. It can be heuristically explained as follows:
Set $u_t = \frac{1}{(1 + ct)^\be}, t\geq 0$. The noise accumulated in epoch $j$ is given by
\[\int_{jT}^{(j+1)T} u_t \si \,d\hat W_t \approx (cjT)^{-\be} \si (\hat W_{(j+1)T} - \hat W_{jT}) = T^{1/2-\be}(jc)^{-\be} \si Z,\]
where
\[Z = \frac{1}{\sqrt{T}}(\hat W_{(j+1)T} - \hat W_{jT}) \sim \cN(0, 1_{d\times d}).\]
If $\be > 1/2$, then $u$ decays faster than the noise accumulates. In this case the accumulated noise vanishes, as $T\to \infty$, since increasing $T$ means we are effectively averaging over more \tiid\ random variables per epoch. On the other hand, if $\be < 1/2$, then $u$ decays too slowly to overcome the noise accumulation. More steps per epoch means more accumulation, so the accumulated noise diverges to infinity, as $T\to \infty$. Finally, at $\be = 1/2$ both effects (decay and noise accumulation) are balanced. 

These different regimes implicitly also exist in other works on stochastic gradient descent (with or without replacement). In particular, usually only the case $\be > 1/2$ is covered (see the end of the following paragraph).

\paragraph{Comparison with existing results}
Our main theorem complements findings by \citet{gurbuzbalaban_why_2021}.
They proved that single shuffle SGDo satisfies 
\[|\chi_k - \hat\thet| \leq \frac{h |\mu(\pi^1)|}{\la} \frac{1}{k^{\be}} + o(k^{-\be}), a.s. \quad k\to \infty,\]
for $\be \in (1/2,1)$.
Here, $\chi$ is given by Equation \ref{eq:genericSGDo} with $\eta_k = hk^{-\be}$ and $\pi^1 = \pi^j, j \in \N$. Further, $\cR$ is given as a sum of $N$ quadratic forms, is $\la$-strongly convex and has its minimum at $\hat\thet$. Moreover, $\mu(\pi) \in \R^d$ is a sum of $\frac12 N(N-1)$ terms depending on $\cR$ and the permutation $\pi$. In general, $|\mu(\pi)|$ can grow with rate $O(N^2)$, as $N\to \infty$.
In contrast, Theorem \ref{thm:mainAlt} suggests a rate of
\[\tilde C N^{1/2-\be} k^{-\be} + o(k^{-\be}), a.s. \quad k\to \infty.\]
where $\tilde C$ is independent of $N$.
They also provide a crude bound for the random reshuffling case:
\[|\chi_k - \hat\thet| \leq \frac{h \sup_{\pi \in \Sym N}|\mu(\pi)|}{\la} \frac{1}{k^{\be}} + o(k^{-\be}), a.s. \quad k\to \infty,\]
where $\Sym N$ is the symmetric group of degree $N$.
However, in the worst case $\sup_{\pi \in \Sym N}|\mu(\pi)| = O(N^2 N!)$, as $N\to \infty$, making this result not very useful for moderately large $N$, say\footnote{The observable universe is estimated to have less than $60!$ particles.} $N > 100$. Naturally, they mention that the constant $\sup_{\pi \in \Sym N}|\mu(\pi)|$ is pessimistic.
Our Theorem \ref{thm:main} suggests a rate of
\[\tilde C N^{1/2-\be}\frac{\sqrt{\log k}}{k^{\be}} + o(\sqrt{\log k}\cdot k^{-\be}), a.s. \quad k\to \infty,\]
for the convergence of SGDo on strongly convex objectives using any shuffling scheme. 
Thus, Theorem \ref{thm:main} suggests good almost sure convergence rates for SGDo even for large sample sizes $N$.
	
Finally, note the restriction $\be > 1/2$ imposed by \citet{gurbuzbalaban_why_2021}. It stems from the application of martingale techniques which require learning rates to be square summable. Indeed, 
\[\sum_{n=1}^\infty \left(\frac{1}{n^\be}\right)^2 < \infty\text{ if and only if } \be > 1/2.\]
Since we do not use any martingale techniques, this barrier only appears implicitly in our main results as the convergence rate factor $T^{1/2-\be}$.
\newcommand{\rpX}{\mathbf X}
\newcommand{\rpaX}{\mathbb X}

\section{Properties of (epoched) Brownian bridges}
In the following we will mostly work with epoched Brownian \emph{bridges}. By the definition they concatenations of Brownian bridges. Recall, that a Brownian bridge is $(1/2-)$-Hölder continuous, that is $(1/2-\ep)$-Hölder continuous for every $\ep > 0$. Together with the following lemma, this implies that epoched Brownian bridges are locally $(1/2-)$-Hölder continuous.

Let $\al \in (0,1)$. In this chapter we denote by $\nrm{\blnk}{\al}$ the $\al$-Hölder seminorm given by
\[\nrm{f}{\al} = \sup_{s,t\in I} \frac{\nrm{f(t)-f(s)}{E}}{|t-s|^\al},\]
where $f : I \to E$ for $E = (\R^d, |\blnk|)$ or\footnote{In contrast to Chapter \ref{chap:mdfdeq}, here we define Hölder norms with respect to the spectral norm for matrices, not the Frobenius norm.} $E = (\R^{d\times d}, \specnrm{\blnk})$ and some interval $I$. Here,
\[\specnrm{A} := \sup_{|x| = 1} |Ax| = \sqrt{\la_{\max}(A^\transp A)}.\]
denotes the \emph{spectral norm} of a square matrix $A$. We also write $\nrm{f}{\al;I} = \nrm{f|_I}{\al}$ when $f$ is defined on a set containing $I$. In the case $\al = 1$ we prefer writing $\nrm{f}{\Lip}$ and $\nrm{f}{\Lip;I}$.
\begin{lem}
\label{lem:concatHoel}
Let $\al \in (0,1)$ and $f,g : [0,1]\to \R^d \in \hoelHZ \al$ be functions with $f(1) = g(0)$. Then the concatenation
\[f\ast g : [0,2]\to \R^d, t\mapsto f(t)1_{[0,1]}(t) + g(t-1)1_{(1,2]}(t) \] 
satisfies $f\ast g\in \hoelHZ \al$ with $\nrm{f\ast g}{\al} \leq 2^{1-\al} (\nrm{f}{\al} \vee \nrm{g}{\al})$.
\end{lem}
\begin{proof}
It suffices to check the Hölder condition for $s < 1 < t$. In this case
\begin{align*}
|f\ast g(t) - f\ast g(s)| \leq & |f\ast g(t) - f\ast g(1)| + |f\ast g(1) - f\ast g(s)| \\
= & |g(t-1) - g(0)| + |f(1) - f(s)|\\
\leq &(\nrm{f}{\al} \vee \nrm{g}{\al})(|t-1|^\al + |1-s|^\al)\\
\leq &2^{1-\al} (\nrm{f}{\al} \vee \nrm{g}{\al})(|t-1| + |1-s|)^\al\\
= &2^{1-\al}(\nrm{f}{\al} \vee \nrm{g}{\al})|t-s|^\al,
\end{align*} 
since $|t-1| + |1-s| = t-1 + 1 - s$.
\end{proof}

\newcommand{\sphre}[1]{\mathbb S^{#1}}
\begin{lem}[Borell-TIS]
\label{lem:borellTIS}
Let $D$ be a topological space and $Q : \Om \times D \to \R^d$ be Gaussian random field, which is almost surely bounded on $D$. 
Define $m = \E\left[\sup_{t\in D} |Q_t|\right]$ and $\si^2 = \sup_{t\in D} \la_{\max}(\Cov(Q_t))$.
Then
\[\P\left(\sup_{t\in D} |Q_t| > x\right) \leq e^{-\frac{(x-m)^2}{2\si^2}}, \quad x > m.\]
\end{lem}
\begin{proof}
We write $\sphre{d-1} = \set{v\in \R^d : |v| = 1}$. Note that
\[|Q_t| = \sup_{v\in \sphre{d-1}} \innp{Q_t}{v},\]
since $|\innp{Q_t}{v}| \leq |Q_t||v| = |Q_t|$ for $v\in \sphre{d-1}$ and because we can pick $v = Q_t/|Q_t|$.
Define
\[\tilde Q : \Om \times D \times \sphre{d-1} \to \R, (\om, t, v) \mapsto \innp{Q_t(\om)}{v}.\]
Then $\tilde Q$ is again a Gaussian random field and almost surely bounded. We have 
\[\E\left[\sup_{(t,v)\in D \times \sphre{d-1}} \tilde Q_{t,v}\right] = m.\]
Moreover, we have $\Var(\innp{Q_t}{v}) = v^\transp \Cov(Q_t) v$, and so
\[\sup_{(t,v)\in D \times \sphre{d-1}} \Var(\innp{Q_t}{v}) = \sup_{t\in D} \sup_{v\in \sphre{d-1}} v^\transp \Cov(Q_t) v = \sup_{t\in D} \la_{\max}(\Cov( Q_t)) = \si^2.\]
The penultimate equality follows because we are maximizing the Rayleigh quotient of $\Cov(Q_t)$.
Now, using the standard Borell-TIS inequality \citep[see][Theorem 2.1.1]{adler2009random} we have 
\[\P\left(\sup_{(t,v)\in D \times \sphre{d-1}} \tilde Q_{t,v} - m > x\right) \leq e^{-\frac{x^2}{2 \si^2}}, \quad x > 0,\]
or equivalently
\[\P\left(\sup_{t\in D} |Q_t| > x\right) \leq e^{-\frac{(x-m)^2}{2\si^2}}, \quad x > m.\]
\end{proof}

\begin{lem}
\label{lem:expttnTailFrmla}
Let $g : [0,\infty) \to \R \in \dC 1$ and $Z$ be a non-negative random variable. Then
\[\E g(Z) = g(0) + \int_0^\infty g'(x)\P(Z > x)\,dx.\]
\end{lem}
\begin{proof}
We have
\[g(z) = g(0) + \int_0^z g'(x)\,dx,\]
and so  
\[\E g(Z) = g(0) + \E\left[\int_0^Z g'(x)dx\right] = g(0) + \int_0^\infty g'(x)\P(Z > x)\,dx.\]
\end{proof}

\begin{lem}
\label{lem:BBferniq}
Let $B : \Om \times [0,1] \to \R^d$ be a Brownian Bridge. Then 
\[\E[e^{a \nrm{B}{\al}^2}] < \infty\]
for all $\al \in (0,1/2)$ and $a \in (0, \frac{1}{2(1-b)b^{1-2\al}})$, where $b = \frac{1-2\al}{2-2\al}$.
\end{lem}

\begin{proof}
Define 
\[Q_{s,t} = \begin{cases}
\frac{B_t - B_s}{|t-s|^\al}, & s\neq t,\\
0, & s = t,
\end{cases}\]
for all $s,t\in [0,1]$, and write $\hat Q := \sup_{s,t\in [0,1]} Q_{s,t}$.
Then $Q$ is a Gaussian random field $\Om \times [0,1]^2\to \R^d$ and $\sup_{s,t\in [0,1]} |Q_{s,t}| = \nrm{B}{\al}$. Thus, by Lemma \ref{lem:borellTIS}
\begin{align*}
\P(\nrm{B}{\al} > x) \leq e^{-\frac{(x-m)^2}{2 \si^2}}, \quad x > m := \E\nrm{B}{\al},
\end{align*}
where $\si^2 := \sup_{s,t\in [0,1]} \la_{\max}(\Cov Q_{s,t})$. Because the components of $B$ are independent, Brownian bridges have stationary increments and using the covariance formula for a one-dimensional Brownian bridge we have
\[\la_{\max}(\Cov(B_t - B_s)) = \Var(B_t^1 - B_s^1) = \Var(B_{t-s}^1) = |t-s|(1 - |t-s|), \quad s,t\in [0,1].\]
Thus,
\[\la_{\max}(\Cov Q_{s,t}) = \begin{cases}
\frac{|t-s|(1 - |t-s|)}{|t-s|^{2\al}},  &s\neq t,\\
0, & s = t
\end{cases} = f(|t-s|), \quad s,t\in [0,1],\]
where $f(b) = (1 - b)b^{1-2\al}$. The function $f$ attains its maximum at $b^* := \frac{1-2\al}{2-2\al}$. Hence $\si^2 = f(b^*)$.
Let $a > 0$. Then Lemma \ref{lem:expttnTailFrmla} implies
\[\E[e^{a \nrm{B}{\al}^2}] = 1 + \int_0^\infty 2 a x e^{ax^2} \P(\nrm{B}{\al} > x)\,dx.\]
Estimating the tail of the integral, we have
\[\int_m^\infty 2 a x e^{ax^2} \P(\nrm{B}{\al} > x)\,dx \leq \int_m^\infty 2 a x e^{ax^2}e^{-\frac{(x-m)^2}{2 \si^2}}\,dx.\]
Since
\[ax^2 -\frac{(x-m)^2}{2 \si^2} = \left(a - \frac{1}{2 \si^2}\right) x^2 + \frac{m}{\si^2}x - \frac{m^2}{2\si^2}\]
the integral converges if $a < \frac{1}{2 \si^2} = \frac{1}{2f(b^*)}$.
\end{proof}

The following lemma gives us one factor in the decay rate of Theorem \ref{thm:main}.
\begin{lem}
\label{lem:epochedBBgrowth}
Let $\al \in (0,1/2)$, $a \in (0, \frac{1}{2(1-b)b^{1-2\al}})$, where $b = \frac{1-2\al}{2-2\al}$, and $(B^j)_{n\in \N_0}$ be a family of Brownian bridges.
Then
\[\max_{j \leq n}\nrm{B^j}{\alpha} \leq a^{-1/2}\sqrt{\log n},\]
for large $n\in \N$, almost surely.
\end{lem}
\begin{proof}
We use Lemma \ref{lem:BBferniq}.
By Markov's inequality
	\[\P(\nrm{B}{\alpha} \geq x) = \P(e^{a \nrm{B}{\alpha}^2} \geq e^{ax^2}) \leq \E[e^{a \nrm{B}{\alpha}^2}] e^{-a x^2},\]
for all $x\in \R$.
	Define $Z_j = \nrm{B^j}{\alpha}, j \in \N$, and $Z_n^* = \max(Z_1,\dots, Z_n)$.
	Then
	\[\P(Z_n^* > x) \leq \sum_{j=1}^n \P(Z_j > x) \lesssim n e^{-a x^2},\]
	uniformly over $x$ and $n$. For any $\ep > 0$ we thus have
	\[\sum_{j=1}^\infty \P(Z_{2^j}^*> \sqrt{\frac{1+\ep}{c}\log 2^j}) \lesssim \sum_{j=1}^\infty 2^{-j\ep} < \infty.\]
	By Borel-Cantelli
	\[\P\left(\limsup_{n\to \infty} \set{Z_n^* > \sqrt{\frac{1+\ep}{a}\log n}}\right) = 0,\]
	that is
	\[\max_{j \leq n}\nrm{B^j}{\alpha} = Z_n^* \leq \sqrt{\frac{1+\ep}{a}\log n},\]
	for large $n\in \N$, almost surely. Finally, by picking a slightly smaller $a$ we can leave out the $+\ep$. However, since we started with an arbitrary $a < \frac{1}{2(1-b)b^{1-2\al}}$ we have
	\[\max_{j \leq n}\nrm{B^j}{\alpha}\leq a^{-1/2}\sqrt{\log n},\]
	for large $n\in \N$, almost surely, for all $a \in (0, \frac{1}{2(1-b)b^{1-2\al}})$.
\end{proof}

\begin{docu}
This is a conjecture, but mostly because I dont know the proper theory of 2D-Young (?) integration.
We say $(B,B)' :  \Om \times [0,1]\to \R$ is a $C$-Brownian bridge if it is a Gaussian process, $B,B'$ are Brownian bridges from $0$ to $0$ and their covariance satisfies
\[\Cov(B_s, B_t') = C(s,t) - st.\]
\begin{prop}
Let $C$ be a $2$-copula, $(B,B')$ be a $C$-Brownian bridge, $\be \in (0,1)$ with $\frac12 + \frac1\beta > 1$ and $f,g\in C^{\be}([0,1])$. Then
\[\Cov\left(\int_0^1 f(t)\,dB_t, \int_0^1 g(t)\,dB_t'\right) = \Cov_{(U,V)\sim C}(f(U),g(V)).\]
\end{prop}
\begin{proof}[Proof sketch]
We have
\begin{align*}
\E[\int_0^1 f(t)\,dB_t, \int_0^1 g(t)\,dB_t'] = & \lim_{|\cP_1|, |\cP_2| \downarrow 0} \sum_{(r,s) \in \cP_1, (t,u) \in \cP_2} f(r)g(t) \E[B_{r,s} B'_{t,u}] \\
= & \int_{[0,1]^2} f(s)g(t) d\E[B_s B'_t]\\
= & \int_{[0,1]^2} f(s)g(t) d(C(s,t) - st)\\
= & \E[f(U)g(V)] - \E[f(U)] \E[g(V)],
\end{align*}
where $(U,V)\sim C$.
\end{proof}
\end{docu}

\section{Young differential equations driven by epoched noise}
In this section we study the properties of Young differential equations with state-independent noise term, specifically driven by an epoched bridge $X$. Let $m \in \N$. We call $X : [0,\infty) \to \R^m$ an \emph{epoched bridge} if $X$ is locally Hölder continuous and $X_n = 0, n\in \N$. None of the arguments in this section directly depend on $X$ being an epoched \emph{Brownian} bridge\footnote{For example, all arguments here apply to $X_t = \sin(\pi t)$.}. Hence, we work without this specific assumption.

We consider Young differential equations of the form
\[dY_t = f_t(Y_t)\,dt + \si_t\,dX_t,\quad t\geq 0, Y_0 \in \R,\]
with $f_t : \R^d\to \R^d$ and $\si_t \in \R^{d\times m}$,
which is strictly speaking a different way of writing the integral equation
\begin{equation}
\label{eq:linearRDE}
Y_t = Y_0  + \int_0^t f_s(Y_s)\,ds + \int_0^t \si_s\,dX_s, \quad t\geq 0.
\end{equation}
Here, 
\[\int_0^t \si_s\,dX_s = \lim_{|\cP| \to 0} \sum_{[r,s]\in \cP} \si_r X_{r,s},\]
where the limit is taken with respect to all partitions of $[0,t]$ with mesh size $|\cP|$, and $X_{r,s} = X_s - X_r$. This is the \emph{Young integral}. If $X\in \hoelHZ \al([0,T])$ and $\si \in \hoelHZ \be([0,T])$ with $\al + \be > 1$, then the Young integral is guaranteed to exist (see Proposition \ref{prop:ynglve}).

To give an idea what is so special about (epoched) bridges consider the Young-Lóeve inequality.
\begin{prop}[Young-Lóeve]
\label{prop:ynglve}
Let $\al,\be \in (0,1]$ with $\al + \be > 1$. Given $X\in \CHoelLoc{\al}$ and $\si\in \CHoelLoc \be$, the Young integral $\int_s^t \si_u\,dX_u$ exists, and we have
\[\left|\int_s^t \si_u\,dX_u - \si_s X_{s,t}\right| \leq \frac{(t-s)^{\al+\be}}{2^{1-(\al + \be)}} \nrm{X}{\al;[s,t]} \nrm{\si}{\be;[s,t]}, \quad 0\leq s\leq t.\]
Further, $\int_s^\blnk \si_u\,dX_u \in \CHoelLoc \al$.
\end{prop}
\begin{proof}
See \citet[Theorem 6.8]{Friz_Victoir_2010} and note that any $\al$-Hölder continuous function $X$ on $[s,t]$ (even if matrix-valued) has finite $1/\al$-variation $\pvar{X}{1/\al}$, with 
\[\pvar{X}{1/\al} \leq (t-s)^\al\nrm{X}{\al}.\]
\end{proof}
Note that for any epoched bridge $X$ we have $X_{n,n+1} = 0$ for all $n\in \N_0$, so in this case Proposition \ref{prop:ynglve} implies
\begin{equation}
\label{eq:ynglvebrdge}
\left|\int_n^{n+1} \si_s\,dX_s\right| \leq \frac{1}{2^{1-(\al + \be)}} \nrm{X}{\al;[n,n+1]} \nrm{\si}{\be;[n,n+1]}, \quad n \in \N_0
\end{equation}
This is a crucial estimate in our convergence arguments (see the proof of Proposition \ref{prop:linearEpochedDecayYDE}).

\subsection{Existence and Uniqueness}
Our first aim is to show existence and uniqueness of a global solution $Y$ to \eqref{eq:linearRDE}.
\begin{prop}
\label{prop:YDEaddExUn}
Suppose we are given the following.
\begin{itemize}
\item $\al, \be \in (0,1]$ with $\al + \be > 1$,
\item $X : [0,\infty) \to \R^m \in \CHoelLoc\al$,
\item $\si : [0,\infty) \to \R^{d\times m} \in \CHoelLoc\be$,
\item $f : [0,\infty) \times\R^d \to \R^d$ is (jointly) measurable, such that
\begin{enumerate}[(a)]
\item $f_t(\blnk) \in \Lip$, uniformly in $t\geq 0$,
\item $f_\blnk(0) \in \LpLoc1$.
\end{enumerate}
\end{itemize}
Then there exists a unique solution $Y : [0,\infty) \to \R^d$ to the Young differential equation
\begin{equation}
\label{eq:YDEadd}
dY_t = f_t(Y_t)\,dt + \si_t\,dX_t,\quad t\geq 0, Y_0 = y,
\end{equation}
and it satisfies $Y\in \CHoelLoc{(\al\wedge \be)^-}([0,\infty), \R^d)$.
\end{prop}

\begin{proof}
Let $T>0, \ga \in (0,\al \wedge \be)$ and define
\[E = \set{Y \in \hoelHZ \ga([0,T], \R^d) : Y_0 = y}.\]
This is a complete metric space when equipped with $d(Y,\tilde Y) = \nrm{Y - \tilde Y}{\ga}$.
Define the map $\Ph : E\to E$ by
\[(\Ph Y)_t = y_0 + \int_0^t f_s(Y_s)\,ds + \int_0^t \si_s \,dX_s.\]
Note that the latter summand is a proper Young integral, since $\al + \be > 1$.
We have
\[|f_s(Y_s)|\leq |f_s(0)| + |f_s(Y_s) - f_s(0)| \leq |f_s(0)| + \nrm{f}{\Lip}|Y_s|,\]
which is locally integrable in $s$. Thus, $\int_0^\blnk f_s(Y_s)\,ds \in \Lip([0,T])$.
Further, $(\Ph Y)_0 = y_0$ and $\int_0^\blnk \si_s \,dX_s \in \hoelHZ \al([0,T]) \subseteq \hoelHZ \ga([0,T])$ by Proposition \ref{prop:ynglve}. Hence, $\Ph$ is well-defined.
For $s,t\in [0,T]$ we estimate
\begin{align*}
|\Ph Y_{s,t} - \Ph \tilde Y_{s,t}| 	\leq& \int_s^t |f_r(Y_r) - f_r(\tilde Y_r)| \,dr \\
									\leq& \nrm{f}{\Lip} \int_s^t |Y_r - \tilde Y_r|\,dr \\
									\leq& \nrm{f}{\Lip} \nrm{Y - \tilde Y}{\ga} \int_s^t (r-s)^\ga\,dr\\
									\leq& \frac{1}{1+\ga}\nrm{f}{\Lip} \nrm{Y - \tilde Y}{\ga} (t-s)^{1+\ga}.
\end{align*}
Thus,
\[|\Ph Y_{s,t} - \Ph \tilde Y_{s,t}|(t-s)^{-\ga} \leq \frac{T}{1+\ga}\nrm{f}{\Lip} \nrm{Y - \tilde Y}{\ga}, \quad s,t\in [0,T],\]
i.e.
\[\nrm{\Ph Y - \Ph \tilde Y}{\ga} \leq \frac{T}{1+\ga}\nrm{f}{\Lip} \nrm{Y - \tilde Y}{\ga},\]
or, in other words, $\Ph$ is Lipschitz with constant bounded by $\frac{T}{1+\ga}\nrm{f}{\Lip}$.
By picking $T = \frac{1+\ga}{2 \nrm{f}{\Lip}}$ we get $\nrm{\Ph}{\Lip} \leq \frac12$. In particular, $\Ph$ is a contraction and has a fixed point $Y\in E$, using the Banach fixed-point theorem. Being a fixed point means it is a solution of \eqref{eq:YDEadd} on $[0,T]$.
If a solution $Y$ of \eqref{eq:YDEadd} exists on $[0,nT]$ for some $n\in \N$, then by applying the same argument with 
\[E = \set{\tilde Y \in \hoelHZ \ga([nT,(n+1)T], \R^d) : \tilde Y_{nT} = Y_{nT}}\]
extends the solution $Y$ to $[0,(n+1)T]$. Thus, a solution $Y$ exists on $[0,\infty)$.

If there are two solutions $Y,\tilde Y$ on some interval $[0,T]$, then
\[|Y_t - \tilde Y_t| \leq \int_0^t |f_s(Y_s) - f_s(\tilde Y_s)| \leq \nrm{f}{\Lip} \int_0^t |Y_s - \tilde Y_s|\,ds,\]
and then Grönwalls inequality implies $Y_t = \tilde Y_t$, for all $t\in [0,T]$.
\end{proof}

\begin{prop}
	\label{prop:varOfConst}
	Suppose we are given the following.
	\begin{itemize}
		\item $\al, \be \in (0,1]$ with $\al + \be > 1$,
		\item $X : [0,\infty) \to \R^m \in \CHoelLoc\al$,
		\item $\si : [0,\infty) \to \R^{d\times m} \in \CHoelLoc\be$,
		\item $A : [0,\infty) \to \R^{d\times d}\in \LpLoc 1 \cap L^\infty$,
		\item $b : [0,\infty) \to \R^d \in \LpLoc 1$.
	\end{itemize}
Let $\ph$ be the unique solution to the linear matrix integral equation
\begin{equation}
	\label{eq:linMatODE}
	\ph_t = 1_{d\times d} + \int_0^t A_s \ph_s\,ds.
\end{equation}
	Then the unique solution $Y : [0,\infty) \to \R^d$ to the Young differential equation
\begin{equation}
	\label{eq:YDElin}
	dY_t = A_t Y_t + b_t\,dt + \si_t\,dX_t, Y_0 \in \R^d.
\end{equation}
is given by
\[Y_t = \ph_t\left(Y_0 + \int_0^t \ph_s^{-1} b_s\,ds + \int_0^t \ph_s^{-1} \si_s\,dX_s\right),\quad t\geq 0.\]
\end{prop}
\begin{proof}
	Define
	\[Z_t = Y_0 + \int_0^t \ph_s^{-1} b_s\,ds + \int_0^t \ph_s^{-1} \si_s\,dX_s, \quad t\geq 0.\]
	Note that $\ph\in \CHoelLoc1$. Thus, the product formula (see \cite{friz2020course} Exercise 7.4) implies
	\begin{align*}
		\ph_t Z_t 	= & \ph_0 Z_0 + \int_0^t (d\ph_s) Z_t + \int_0^t \ph_s\,dZ_s \\
		= & \ph_0 Z_0 + \int_0^t A_s \ph_s Z_s\,ds + \int_0^t b_s\,ds + \int_0^t \si_s dX_s.
	\end{align*}
	Hence, $Y = \ph Z$ is a solution to \eqref{eq:linearRDE}. Uniqueness follows from Proposition \ref{prop:YDEaddExUn}.
\end{proof}
We can transform our main equation \eqref{eq:epochedCvx} into the simpler form (see Lemma \ref{lem:epochedCvxTransf} for details)
\[dY_t = -\tilde u_t \nabla \tilde \cR(Y_t)\,dt + \tilde u_t dX_t,\]
Here, $X$ is an epoched Brownian bridge, $\tilde u_t = (1 + tT)^{-\be}$ and $\tilde \cR$ is a random function satisfying the same conditions as $\cR$ in Theorem \ref{thm:main}, almost surely, except its global minimum is at $0$.
Thus, we will work mainly with equations of this form from now on.

\begin{docu}
\begin{rem}
Since $\si$ does not depend on $Y$ the present theory only requires the theory of Young integration and no rough path techniques. However, we can still apply various results from the theory of ($\alpha$-Hölder) rough paths, which is useful due to larger body of literature on this topic compared to the Young case. Note that any locally Lipschitz in time function $Z$ is trivially an $\rpX$-controlled rough path with remainder term $R_{s,t}^Z = Z_t - Z_s$ and Gubinelli derivative $Z' = 0$, since we have $\nrm{R_{s,t}^Z}{2\al} \lesssim \nrm{Z}{\Lip} < \infty$. Here, $X$ has been lifted to a rough path $\rpX$ in some way. Therefore, the rough integral against $\rpX$ coincides with the Young integral 
\[\int_0^t Z_s\,d\rpX_s = \int_0^t Z_s\,dX_s\]
and the specific lift $\rpX$ of $X$ plays no role.
Further, the stipulation that $\al \in \left(\frac13, \frac12\right]$ is also unnecessary. All results in this section are true for any locally Hölder continuous path $X$.
\end{rem}

\paragraph{Periodic drivers}
Let $T > 0$ and suppose $X$ has $T$-periodic increments, \tIe{} $X_{t+T} - X_{s+T} = X_t - X_s, s,t\geq 0$, and $X_0 = 0$.

When considering the equations of the form \eqref{eq:linearRDE} it is usually not a restriction to assume that $X$ is $1$-periodic, that is $X_{t+1} = X_t, t\geq 0$. This is because if $Y$ satisfies \eqref{eq:linearRDE}, then for $\tilde Y_t := Y_{tT}$ and $\tilde X_t = X_{tT} - tX_T$ we have
\begin{align*}
\tilde Y_t = & Y_0 + T\int_0^t A_{sT} Y_{sT} + b_{sT}\,ds + \int_0^t \si_{sT} dX_{sT} \\
= & \tilde Y_0 +  \int_0^T T A_{sT} \tilde Y_s + T b_{sT} + \si_{sT} X_T\,ds + \int_0^t \si_{sT} \,d\tilde X_t.
\end{align*}

\begin{lem}
Suppose $Y$ solves \eqref{eq:linearRDE}.
Then $\tilde Y_t = Y_{tT}$ solves
\begin{equation}
\label{eq:linearRDE1period}
d\tilde Y_t = \tilde A_t Y_t + \tilde b_t\,dt + \tilde \si_t\,d\tilde X_t, \quad \tilde Y_0 = Y_0.
\end{equation}
where
\[\tilde A_t = T A_{tT}, \tilde b_t = Tb_{tT} + \si_{tT} X_T, \tilde \si_t = \si_{tT},\]
and the driver $\tilde X_t = X_{tT} - tX_T$ is $1$-periodic.
Conversely, if $\tilde Y$ solves \eqref{eq:linearRDE1period} for $1$-periodic $\tilde X$ and we are given $X_T\in \R^d$, then $Y_t := Y_{t/T}$ solves \eqref{eq:linearRDE}, where
\[A_t = \frac1T\tilde A_{t/T}, b_t = \frac1T \tilde b_{t/T}- \tilde \si_t X_T, \si_t = \tilde \si_{t/T},\]
with $X_t = \tilde X_{t/T} + \frac t T X_T$ having $T$-periodic increments.
\end{lem}

Suppose now $X$ is $1$-periodic with $X_0 = 0$, and that $b = 0$. For every $Z\in \Lip_{\text{loc}}$ we have
\begin{align*}
	\int_{n}^{n+1} Z_s \,dX_s = & \lim_{|\cP|\to 0} \sum_{[s,t] \in \cP} Z_{s+n} (X_{t+n} - X_{s+n}) = \lim_{|\cP|\to 0} \sum_{[s,t] \in \cP} Z_{s+n} (X_{t} - X_{s}) \\
	= & \int_0^1 Z_{s+n}\,dX_s.
\end{align*}
Thus, for $t\geq 0$,
\begin{align*}
	\int_0^t Z_s\,dX_s = & \sum_{n=0}^{\floor t - 1} \int_n^{n+1} Z_s \,dX_s + \int_{\floor t}^t Z_s\,dX_s \\
	= & \sum_{n=0}^{\floor t - 1} \int_0^1 Z_{s+n} \,dX_s + \int_0^{\frk t} Z_{s + \floor t}\,dX_s.
\end{align*}
Hence, the unique solution to \eqref{eq:linearRDE} satisfies
\begin{equation}
\label{eq:linearRDEsol1Per}
Y_t= \ph_t Y_0 + \sum_{n=0}^{\floor t-1}\ph_t \int_0^1  \ph_t^{s+n} \si_{s+n}\,dX_s + \int_0^{\frk t} \ph_t^{s+\floor t} \si_{s + \floor t}\,dX_s, \quad t\geq 0,
\end{equation}
where $\ph_t^s := \ph_t(\ph_s)^{-1}$, and
\begin{equation}
\label{eq:linearRDEsol1PerRec}	
Y_t = \ph_{\frk t} Y_{\floor t} + \int_0^{\frk t} \ph_t^{s + \floor t} \si_{s+\floor t}\,dX_s, \quad t\geq 0.
\end{equation}
\end{docu}

\begin{docu}
\subsection{Ornstein-Uhlenbeck process with periodic noise}
Let $\ka \in \R^{d\times d}$ be symmetric and positive definite, and set
\[A_t = - \ka, \quad b_t  = 0,\quad \si_t = 1_{d\times d}, \quad t\geq 0.\]
Then \eqref{eq:linearRDE} becomes
\begin{equation}
\label{eq:linearRDEc}
	dY_t = -\ka Y_t\,dt + dX_t, \quad Y_0 \in \R, t\geq 0,
\end{equation}
The solution $Y$ can be viewed as Ornstein-Uhlenbeck process driven by a $1$-periodic process $X$.

Denote the unique solution of \eqref{eq:linearRDEc} with initial value $y\in \R^d$ by $(Y_t(y))_{t\geq 0}$.
We are looking for an initial value $y\in \R^d$ is, such that $Y_1(y) = Y_0(y) = y$. We have
\[0 = Y_0(y) - Y_1(y) = (1 - e^{-\ka})y  - \int_0^1 e^{-\ka (1 - s)}\,dX_s\]
if and only if
\[y = y^* := (1 - e^{-\ka})^{-1} \int_0^1 e^{-\ka (1 - s)}\,dX_s.\]
Further,
\begin{align*}
	Y_{n+1}(y^*) = & Y_n(y^*) e^{-\ka}+ \int_0^1 e^{-\ka (1-s)}\,dX_s\\
	= & Y_1(Y_n(x^*))
\end{align*}
Therefore,
\[Y_n(y^*) = y^*, \quad n \in \N_0,\]
and thus the solution to \eqref{eq:linearRDEc} with initial value $y^*$ (see \eqref{eq:linearRDEsol1PerRec})
\[Y_t(y^*) = y^* e^{-\ka \frk t} + \int_0^{\frk t} e^{-\ka (\frk{t}-s)}\,dX_s\]
is $1$-periodic, that is
\[Y_{t+1}(y^*) = Y_t(y^*),\quad t\geq 0\]
Thus, $Y(y^*)$ is a \emph{closed trajectory} of the dynamical system and its image
\[O = \set{Y_t(y^*) : t \geq 0} = \set{Y_t(y^*) : t\in [0,1]}\]
is a \emph{closed orbit}.

\begin{prop}
For every $y\in \R^d$ we have
\[d(Y_t(y), O) = \inf_{p\in O} |Y_t(y) - p| \leq |y-y^*| e^{-\la_{\min}(\ka) t}, t\geq 0.\]	
In particular, $Y_t(y)$ converges to $O$, as $t\to \infty$, for every initial value $y\in \R^d$.
\end{prop}

\begin{proof}
The solution to 
\[d(Y_t(y) - Y_t(\tilde y)) = - \ka (Y_t(y) - Y_t(\tilde y))\,dt\]
is given by
\[Y_t(y) - Y_t(\tilde y) = (y-\tilde y)e^{-\ka t}.\]
Hence
\[|Y_t(y) - Y_t(\tilde y)|  \leq |y-\tilde y| \specnrm{e^{-\ka t}} \leq |y-\tilde y| e^{-\la_{\min}(\ka) t}, \quad t\geq 0,\]
and
\[d(Y_t(y), O) \leq |Y_t(y) - Y_t(y^*)| \leq |y-y^*| e^{-\la_{\min}(\ka) t}, \quad t\geq 0.\]
\end{proof}
\end{docu}

\subsection{Cooling down under epoched bridge noise}
\subsubsection{Preliminaries}
For some asymptotic integral estimates we use the theory of regular variation \citep[see][for more information]{Bingham_Goldie_Teugels_1987}.
A function $f : [0,\infty) \to (0,\infty)$ is called \emph{regularly varying of index} $\rho$ if $f$ is measurable and 
\[\lim_{t\to \infty} \frac{f(ct)}{f(t)} \to c^\rho, \quad c > 0.\]
Further, we call $f$ \emph{slowly varying} if it is regularly varying of index $\rho = 0$.
If $f$ is regularly varying, then  $f$ and $1/f$ are locally bounded and locally integrable on $[t_0,\infty)$ for some $t_0 \geq 0$. Moreover,
we can write 
\[f(t) = t^\rho \ell(t), \quad t > 0\]
where $\ell$ is slowly varying. 

If $f$ is regularly varying and $f\sim g$, then $g$ is also regularly varying with the same index. In particular, if $g = o(f)$ and $f$ is regularly varying of index $\rho$, then so is $f + g$ (provided $f+g > 0$ everywhere).

If $f$ is regularly varying with negative index, then $f(t) \to 0$, as $t\to \infty$.

If $\ell$ is slowly varying, then $\ell(t) = o(t^{\al}), t\to \infty$ for any $\al > 0$.
Examples of slowly varying functions include $\log(t)^{\al}$ for all $\al \in \R$.
\begin{lem}
\label{lem:exp-Ut}
Let $\be \in (0,1)$ and $u$ be regularly varying with index $-\be$ and define $U_t = \int_0^t u_s\,ds$. Then 
\[e^{-U_t} = o(f(t)), \quad t\to \infty,\]
for any regularly varying function $f$.
\end{lem}
\begin{proof}
Writing $u_t = t^{-\be} \ell(t)$ for large $t$, we have by L'Hôpital's rule
\[\lim_{t\to \infty} \frac{U_t}{\log t} = \lim_{t\to \infty}  t u_t = \lim_{t\to \infty} t^{1-\be} \ell(t) = \infty.\]
Now, let $\al \in \R$. Then $-U_t + \al \log t \to -\infty$, and so $e^{-U_t} t^\al \to 0$, as $t\to \infty$.
If $f$ is regularly varying of index $\al$, then $e^{-U_t} = o(t^{-|\al| - 1}) = o(f(t))$ as $t\to \infty$.
\end{proof}

\begin{prop}
\label{prop:laplaceIntegralBnd}
Let $f$ and $u$ be regularly varying functions with indices $-\rho, -\beta < 0$ and $\beta < 1$. Suppose further that $f$ is locally bounded and $u\in \LpLoc{1}$ is non-increasing.
Then we have
\[\int_0^t f(s)e^{-U_t^s}\,ds \leq \frac{f(t)}{u(t)} + o\left(\frac{f(t)}{u(t)}\right),\quad t\to \infty,\]
where $U_t^s = \int_s^t u(s)\,ds$.
\end{prop}
\begin{proof}
Since $u$ is non-increasing, $U$ is concave and we have
\[U(s) \leq U(t) + u(t)(s-t), \quad s,t\geq 0,\]
where $U(t) = U_t^0$.
Therefore,
\begin{equation}
\label{eq:laplaceIntegralBndInitEst}
\int_0^t f(s)e^{-U_t^s}\,ds \leq \int_0^t f(s)e^{-(t-s)u(t)}\,ds = f(t)\int_0^t \frac{f(t-s)}{f(t)} e^{-su(t)}\,ds.
\end{equation}
Let $\tau : [0,\infty) \to [0,\infty)$ be non-increasing, such that 
\begin{equation}
\label{eq:timeTauCond}
\frac{\tau_t}{t}\to 0,\,\tau_t u(t) \to \infty, \quad t\to \infty.
\end{equation}
In particular, $\tau_t\to \infty$ since $u(t)\leq u(0), t\geq 0$. We make a particular choice of $\tau$ towards the end.
We split the integral on the RHS of Inequality \eqref{eq:laplaceIntegralBndInitEst} into a main part $\int_0^{\tau_t}\dots\,ds$ and a tail part $\int_{\tau_t}^t\dots\,ds$.

Let us first estimate the main part. Because $f$ is regularly varying with index $-\rho$, we have 
\[\lim_{t\to \infty} \sup_{c\in [a,\infty)} \left|\frac{f(ct)}{f(t)} - c^{-\rho}\right| = 0,\]
for all $a > 0$ \citep[Theorem 1.5.2]{Bingham_Goldie_Teugels_1987}.
Since $t-s = t(1-s/t)$ we have
\begin{align*}
\sup_{s\in (0,\tau_t]} \left|\frac{f(t-s)}{f(t)} - 1\right| = & \sup_{c\in [1-\frac{\tau_t}{t}, 1)} \left|\frac{f(ct)}{f(t)} - 1\right| \\
															\leq & \sup_{c\in [1-\frac{\tau_t}{t}, 1)} \left|\frac{f(ct)}{f(t)} - c^{-\rho}\right| + \sup_{c\in [1-\frac{\tau_t}{t}, 1)} |c^{-\rho} - 1| \\
															\to & 0,
\end{align*}
because $\frac{\tau_t}{t}\to 0$, as $t \to \infty$.
Hence,
\[\int_0^{\tau_t} \frac{f(t-s)}{f(t)} e^{-su(t)}\,ds \sim \int_0^{\tau_t} e^{-su(t)}\,ds = \frac{1}{u(t)} (1 - e^{-\tau_t u(t)}) \sim \frac{1}{u(t)}\]
as $t\to \infty$.

To estimate the tail integral let $\ep > 0$.
By Potter's theorem \citep[Theorem 1.5.6 (iii)]{Bingham_Goldie_Teugels_1987}, there exists a $t_0 \geq 0$ with
\[\frac{f(r)}{f(t)} \lesssim \left(\left(\frac r t\right)^{-\rho + \ep} \vee \left(\frac r t\right)^{-\rho - \ep}\right) = \left(\frac t r\right)^{\rho + \ep} \leq t_0^{-(\rho + \ep)}t^{\rho + \ep},\]
uniformly over $t\geq r\geq t_0$. In particular, by writing $r = t - s$ we have
\[\sup_{s\in [0,t-t_0]}\frac{f(t-s)}{f(t)} \lesssim t^{\rho + \ep},\]
uniformly over large $t$.
Since $f$ is locally bounded, we have
\[\sup_{s\in [t-t_0,t]}\frac{f(t-s)}{f(t)} \lesssim \frac1{f(t)} \sim \ell(t)t^{\rho}, \quad t\to \infty,\]
for some slowly varying function $\ell$. Hence,
\[\sup_{s\in [0, t]}\frac{f(t-s)}{f(t)} \lesssim t^{\rho + \ep}\ell(t),\]
uniformly over large $t$, for slowly varying $\ell$.
Thus,
\[\int_{\tau_t}^t \frac{f(t-s)}{f(t)} e^{-su(t)}\,ds \lesssim \ell(t)t^{\rho + \ep}\int_{\tau_t}^\infty e^{-s u_t}\,ds = \frac{1}{u(t)} \ell(t)t^{\rho +\ep} e^{-\tau_t u(t)},\]
uniformly over large $t$.
Finally, define $\tau_t = \frac{(\rho + 2\ep) \log t}{u(t)}$. Then the first convergence in \eqref{eq:timeTauCond} is satisfied because $u$ is regularly varying with index $-\be \in (-1,0)$. The second follows from $\log t \to \infty$, as $t\to \infty$.
Moreover, $t^{\rho+\ep}e^{-\tau_t u(t)} = t^{-\ep}$ and so 
\[\int_{\tau_t}^t \frac{f(t-s)}{f(t)} e^{-su(t)}\,ds = o\left(\frac{1}{u(t)}\right), \quad t\to \infty.\]
Using Inequality \eqref{eq:laplaceIntegralBndInitEst} we conclude
\[\int_0^t f(s)e^{-U_t^s}\,ds \leq \frac{f(t)}{u(t)} + o\left(\frac{f(t)}{u(t)}\right),\quad t\to \infty.\]
\end{proof}

\begin{lem}
\label{lem:sumVsInt}
Let $a,b\in \N_0$ with $a < b$ and $f : [a,b]\to \R$ be integrable with finite $1$-variation $\pvar f1$. Then
	\[\left|\sum_{n = a+1}^b f(n) - \int_a^b f(t)\,dt\right| \leq \pvar f1.\]
\end{lem}
\begin{proof}
	We calculate 
	\begin{align*}
		\sum_{n = a+1}^b f(n) 	= &\sum_{n=a}^{b-1} f(n+1) \\
		= &\sum_{n=a}^{b-1} \int_n^{n+1} f(t)\,dt + \sum_{n=a}^{b-1} \left(f(n+1)-\int_n^{n+1} f(t)\,dt\right)
	\end{align*}
	Note that
	\[\left|f(n+1)-\int_n^{n+1} f(t)\,dt\right| \leq \sup_{t\in [n,n+1)} |f(t) - f(n+1)|.\]
	Let $\ep > 0$. There exist $t_a,\dots, t_{b-1}$ with $t_n \in [n,n+1)$, such that
	\[\sup_{t\in [n,n+1)} |f(t) - f(n+1)| \leq |f(t_n) - f(n+1)| + \ep.\]
	Then
	\begin{align*}
		\left|\sum_{n=a}^{b-1} \left(f(n+1)-\int_n^{n+1} f(t)\,dt\right)\right| \leq \pvar f1 + (b-a)\ep.
	\end{align*}
	Since $\ep > 0$ was arbitrary, the desired conclusion follows.
\end{proof}

Now, let $\be \in (0,1), c > 0$ and consider $u : [0,\infty) \to [0,1], t\mapsto \frac{1}{(1+ct)^\be}$.
Given a positive definite and symmetric matrix $\ka$, the unique solution to the ODE
\[\dot \ph_t^s = -u_t \ka \ph_t^s\, \quad t\geq s, y_s = 1_{d\times d}\]
is given by $\ph_t^s = e^{- \ka U_t^s}$, where $U_t^s = \int_s^t u_r\,dr$, and we have
\begin{equation}
\label{eq:linODEest}
\specnrm{\ph_t^s} = \la_{\max}(\ph_t^s) \leq e^{-\la U_t^s},
\end{equation}
where $\la := \la_{\min}(\ka)$. In particular, $\ph_t^s$ converges to $0$, as $t\to \infty$.

\begin{lem}
\label{lem:lrScheduleEst}
We have
\begin{enumerate}[(a)]
\item $u \in \Lip^1([0,\infty))$,
\item $u$ is strictly decreasing, convex and $\lim_{t\to \infty} u_t = 0$,
\item $U$ is concave and $\lim_{t\to \infty} U_t = \infty$,
\item $|\dot u_t| = c\be u_t^{2 + \ga}$ for all $t\geq 0$, where $\ga = \frac{1-\be}{\be} > 0$,
\item \[\nrm{u_{\cdot} \ph^{\cdot}_t}{\Lip;[k,(k+1)\wedge t]} \leq (\la_{\max}(\ka)  + c \be u_k^\ga) u_k^2 e^{-\la U_t^{(k+1)\wedge t}},\]
for all $t\geq 1$ and $k\leq t$,
In particular, $\nrm{u_{\cdot} \ph^{\cdot}_t}{\Lip;[k,(k+1)\wedge t]} = o(u_t), t\to \infty$.
\item For all $\rho > 1$ and $t\geq 1$ we have
\[\sum_{k=0}^{\floor t-1} u_k^{\rho} e^{-\la U^{k+1}_t} \leq I_t(\rho) + I_t(\rho + 1) + \rho c\be I_t(\rho + \ga + 1) + e^{-\la U_t},\]
where $I_t(\al) = \int_0^{\floor t-1} u_s^\al e^{-\la U^{s+1}_t}\,ds$.
\item $I_t(\rho) \leq \la^{-1} (ct)^{-(\be(\rho - 1))} + o(t^{-(\be(\rho - 1))}), t\to \infty$, for all $\rho > 1$.
\item $e^{-\la U_t} = o(t^{-\al}), t\to \infty$, for all $\al > 0$.
\item \[\sum_{k=0}^{\floor t-1} \nrm{u_\cdot \ph^\cdot_t}{\Lip;[k,k+1]} \leq \condNr{\ka} (ct)^{-\be} + o(t^{-\be}),\]
as $t\to \infty$.
\end{enumerate}
\end{lem}

\begin{docu}
$e^{-\la U_t} = o(t^{-\al}), t\to \infty$, for all $\al > 0$ implies $e^{-\la U_t} = o(\ell(t))$ for some slowly varying function $\ell$ by Theorem 2.3.6. Binghzam (a result due to Vuilleumier).
\end{docu}

\begin{proof}
\begin{enumerate}[(a)]
\item $u$ is differentiable with $\dot u_t = -c \be(1+t)^{-(1+\be)}$ and $|\dot u_t| \leq \be$,
\item Straightforward.
\item We have 
\[U_t = \frac{1}{1-\be}\left((1+t)^{1-\be} - 1\right),\]
so $\lim_{t\to \infty} U_t = \infty$. Concavity follows from $u$ being strictly decreasing.
\item $|\dot u_t| = c \be (1+t)^{-(1+\be)}= c \be (1+t)^{-(1-\be)}(1+t)^{-2\be} = c \be u_t^{2+\ga}$ for all $t\geq 0$,
\item Let $f_s = u_s \ph_t^s$. Then 
\[\dot f_s = (\dot u_s 1_{d\times d} + u_s^2 \ka)\ph_t^s,\]
and so 
\[\specnrm{\dot f_s} \leq \specnrm{\dot u_s 1_{d\times d} + u_s^2 \ka}\specnrm{\ph_t^s}\leq (|\dot u_s| + u_s^2 \specnrm \ka)e^{-\la U_t^s} =(\specnrm \ka + c \be u_s^\ga)  u_s^2 e^{-\la U_t^s},\]
for all $0\leq s\leq t$. Taking the supremum over $[k,k+1]$ for each factor individually yields the estimate.
\item Set $n = \floor t$. By applying Lemma \ref{lem:sumVsInt} we have
\[e^{-\la U_t} \sum_{k=0}^{n-1} u_k^{\rho} e^{\la U_{k+1}}\leq e^{-\la U_t} \pvar{(u^\rho e^{\la U_{\blnk + 1}})|_{[0,n-1]}}1 + e^{-\la U_t} + I_t(\rho).\]
Since
\[|\der_s (u^\rho_s e^{\la U_{s+1}})| = (\rho u_s^{\rho - 1}|\dot u_s|+ u_s^{\rho + 1})e^{\la U_{s+1}} \leq u_s^{\rho + 1}(1 + \rho c \be u_s^\ga) e^{\la U_{s+1}},\]
we conclude
\[e^{-\la U_t} \pvar{(u^\rho e^{\la U_{\blnk + 1}})|_{[0,n-1]}}1 \leq I_t(\rho + 1) + \rho c \be I_t(\rho + \ga + 1).\]
\item Proposition \ref{prop:laplaceIntegralBnd} implies
\[I_t(\rho) \leq \int_1^t u_{s-1}^\rho e^{-\la U_t^s} \leq \frac{u_{t-1}^\rho}{\la u_t} + o\left(\frac{u_{t-1}^\rho}{u_t}\right), \quad t\to \infty.\]
Now observe that for $c = 1$
\[\frac{u_{t-1}^\rho}{u_t} = u_{t-1}^{\rho-1}\left(1 + \frac{1}{t}\right)^\be = t^{-(\be(\rho - 1))} + o(t^{-(\be(\rho - 1))}), t\to \infty,\]
so for general $c > 0$
\[\frac{u_{t-1}^\rho}{u_t} = (ct)^{-(\be(\rho - 1))} + o(t^{-(\be(\rho - 1))}), t\to \infty.\]
\item Follows from Lemma \ref{lem:exp-Ut}.
\item By applying (e) and (f) we have
\begin{align*}
	\sum_{k=0}^{n-1} \nrm{u_\cdot \ph^\cdot_t}{\Lip;[k,k+1]} \leq & \sum_{k=0}^{n-1} u_k^2(\la_{\max}(\ka) +  \be u_k^\ga) e^{-\la U_t^{(k+1)}} \\
	\leq & \la_{\max}(\ka)(I_t(2) + I_t(3) + 2 c \be I_t(3 + \ga) + e^{-\la U_t}) \\
	&+ \be(I_t(2+\ga) + I_t(3+\ga) + (2+\ga)c \be I_t(3 + 2\ga) + e^{-\la U_t}).
\end{align*}
We conclude the desired result using (g) and (h).
\end{enumerate}
\end{proof}

\subsubsection{Convergence results}

\begin{prop}
\label{prop:linearEpochedDecayYDE}
Let $X$ be a locally $\al$-Hölder epoched bridge and $Y$ be the solution to the linear \tYDE{}
\[dY_t = - u_t\ka Y_t\,dt + u_t\,dX_t, \quad Y_0 \in \R, t\geq 0.\]
Then
\[|Y_t|\leq \left(\frac{1}{1-2^{-\al}}\condNr{\ka} + 1\right)c^{-\be}\frac{x^*_t}{t^{\be}} + o\left(x^*_t t^{-\be}\right), \quad t\to \infty,\]
where $x^*_t := \max_{k\leq t} \nrm{X}{\alpha;[k,(k+1)\wedge t]}$.
\end{prop}

\begin{proof}
Let $t\geq 0$ and $n = \floor t$. By Proposition \ref{prop:varOfConst} we have
\[Y_t = \ph_t Y_0 + \int_n^t u_s \ph^s_t\,dX_s + \sum_{k=0}^{n-1} \int_0^1 u_{s+k}\ph^{s+k}_t \,dX_{s+k}, \quad n \in \N.\]
We estimate using the Young-Lóeve inequality in its original form (Proposition \ref{prop:ynglve}) and in the form \eqref{eq:ynglvebrdge} (with $\be = 1$), as well as Inequality \eqref{eq:linODEest}
\begin{align*}
|Y_t| \leq & |Y_0| e^{-\la U_t} + (|u_n \ph_t^n X_{n,t}| + C\nrm{u_{\cdot} \ph^{\cdot}_t}{\Lip;[n,t]} \nrm{X}{\alpha;[n,t]}) + C\sum_{k=0}^{n-1} \nrm{u_{\cdot} \ph^{\cdot}_t}{\Lip;[k,k+1]} \nrm{X}{\alpha;[k,k+1]},
\end{align*}
where $C = \frac{1}{1-2^{-\al}}$.
We have $e^{-\la U_t} = o(t^{-\be})$ by Lemma \ref{lem:lrScheduleEst} (h).
Further,
\[|u_n \ph_t^n X_{n,t}| \leq u_n \specnrm{\ph_t^n}|X_{n,t}| \leq u_n\cdot 1 \cdot (t-n)^\al \nrm{X}{\al;[n,t]} =  (x^*_tt^{-\be} + o(x^*_tt^{-\be})) ,\]
$t\to \infty$, and
\[\nrm{u_{\cdot} \ph^{\cdot}_t}{\Lip;[n,t]} \nrm{X}{\alpha;[n,t]} = o(x_t^*t^{-\be}),\quad t\to \infty,\]
by Lemma \ref{lem:lrScheduleEst} (e).
Finally,
\[\sum_{k=0}^{n-1} \nrm{u_{\cdot} \ph^{\cdot}_t}{\Lip;[k,k+1]} \nrm{X}{\alpha;[k,k+1]} \leq  \condNr{\ka}\frac{x_t^*}{t^{\be}} + o(x_t^* t^{-\be}),\quad  t\to \infty, \]
by Lemma \ref{lem:lrScheduleEst} (i).
\end{proof}

\begin{prop}
\label{prop:linToCvxDecay}
Let $\cR : \R^d \to \R \in \dC 2$ be $\la$-strongly convex and $L$-smooth with $\nabla \cR(0) = 0$ and $\nabla^2 \cR$ Hölder continuous.
Let $X$ be locally Hölder continuous and assume that $X$ does not vanish on any closed interval of positive measure.
Let $Y_0 = Z_0 \in \R^d$, and $Y, Z$ be the solutions to the Young differential equations
\begin{align*}
dY_t = &-u_t \nabla \cR(Y_t)\,dt + u_t\,dX_t,\\
dZ_t = &- u_t \nabla^2 \cR(0) Z_t\,dt + u_t\,dX_t, \quad t\geq 0.
\end{align*}
Let $f$ be regularly varying with negative index and assume $|Z_t| \leq f(t), t\to \infty$.
Then also
\[|Y_t| \leq f(t) + o(f(t)), \quad t\to \infty.\]
\end{prop}

\begin{docu}
If $\nabla^2 \cR$ is not globally Hölder continuous, then we still can show that $Y$ converges, except with
\[|Y_t|\leq (2L\vee 1)f(t) + o(f(t)).\]
This is because $|r(y)|\leq 2 L|y|$ and then the penultimate estimate yields $|\delt_t|\leq \leq 2Lf(t) + o(f(t))$.
\end{docu}

\begin{proof}
Firstly, assume $\cR$ is not quadratic. Otherwise, $Y = Z$ and we are done.
Now, using Hadarmard's lemma we have
\[r(y) := \nabla \cR(y) - \nabla^2 \cR(0) y = \int_0^1 (\nabla^2 \cR(ty) - \nabla^2 \cR(0))y\,dt.\]
Thus, the Hölder continuity of $\nabla^2 \cR$ implies
\[|\nabla^2 \cR(ty) - \nabla^2 \cR(0)|\lesssim |ty|^\ga \leq |y|^\ga, \quad t\in [0,1], y\in \R^d,\]
for some $\ga\in (0,1]$. Thus,
\begin{equation}
\label{eq:cvxRemEst}
|r(y)|\lesssim |y|^{1+\ga}
\end{equation}
uniformly over $y\in \R^d$, and we can write
\[dY_t = -u_t (\ka Y_t + r(Y_t))\,dt + u_t\,dX_t, \quad t\geq 0,\]
where $\ka := \nabla^2 \cR(0)$.
Let $\delta = Y - Z$. Then
\[\dot \delt_t = - u_t \ka \delt_t - u_t r(Y_t).\]
Furthermore,
\begin{align*}
\frac12\der_t(|\delt_t|^2) 	= \frac12 \der_t \innp{\delt_t}{\delt_t} = \innp{\dot \delt_t}{\delt_t} = &-u_t \innp{\ka \delt_t + r(Y_t)}{\delt_t} \\
						= &-u_t \innp{\ka \delt_t + r(Y_t) - r(Z_t)}{\delt_t} + u_t \innp {r(Z_t)}{\delt_t}, \quad t\geq 0.		
\end{align*}
Since $\cR$ is $\la$-strongly convex we have
\[\innp{\ka y + r(y) - (\ka z + r(z))}{y-z} = \innp{\nabla \cR(y) - \nabla \cR(z)}{y-z} \geq \la |y-z|^2, \quad y,z\in \R^d.\]
Hence, writing $v = |\delt|$,
\[\dot v_t v_t = \frac12\der_t(v_t^2) \leq - u_t \la v_t^2 + u_t |r(Z_t)|v_t,\]
and so
\begin{equation}
\label{eq:|delta|ineq}
\dot v_t \leq -u_t \la v_t + u_t |r(Z_t)|,
\end{equation}
for all $t\geq 0$, such that $\delt_t \neq 0$.
The set
\[\set{t \geq 0 : \delt_t = 0}\]
has Lebesgue measure zero. To show this note that if $\delta_t = 0$, then
\[\dot \delt_t = - u_t r(Y_t).\]
Assume $\delt = 0$ on an interval $[t,w]$. Then 
\[\dot \delt_s = - u_s r(Y_s) = 0,\quad s\in [t,w].\] 
Since $\cR$ is not quadratic we have $r(y) = 0$ if and only if $y = 0$.
Together with $u > 0$ everywhere this implies $Y = 0$ on $[t,w]$.
Thus,
\[Y_s = Y_t + \int_t^s u_v \,dX_v = \int_t^s u_v \,dX_v\]
implying $X = 0$ on $[t,w]$, which we assumed to be impossible.
Thus, $\delt_t = 0$ only at isolated points $t\geq 0$.
Hence, the set of $\delt$s zeros has measure $0$.

Moving on, define the integrating factor $I_t = e^{\la U_t}$. Then using Inequality \eqref{eq:|delta|ineq}
\[\der_t (I_tv_t) = I_t \dot v_t + \la u_t v_t I_t \leq u_t |r(Z_t)| I_t,\]
for almost all $t\geq 0$.
Hence,
\[|\delt_t|e^{\la U_t} = I_t v_t \leq \int_0^t u_s|r(Z_s)|e^{\la U_s}\,ds.\]
Note that the function $\tilde f = uf^{1+\ga}$ is again regularly varying with negative index.
Thus, using Inequality \eqref{eq:cvxRemEst} and Proposition \ref{prop:laplaceIntegralBnd} for the function $\tilde f$,
\[|\delt_t| \leq \int_0^t u_s e^{-\la U_t^s} |Z_s|^{1+\ga}\,ds \leq \int_0^t u_s e^{-\la U_t^s} f(s)^{1+\ga}\,ds = O\left(\frac{\tilde f(t)}{u(t)}\right) = o(f(t)), \quad t\to \infty.\]
We conclude
\[|Y_t| \leq |\delt_t| + |Z_t| \leq f(t) + o(f(t)), \quad t\to \infty.\]
\end{proof}

\begin{cor}
\label{cor:epochedCvxSimplConv}
Let $X$ be a locally $\al$-Hölder epoched bridge that does not vanish on any closed interval of positive measure, and such that
\[\max_{k\leq t} \nrm{X}{\alpha;[k,(k+1)\wedge t]} \leq \ell(t), \quad t\to \infty,\]
for some slowly varying function $\ell$.
Further, let $\cR : \R^d \to \R \in \dC 2$ be $\la$-strongly convex and $L$-smooth with $\nabla \cR(0) = 0$ and $\nabla^2 \cR$ Hölder continuous.
If $Y$ is the solution to the \tYDE{}
\[dY_t = - u_t \nabla \cR(Y_t)\,dt + u_t\,dX_t, \quad Y_0 \in \R, t\geq 0,\]
then
\[|Y_t|\leq \left(\frac{1}{1-2^{-\al}} \frac{L}{\la} + 1\right)c^{-\be}\frac{\ell(t)}{t^{\be}} + o\left(\ell(t) t^{-\be}\right), \quad t\to \infty.\]
\end{cor}
\begin{proof}
We apply Proposition \ref{prop:linearEpochedDecayYDE} to the linear ODE
\[dZ_t = - u_t \nabla^2 \cR(0) Z_t\,dt + u_t\,dX_t.\]
Then, Proposition \ref{prop:linToCvxDecay} implies the desired conclusion.
\end{proof}

\begin{docu}
Stefan 22.06.2025
\begin{itemize}
\item I tried to make the conditions on $u$ generic initially, but I just ended up having to add more and more conditions to get easy to understand convergence rate, and its not clear what functions even satisfy all these conditions. So in the end I settled on $(1+t)^{-\be}$. Probably most LR schedules that could theoretically work arent even that interesting in practice anyway. Comment 12.07.: regularly varying $u$ works surely, but then I have to state everything asymptotically from the get go and I didnt want that in case I want to derive less asymptotic results too. And frankly whos gonna use an LR schedule with extra logs.
\item There is no way the convergence rate I have deduced is tight, at least not for certain cases like single shuffle or RR. For single shuffle or any periodic $X$ we can leave out $\sqrt{\log t}$ entirely. Moreover, RR should be even better than SS according to SGDo literature. Even without the $\sqrt{\log t}$ the $2\be - 1$ should be improvable to $\be$ (see Gürbüzbalaban2019). - Comment: 30.06.  fixed, but I dont think I can prove a non-asymptotic rate. - Comment 12.07.: Gürbüzbalabans rates are also asymptotic, I just misunderstood their result. 
\item It is highly doubtful that you can deduce convergence under the general Robbins-Monroe conditions this way, because of the $\sqrt{\log t}$. Maybe for periodic driver it could work?
\item For $u_t = \frac{1}{\log(1+t)}$ or $u_t = 1/(1+t)$. What happens then?
\end{itemize}
\end{docu}

\section{Proof of the main theorem}
Firstly, let us prove that $(\nabla \cR)^{-1}$ is actually well-defined.
\begin{lem}
Let $\la > 0$. Suppose $\cR$ is $\la$-strongly convex with Lipschitz gradient. Then $\nabla \cR : \R^d\to \R^d$ is bijective.
\end{lem}
\begin{proof}
Strong convexity implies strong monotonicity, that is
\[\innp{\nabla \cR(x) - \nabla \cR(y)}{x-y} \geq \la |x-y|^2, \quad x,y\in \R^d.\]
In particular, $\nabla \cR$ is injective.
To show surjectivity we use the Browder-Minty theorem \citep[see][Theorem 10.49]{renardy2006introduction}, identifying $\R^d$ with its dual space. Indeed, $\nabla \cR$ is monotone, as shown before. Also since $\nabla \cR$ is Lipschitz, it is in particular continuous and preserves bounded sets. To show coercivity, note that strong convexity of $\cR$ implies
\[\cR(0) \geq \cR(x) + \innp{\nabla \cR(x)}{0-x} + \frac{\la}{2}|x|^2, \quad x\in \R^d.\]
That is,
\[\innp{\nabla \cR(x)}{x} \geq \cR(x) - \cR(0) + \frac{\la}{2} |x|^2.\]
In particular,
\[\lim_{x\to 0}\frac{\innp{\nabla \cR(x)}{x}}{|x|} = \infty.\]
Hence, $\nabla \cR$ is coercive, and thus also surjective.
\end{proof}
Now, let us transform equation \eqref{eq:epochedCvx} into a simpler form.
We can rewrite
\[dY_t = - u_t (\nabla \cR(Y_t) - T^{-1/2}\si Z)\,dt + u_t \sqrt T \si dX_{t/T},\]
or equivalently
\[dY_{tT} = - u_{tT}\nabla \hat \cR(Y_{tT})\,dt + u_{tT} \sqrt T \si dX_t,\]
where $Z = \frac{1}{\sqrt T}\hat W_T \sim \cN(0,1_{d\times d}), \hat W_t = \sqrt T X_{t/T} + \frac t {\sqrt T} Z$ and $X$ is an epoched Brownian bridge independent of $Z$,
and $\hat \cR(y) =  \cR(y) - T^{-1/2} \si Z y$. Note that 
\[(\nabla \hat \cR)^{-1}(0) = (\nabla \cR - T^{-1/2}\si Z)^{-1}(0) = (\nabla \cR)^{-1}(T^{-1/2}\si Z).\]
Define 
\[\tilde Y_t = \frac{1}{\sqrt T} \si^{-1}(Y_{tT} - (\nabla \hat \cR)^{-1}(0)), \quad t\geq 0.\]
Then
\[d\tilde Y_t = - u_{tT} \frac{1}{\sqrt T}\si^{-1} \nabla \hat \cR(\sqrt T \si \tilde Y_t + (\nabla \hat \cR)^{-1}(0))\,dt + u_{tT}dX_t, \quad t\geq 0.\]
Equivalently, we can write
\[d\tilde Y_t = - u_{tT} \nabla \tilde \cR(Y_t)\,dt + u_{tT}\,dX_t,\]
where
\begin{align*}
\tilde \cR(y) :=& T^{-1}\si^{-2} \hat \cR(\sqrt T \si y + (\nabla \hat \cR)^{-1}(0)) \\
			= & T^{-1}\si^{-2} \cR(\sqrt T \si y + T^{-1}\si \hat W_T) - T^{-1}\si \hat W_T y, \quad y\in \R^d.
\end{align*}
Let us summarize this procedure in a proposition.
\begin{lem}
	\label{lem:epochedCvxTransf}
	Let $Y$ be the solution to \eqref{eq:epochedCvx}. Then 
	\[\tilde Y_t = \frac{1}{\sqrt T}\si^{-1}(Y_{tT} - (\nabla \cR)^{-1}(T^{-1}\si \hat W_T))\]
	is the unique solution to the \tYDE{}
	\[d\tilde Y_t = - \tilde u_t \nabla \hat \cR(\tilde Y_t) + \tilde u_t\,dX_t, \quad t\geq 0,\]
	where $\tilde u_t = u_{tT}$ and
	\[\tilde \cR(y) =  T^{-1}\si^{-2} \cR(\sqrt T \si y + T^{-1}\si \hat W_T) - T^{-1}\si \hat W_T y, \quad y\in \R^d.\]
\end{lem}

\begin{proof}[Proof of Theorems \ref{thm:main} and \ref{thm:mainAlt}]
Recall the definition of $Y$ in \eqref{eq:epochedCvx}. Apply Lemma \ref{lem:epochedCvxTransf}, then
\[Y_t = \sqrt T \si \tilde Y_{t/T} + (\nabla \cR)^{-1}(T^{-1}\si \hat W_T).\]
Note that $X$ does not vanish on any closed interval of positive measure, almost surely.
Suppose for now we are given slowly varying function $\ell$ with
\begin{equation}
\label{eq:EBMdriverSlow}
\max_{k\leq t} \nrm{X}{\alpha;[k,(k+1)\wedge t]} \leq \ell(t), \quad a.s.,t\to \infty.
\end{equation}
By Corollary \ref{cor:epochedCvxSimplConv}
\[\left|Y_t - (\nabla \cR)^{-1}(T^{-1}\si \hat W_T)\right| \leq \sqrt T |\si| \left(\frac{1}{1-2^{-\al}}\frac{L}{\la} + 1\right)(cT)^{-\be}\frac{\ell(t)}{t^{\be}} + o\left(\ell(t) t^{-\be}\right), \quad t\to \infty.\]
Here, we used that $\nabla^2 \tilde \cR(0) = \nabla^2 \cR( (\nabla \cR)^{-1}(T^{-1}\si \hat W_T))$.

We can find a slowly varying function $\ell$ such that Inequality \eqref{eq:EBMdriverSlow} holds true. Indeed, by Lemma \ref{lem:epochedBBgrowth} we can set 
\[\ell(t) :=  a^{-1/2}\sqrt{\log t} + g(t) \geq a^{-1/2}\sqrt{\log{(\floor t + 1)}},\] 
for $a \in (0, \frac{1}{2(1-b)b^{1-2\al}})$, where $b = \frac{1-2\al}{2-2\al}$, and
\[g(t) = a^{-1/2}(\sqrt{\log{(\floor t + 1)}} - \sqrt{\log t})   = o(\sqrt{\log t}), \quad t\to\infty.\]
If we pick $\al = 0.42$, $a = 0.8 \in (0, 0.858581) = (0, \frac{1}{2(1-b)b^{1-2\al}})$, then
\[a^{-1/2} = 1.11803 < 1.2, \quad a^{-1/2} \frac{1}{1-2^{-\al}} = 4.61727 < 4.7,\]
proving Theorem \ref{thm:main} (the second constant cannot be lowered much further).
Assume now there exists a number $J\in \N$, such that $\cI := \set{(W_{(j+t)T}-W_{jT})_{t\in [0,1]} : j\in \N}|$ satisfies $|\cI| = J$, almost surely. Then we can instead set $\ell(t) = \max_{w\in \cI} \nrm{w}{\al}, t\geq 0$ in Inequality \eqref{eq:EBMdriverSlow}, proving Theorem \ref{thm:mainAlt}.
\end{proof}

\chapter{On the weak convergence of shuffled random walks}

\label{chap:weakshuffle}
In this chapter, we study scaling limits of random walks that share the same increments up to a (random) permutation.
Assuming convergence of the sequence of random permutations to a deterministic $J$-dimensional permuton $\mu$ (i.e. a probability measure on $[0,1]^J$ with uniform marginals), we show weak convergence of these random walks to a $J$-dimensional Gaussian process $W$. The covariance of two different components of $W$ is then given by the bivariate marginals of the joint distribution function of $\mu$. Conversely, given a $J$-dimensional permuton $\mu$ we can find a sequence of (higher dimensional) random permutations converging to $\mu$. Thus, we further prove that every Gaussian process $W$, whose covariance function is determined by $\mu$, is the scaling of some family of random walks sharing the same increments up to a random permutation.
Finally, we show that epoched Brownian processes determined by a permuton $\mu$ arise as scaling limits of random walks that use only finitely many distinct increments.

\section{Introduction}

Consider a probability space $(\Om, \cF, \P)$ and a sequence of i.i.d.\ random variables $(Z_n)_{n\in \N}$  with $\E Z_1 = 0, \Var Z_1 = 1$ and $\E[Z_1^4] < \infty$. Denote the group of permutations of size $N$ by $\Sym N$. Given $N\in \N$ and a family of permutations of increasing size, that is $(\tau_N)_{N\in \N}$ with $\tau_N \in \Sym N$, define
\[S_n^{N, \tau} =\sum_{k=1}^n Z_{\tau_N(k)}, \quad n\in \set{0,\dots, N}, N\in \N.\]
Then $S^{N,\tau}$ is a simple random walk using the shuffled increments $(Z_{\tau_N(i)})_{i=1}^N$. Using appropriate scaling and interpolation we expect this random walk to converge weakly to a Brownian motion, as $N\to \infty$. We are interested in the limiting behavior of the joint distribution of the random vector $(S^{N, \id{N}}, S^{N,\tau})_{N\in \N}$ consisting of the random walk with unshuffled increments and the one with increments shuffled by $(\tau_N)_{N\in \N}$, appropriately scaled and interpolated.

More generally, let $J\in \bar \N := \N \cup \set{\infty}$. We call a sequence of $N$-permutations $(\si^j)_{j < J} = (\si^j)_{j \in [0,J)\cap \N_0}$ a \emph{$J$-dimensional permutation\footnote{There are several possible interpretations for the term \enquote{$J$-dimensional permutation}. We follow the convention by \citet{borga_high-dimensional_2025}, \emph{except} what we call $J$-dimensional they call $J-1$ dimensional.} of size $N$}. We let $\dSym JN$ denote the set of $J$-dimensional permutation of size $N$. Now, consider a sequence of \emph{random} $J$-dimensional permutations\footnote{Without loss of generality we may set $\si_N^0 = \id{N}, N \in \N$, but since it does not simplify any of the upcoming arguments we will not. However, it can be useful in examples to make this choice, especially in the case $J = 2$.} of increasing size $\si = (\si_N : \Om \to \dSym JN)_{N\in \N}$. Equivalently, $\si$ is a random element of $\prod_{N\in \N} \dSym  JN$. Define the family of $\R^J$-valued stochastic processes $(X^N : \Om \times [0,1] \to \R^J)_{N\in \N}$ by
\[X^N_t =  (X^{N,0}_t, X^{N,1}_t, \dots) := \frac{1}{\sqrt N}(S^{N,\si^j}_{\floor{Nt}})_{j<J}, \quad t\in [0,1], N\in \N.\]
Here, we equip $\R^J= \prod_{j<J} \R$ with the product topology. Note that $\R^J$ is a Polish space. 
Our goal is to find sufficient conditions, such that $(X^N)_{N\in \N}$ converges in distribution in the Skorokhod space $\cD([0,1], \R^J)$, as $N\to \infty$, to a Gaussian process $W$.

To our knowledge, scaling limits of random walks with shared increments up to a permutation have not been studied in this form before. We note however that Gaussian limits with the same law as $W$ (see our main theorem \ref{thm:shuffledDonsker}) have been studied in the theory of empirical processes, albeit with no connection to permutations or random walks, and only for finite $J$. Consider for simplicity $J = 2$. Given a $2$-copula $C$ (i.e. the distribution function of a $2$-dimensional permuton) we call a random field $B : \Om \times [0,1]^2\to \R$ a \emph{$C$-Brownian bridge} if 
\[\E[B_{u,v} B_{u',v'}] = C(u\wedge u', v\wedge v') - C(u,v) C(u',v'),\quad u,u',v,v'\in [0,1].\]
Then the $2$-dimensional Gaussian process $W$ in Theorem \ref{thm:shuffledDonsker} is given by
\[W_t = \mat{W^0_t\\ W^1_t} = \mat{B_{t,1}\\ B_{1,t}} + t\mat{V\\V}, \quad t\in [0,1],\]
for some $C$-Brownian bridge $B$ and a standard Gaussian random variable $V$ independent of $B$. Here, $C = F^{01}_\mu$ is the distribution function of the limiting permuton $\mu$ determining the law of $W$. 
Indeed, in this case the components of $W$ are Brownian motions with
\[\E[W_s^0W_t^1] = \E[(B_{s,1} + sV)(B_{1,t} + tV)] = \E[B_{s,1}B_{1,t}] + st = C(s,t) - st + st = C(s,t),\]
for $s,t\in [0,1]$. This coincides with the covariance formula in Theorem \ref{thm:shuffledDonsker}.
\section{The story in two dimensions}
Before we start with the general theory, let us heuristically explore the case $J = 2$ first. Let $(Z_k)_{k\in \N_0}$ be an \tiid\ sequence with $\E Z_0 = 0, \Var Z_0 = 1$ and finite fourth moment, and let $(\pi_N : \Om \to \Sym N)_{N\in \N}$ be a sequence of random permutations increasing in size, independent of $(Z_k)_{k\in \N}$.
We consider the following two random sums with shared increments
\[S_n = \sum_{k=1}^n Z_k, \quad S_n^{N,\pi} = \sum_{k=1}^n Z_{\pi_N(k)}, \quad n \in \set{0,\dots, N}, N\in \N.\]
Let 
\[X_t^N = (X_t^{N,0}, X_t^{N,1}) = \frac 1{\sqrt{N}} (S_{\floor{Nt}}, S^{N,\pi}_{\floor{Nt}}), \quad t\geq 0, N\in \N.\]
Then, it is well-known that $X^{N,0}$ and $X^{N,1}$ converge to Brownian motions $W^0$ and $W^1$ in law. We are interested in the convergence of the joint distribution $X^N$. A straightforward computation shows
\[\Cov(X_s^{N,0}, X_t^{N,1}) = \frac1N \E|[Ns] \cap \pi_N[Nt]|,\quad s,t\in [0,1],\]
where $\tau[x] = \set{\tau(1), \dots, \tau(\floor x)}$ for $x \geq 0$.
Looking closely at the quantity $\frac1N|[Ns] \cap \pi_N[Nt]|$ we notice that it is the joint distribution function of the empirical measure
\[\hat \mu_{(\id{N}, \pi_N^{-1})} = \frac1N \sum_{k=1}^N \idK_{\blnk}(k/N, \pi_N^{-1}(k)/N)\]
on the unit square $[0,1]^2$ (notice the inverse $\pi^{-1}_N$). Here, the joint distribution function of a probability measure $\nu$ on $[0,1]^2$ is given by
\[F_\nu(s,t) = \nu([0,s]\times [0,t]), \quad s,t\in [0,1].\]
Suppose $\hat \mu_{(\id{N}, \pi_N^{-1})}$ converges weakly in the space of probability measures on $[0,1]^2$ to a deterministic probability measure $\mu$. 
Then we can actually already conclude that $X^N$ converges weakly to a Gaussian process $W = (W^0, W^1)$.
The limiting measure $\mu$ necessarily has uniform marginals on $[0,1]$. A probability measure on $[0,1]^2$ with uniform marginals is called a (two-dimensional) \emph{permuton}.
In the proof for convergence we use the fact that the distribution function of $\hat \mu_{(\id{N}, \pi_N^{-1})}$, i.e.\ $(s,t)\mapsto \frac1N|[Ns] \cap \pi_N[Nt]|$, must also converge, in fact uniformly, to the distribution function $F_\mu$ of $\mu$. The distribution function of a permuton is called a ($2-$) \emph{copula}.
The copula $F_\mu$ determines the covariance function of $W$, via
\[\Cov(W_s^0,W_t^1) = \lim_{n\to \infty} \frac1N \E|[Ns] \cap \pi_N[Nt]| = F_\mu(s,t), \quad s,t\in [0,1].\]
Further, we find it more convenient to work with permutons throughout. For $N\in \N$, we define the permuton
\[\mu_{(\id{N}, \pi_N^{-1})} = \frac1N \sum_{k=1}^N \cU\left[\frac{k-1}{N},\frac k N\right]\otimes \cU\left[\frac{\pi_N^{-1}(k)-1}{N},\frac{\pi_N^{-1}(k)}{N}\right].\] 
Here, $\cU[a,b]$ is the uniform distribution on $[a,b]$. 
Compared to the empirical measure $\hat \mu_{(\id{N}, \pi_N^{-1})}$, we essentially replaced all mass points with squares of side length $1/N$. Since 
\[\nrm{F_{\mu_{(\id{N}, \pi_N^{-1})}} - F_{\hat \mu_{(\id{N}, \pi_N^{-1})}}}{\infty} \leq \frac2N,\quad N\in \N,\]
we can work with the copulas $F_{\mu_{(\id{N}, \pi_N^{-1})}}$ instead of $F_{\hat \mu_{(\id{N}, \pi_N^{-1})}}$ throughout the proof.

Now, let us invert the situation. Suppose we are given a permuton $\mu$. Then it well-known from the theory of permutons and large permutations that we can construct a sequence of random permutations $(\pi_N)_{N\in \N}$ such that $\mu_{(\id{N}, \pi_N^{-1})}$ converges weakly to $\mu$. 

In fact, consider an \tiid\ sequence $(U_N : \Om \to [0,1]^2)_{N\in \N}$ with $U_N\sim \mu$. Given $N\in \N$ and $v\in \R^N$ define a permutation $\Perm(v)$ of size $N$ such that it orders the elements of $v$ from least to greatest. Ties are ranked according to their original position in $v$, ensuring $\Perm(v)$ is an actual permutation.
Set $\si^j_N = \Perm(U^j_1,\dots, U^j_N), j = 1,2$, i.e.\ define permutations $\si^1_N, \si^2_N$ by sorting the random vectors $(U_1^1,\dots, U_N^1),(U_1^2,\dots, U_N^2)$ . By defining $\pi_N^{-1} = \si^2_N\circ (\si^1_N)^{-1}$ we can indeed show that $\mu_{(\id{N}, \pi_N^{-1})}$ converges weakly to $\mu$. 
Thus, given a two-dimensional Gaussian process $W$ with $\Cov(W_s^1,W_t^2)= F_\mu(s,t)$ and Brownian marginals, we can construct $S$ and $X$ in such a way that $X^N\to W$ in law.
\section{Main theorem}
Let us rewrite
\[S_n^{N,\si^j} = \sum_{k=1}^N \idK_{\si^j_N[n]}(k)Z_k, \quad j < J, 0\leq n\leq N.\]
Here and in the following we write $[x] := \N \cap [0,x]$, for all $x\geq 0$. So given $\tau \in \Sym N$, where $N \geq \floor x$, we have $\tau[x] = \set{\tau(1), \dots, \tau(\floor x)}$.
The covariance between the components of $X$ satisfies 
\begin{align*}
\Cov(X^{N,i}_s, X^{N,j}_t) 	= & \frac1N\sum_{k=1}^N \sum_{l=1}^N \Cov(\idK_{\si^i_N[Ns]}(k) Z_k, \idK_{\si^j_N[Nt]}(l)Z_l)\\
							= & \frac1N\sum_{k=1}^N \E[\idK_{\si^i_N[Ns]}(k)\idK_{\si^j_N[Nt]}(k)]\\
							= & \frac1N \E | \si^i_N[Ns] \cap \si^j_N[Nt]|, \quad i,j < J,
\end{align*}
since by the law of total covariance
\begin{align*}
\Cov(\idK_{\si^i_N[Ns]}(k) Z_k, \idK_{\si^j_N[Nt]}(l)Z_l) = & \E[\Cov(\idK_{\si^i_N[Ns]}(k) Z_k, \idK_{\si^j_N[Nt]}(l)Z_l|\si)] \\
															&  + \Cov(\E[\idK_{\si^i_N[Ns]}(k) Z_k|\si], \E[\idK_{\si^j_N[Nt]}(l)Z_l|\si]) \\
														= & \E[\idK_{\si^i_N[Ns]}(k)\idK_{\si^j_N[Nt]}(k)] \delt_{k,l}, \quad i,j<J, k,l \in [N].
\end{align*}
Thus, if $X^N$ converges in distribution, then $\frac1N \E | \si^i_N[Ns] \cap \si^j_N[Nt]|$ also has to converge for all $i,j,s,t$, under uniform integrability assumptions. Conversely, this suggests that the sequence of $J$-dimensional permutations $(\si_N)_{N\in \N}$ needs to converge in a certain sense.

Given a Polish space $S$, we consider the space of probability measures $\cP(S)$ equipped with the topology of weak convergence, i.e. $\nu_n\to \nu$ if
\[\lim_{n\to \infty} \int_S f \,d\mu_n = \int_S f \,d \mu,\]
for all continuous and bounded $f : S \to \R$. Equivalently, $\nu_n \to \nu$ if $\nu_n(A) \to \nu(A)$ for all $A\in \cF_S$ with $\nu(\partial A) = 0$. The space $\cP(S)$ is again Polish and thus we may consider weak convergence in $\cP\cP(S)$.
Given random measures $\nu_n, \nu : \Om \to \cP(S), n\in \N$, we say $\nu_n$ converges to $\nu$ in distribution if $\nu_n\P \to \nu\P$ in $\cP \cP(S)$. Here, $\nu\P(A) = \P(\nu^{-1}(A)),A \in \cF$.

Given $J\in \bar \N$ and $\tau \in \dSym JN$ we define the empirical measure $\hat \mu_\tau \in \cP([0,1]^J)$ by
\[\hat \mu_\tau = \frac 1N \sum_{k=1}^N \idK_{\blnk}((\tau^j(k)/N)_{j<J}).\]
This measure has discrete uniform marginals on the set $\{\frac1N, \dots, \frac NN\}$. As $N\to \infty$ the marginals converge to the uniform distribution on $[0,1]$.
Alternatively, it can be more convenient to have continuous uniform marginals even before taking the limit. We call a probability measure on $[0,1]^J$ with $\cU{[0,1]}$-marginals a \emph{$J$-dimensional permuton} or \emph{$J$-permuton} for short. With $\tau \in \dSym JN$ we associate the following $J$-permuton
\[\mu_\tau = \frac1N \sum_{k=1}^N\bigotimes_{j<J} \cU\left(\frac1N[\tau^j(k)-1,\tau^j(k)]\right),\]
where $\cU{[a,b]}$ denotes the uniform distribution on $[a,b]$. Here, $\bigotimes$ denotes the product measure operation. Effectively, this replaces the mass points from the empirical measure of $\tau$ with hypercubes of side length $1/N$.

For $J<\infty$ and any finite-dimensional measure $\nu \in \cP([0,1]^J)$ we consider the joint distribution function
\[F_\nu(t) = \nu([0,t_0]\times \dots \times [0,t_{J-1}]), \quad t = (t_0,\dots, t_{J-1}) \in [0,1]^J.\]
Given $\nu \in \cP([0,1]^J)$ and a tuple $a \in \N^*$ with length $n = |a| < J$ we consider the projection $\nu^a\in \cP([0,1]^n)$ given by
\[\nu^a(A) := \pr^a\nu(A) = \nu((\pr^a)^{-1}(A)), \quad A\in \cF_{[0,1]^n},\]
where 
\[\pr^a : \R^J \to \R^n, (x_j)_{j\in \N} \mapsto (x_{a_1},\dots, x_{a_n}).\]
We also write $F^a_\nu:= F_{\nu^a}$ and define the (joint) distribution function of an infinite-dimensional measure $\nu\in \cP([0,1]^\N)$ by
\[F_\nu : [0,1]^* \to [0,1], t\mapsto F_{\nu}^{1\dots |t|}(t).\]
Any $\nu \in \cP([0,1]^J)$ is uniquely determined by its distribution function (even for $J = \infty$).
The distribution function of $J$-permutons are known as $J$-\emph{copulas} for finite $J$. Accordingly, we also call the distribution function of an $\infty$-permuton an $\infty$-\emph{copula}.

Write $\tau^{-1} := ((\tau^j)^{-1})_{j < J}$ for any $J$-dimensional permutation $\tau$.
We can relate the bivariate marginals of the empirical measure $\hat \mu_{\si_N^{-1}}$ to the covariance of the process $X^N$.
Specifically, we will show that for $i,j < J$ (see Lemma \ref{lem:empPermApprox} (a) below)
\[\hat F_N^{ij}(s,t) := F_{\hat \mu_{\si^{-1}_N}}^{ij}(s,t) = \frac1N |\si^i_N[Ns] \cap \si^j_N[Nt]|, \quad s,t\in [0,1].\]
The case $i = j$ is included and we have
\[\hat F^{ii}_N(s,t) = \frac1N |[Ns] \cap [Nt]| \to s\wedge t, \text{ as }  N\to \infty.\]
By Lemma \ref{lem:empPermApprox} (b) below, the convergence of the bivariate distribution functions $\hat F_N^{ij}$ is equivalent to the convergence of the bivariate marginals of the $\infty$-permutons $\mu_{\si^{-1}_N}$. Nevertheless, we require convergence of not just the bivariate marginals in the following.
\begin{assum}
\label{assum:limitCopulas}
The sequence of random $J$-permutons $(\mu_{\si_N^{-1}})_{N\in \N}$ converges to a deterministic measure in distribution, i.e. there exists a $\mu \in \cP([0,1]^\N)$, such that 
\[\mu_N := \mu_{\si_N^{-1}} \to \mu,\]
in distribution, as $N\to \infty$.
\end{assum}
In this case $\mu$ is also a $J$-permuton and, since the limit is deterministic, the convergence is also in probability, with respect to the weak convergence topology. 

Recall that for any set $T$ a random field $Y : \Om \times T \to \R$ is called Gaussian if $(Y_{t_1},\dots, Y_{t_m})$ is Gaussian, for all $t_1,\dots, t_m\in T$. We say $Y$ is centered if $\E Y_t = 0$ for all $t\in T$.

Let $(U^j)_{j<J}$ be a family of $\Unif{[0,1]}$-random variables and define 
\[A_{(t,j)} = \set{U^j \leq t}, \quad j < J, t\in [0,1].\]
Then the kernel on the set $[0,1] \times ([0,J) \cap \N_0)$ given by
\[K((s,i), (t,j)) = \P(A_{(s,i)} \cap A_{(t,j)}) = \P(U^i \leq s, U^j\leq t) = \E[\idK_{A(s,i)}\idK_{A(t,j)}], \quad i,j < J, s,t\in [0,1]\]
is positive semi-definite. This is because for every finite family $(t_1,j_1),\dots, (t_m,j_m)$, the matrix 
\[(K((t_k,j_k), (t_l,j_l)))_{k,l \in \set{1,\dots, m}}\]
is the Gram matrix of the indicator functions $\idK_{A(t_1,j_1)},\dots, \idK_{A(t_m,j_m)}$ in the Hilbert space\footnote{The inner product is $(X,Y)\mapsto \E[XY]$.} $\cL^2(\Om, \P)$ and thus positive semi-definite.
By choosing $U^j\sim \mu, j < J$, we conclude that there exists a stochastic process $W : \Om \times [0,1] \to \R^J$, such that
\[\Om \times [0,1]\times ([0,J) \cap \N_0) \to \R, (\om, t, j)\mapsto W^j_t(\om)\]
is centered Gaussian, and
\[\Cov(W^i_s, W^j_t) = \P(U^i \leq s, U^j \leq t) = F^{ij}_{\mu}(s,t), \quad s,t\in [0,1], i,j < J.\]
Our goal is first to show that $W$ is indeed a limit of the sequence of processes $(X^N)_{N\in \N}$ in distribution.
We want also show that conversely, if we are given a $J$-permuton $\mu$, then we can find a sequence of random $J$-dimensional permutations $(\si_N)_{N\in \N}$, such that Assumption \assref{assum:limitCopulas} is satisfied.

\begin{satz}
\label{thm:shuffledDonsker}
Suppose we are given a sequence $(\si_N)_{N\in \N}$ of random $J$-dimensional permutations, such that \assref{assum:limitCopulas} holds true.
Then $(X^N_t)_{t\in [0,1]}$ converges in distribution to a centered Gaussian process $W$, as $N\to \infty$, with
\begin{equation}
\label{eq:covShuffledDonsker}
\Cov(W^i_s, W^j_t) = F^{ij}_{\mu}(s,t), \quad s,t\in [0,1], i, j < J.
\end{equation}
Conversely, given a centered Gaussian process $W$ and a $J$-permuton $\mu$, such that \eqref{eq:covShuffledDonsker} holds true, there exists a sequence $(\si_N)_{N\in \N}$ of random $J$-dimensional permutations, such that \assref{assum:limitCopulas} holds true and so $(X^N)_{t\in [0,1]}$ converges in distribution to $W$.
\end{satz}

\begin{example}
\label{ex:basiccopulas}
We give several examples for $(\si_N)_{N\in \N}$ such that \assref{assum:limitCopulas} is fulfilled and thus the convergence in Theorem \ref{thm:shuffledDonsker} holds true.
\begin{enumerate}[(a)]
\item Let $J = \infty$ and $\si_N = (\id{N})_{j\in \N}$, so that no shuffling occurs at all. Then
$\mu_N$ converges to the (deterministic) singular measure $\mu$ given by the \emph{comonotonicity copula}
\[F_\mu(t) = t_0 \wedge \dots \wedge t_{m-1}, \quad t\in [0,1]^m, m < J.\]
The process $W$ is equal in distribution to $(\tilde W)_{j\in \N}$ where $\tilde W$ is a one-dimensional Brownian motion on $[0,1]$.
\item Let $ J = \infty$. Suppose $(\si^j_N)_{j\in \N}$ is \tiid\ with $\si^1_N \sim \Unif{\Sym N}$ for all $N\in \N$. Then one can show that $\mu_N$ converges to the infinite product of the Lebesgue measure on $[0,1]$ with itself, i.e.
\[\mu = \Unif{[0,1]} \otimes \Unif{[0,1]} \otimes \dots\]
That is, its joint distribution function is the \emph{independence copula}
\[F_\mu(t) = \prod_{k=0}^{m-1} t_k, \quad t\in [0,1]^m, m < J.\]
\item Set $J=2$.
Define $\tau \in \Sym N$ by $\tau_N(n) = N - n + 1, n \in \N$, that is $\tau$ puts the elements of $\set{1,\dots, N}$ into reverse order. Note that $\tau_N^{-1} = \tau_N$. The distribution function of the sequence of $2$-permutons $\mu_{(\id{N}, \tau_N)}$ converges to the \emph{countermonotonicity copula} 
\[\cW(s,t) = \max(s + t - 1, 0)\]
There is no direct higher-dimensional analogue to this because reversing the order twice leaves us with the original order. 

\end{enumerate}
\end{example}
We justify these claims (see in particular Example \ref{ex:basiccopulas2}) and give further examples in Subsection \ref{sec:constructingexamples}.

\section{Permutons and copulas}
In this section we give additional background on (higher-dimensional) permutons and copulas. None of this material is really new, except for the simple extension to $J = \infty$.
For permutons we follow \citet{borga_high-dimensional_2025} and for copulas \citet{durante2015principles}.

\subsection{Approximation by random permutations}

The following lemma explains that while our definition of $J$-dimensional permutations and their associated permutons differs slightly from the typical definitions, in particular the ones by \citet{borga_high-dimensional_2025}, the difference is not crucial. This is also the reason why we said we \emph{can}    set $\si^N_0 = \id{N}$, but do not \emph{need} to do that.
\begin{lem}
\label{lem:permtnStab}
Let $N\in \N, J\in \bar \N, \tau \in \Sym N^J$ and $\pi \in \Sym N$. Then $\mu_{\tau\circ \pi} = \mu_\tau$, where $\tau\circ \pi = (\tau^j\circ \pi)_{j<J}$. In particular, the $J$-permuton $\mu_\tau$ is completely determined by a $J-1$-dimensional permutation via
\[\mu_\tau = \mu_{(\id{N}, \tau^1\circ (\tau^0)^{-1}, \tau^2\circ (\tau^0)^{-1}, \dots)}.\]
\end{lem}
\begin{proof}
We have
\begin{align*}
\mu_{\tau\circ \pi} 	= & \frac1N \sum_{k=1}^N\bigotimes_{j<J} \cU\left(\frac1N[(\tau^j(\pi(k)))-1,\tau^j(\pi(k))]\right) \\
						= & \frac1N \sum_{k=1}^N\bigotimes_{j<J} \cU\left(\frac1N[(\tau^j(k))-1,\tau^j(k)]\right),
\end{align*}
where in the last step we permuted the summands by $\pi$ (which leaves the sum unchanged).
\end{proof}

\begin{lem}
	\label{lem:empPermApprox}
	Let $N, J \in \N$ and $\si\in \dSym JN$.
	Then the following hold true.
	\begin{enumerate}[(a)]
		\item $F_{\hat \mu(\si^{-1})}(t) = \frac1N |\bigcap_{j < J} \si^j[Nt_j]|, \quad t\in [0,1]^J$,
		\item $\nrm{F_{\mu_\si} - F_{\hat \mu_\si}}{\infty} \leq \frac J N$.
	\end{enumerate}
\end{lem}

\begin{proof}
	\begin{enumerate}[(a)]
		\item 
		Let $t = (t_0,\dots, t_{J-1})\in [0,1]^J$. Then 
		\begin{align*}
			N F_{\hat \mu_{\si^{-1}}}(t) = & \sum_{k=1}^N \prod_{j=1}^J \idK_{[0,t_j]}\left(\frac{(\si^j)^{-1}(k)}{N}\right)\\
			= & |\set{k \leq N : (\si^j)^{-1}(k) \leq Nt_j, j < J}|\\
			= & |\set{k \leq N : k = \si^j(l_j), l_j \leq Nt_j, j < J\text{ for some } l_1,\dots, l_J  \leq N}|\\
			= & |\bigcap_{j \leq J} \si^j[Nt_j]|.
		\end{align*}
		\item 
		Let $t\in [0,1]^J$. Then\footnote{In the second to last equality, the first colon is for the set builder notation $\set{\nu : \phi(\nu)}$ and the other one is part of the $\forall$-quantifier.}
		\begin{align*}
			F_{\mu_\si}(t) = & \mu_{\si}([0,t_0]\times \dots \times [0,t_{J-1}])\\
			= & \frac1N \sum_{k=1}^N\prod_{j=0}^{J-1} \cU\left(\frac1N[\si^j(k)-1,\si^j(k)]\right)(0,t_j) \\
			= & \frac1N \sum_{k=1}^N\prod_{j=0}^{J-1} (\si^j(k) \wedge N t_j - (\si^j(k) - 1) \wedge N t_j) \\
			= & \frac 1N |\set{k \leq N : \forall j< J : \si^j(k)\leq \floor{Nt_j}}| + R_N\\
			= & F_{\hat \mu_\si}(t) + R_N,
		\end{align*}
		where the first summand accounts for all the $1/N\times \dots \times 1/N$ hypercubes fully contained in $[0,t_1]\times \dots \times [0,t_J]$ and $R_N$ accounts for the smaller rectangular hypercuboids still left over in $[0,t_1]\times \dots \times [0,t_J]$. More precisely,
		consider $K = \set{k \leq N : \exists j \leq J : \si^j(k) = \floor{Nt_j} + 1}$. This is the image of the map $[J] \to K, j \mapsto (\si^j)^{-1}(\floor{Nt_j} + 1)$ and so $|K| \leq J$. Thus,
		\[R_N =  \frac 1N \sum_{k \in K}\prod_{j=1}^J (\si^j(k) \wedge N t_j - (\si^j(k) - 1) \wedge N t_j) \leq \frac{|K|}{N} \leq \frac{J}{N}. \]
	\end{enumerate} 
\end{proof}

Lemma \ref{lem:empPermApprox} (b) says that the Kolmogorov-Smirnov distance of $\mu_\si$ and $\hat \mu_\si$ is bounded by $J/N$. This cannot be improved to a bound on, say, the total variation distance, since $\mu_\si(\supp \hat \mu_\si) = 0$.

%

\begin{lem}
\label{lem:inftypermtonMetric}
Let $J\in \bar \N$. Suppose we are given a sequence of $J$-permutons $(\nu_n)_{n\in \N}$ and a $J$-permuton $\nu$. Then $\nu_n \to \nu$ weakly if and only if
\begin{equation}
\label{eq:inftypermtonMetric}
\lim_{n\to \infty}\sum_{j=1}^J 2^{-j}\nrm{F^{1\dots j}_{\nu_n} - F^{1\dots j}_\nu}{\infty} = 0.
\end{equation}
\end{lem}
\begin{proof}
For $J < \infty$, we have $\nu_n\to \nu$ weakly if and only if $F_{\nu_n} \to F_{\nu}$ in $\cC([0,1]^J, [0,1], \nrm{\blnk}{\infty})$, by \citet{borga_high-dimensional_2025} Proposition 2.1, proving the finite dimensional case.

Now, suppose $J = \infty$.
Observe that as a subspace of $\R^\N$, weak convergence of probability measures on $[0,1]^\N$ already follows from the convergence of their finite-dimensional distributions.
The following defines a metric $d$ on the product space $\prod_{j=1}^\infty \cC([0,1]^j, [0,1])$
\[d((f_j)_{J\in \N}, (g_j)_{J\in \N})= \sum_{j=1}^\infty 2^{-J}\nrm{f_j - g_j}{\infty},\]
which induces the product topology. Thus, as $n\to \infty$,
\begin{align*}
\nu_n \to \nu & \Ioif \fa j\in J : \nu_n^{1\dots j}\to \nu^{1\dots j} \\
			& \Ioif \fa j \in J : F_{\nu_n}^{1\dots j}\to F_{\nu}^{1\dots j}\text{ in } \cC,\\
			&\Ioif d((F_{\nu_n}^{1\dots j})_{j\in \N}, (F_{\nu}^{1\dots j})_{j\in \N}) \to 0.
\end{align*}
\end{proof}

So far we considered the convergence of a given sequence of random permutations to some permuton. Next, we want to reverse this process. 
Specifically, Let $J\in \bar \N$ and $\mu \in \cP([0,1]^J)$ be a $J$-permuton. 
Our aim is to construct a sequence of random $J$-dimensional permutations $\si = (\si_N : \Om \to \Sym N^J)_{N\in \N}$ such that the associated sequence of $J$-permutons converges to $\mu$.

To this end, consider an \tiid\ sequence $(U_N : \Om \to [0,1]^J)_{N\in \N}$ with $U_N \sim \mu$. In other words, we consider a (possibly infinite) matrix of random variables uniformly distributed on $[0,1]$:

\begin{equation}
	\label{term:matrixU}
	\mat{U^1_1 & U_1^2 & \dots & U_1^j & \dots \\ 
		\vdots & \vdots & \ddots & \vdots & \dots \\
		U_N^1 & U_N^2 & \dots & U_N^j & \dots
	}
\end{equation}
Here, the rows are independent realizations of a discrete-time permuton process, that is a process with uniform marginals\footnote{Here, the marginals are uniform on $[0,1]$. In some other works permuton processes have uniform marginals on $[-1,1]$. The difference is mostly cosmetic.}, with distribution $\mu$.

\newcommand{\perm}{\operatorname{Perm}}
Given any $N\in \R^N$ and $v\in \R^N$ define $\Perm(v) \in \Sym N$ by 
\begin{equation}
	\label{eq:approxPrmttns}
	\Perm(v)(k) = 1 + |\set{l\leq N : v_l < v_k}| + |\set{l < k : v_l = v_k}|, \quad k \leq N.
\end{equation}
In other words, $\Perm(v)(k)$ is the (stable) rank of $v_k$ in the vector $v$. Thus, $\Perm(v)$ sorts the vector $v$ from lowest to highest. Entries that are tied (equal) are ranked according to their original position. This ensures that $\Perm(v)$ is indeed a permutation.

Now, set $\si^j_N := \Perm(U^j_1,\dots U^j_N), j < J$.
Because $\P(U^j_l = U^j_k) = 0$ for all $l,k< N$, we have
\[\si^j_N(k) = 1 + |\set{l\leq N : U_l^j < U_k^j}|,\quad a.s., j < J, k \leq N.\]

Given $v\in \R^N$ we write $v_{(k)} = v_{\Perm(v)(k)}$. Then $v_{(1)} \leq \dots \leq v_{(N)}$. Consider independent $U_1,\dots, U_N \sim \Unif{[0,1]}$. Then, because the sequence is exchangeable, we have $\P(U_1 = U_{(n)}) = \dots = \P(U_N = U_{(n)})$ for all $n \leq N$. Thus, $\P(U_m = U_{(n)}) = \frac1N$ for all $m,n\leq N$. In other words, $\Perm(U_1,\dots, U_N) \sim \Unif{\Sym N}$. 

Consequently, $\si^j_N \sim \Unif{\Sym N}$ for all $j<J$. Aside from the marginals, the joint distribution of $(\si^j_N)_{j < J}$ is determined by $\mu$. In fact, the random $J$-dimensional permutations $\si_N$ approximate $\mu$ in the following sense.
\begin{prop}
	Let $J\in \N$, $\mu$ be a $J$-dimensional permuton,  $(U_N : \Om \to [0,1]^J)_{N\in \N}$ an \tiid\ sequence with $U_N \sim \mu$ and let $\si = (\si_N : \Om \to \Sym N^J)_{N\in \N}$ be defined by 
	\[\si^j_N := \Perm(U^j_1,\dots U^j_N), \quad j < J, N\in \N.\] Then
	\[\P(\nrm{F_{\mu(\si_N)} - F_\mu}{\infty} > 4JN^{-1/4}) \leq e^{-\sqrt N},\]
	for large $N$.
\end{prop}
\begin{proof}
	This was proven by \citet[Proposition 2.6.]{borga_high-dimensional_2025}. More precisely, this exact estimate is found at the end of their proof of that proposition.
\end{proof}

\begin{prop}
	\label{prop:convRandSortPerm}
	Let $J\in \bar\N$, $\mu$ be a $J$-dimensional permuton,  $(U_N : \Om \to [0,1]^J)_{N\in \N}$ an \tiid\ sequence with $U_N \sim \mu$ and let $\si = (\si_N : \Om \to \Sym N^J)_{N\in \N}$ be defined by 
\[\si^j_N := \Perm(U^j_1,\dots U^j_N), \quad j < J, N\in \N.\] Then
	\[\mu_{\si_N} \to \mu, \quad a.s.,\]
	as $N\to \infty$.
\end{prop}
\begin{proof}
	Consider the case $J < \infty$ first. For some sufficiently large $N_0$ we have
	\[\sum_{N=N_0}^\infty \P(\nrm{F_{\mu(\si_N)} - F_\mu}{\infty} > 4JN^{-1/4}) \leq \sum_{N=N_0}^\infty e^{-\sqrt N} < \infty.\]
	Thus, by Borel-Cantelli
	\[\nrm{F_{\mu(\si_N)} - F_\mu}{\infty} \leq 4JN^{-1/4} \text{ for large } N, \quad a.s.\]
	In particular, $\mu_{\si_N} \to \mu, a.s.$, as $N\to \infty$.
	
	In the case $J = \infty$ the convergence $\mu_{\si_N} \to \mu, a.s.$ follows from the convergence of the finite-dimensional distributions $\mu_{\si_N}^{1\dots J} \to \mu^{1\dots J}, a.s., J \in \N$, as $N\to \infty$.
\end{proof}

\subsection{Constructing examples}
\label{sec:constructingexamples}
\begin{lem}
\label{lem:incrsngUnif}
Let $Z : \Om \to \R$ be a random variable and $\ph$ be an increasing function. Then $\ph(Z) \sim \Unif{[0,1]}$ if and only if 
\[\ph(t) = F_Z(t) := \P(Z\leq t), \quad \P_Z\text{-}a.s.\]
Similarly, if $\psi$ is decreasing, then $\psi(Z) \sim \Unif{[0,1]}$ if and only if 
\[\psi(t) = 1 - F_Z(t) = \P(Z <t), \quad \P_Z\text{-}a.s.\]
\end{lem}
\begin{proof}
Assume $\ph(Z) \sim \Unif{[0,1]}$. Since $\ph$ is increasing, the generalized inverse exists $\ph^{-1}(t) = \inf\set{z\in \R : \ph(z) \geq t}$ and is increasing. We have
\[t = \P(\ph(Z)\leq t) = \P(Z \leq \ph^{-1}(t)) = F_Z(\ph^{-1}(t)), \quad t\geq 0.\]
This implies $\ph^{-1} = F_Z^{-1}$, Lebesgue almost everywhere, and so $\ph = F_Z$, $\P_Z$-almost surely.
The converse is well known. The statement for decreasing $\psi$ is proven similarly (by duality).
\end{proof}

\begin{lem}
\label{lem:indpnceCopula}
Let $U_1,\dots, U_m$ be random variables with $U_1,\dots, U_m\sim \Unif{[0,1]}$. Then the following are equivalent:
\begin{enumerate}[(i)]
\item $U_1,\dots, U_m$ are independent,
\item $\P_{(U_1,\dots, U_m)} = \Unif{[0,1]}^{\otimes m}$,
\item $\P(U_1\leq t_1,\dots, U_m \leq t_m) = \prod_{j=1}^m t_j,\quad t\in [0,1]^m$,
\end{enumerate}
\end{lem}
\begin{proof}
Straightforward.
\end{proof}

\begin{lem}
\label{lem:comonCopula}
Let $U_1,\dots, U_m$ be random variables with $U_1,\dots, U_m\sim \Unif{[0,1]}$. Then the following are equivalent:
\begin{enumerate}[(i)]
\item $U_1 = \dots = U_m, a.s.$,
\item $\P(U_1 \leq t_1,\dots, U_m \leq t_m) = t_1 \wedge \dots \wedge t_m,\quad t\in [0,1]^m$,,
\end{enumerate}
\end{lem}
\begin{proof}
We apply Theorem 2.5.7 by \cite{durante2015principles}. It implies that (ii) is equivalent to the following statement: there exists a probability space $(\Om', \cF_{\Om'}, \P')$, a random variable $Z : \Om' \to \R$ and increasing functions $\ph_j : \R \to \E, j = 1,\dots, m$, such that 
\[\P_{(U_1,\dots, U_m)} = \P'_{(\ph_1(Z),\dots, \ph_m(Z))}.\]
Since $\ph_j(Z)$ is uniform, we must have $\ph_j(t) = \P(Z\leq t), \P'_Z\text{-}a.s$ by Lemma \ref{lem:incrsngUnif}. Hence, (ii) is equivalent to $U_1 = \dots = U_m, \P\text{-}a.s.$
\end{proof}

\begin{lem}
\label{lem:cntrmonCopula}
Let $U,V$ be random variables with $U,V\sim \Unif{[0,1]}$. Then the following are equivalent:
\begin{enumerate}[(i)]
\item $U = 1-V, a.s.$,
\item $\P(U\leq s, V\leq t) = \cW(s,t) := (s+t-1)\vee 0$.
\end{enumerate}
\end{lem}
\begin{proof}
We apply Theorem 2.5.13 by \cite{durante2015principles}. It implies that (ii) is equivalent to the following statement: there exists a probability space $(\Om', \cF_{\Om'}, \P')$, a random variable $Z : \Om' \to \R$ an increasing function $\ph : \R \to \R$ and a decreasing function $\psi : \R\to \R$, such that 
\[\P_{(U,V)} = \P'_{(\ph(Z),\psi(Z))}.\]
Since $\ph(Z), \psi(Z)$ are uniform, we must have 
\[\ph(t) = \P(Z\leq t) \text{ and } \psi(t) = \P(Z > t) = 1 - \P(Z \leq t), \quad \P'_Z\text{-}a.s,\]
by Lemma \ref{lem:incrsngUnif}. Hence, (ii) is equivalent to $U = 1-V, \P\text{-}a.s.$
\end{proof}

\begin{example}
\label{ex:basiccopulas2}
\begin{enumerate}[(a)]
\item Consider $\mu$ given by the comonotonicity copula
\[F_\mu(t) = t_1 \wedge \dots \wedge t_m, \quad t\in [0,1]^m, m < J\]
Given a random variable $U\in \Unif{[0,1]}$ we have $(U,\dots, U) \sim \mu^{1\dots m}$. Conversely, if $(U_1,\dots, U_m) \sim \mu^{1\dots m}$, then $U_1 = \dots = U_m$ almost surely by Lemma \ref{lem:comonCopula}. Thus, up to almost sure equality, the matrix \eqref{term:matrixU} is given by
\[\mat{U_1 & U_1 & \dots & U_1 & \dots \\ 
	\vdots & \vdots & \ddots & \vdots & \dots \\
	U_N & U_N & \dots & U_N & \dots
}\]
Then $\si_N^1 = \si_N^2 = \dots = \Perm(U_1,\dots, U_N)\sim \Unif{\Sym N}$, and so
\[\mu_{\si_N} = \mu_{(\si_N^1)_{j\in \N}} = \mu_{(\id{N})_{j\in \N}}\]
by Lemma \ref{lem:permtnStab}.
Thus, $\mu_{(\id{N})_{j\in \N}}\to \mu, a.s.,$ as $N\to \infty$, by Proposition \ref{prop:convRandSortPerm}.
\item Consider $\mu = {\Unif{[0,1]}}^{\otimes J}$, i.e. the $J$-fold product of the Lebesgue measure on $[0,1$] with itself. Its distribution function is the independence copula
\[F_\mu(t) = \prod_{j=1}^m t_j, \quad t\in [0,1]^m, m < J.\]
Suppose we are given $U_1,\dots, U_m\in \Unif{[0,1]}$. Then $(U_1,\dots, U_m) \sim \Unif{[0,1]}^{\otimes m}$ if and only if $U_1,\dots, U_m$ are independent by Lemma \ref{lem:indpnceCopula}.
Thus, the matrix \eqref{term:matrixU} consists entirely of \tiid\ random variables. Hence, $(\si^j_N)_{j\in \N}$ is \tiid\ with $\si^j_N \sim \Unif{\Sym N}$ and $\mu_{(\si^j_N)_{j\in \N}} \to \mu, a.s.,$ as $N\to \infty$, by Proposition \ref{prop:convRandSortPerm}.
\item Consider random variables $U, V\sim \Unif{[0,1]}$ and the countermonotonicity copula
\[\cW(s,t) = (s + t - 1) \vee 0, \quad s,t\in [0,1].\]
We have $(U,V) \sim \cW$ if and only if $U = 1 - V, a.s.$ by Lemma \ref{lem:cntrmonCopula}. Thus, for $J = 2$ the matrix \eqref{term:matrixU} is almost surely given by
\[\mat{U_1 & 1 - U_1 \\ 
	\vdots & \vdots \\
	U_N & 1- U_N
}\]
Since $(1-U_1,\dots, 1- U_N)$ is almost surely in reverse order compared to $(U_1,\dots, U_N)$, we have $\si^1_N = \tau_N \circ \si^0_N, a.s.$, where $\tau_N(k) = N-k+1$ is the \emph{reversal permutation} of size $N$. By Proposition \ref{prop:convRandSortPerm} we have 
\[\mu_{(\id{N}, \tau_N)} = \mu_{(\si^0_N, \tau_N\circ \si^0_N)} \to \P_{(U,1-U)}, a.s.,\]
as $N\to \infty$, where $U\sim \Unif{[0,1]}$.

Reversing the order twice leaves you with the original order. So there is no direct higher-dimensional analogue to $\cW$. Next, we look at two simple generalizations of $\cW$ to infinite dimensions. 

\item Given $U\in \Unif{[0,1]}$, the sequence $(U,1-U,U,1-U,\dots)$ has the distribution $\mu$ with
\[F_\mu(t_1,\dots, t_{2m}) = \cW(t_1,t_2)\wedge \dots \wedge \cW(t_{2m-1},t_{2m}), \quad t\in [0,1]^{2m}, m\in \N,\]
the matrix \eqref{term:matrixU} is almost surely given by
\[\mat{U_1 & 1- U_1 & U_1 & 1-U_1 & \dots \\ 
	\vdots & \vdots & \ddots & \vdots & \dots \\
	U_N & 1 - U_N & U_N & 1-U_N & \dots
},\]
$\si^{2j}_N = \si^0_N, a.s.$ and $\si^{2j+1}_N = \tau_N\circ \si^0_N, a.s.$ for all $j\in \N$. Hence,
\[\mu_{(\id{N}, \tau_N, \id{N}, \tau_N, \dots)} = \mu_{(\si^0_N, \tau_N\circ \si^0_N, \si^0_N, \tau_N\circ \si^0_N, \dots)} \to \mu, a.s.,\]
as $N\to \infty$.
\item Given an \tiid\ sequence $(U^j)_{j\in \N}$ with $U^j \sim \Unif{[0,1]}$, the sequence $(U^1,1-U^1,U^2,1-U^2,\dots)$ has the distribution $\mu$ with
\[F_\mu(t_1,\dots, t_{2m}) = \prod_{j=1}^m\cW(t_{2j-1},t_{2j}), \quad t\in [0,1]^{2m}, m\in \N,\]
the matrix \eqref{term:matrixU} is almost surely given by
\[\mat{U_1^1 & 1- U_1^1 & U_1^2 & 1-U_1^2 & \dots \\ 
	\vdots & \vdots & \ddots & \vdots & \dots \\
	U_N^1 & 1 - U_N^1 & U_N^2 & 1-U_N^2 & \dots
},\]
and $\si^{2j+1}_N = \tau_N\circ \si^{2j}_N, a.s.$ for all $j\in \N$. Hence,
\[\mu_{(\si^0_N, \tau_N\circ \si^0_N, \si^2_N, \tau_N\circ \si^2_N, \dots)} = \mu_{(\si^j_N)_{j\in \N}} \to \mu, a.s.,\]
as $N\to \infty$.
\end{enumerate}
\end{example}

The set of $\infty$-permutons is much richer than the examples so far indicate. To illustrate this point, we discuss \emph{Archimedean copulas} next.

A function $\ph : [0,\infty) \to [0,1]$ is called an \emph{(additive) generator} if it is continuous, decreasing with $\ph(0) = 1, \lim_{t\to \infty} \ph(t) = 0$, and strictly decreasing on $[0,t_0]$ with $t_0 = \inf\set{t > 0 : \ph(t) = 0}$. The \emph{pseudo-inverse} $\ph^{(-1)} : [0,1]\to [0,\infty)$ of a generator $\ph$ is defined by
\[\ph^{(-1)}(t) =  \begin{cases}
\ph^{-1}(t), & t \in (0,1],\\
t_0, & t = 0.
\end{cases}\]
Note that $\ph^{(-1)}(\ph(t)) = t \wedge t_0, t\geq 0$.

A function $\ph : [0,\infty) \to [0,1]$ is \emph{completely monotone} if it is continuous with $\ph \in \dC{\infty}((0,\infty))$, and $(-1)^k f^{(k)}(x) \geq 0$ for all $x > 0$.

\begin{prop}
\label{prop:archmdnInftyCopula}
Let $\ph : [0,\infty) \to [0,1]$ be an additive generator. Then $\ph$ is completely monotone if and only if 
\[C(t) = \ph\left(\sum_{k=1}^{|t|} \ph^{(-1)}(t_k)\right), \quad t\in [0,1]^*\]
defines an $\infty$-copula.
\end{prop}
\begin{proof}
This is immediate by Corollary 6.5.14 by \cite{durante2015principles}.
\end{proof}
In the setting of Proposition \ref{prop:archmdnInftyCopula}, the $\infty$-copula $C$ is called \emph{Archimedean} with generator $\ph$.
\begin{example}
The following are well-known examples of (families of) completely monotone generators and their associated Archimedean $\infty$-copulas.
\begin{center}
\begin{tabular}{c|c|c|c|c}
	Family & $\ph^{(-1)}(u)$ & $\ph(v)$ & $C(t)$ & $\thet$ \\
	\hline
	Clayton & $ \frac{u^{-\thet} - 1}{\thet}$ & $(1 + \thet v)^{-1/\thet}$ & $\left(\sum_{k=1}^{|t|}(t_k^{-\thet} - 1) + 1\right)^{-1/\thet}$ &  $ > 0$\\
	Gumbel & $(-\log u)^\thet$ & $\exp(-v^{1/\thet})$ & $\exp\left(-\left(\sum_{k=1}^{|t|}(-\log t_k)^\thet\right)^{1/\thet}\right)$  &  $\geq 1$\\
	Frank & $-\log\left(\frac{e^{-\thet u} - 1}{e^{-\thet}-1}\right)$ &  $-\frac1\thet \log(1 - (1 - e^{-\thet})e^{-v})$ & $-\frac1\thet \log\left(1 + \frac{\prod_{k=1}^{|t|} (e^{-\thet t_k} - 1)}{(e^{-\thet}-1)^{|t|-1}}\right)$ & $> 0$ \\
\end{tabular}
\end{center}
To reiterate, any of these examples define an $\infty$-permuton $\mu$ and Proposition \ref{prop:convRandSortPerm} defines a sequence of random infinite-dimensional permutations $(\si_N)_{N\in \N}$ with $\mu_{\si_N} \to \mu$ almost surely. After picking such a family one can conceivable tune the parameter $\thet$, such that the limiting process in main result Theorem \ref{thm:shuffledDonsker} has whatever properties one desires more.
\end{example}

There are many other types of examples, including extreme-value copulas and elliptic copulas. Moreover, Sklar's theorem lets us turn \emph{any} joint distribution on $\R^J$ with continuous marginals into a unique $J$-copula. Many of these examples and Sklar's theorem can be potentially generalized directly to the case $J =\infty$. Alternatively, one can combine finite-dimensional examples by independence (just putting independent vectors together) or using the Markov product of copulas (essentially the gluing operation from optimal transport theory). For more information we refer to \citet{durante2015principles}.

There are also numerous examples that come from studying large permutations, such as the limits of sequences of square permutations. Other examples, such as the limit of the Baxter permutations, cannot be used directly because the limit $\mu$ is \emph{random}. However, in principle one can always take the intensity measure 
\[\E\mu : \cF \to [0,1], A\mapsto \E[\mu(A)]\]
as an example instead. Although note that the sequence of permutations converging to $\E\mu$ is in general different from the one for $\mu$. See \citet{borga2021randompermutationsgeometric} for more on these topics.

Finally, note that even $\mu_\tau$ for some $J$-dimensional permutation $\tau \in \Sym N^J$ is a legitimate example for the limit $\mu$ in \assref{assum:limitCopulas}. The so called \emph{shuffles of Min} provide another \enquote{finite-permutation} example \citep[see again][for more information]{durante2015principles}.

\begin{rem}
Even though the limiting process $W$ in Theorem \ref{thm:shuffledDonsker} depends only on the bivariate marginals $F_\mu^{ij}$, we insisted on discussing $J$-copulas also for $J>2$. This is because not all families of $2$-copulas $(C^{ij})_{i<j<J}$ are \emph{compatible}. We call a family $(C^{ij})_{i<j<J}$ of $2$-copulas compatible if there exists a $J$-copula permuton $\mu$ such that $F_\mu^{ij} = C^{ij}, i<j<J$, that is if they are genuinely the bivariate marginals of some $J$-permuton or $J$-copula. For, the family
\[C^{12}(s,t) = \cW, C^{23} = \cW, C^{13} = \cW\]
consisting of only the countermonotonicity copula is \emph{not} compatible, that is there exists no $3$-permuton with these bivariate marginals. The reason is if $U_1,U_2,U_3$ are uniform, such that $(U_1,U_2) \sim \cW$ and $(U_2,U_3)\sim \cW$, then $U_1 = 1 - U_2 = U_3,a.s.$, so $(U_1,U_3)\nsim \cW$.
From the viewpoint of permutations the idea is: if you reverse the order twice, you end up with the original order (the reversal permutation satisfies $\tau_N\circ \tau_N = \id{N}$).
\begin{docu}
Let us consider the convergence of the bivariate marginals $F_{\mu_N}^{ij}$ again. If $(\mu_N)_{N\in \N}$ is a sequence of \emph{deterministic} $J$-permutons, such that $F_{\mu_N}^{ij}$ converges for all $i,j < J$, then their limits must be compatible. In other words, while $(\mu_N)_{N\in \N}$ may not converge, there must be another sequence $(\tilde \mu_N)_{N\in \N}$ with the same bivariate marginals which converges to some $J$-permuton $\mu$.
This is because the space of $J$-permutons is compact. Thus, the sequence has a convergent subsequence $(\mu_{N_k})_{k\in \N}$ with limit $\mu$. Then 
\[\mu^{ij} = \lim_{k\to \infty} \mu_{N_k}^{ij} = \lim_{N\to \infty} \mu_{N}^{ij}, \quad i,j < J.\]
However, if the sequence $(\mu_N)_{N\in \N}$ is \emph{random}, then it cannot be guaranteed that there is a deterministic limit. Just consider two different $J$-permutons that have all the same bivariate marginals (say the independence copula). Now, take a mixture at every level $N$. Then all bivariate marginals converge to a deterministic limit (because they are constant), but the limit of the sequence is still a mixture and thus random.
\end{docu}
\end{rem}

\section{Proof of the main theorem}
First, let us note some consequences of Assumption \assref{assum:limitCopulas}.
\begin{lem} 
	\label{lem:permtonCDFconv}
	Assume \assref{assum:limitCopulas} holds true. Then
	\[\nrm{F_{\mu_N}^a - F_\mu^a}{\infty} \to 0, \text{in probability},\]
	as $N\to \infty$, for all $a\in \set{j\in \N : j < J}^*$.
\end{lem}
\begin{proof}
	Because $\mu_N \to \mu$ in probability we also have
	\[\sum_{j=1}^J 2^{-j}\nrm{F^{1\dots j}_{\mu_N} - F^{1\dots j}_\mu}{\infty} \to 0, \text{in probability},\]
	as $N\to \infty$, using Lemma \ref{lem:inftypermtonMetric}.
\end{proof}

\begin{cor}
	\label{lem:empPermtonCDFconv}
	Assume \assref{assum:limitCopulas} holds true. Then
	\[\nrm{\hat F_{\mu_N}^a - F_\mu^a}{\infty} \to 0, \text{in probability},\]
	as $N\to \infty$, for all $a\in \set{j\in \N : j < J}^*$.
\end{cor}
\begin{proof}
	Combine Lemma \ref{lem:empPermApprox} (b) with Lemma \ref{lem:permtonCDFconv}.
\end{proof}

In the following we also denote distributions with the more traditional notation $\P_Z := Z \P$ and we write $\P^N_X := \P_{X_N}$.
To prove the convergence result in Theorem \ref{thm:shuffledDonsker} it is sufficient to show two things:
\begin{enumerate}[(a)]
\item The family $\left(\P^{N}_X\right)_{N \in \mathbb{N}}$ is tight.
\item The finite-dimensional distributions of $\P^{N}_X$ converge to those of $\mathbb{P}_{W}$, that is, given $t_{1}, \ldots, t_{M} \in[0,1]$ and $j < J$ we have
\[\mat{X_{t_{1}}^{N,0} & \cdots & X_{t_{1}}^{N,j} \\
	\vdots & \ddots & \vdots \\
	X_{t_{M}}^{N,0} & \cdots & X_{t_{M}}^{N,j}} \Rightarrow
\mat{W_{t_{1}}^{0} & \cdots & W_{t_{1}}^{j} \\
	\vdots & \ddots & \vdots \\
	W_{t_{M}}^{0} & \cdots & W_{t_{M}}^{j}}\]
as $N \rightarrow \infty$.
Here and in the following, $W$ is the stochastic process introduced in Theorem \ref{thm:shuffledDonsker}.
\end{enumerate}

Our approach for showing tightness is similar to the proof of Theorem 14.1. in \cite{billingsley1999convergence}.

\begin{lem}
\label{lem:tightness}
The sequence $\left(\P^{N}_X\right)_{N \in \mathbb{N}}$ is tight.
\end{lem}
\begin{proof}
Given $j<J$ the tightness of $(\P_{X_N^j})_{N\in \N}$ follows from the usual argument for Donsker's theorem. Here it is for completeness:
For $0 \leq s \leq t \leq 1$ and $N \in \N$ we have
\begin{align*}
	\mathbb{E}\left[\left|X_{t}^{N, j}-X_{s}^{N, j}\right|^{2} \mid \si_N\right] & =\frac{1}{N} \mathbb{E}\left[\left|\sum_{k=\lfloor N s\rfloor+1}^{\lfloor N t\rfloor} Z_{\si^j_N(k)}\right|^{2} \mid \si_N\right] \\
	& =\frac{1}{N} \mathbb{E}\left[\sum_{k=\lfloor N s\rfloor+1}^{\lfloor N t\rfloor} Z_{\si^j_N(k)}^{2} \mid \si_N\right] \\
	& =\frac{1}{N}(\lfloor N t\rfloor-\lfloor N s\rfloor) .
\end{align*}
Fix $0 \leq s \leq u \leq t \leq 1$. Note that $X^{N,j}$ has independent increments conditional on $\si_N$. Thus,
\begin{align*}
\mathbb{E}\left[|X_{u}^{N, j}-X_{s}^{N, j}|^{2}|X_{t}^{N, j}-X_{u}^{N, j}|^{2}|\si_N\right] & =\frac{1}{N^{2}}(\lfloor N u\rfloor-\lfloor N s\rfloor)(\lfloor N t\rfloor-\lfloor N u\rfloor) \\
& \leq\left(\frac{\lfloor N t\rfloor-\lfloor N s\rfloor}{N}\right)^{2}.	
\end{align*}
For $t-s \geq \frac{1}{N}$ the RHS is bounded by $4(t-s)^{2}$. Otherwise $\lfloor N u\rfloor=\lfloor N s\rfloor$ or $\lfloor N t\rfloor=\lfloor N u\rfloor$ and the LHS of the inequality vanishes. By Theorem 13.5. in \cite{billingsley1999convergence} (with $\alpha=\beta=1$ and $F(t)=2 t$) the sequence of projected measures $(\P_{X_N^j})_{N\in \N}$ on $\cD([0,1],\R)$ is tight, for all $j<J$.

Suppose $J < \infty$ and let $\ep > 0$. Then the tightness of $(\P_{X_N^j})_{N\in \N}$ for all $j < J$ implies there exist compacts $K_{1}, \ldots K_{J} \subseteq \cD([0,1], \mathbb{R})$, such that $\mu_{1}^{N}\left(K_{1}^{c}\right) \ldots \mu_{J}^{N}\left(K_{J}^{c}\right) \leq \ep / J$ for all $N \in \N$. By Tychonoff's theorem, the product $K_1\times \dots K_J$ is compact in $\cD([0,1], \R)^J$, and
\[\P_X^N((K_{1} \times \cdots \times K_{J})^c) \leq \sum_{j=1}^{J} \P_X^N(\mathbb{R} \times \cdots \times K_{j}^{c} \times \cdots \times \mathbb{R}) \leq \varepsilon\]
for all $N \in \mathbb{N}$. Since $\cD([0,1],\R)^J$ is homeomorphic to $\cD([0,1],\R^J)$, the sequence $\left(\P^{N}_X\right)_{N \in \mathbb{N}}$ is tight as well.

Now, suppose $J = \infty$. For every finite sequence $a \in \R^j$ we define the continuous function
\[\Pi_a : \cD([0,1],\R^\N) \to \cD([0,1], \R)\]
by
\[(\Pi_a x)_t = \sum_{k=1}^j a_k x^k_t, \quad t\geq 0.\]
and by abuse of notation we denote the map $\Pi_a : \cD([0,1],\R^j) \to \cD([0,1], \R)$ given by the same formula by the same symbol.
Given $a\in \R^j$ and $j\in \N$, we have $\Pi_a\P_X^N = \Pi_a \P_{X^{1\dots j}}^N$. Hence,
the family $(\Pi_a\P_X^N)_{N\in \N}$ consists of pushfowards of a tight family of measures along a continuous map. Therefore, it is tight as well. By Mitoma's criterion \citep[see][Theorem 4.1]{mitoma_tightness_1983} we conclude\footnote{For the application of their theorem, we let $E$ be the countable nuclear Hilbert space of finite sequences. Then its strong dual $E'$ is the space $\R^\N$ (up to linear homeomorphism)} that the family $(\P_X^N)_{N\in \N}$ is tight.
\end{proof}

To prove the convergence of the finite-dimensional distributions we use the following central limit theorem.
\begin{prop}
	\label{prop:CLTchandra}
	Let $(\xi_k^N)_{N\in \N, k\leq N}$ be a triangular array of centered $\R$-valued random variables. For every $N \in \N$ define
	\[V_N = \sum_{k=1}^N \Var(\xi_k^N).\]
	Assume 
	\begin{enumerate}[(a)]
		\item $|V_N| \to \infty$, as $N\to \infty$,
		\item $\E[\xi_{k}^N \sum_{l\neq k}^N \xi_{l}^N | \sum_{l\neq k}^N \xi_{l}^N]\geq 0, a.s.$, for all $k\leq N\in \N$,
		\item $\sum_{k,l} \Cov((\xi^N_{k})^2, (\xi^N_{l})^2) = o(|V_N|^2)$,
		\item $\sum_{k,l} \Cov((\xi^N_{k}), (\xi^N_{l})) = o(|V_N|)$.
	\end{enumerate}
	Then $V_N^{-1/2} \sum_{k=1}^N \xi_k^N \to \cN(0,1)$ in distribution, as $N\to \infty$.
\end{prop}
\begin{proof}
See \cite{chandrasekhar_general_2023} Corollary 1.
\end{proof}
In the proof of the next lemma we work with random arrays. Given $d,e\in \N^*, X : \Om \to \R^{\Pi d}$ and $Y: \Om \to \R^{\Pi e}$ we define the (cross-) covariance $\Cov(X,Y) \in \R^{(\Pi e)\times(\Pi d)}$ by
\[\Cov(X,Y)_{ij} = \Cov(X_i, Y_j), \quad i \leq d, j\leq e\]
provided $\E[|X_i Y_j|^2|]<\infty$ for all $i\leq |d|, j\leq |e|$.
In similar vein one defines conditional (cross-) covariance. Note that given $A \in \R^{\Pi(df)}$ and $B\in \R^{\Pi(eg)}$, we have
\begin{equation}
\label{eq:covOfFrobenius}
\Cov(\innp{X}{A}, \innp{Y}{B}) = \innp{\Cov(X,Y)}{A\otimes B} \in \R,
\end{equation}
provided all relevant terms are defined. In particular, if $X,Y$ are real-valued we have
\[\Cov(XA, YB) = \Cov(X,Y)(A\otimes B).\]
Analogous properties hold true for conditional covariance.
\begin{lem}
\label{lem:finitedimdist}
	For all $t_{1}, \ldots, t_{M} \in[0,1]$ and $j < J$ we have convergence
	\[\mat{X_{t_{1}}^{0} & \cdots & X_{t_{1}}^{j} \\
		\vdots & \ddots & \vdots \\
		X_{t_{M}}^{0} & \cdots & X_{t_{M}}^{j}} \Rightarrow
	\mat{W_{t_{1}}^{0} & \cdots & W_{t_{1}}^{j} \\
		\vdots & \ddots & \vdots \\
		W_{t_{M}}^{0} & \cdots & W_{t_{M}}^{j}}\]
	as $N \rightarrow \infty$.
\end{lem}

\begin{proof} Fix $t_{1}, \ldots, t_{M} \in[0,1]$. We fix $J\in \N$ and write $\si^1, \si^2,\dots, \si^J$ instead of $\si^0, \si^1,\dots, \si^{J-1}$ (and similarly for $X$ and $W$) in this proof to simplify the notation. 
By a variation of the Cramér-Wold theorem (cf.\ \cite{kallenberg2021foundations} Corollary 6.5) it suffices to prove
\[\frac{1}{\sqrt{N}} \sum_{m=1}^{M} \sum_{j=1}^{J} u_{m} v_{j} S_{\left\lfloor N t_{m}\right\rfloor}^{N, \si^{j}} \Rightarrow \sum_{m=1}^{M} \sum_{j=1}^{J} u_{m} v_{j} W_{t_{m}}^{j}\]
as $N \rightarrow \infty$, for all $u\in \R^M$ and $v \in \R^{J}$. To this end we want to apply Proposition \ref{prop:CLTchandra}.
 Fix $u, v \in \R^{J}$ and assume wlog $|u|=|v|=1$. We have
\begin{align*}
	\frac{1}{\sqrt{N}} \sum_{m=1}^{M} \sum_{j=1}^{J} u_{m} v_{j} S_{\left\lfloor N t_{m}\right\rfloor}^{N, \si^{j}} & =\frac{1}{\sqrt{N}} \sum_{k=1}^{N}\left(\sum_{m=1}^{M} \sum_{j=1}^{J} u_{m} v_{j} \idK_{\si^{j}_N[N t_{m}]}(k)\right) Z_{k} .
\end{align*}
Given $k\leq N\in \N$ define the matrix $A^N(k) \in \R^{M\times J}$ by
\[A_{m,j}^N(k) = \idK_{\si^{j}_N[N t_{m}]}(k), \quad m \leq M, j\leq J\]
 and set
\[\xi_k^N = \innp{A^N(k)}{u\otimes v} Z_k, \quad k\leq N \in \N,\]
so that we may write
\[\frac{1}{\sqrt{N}} \sum_{m=1}^{M} \sum_{j=1}^{J} u_{m} v_{j} S_{\left\lfloor N t_{m}\right\rfloor}^{N, \si^{j}} = \frac{1}{\sqrt N}\sum_{k=1}^N \xi_k^N.\]
By Equation \eqref{eq:covOfFrobenius}
\[\Cov(\xi_k^N, \xi_l^N) = \innp{\Cov(A^N(k)Z_k, A^N(l)Z_l)}{u\otimes v\otimes u \otimes v}\]
and by the law of total (cross-) covariance
\begin{align*}
\Cov(A^N(k)Z_k, A^N(l)Z_l) = & \E[\Cov(A^N(k)Z_k, A^N(l)Z_l|\si_N)] \\
& + \Cov(\E[A^N(k)Z_k|\si_N], \E[A^N(l)Z_l|\si_N]) \\
= & \E[A^N(k)^{\otimes 2}] \delt_{k,l}.
\end{align*}
Hence, the sequence $(\xi_k^N)_{k\leq N}$ is pairwise uncorrelated for all $N \in \N$. Moreover,
\begin{align*}
\E[A^N(k)^{\otimes 2}]_{m,j,m',j'} = &\E[\idK_{\si^j_N[N t_{m}]}(k) \idK_{\si^{j'}_{N}[N t_{m'}]}(k)] \\
= & \E[\idK_{\si^j_N[N t_{m}] \cap \si^{j'}_{N}[N t_{m'}]}(k)],
\end{align*}
and so
\begin{align*}
\Var\left(\frac{1}{\sqrt N}\sum_{k=1}^N \xi_k^N\right) = & \frac1N\sum_{k=1}^N \sum_{m,m',j,j'}  u_m u_{m'} v_j v_{j'} \E[\idK_{\si^j_N[N t_{m}] \cap \si^{j'}_{N}[N t_{m'}]}(k)]  \\
= & \sum_{m,m',j,j'} u_m u_{m'} v_j v_{j'} \E \hat F^{jj'}_N(t_m, t_m').
\end{align*}
By Corollary \ref{lem:empPermtonCDFconv} and since any family of joint CDFs is bounded by $1$
\begin{equation}
\label{eq:varConv}
\lim_{N\to \infty} \Var\left(\frac{1}{\sqrt N}\sum_{k=1}^N \xi_k^N\right) = \sum_{m,m',j,j'} u_m u_{m'} v_j v_{j'} F^{j,j'}_{\mu}(t_m, t_m').
\end{equation}
Using the identity
\[\innp{A}{B}\innp{C}{D} = \innp{A\otimes C}{B\otimes D}, \quad A,B,C,D\in \R^{\Pi d}\]
we further have
\begin{align*}
\Cov((\xi_k^N)^2, (\xi_l^N)^2) = & \Cov(\innp{A^N(k)^{\otimes 2} Z_k^2}{(u\otimes v)^{\otimes 2}},\innp{A^N(l)^{\otimes 2} Z_l^2}{(u\otimes v)^{\otimes 2}}) \\
= &\innp{\Cov(A^N(k)^{\otimes 2} Z_k^2, A^N(l)^{\otimes 2} Z_l^2)}{(u \otimes v)^{\otimes 4})},
\end{align*}
where
\begin{align*}
\Cov(A^N(k)^{\otimes 2} Z_k^2, A^N(l)^{\otimes 2} Z_l^2) = & \E[\Cov(A^N(k)^{\otimes 2} Z_k^2, A^N(l)^{\otimes 2} Z_l^2|\si_N)] \\
& + \Cov(\E[A^N(k)^{\otimes 2} Z_k^2|\si_N], \E[A^N(l)^{\otimes 2} Z_l^2|\si_N])\\
= & \E[\Cov(Z_k^2, Z_l^2) A^N(k)^{\otimes 2}\otimes A^N(l)^{\otimes 2}] \\
& + \Cov(A^N(k)^{\otimes 2}, A^N(l)^{\otimes 2})\\
= & \E[A^N(k)^{\otimes 4}] \E[Z_1^4] \delt_{k,l} + \Cov(A^N(k)^{\otimes 2}, A^N(l)^{\otimes 2}).
\end{align*}
Thus,
\begin{align}
\label{eq:estSumOfCovOfSq}
\sum_{k,l=1}^N \Cov((\xi_k^N)^2, (\xi_l^N)^2)=& \E[Z_1^4] \innp{\sum_{k=1}^N \E[ A^N(k)^{\otimes 4}]}{(u\otimes v)^{\otimes 4}} \nonumber \\
& + \innp{\Cov\left(\sum_{k=1}^N A^N(k)^{\otimes 2}, \sum_{l=1}^N A^N(l)^{\otimes 2}\right)}{(u\otimes v)^{\otimes 4}}.
\end{align}
Using the Cauchy-Schwarz inequality and the estimates
\begin{align*}
|A^N(k)| = & \left(\sum_{m,j=1}^{M,J} \idK_{\si^j_N[N t_{m}]}(k)\right)^{1/2} \leq \sqrt{MJ}, \\
|\E[A^N(k)^{\otimes 4}]|\leq & \E[|A^N(k)|^4] \leq M^2 J^2, \\
\end{align*}
we see that the first summand in \eqref{eq:estSumOfCovOfSq} is bounded by $\E[Z_1^4] N M^2J^2$. Regarding the second summand, observe that
\[\frac1N\sum_{k=1}^N A^N(k)^{\otimes 2}_{m,j+1,m',j'+1} = \frac1N|\si^j_N[Nt_m] \cap \si^{j'}_N[Nt_{m'}]| = \hat F_N^{jj'}(t_m, t_{m'}).\]
Therefore,
\begin{align*}
&\Cov\left(\sum_{k=1}^N A^N(k)^{\otimes 2}, \sum_{l=1}^N A^N(l)^{\otimes 2}\right)_{m_1,j_1,m_2,j_2,m_3,j_3,m_4,j_4} \\
= & N^2\Cov(\hat F_N^{j_1j_2}(t_{m_1}, t_{m_2}), \hat F_N^{j_3j_4}(t_{m_3}, t_{m_4})).
\end{align*}
Note that $\hat F_N^{ij}$ converges to a deterministic function in probability by Corollary \ref{lem:empPermtonCDFconv}, for all $i,j\in \N$. Thus, the covariance $\Cov(\hat F_N^{j_1j_2}(t_{m_1}, t_{m_2}), \hat F_N^{j_3j_4}(t_{m_3}, t_{m_4}))$ vanishes as $N\to \infty$ and we have
\[\lim_{N\to \infty} \frac{1}{N^2}\Cov\left(\sum_{k=1}^N A^N(k)^{\otimes 2}, \sum_{l=1}^N A^N(l)^{\otimes 2}\right) = 0,\]
Hence, we can conclude
\[\sum_{k,l} \Cov((\xi_k^N)^2, (\xi_l^N)^2) = o(N^2), \quad N\to \infty.\]
Finally, recall that if $X$ and $Y$ are random variables and $\cF$ a $\si$-algebra, such that $X$ is independent of $\si(Y,\cF)$, then
\[\E[XY|\cF] = \E[X]\E[Y|\cF], \text{a.s.}\]
Hence,
\[\E\left[\xi_k^N|\sum_{k\neq l}^N \xi_l^N\right] = \E\left[Z_k\innp{A^N(k)}{u\otimes v}|\sum_{k\neq l}^N \xi_l^N\right] = 0, \text{a.s.}\]
Thus, by Proposition \ref{prop:CLTchandra} we have 
\[\frac{\sum_{k=1}^N \xi_k^N}{\left(\Var\left(\sum_{k=1}^N \xi_k^N\right)\right)^{1/2}} \Rightarrow \cN(0, 1), \quad N\to \infty.\]
With \eqref{eq:varConv} we conclude further
\begin{align*}
\frac{1}{\sqrt N}\sum_{k=1}^N \xi_k^N = &  \left(\Var\left(\frac{1}{\sqrt N}\sum_{k=1}^N \xi_k^N\right)\right)^{1/2}\frac{\sum_{k=1}^N \xi_k^N}{\left(\Var\left(\sum_{k=1}^N \xi_k^N\right)\right)^{1/2}}\\
\to &\cN\left(0, \sum_{m,m',j,j'} u_m u_{m'} v_j v_{j'} F^{jj'}_\mu(t_m, t_m')\right)
\end{align*}
weakly, as $N\to \infty$.
On the other hand, we also have
\begin{align*}
\Var\left(\sum_{m=1}^M \sum_{j=1}^J u_m v_j W_{t_m}^j\right) = & \sum_{m,m',j,j'} u_m u_{m'} v_j v_{j'} \Cov(W_{t_m}^j, W_{t_{m'}}^{j'})\\
= & \sum_{m,m',j,j'} u_m u_{m'} v_j v_{j'} F^{jj'}_\mu(t_m, t_m').
\end{align*}
Consequently, the result follows.
\end{proof}

\begin{proof}[Proof of Theorem \ref{thm:shuffledDonsker}]
Lemmas \ref{lem:tightness} and \ref{lem:finitedimdist} imply the weak convergence of $X$ to $W$. The converse statement follows from Proposition \ref{prop:convRandSortPerm}.
\end{proof}

\section{Epoched Brownian processes as scaling limits}
Our main motivation for the theory in this chapter is to construct an analogue to Brownian motion in the approximation of SGD by an SDE for the finite-data without replacement case.
Specifically, SDEs driven by Brownian motion are weak approximation to one-pass SGD, where the data is an infinite \tiid\ sequence drawn from the population. If we are instead given a finite \tiid\ sequence, then the corresponding SDE driver is an \emph{epoched Brownian motion}. The significance of the limiting process $W$ in Theorem \ref{thm:shuffledDonsker} is that its components are essentially the epochs of an epoched Brownian motion. Thus, epoched Brownian motions arise as limits of random walks with finitely many increments. After being used up in the first epoch, the are used again, perhaps (randomly) permuted.

In this section we want to show how epoched Brownian bridges arise as scaling limits as an application of Theorem \ref{thm:shuffledDonsker}. More precisely, we want to prove the following statement.
\begin{satz}
\label{thm:epochedBBdonsker}
Suppose $(Z_n)_{n\in \N_0}$ is a sequence of \tiid\ random variables with $\E Z_n = 0, \Var Z_n = 1$ and $\E[Z_n^4] < \infty$. Further, suppose we are given a sequence $(\si_N)_{N\in \N}$ of random infinite-dimensional permutations, independent of $(Z_n)_{n\in \N_0}$, (where $\si_N^j$ is defined on the set $\set{0,\dots, N-1}$ for all $j\in \N$) such that \assref{assum:limitCopulas} holds true.
Define\footnote{We set $\sum_{k=a}^b \dots := 0$ for $a,b\in \IZ$ with $a > b$.}
\[\tilde X_t^N = \frac{1}{\sqrt N}\sum_{k=0}^{\floor{Nt}-1} Z_{\si_N^{\floor{k/N}}(k \modu N)} - \frac{t}{\sqrt N}\sum_{k=0}^{N-1} Z_k, \quad t\geq 0, N\in \N.\]
	Then there exists a jointly Gaussian family of Brownian bridges $(B^j : \Om \times [0,1]\to \R)_{j\in \N_0}$ from $0$ to $0$, with
	\[\Cov(B^i_s,B^j_t) = F_\mu^{ij}(s,t) - st,\quad i\neq j, s,t\in [0,1],\]
	such that $\tilde X$ converges in distribution to the centered Gaussian process $\tilde B : \Om \times [0,\infty) \to \R$ given by
	\[\tilde B_t = B_{\frk t}^{\floor t},\quad t\geq 0.\]
\end{satz}

Let $b > a > 0$. The $J_1$-metric, which induces the topology of the Skorokhod space $\cD([a,b], \R)$, is given by
\[d_{J_1}(f,g) = \inf_{\la} (\nrm{\la - \id{[a,b]}}{\infty} \vee \nrm{f - g\circ \la}\infty),\quad f,g\in \cD([a,b], \R),\]
where the infimum is taken over all homeomorphisms $\la : [a,b]\to [a,b]$ with $\la(a) = a$ and $\la(b) = b$. We call these \emph{time changes} for simplicity.
Thus, we can induce the topology on $\cD([0,1], \R^\N)$ by equipping this space with the metric
\[d_{\cD([0,1], \R^\N)}(f,g) = \sum_{j=0}^\infty 2^{-j}(1\wedge d_{J_1}(f^j,g^j)), \quad f,g\in \cD([0,1], \R^\N).\]
Denote the subspace of $\cD([a,b], \R)$ of all functions $f$ with $f(a) = f(b) = 0$ by $\cD_0([a,b])$. We call these functions \emph{\tCadl{} loops}.
We can equip $\cD_0([0,1])^\N$ with the metric sharing the same formula as before
\[d_{\cD_0([0,1])^\N}(f,g) = \sum_{j=0}^\infty 2^{-j}(1\wedge d_{J_1}(f^j,g^j)), \quad f,g\in D_0^\N.\]
For $M\in \bar \N$ we define $\Ph_M : \cD([0,1], \R^M) \to \cD_0([0,1])^M$ by
\begin{equation}
\label{eq:PhiM}
(\Ph_M f)(t) := f(t) - f(0)- tf(1), \quad f\in \cD([0,1], \R^M), t\geq 0,
\end{equation}
and write $\Ph := \Ph_\infty$.
\begin{lem}
\label{lem:toBridge}
The function $\Ph : \cD([0,1], \R^\N) \to \cD_0([0,1])^\N$ is Lipschitz.
\end{lem}
\begin{proof}
First, note that $\Ph f = (\Ph_1 f^j)_{j\in \N_0}$. Certainly, $\Ph_1 f^j$ is \tCadl{} if $f$ is, for all $j\in \N$. For any time change $\la : [0,1]\to [0,1]$ we have
\begin{align*}
|\Ph_1 f(t) -\Ph_1 g(\la(t))| 	= & |f(t) - f(0) - tf(t) - (g(\la(t)) - g(0) - tg(1))|\\
										\leq &|f(t) - g(\la(t))| + |f(0) - g(0)| + |f(1)-g(1)|\\ 
										\leq &3 d_{J_1}(f,g), \quad t\in [0,1].
\end{align*}
Hence, $\Ph_1$ is Lipschitz with constant $3$ with respect to the $d_{J_1}$ metric. Now, consider $\Ph$. We have
\begin{align*}
d_{\cD([0,1], \R^\N)}(\Ph f,\Ph g) = & \sum_{j=0}^\infty 2^{-j}(1\wedge d_{J_1}(\Ph_1 f^j,\Ph_1 g^j)) \\
\leq & 3 \sum_{j=0}^\infty 2^{-j}(1\wedge d_{J_1}(f^j,g^j)) \\
= & 3 d_{\cD_0([0,1])^\N}(f,g),
\end{align*}
for all $f,g\in \cD([0,1], \R^\N)$.
\end{proof}

\begin{lem}
\label{lem:skorokShift}
Let $a,b,c > 0$ with $a < b$, and define the shift operator 
\[T_c : \cD([a,b],\R) \to \cD([a+c,b+c], \R),\]
by
\[T_c f(t) = f(t-c), \quad t\in [a+c,b+c].\]
Then $T_c$ is an isometry, that is
\[d_{J_1}(f,g) = d_{J_1}(T_c f, T_c g), \quad f,g\in \cD([a,b], \R).\]
\end{lem}
\begin{proof}
Let $\la : [a,b]\to [a,b]$ be a time change. Define
\[\la_c : [a+c,b+c]\to [a+c,b+c], t\mapsto c + \la(t-c).\]
Then $\la_c$ well-defined, a homeomorphism, $\la_c(a+c) = a+c$ and $\la_c(b+c) = b+c$.
Then
\[\sup_{t\in [a+c,b+c]} |\la_c(t) - t| = \sup_{t\in [a,b]} |c + \la(t) - (t+c)| = \sup_{t\in [a,b]} |\la(t) - t|,\]
and for all $f,g\in \cD([a,b], \R)$,
\begin{align*}
\sup_{t\in [a+c,b+c]} |T_c f(t) - T_c g(\la_c(t))| = & \sup_{t\in [a+c,b+c]} |f(t-c) - g(c+ \la(t-c) - c)| \\
= & \sup_{t\in [a,b]} |f(t) - g(\la(t))|.
\end{align*}
Since $\la \mapsto \la_c$ is a bijection between the sets of time changes, taking infima over $\la$ and $\la_c$ yields the isometry property.
\end{proof}

Given $0 \leq a \leq b \leq c$ and functions $f : [a,b]\to \R, g : [b,c]\to \R$ with $f(b) = g(b)$, we define their \emph{concatenation} $f\ast g : [a,c]\to \R$ by
\[(f\ast g)(t) = \begin{cases}
	f(t), & t \in [a,b], \\
	g(t), & t\in (b,c].
\end{cases}\]
The condition $f(b) = g(b)$ ensures that $f\ast g$ is continuous at $b$.
\begin{lem}
\label{lem:skorokConcat}
Let $0 \leq a \leq b \leq c$. Then the concatenation operator for \tCadl{} loops
\[\ast : \cD_0([a,b]) \times \cD_0([b,c]) \to \cD_0([a,c]),\]
is continuous.
\end{lem}
\begin{proof}
Let $f\in \cD_0([a,b])$ and $g\in \cD_0([b,c])$.
Consider time changes $\la_f : [a,b]\to [a,b]$ and $\la_g : [b,c] \to [b,c]$. Then $\la_f \ast \la_g : [a,c]\to [a,c]$ is a time change as well (continuity is preserved since the endpoints are fixed). Suppose $(f_n)_{n\in \N}, (g_n)_{n\in \N}$ are sequences in $\cD_0([a,b])$ and $g\in \cD_0([b,c])$ respectively, such that $f_n \to f$ and $g_n \to g$, $n\to \infty$. Then
\begin{align*}
& \nrm{\la_f \ast \la_g - \id{[a,c]}}{\infty} \vee \nrm{f_n\ast g_n - (f\ast g)\circ (\la_f \ast \la_g)}{\infty} \\
\leq & \nrm{\la_f - \id{[a,b]}}{\infty} \vee \nrm{f_n - f\circ \la_f}{\infty} \vee \nrm{\la_g - \id{[b,c]}}{\infty} \vee \nrm{g_n - g\circ \la_g}{\infty}.
\end{align*}
Hence,
\[\lim_{n\to \infty} d_{J_1}(f_n\ast g_n, f\ast g) = 0,\]
as desired.
\end{proof}

\begin{lem}
\label{lem:toEpoched}
The function $\Psi : \cD_0([0,1])^\N \to \cD([0,\infty), \R)$ given by 
\[(\Psi f)(t) = f_{\frk t}^{\floor t}, \quad t\geq 0, f\in \cD_0([0,1])^\N\]
is continuous.
\end{lem}
\begin{proof}
Firstly, given $j\in \N_0$, define $\Psi_j : \cD_0([0,1])^j \to \cD([0,j],\R)$ by
\[(\Psi_j f)(t) = f_{\frk t}^{\floor t}, \quad t\in [0,j], f\in \cD_0([0,1])^j.\]
Then
\[\Psi_j = f^0 \ast T_1 f^1 \ast \dots \ast T_{j-1} f^{j-1},\]
and so $\Psi_j$ is continuous by Lemmas \ref{lem:skorokShift} and \ref{lem:skorokConcat}.

Consider $f\in \Om^\N$. Since $f^j$ is \tCadl{} for all $j\in \N_0$, so is $\Psi f$. No jumps occur at integer points $t\in \N$, since $f^j(0) = f^j(1) = 0$ for all $j\in \N$. Note that the topology on $\cD([0,\infty), \R)$ is induced by the metric
\[d_{\cD([0,\infty), \R)}(f,g) = \sum_{j=1}^\infty 2^{-j}(1\wedge d_{J_1}(f|_{[0,j]},g|_{[0,j]})), \quad f,g\in \cD([0,\infty), \R).\]
Note that $(\Psi f)|_{[0,j]} = \Psi_j f|_{[0,j]}, j \in \N_0$.
Let $\ep > 0$. There exists an $M\in \N$, such that $\sum_{j=M+1}^\infty 2^{-j} < \ep/2$. Recall the definition of $\Ph_M$ (Equation \eqref{eq:PhiM}). Since $\Ph_M$ is continuous, there exists a neighborhood $V \subseteq \cD_0([0,1])^M$ of $f|_{[0,M]}$, such that
\[d_{J_1}(f|_{[0,M]}, g) < \frac \ep 2, \quad g\in V.\]
Further, $V$ is the projection of a neighborhood $U\subseteq \cD_0([0,1])^\N$. Then, for all $g\in U$ we have
\[d(\Psi f, \Psi g) = \sum_{j=1}^\infty 2^{-j}(1\wedge d_{J_1}(\Psi_j f|_{[0,j]},\Psi_j g|_{[0,j]})) \leq \sum_{j=1}^M 2^{-j} \frac \ep 2 + \sum_{j=M+1}^\infty 2^{-j} < \ep.\]
Hence, $\Psi$ is continuous.
\end{proof}

\begin{proof}[Proof of Theorem \ref{thm:epochedBBdonsker}]
First, apply Theorem \ref{thm:shuffledDonsker} to get the convergence $X^N \to W, N\to \infty$ in distribution. Consider Lemma \ref{lem:toBridge}. We have
\begin{align*}
\Ph(X^N)_j(t) 	= & X^{N,j}_t - X_0^{N,j} - tX_1^{N,j} \\
				= & \frac1{\sqrt N} \sum_{k=0}^{\floor{Nt}-1} Z_{\si_N^j(k)} - \frac{t}{\sqrt N} \sum_{k=0}^{N-1} Z_{\si^j_N(k)} \\
				= & \frac1{\sqrt N} \sum_{k=0}^{\floor{Nt}-1} Z_{\si_N^j(k)} - \frac{t}{\sqrt N} \sum_{k=0}^{N-1} Z_k, \quad j\in \N_0, t\in [0,1],
\end{align*}
Further,
\[\Ph(W)_j(t) = W_t^j - t W_1^j, \quad j\in \N_0, t\in [0,1].\]
defines a Brownian bridge from $0$ to $0$.
The continuous mapping theorem applied with the continuous function from Lemma \ref{lem:toBridge} implies
\[\Ph(X^N) \to \Ph(W) =: B,\]
in distribution, as $N\to \infty$, where $B = (B^j : \Om \times [0,1]\to \R)$ is a jointly Gaussian family of Brownian bridges.
Next, we calculate
\begin{align*}
\sqrt N(\Psi \circ \Ph)(X^N)(t) = & \sqrt N \Ph(X^N)_{\floor t}(\frk t)\\
 = &  \sum_{k=0}^{\floor{N\frk t}-1} Z_{\si_N^{\floor t}(k)} - \frk t\sum_{k=0}^{N-1} Z_k \\
= &  \sum_{k=N\floor t}^{\floor{N\frk t}-1 + N\floor t} Z_{\si_N^{\floor t}(k - N\floor t)} - \frk t\sum_{k=0}^{N-1} Z_k \quad t\in [0,1], N\in \N, 
\end{align*}
Using the identity $\floor{Nt} = \floor{N\frk t} + N\floor t$ for all $N\in \N$ and $t\geq 0$, we further obtain
\begin{align*}
\sqrt N(\Psi \circ \Ph)(X^N)(t) = &  \sum_{k=N\floor t}^{\floor{N t}-1 } Z_{\si_N^{\floor{k/N}}(k\modu N)} + \sum_{k=0}^{N\floor t-1 } Z_{\si_N^{\floor{k/N}}(k\modu N)} \\
& -  \sum_{k=0}^{N\floor t-1 } Z_{\si_N^{\floor{k/N}}(k\modu N)} - \frk t \sum_{k=0}^{N-1} Z_k \\
= & \sum_{k=0}^{\floor{Nt}-1} Z_{\si_N^{\floor{k/N}}(k \modu N)} - \floor t \sum_{k=0}^{N-1} Z_k - \frk t\sum_{k=0}^{N-1} Z_k\\
= & \sum_{k=0}^{\floor{Nt}-1} Z_{\si_N^{\floor{k/N}}(k \modu N)} - t\sum_{k=0}^{N-1} Z_k, \quad t\in [0,1], N\in \N.
\end{align*}
On the other hand,
\[(\Psi \circ \Ph)(W)(t) = \Psi(B)(t) = B_{\frk t}^{\floor t}, \quad t\geq 0.\]
Applying the continuous mapping theorem for the continuous function from Lemma \ref{lem:toEpoched} yields
\[(\Psi \circ \Ph)(X^N) \to \Psi(B),\]
in distribution, as $N\to \infty$.
\end{proof}

\begin{docu}
\paragraph{An inequality for copulas}
This is likely not new considering $C$-Brownian bridges are a known concept.
\begin{prop}
Let $d\in \N$, $C$ be a $d$-copula and $t_1\leq \dots \leq t_{d+1}$ be elements of $[0,1]$. Then
\[\det\mat{t_1 & C^{1,2}(t_1, t_2) & \dots & C^{1,d+1}(t_1,t_{d+1})\\
C^{1,2}(t_1,t_2) & t_2 & \ddots & \vdots\\
\vdots & \ddots & \ddots & \vdots \\
C^{d+1,1}(t_1,t_{d+1}) & \dots & \dots & t_{d+1}
}
\geq 0,\]
that is 
\[\sum_{\si \in \cA_{d+1}} \prod_{k=1}^{d+1}C^{k, \si(k)}(t_k, t_{\si(k)}) \geq \sum_{\si \in \Sym{d+1} \setminus \cA_{d+1}} \prod_{k=1}^{d+1} C^{k, \si(k)}(t_k, t_{\si(k)}),\]
where $\cA_{d+1}$ is the alternating subgroup of $\cS_{d+1}$ containing the permutations with sign $1$.
\end{prop}
\end{docu}

\chapter{Appendix}
\section{A Remark on Kurtosis}
\label{sec:kurtosis}
The \emph{kurtosis} of a distribution is its standardized fourth central moment. That is, given a random variable $Z$ with $\E Z^4 < \infty$ it is defined by
\[\Kurt Z = \frac{\E[(Z - \E[Z])^4]}{(\Var Z)^2}.\]
Note that $\Kurt Z \geq 1$ by Jensen's inequality. Further, kurtosis is invariant under affine transformations, \tIe
\[\Kurt(aZ + b) = \Kurt(Z).\]
This property is of great importance in regards to machine learning, because this means that the typical pre-processing steps of centering and dividing by the standard deviation do not affect the kurtosis of the features (or labels).
In other words, the presence of $\Kurt \bx$ in the expression for $\Sigma(\theta)$ cannot be explained away by a standardization of $\bx$.

For convenience, here is a list of common distributions and their kurtosises.
\begin{center}
	\begin{tabular}{c|c|c|c|c|c|c}
		Dist.& $\operatorname{Exp}(\la)$ & $\operatorname{Poi}(\la)$ & $\chi^2_n$ & $\cN(\mu,\si^2)$  & $\cU[a,b]$ & $\operatorname{Lognormal}(\mu,\si^2)$\\
		\hline
		Kurt. & $9$  & $3+\frac1\la$  & $3+ \frac{12}{n}$  & $3$  & $\frac95$  & $e^{4\si^2} + 3e^{3\si^2} + 3e^{2\si^2} - 3$ \\
	\end{tabular}
\end{center}
Further, if $p\in [0,1]$ and $Z\sim \Bin(1,p)$, then
\[\Kurt Z = \frac{3p^2 - 3p + 1}{p(1-p)}\]
which has minimum $1$ at $\frac12$. That is, a symmetric Bernoulli attains the smallest possible Kurtosis of $1$. 

If $\Kurt Z = 3$, then we say $Z$ (or its distribution) is \emph{mesokurtic}. If $\Kurt Z > 3$, then $Z$ is called \emph{platykurtic} and we call $Z$ \emph{leptokurtic} for $\Kurt Z < 3$. These terms also delineate the settings for the error expansions in Section \ref{sec:batchSize1}.

Finally, we remark that the common interpretation of kurtosis as heaviness of the tails of a distribution is somewhat misleading. Let us suppose the distribution of $Z$ is unimodal, for simplicity. Then, according to \citet{Balanda1988}, kurtosis is \enquote{vaguely [...] the location- and scale-free movement of probability mass from the shoulders of a distribution into its center and tails [...]}, \tIe{} higher kurtosis implies \emph{both} higher peakedness as well as heavier tails. The term \emph{shoulders} refers roughly to the area between the tails and the center.
For multimodal distributions, the interpretation of kurtosis is a lot more involved or perhaps not even well understood. We will restrict our attention to unimodal distributions only (which includes all previous examples).

\section{Extension of $\dC l$ maps}
\label{sec:clext}
Consider bounded intervals $I_1,\dots, I_m, \Thet = I_1\times \dots I_m \times \R^{d-m}$ and a \tFre{} space $F$.
In this Section we demonstrate why functions $f\in \dC l(\Thet, F)$ space can be smoothly extended to an open set containing $\Thet$.

\begin{lem}
\label{lem:FrechetSmoothness}
Let $F$ be a \tFre{} space, $d\in \N, U \subseteq \R^d$ be open, and $f : U\to F$ a function. Then $f\in \dC l(U, F)$ if and only if $\ell \circ f \in \dC l(U,\R)$ for all continuous linear functionals $\ell \in F'$.
\end{lem}
\begin{proof}
	Omitted.
\end{proof}

\begin{lem}
	\label{lem:FrechetExtension}
	Let $(F, (\nrm{\blnk}{p})_{p\in \N})$ be graded \tFre{} space, $l,m,d\in \N, I_1,\dots, I_m$ be bounded intervals, and define
	\[\Theta = I_1\times\cdots\times I_m\times \R^{d-m} \subset \R^d.\]
	Let $f: \Theta \to F \in \dC l(\Thet, F)$.
	Then there exist an open set $U\subset\R^d$ with $\Theta\subset U$ and a map $\tilde f\in \dC l(U,F)$ such that $\tilde f|_{\Thet} = f$.
\end{lem}

\begin{proof}[Proof sketch]
We first treat the case
\[\Theta = [0,\infty) \times \R^{d-1},\]
Given $a\in \N_0$, \citet[Theorem 5.19]{adams2003sobolev} construct a linear extension operator
\[\tilde \cE_a : W^{l,p}(\Thet, \R) \to W^{l,p}(\R^d)\]
using the reflection formula
\[(\tilde \cE_a g)(x):=
\begin{cases}
	g(x), & x> 0,\\
\sum_{j=0}^{l+1} (-1)^a \la_j g(-j, x_1, x_2,\dots, x_n), & x<0,
\end{cases}\]
for suitably chosen constants $(\la_j)_{j=0,\dots,{l+1}}$ so that the derivatives up to order $l$ match at $0$. Write $\tilde \cE = \tilde \cE_0$.
They show that given $g\in \dC l(\Thet)$ we have
\begin{equation}
\label{eq:adamsprop}
\cE g \in \dC l(\R^d) \text{ and } \der^\al \cE g = \cE_{\al(1)} \der^\al g, \quad |\al|\leq l.
\end{equation}


Now, let $f:\Theta\to F\in \dC l(\Thet, F)$. Note that by definition $\der^\al f$ extends continuously to $\Theta$ for $|\al|\leq k$.
Define $\cE f$ be the same reflection formula as above, i.e.\
\[(\cE f)(x) :=
\begin{cases}
	f(x), & x> 0,\\
	\sum_{j=0}^{l+1} \la_j f(-j, x_1, x_2,\dots, x_n), & x<0,
\end{cases}\]
for suitably chosen constants $(\la_j)_{j=0,\dots,{l+1}}$.
For $\ell \in F'$, we have
\[(\ell\circ \cE f)(x) =
\begin{cases}
	(\ell \circ f)(x), & x> 0,\\
	\sum_{j=0}^{l+1} \la_j (\ell \circ f)(-j, x_1, x_2,\dots, x_n), & x<0,
\end{cases}\]
that is $\ell \circ \cE f = \tilde \cE(\ell \circ f)$, where $\tilde \cE$ is the extension operator for $\R$-valued functions.

Note that $\ell \circ f\in \dC l (\Theta,\R)$ by Lemma \ref{lem:FrechetSmoothness}. Property \eqref{eq:adamsprop} implies $\tilde \cE(\ell \circ f)\in \dC l(\R^d, \R)$. Since $\ell \in F'$ was arbitrary, we conclude $\cE f\in \dC l(\R^d, F)$ by Lemma \ref{lem:FrechetSmoothness}.

The case $\Thet = [a,b] \times \R^{d-1}$ whith $a\leq b$ can be reduced to the case $\Thet = [0,\infty) \times \R^{d-1}$ using a smooth partition of unity of $\R$. By iterating this construction we can treat the case $\Theta = I_1\times\cdots\times I_m\times \R^{d-m}$ for closed intervals $I_1,\dots, I_m$. 
	
	Finally, consider
$\Theta = I_1\times\cdots\times I_m\times \R^{d-m} \subset \R^d$ where $I_1,\dots, I_m$ are arbitrary bounded intervals. We can enlarge $\Thet$ to the closure
\[\overline \Thet = \overline{I_1}\times\cdots\times \overline{I_m}\times \R^{d-m}.\]
Since $f$ and its derivatives can be continuously and uniquely extended to $\overline \Thet$, we can use the extension property for closed intervals we have deduced.
Thus, there exists an open set $U\subseteq \R^d$ with $\overline \Thet \subseteq U$ and a function $\tilde f\in \dC l(U,F)$ which restricts to this extension of $f$ on $\overline{\Thet}$. In particular, $\tilde f|_{\Thet} = f$.
\end{proof}

\bibliographystyle{abbrvnat}
\bibliography{thesisbib}
\addcontentsline{toc}{chapter}{References}

\cleardoublepage


\end{document}